\documentclass[reqno,11pt]{amsart}
\usepackage{changes}
\usepackage{lipsum}
\usepackage{bm,cancel,ulem}
\usepackage{amssymb,amsfonts}
\usepackage{amsmath,latexsym}
\usepackage{xcolor}
\usepackage{pdfsync}
\usepackage{mathrsfs}
\usepackage[numbers,sort&compress]{natbib}

\usepackage[colorlinks=true]{hyperref}
\hypersetup{
citecolor=[HTML]{006666},
linkcolor=[HTML]{8B0000},
urlcolor =[HTML]{800000}
}
\newtheorem{thm}{Theorem}[section]
\newtheorem{cor}[thm]{Corollary}
\newtheorem{lemma}[thm]{Lemma}
\newtheorem{prop}[thm]{Proposition}
\theoremstyle{definition}
\newtheorem{defn}[thm]{Definition}
\newtheorem{rem}[thm]{Remark}
\numberwithin{equation}{section}
\newtheorem{claim}[thm]{Claim}

\newcommand{\mR}{\mathbb R}

\newcommand{\dx}{\,{\rm d}x}

\newcommand{\bn}{\bar{\nabla}}

\newcommand{\al}{\alpha}
\newcommand{\ep}{\epsilon}
\newcommand{\f}{\frac}
\newcommand{\q}{\quad}
\newcommand{\na}{\nabla}
\newcommand{\un}{\underbrace}

\renewcommand{\leq}{\leqslant}
\renewcommand{\geq}{\geqslant}
\renewcommand{\d}{\,{\rm d}}

\def\XXint#1#2#3{{\setbox0=\hbox{$#1{#2#3}{\int}$}	\vcenter{\hbox{$#2#3$}}\kern-.5\wd0}}

\newcommand{\les}{\lesssim}
\newcommand{\p}{\partial}

\def\be{\begin{equation}}
\def\ee{\end{equation}}
\def\bes{\begin{equation*}}
\def\ees{\end{equation*}}
\def\bs{\begin{split}}
\def\es{\end{split}}
\def\bali{\begin{aligned}}
\def\eali{\end{aligned}}

\def\bR{{\mathbb R}}

\def\mR{\mathbb{R}}

\def\un{\underbrace}
\def\ol{\overline}
\def\al{\alpha}

\def\la{\lambda}

\def\t{\tilde}
\def\wt{\widetilde}
\def\th{\theta}

\def\G{\Gamma}

\def\Dl{\Delta}

\def\ls{\lesssim}
\def\gs{\gtrsim}

\def\i{\infty}
\def\p{\partial}
\def\f{\frac}
\def\na{\nabla}
\def\o{\omega}
\def\O{\Omega}

\def\q{\quad}
\def\qq{\qquad}

\def\m{\mathcal}
\def\mK{\mathcal{K}}

\allowdisplaybreaks
\def\cd{\cdot}
\def\les{\lesssim}

\def\bm{\boldsymbol}
\def\dx{\,\mathrm{d}x}
\def\d{\,\mathrm{d}}
\def\ba{\begin{equation}\begin{aligned}}
	\def\ea{\end{aligned}\end{equation}}
\def\bn{\[\begin{aligned}}
\def\en{\end{aligned}\]}
\def\l{\label}
\def\ed{\buildrel\hbox{\footnotesize def}\over =}
\def\ef{\eqref}
\def\t{\tilde}
\def\pr{^\prime}
\def\mF{\m{F}}
\def\mO{\m{O}}
\def\mK{\m{K}}
\def\pr{\prime}
\DeclareRobustCommand{\blll}[1]{\text{\boldmath$#1$}}
\begin{document}


\title[NS flows with Navier total-slip boundary]{On the axisymmetric Navier-Stokes flow passing a cone with the total-slip boundary condition}%

\author[Z. Li]{Zijin Li}
\address[Z. Li]
{School of Mathematics and Statistics, Nanjing University of Information Science and Technology, Nanjing
	210044, China}
\email{zijinli@nuist.edu.cn}

\author[X. Yang]{Xin Yang}
\address[X. Yang]
{School of Mathematics, Southeast University, Nanjing, 211189, China; }
\email{xinyang@seu.edu.cn}

\author[Q. Zhang]{Qi S. Zhang}
\address[Q. Zhang]
{Department of mathematics, University of California, Riverside, CA 92521, USA;}
\email{qizhang@math.ucr.edu}

\date{\today}		
%

\begin{abstract}
	(A) It is known that among the currently unresolved cases of the axially symmetric Navier-Stokes equations (ASNS), the most relatively tractable one is where the fluid passes the exterior of a cone.
	In this paper, we investigate this case with Navier total-slip boundary condition. We show that there exists an absolute constant $C_* > 0$ 
	such that if
	\[
	\sup_{x\in D}r|v_{0,\th}(x)|\leq C_* \quad\text{and}\quad \int_{D} r v_{0,\theta}(x) \d x = 0,
	\]
	then there exists a unique global bounded strong solution with finite energy. Note that, for the initial velocity, there is neither a size restriction on other components, nor a parity assumption. There are four key ingredients in the proof.
	\begin{itemize}
		\item[(1)] Three new good unknowns are introduced, and a self-closed energy estimate for them is derived.
		
		\item[(2)] An elliptic estimate for pressure is established to control boundary terms arising from the boundary condition.
		
		\item[(3)] A De Giorgi iteration scheme is applied to establish the boundedness of $rv_\th$.
		
		\item[(4)] A new anisotropic Hardy's inequality is derived for weighted mean-zero functions to overcome the lack of parity of $\bm{v}$.
	\end{itemize}
	
	(B) Based on (A), we introduce and prove the \emph{controlled regularity} for the above problem, i.e. for suitable initial data without any smallness assumption, there exists an external force supported away from the axis of symmetry such that the corresponding problem admits a global strong solution. This seems to add a little weight to the regularity scenario for ASNS, since the force is supported away from the axis which is the only place regularity may break down. We also prove that  if there exists a solution that blows up in finite time, an unstable blow-up solution must exist.
\end{abstract}

\maketitle
%
%
%

\vskip 0.2cm

\tableofcontents

\section{Introduction}
\subsection{The problem and related works}
The goal of the paper is to construct a class of global bounded solutions to the axially symmetric Navier-Stokes equations, abbreviated as ASNS henceforth:
\be\label{eqasns}
\left\{
\begin{aligned}
	&\Big( \Delta-\frac{1}{r^2} \Big)
	v_r - (v_r \p_r + v_3 \p_{x_3}) v_r + \frac{v_{\theta}^2}{r} - \partial_r P - \p_t  v_r = 0,  \\
	&\Big(\Delta-\frac{1}{r^2} \Big) v_{\theta} - (v_r \p_r + v_3 \p_{x_3} )v_{\theta} - \frac{v_{\theta} v_r}{r} - \partial_t v_{\theta} = 0,\\
	&\Delta v_3 - (v_r \p_r + v_3 \p_{x_3})v_3 - \p_{x_3} P - \p_t v_3 = 0,\\
	&\frac{1}{r} \p_r (rv_r) +\p_{x_3}v_3 = 0.
\end{aligned}
\right.
\ee
Here,  $\bm{v} = v_{r}\bm{e_{r}} + v_{\th}\bm{e_{\th}} + v_{3}\bm{e_{3}}$ is the velocity in the cylindrical system with the standard basis $\{\bm{e_{r}}, \bm{e_{\th}}, \bm{e_{3}}\}$, where for any $ x=(x_1,x_2,x_3)\in\bR^3 $, $ r=\sqrt{x_1^2+x_2^2}$, $\theta=\arctan\f{x_2}{x_1}$ and
\bn
\bm{e_r}=(x_1/r, x_2/r, 0)^T,\quad \bm{e_\th}=(-x_2/r, x_1/r, 0)^T,\quad \bm{e_3}=(0,0,1)^T.
\en
The characterization of ASNS is that the components $v_{r}$, $v_{\th}$ and $v_{3}$ are independent of the azimuthal angle $\th$. Although ASNS is a special case of the full 3D Navier-Stokes equations,
\be
\label{nse}
\Delta \bm{v} -  (\bm{v}\cdot \nabla) \bm{v} - \nabla P -\partial_t \bm{v} =0, \quad \text{div} \, \bm{v}=0,
\ee
the regularity problem of ASNS is still open in general settings. In the last several decades, there has been an outburst of research on ASNS,  see e.g. \cite{La, UY, CSTY1, CSTY2, KNSS, HLL, CFZ, LZ17, Weid, Zha22, CLZJMPA} and the references therein.  

After it was realized in \cite{LZ17, CFZ} that ASNS is essentially a critical system, there is expectation that the regularity problem is more promising.
In this direction, the regularity problem was solved in \cite{Zha22} for a cusp domain under the Navier total-slip boundary condition. 
This is the first time that the regularity problem of ASNS is settled when the essential difficulty is beyond that in 2D.  
Earlier, the regularity problem of the 3D Navier-Stokes equations is also solved in \cite{MTL90} for solutions with helical symmetry. 
Such an assumption makes the classical 2D Ladyzhenskaya's inequality available in 3D, therefore the fundamental obstacle of the 3D regularity problem is absent in this situation.
One may feel that the cusp domain in \cite{Zha22} is too special. Subsequently, the authors in \cite{LPYZZZ24} studied a more realistic domain 
(see the blue region in Figure \ref{Fig,domain-cyl} or Definition \ref{Def, domain}). By only restricting the swirl component of the initial velocity to be small, \cite{LPYZZZ24} justified the global existence of bounded axially symmetric solutions to the Navier-Stokes equations.

\begin{figure}[!ht]
	\includegraphics[scale=0.22]{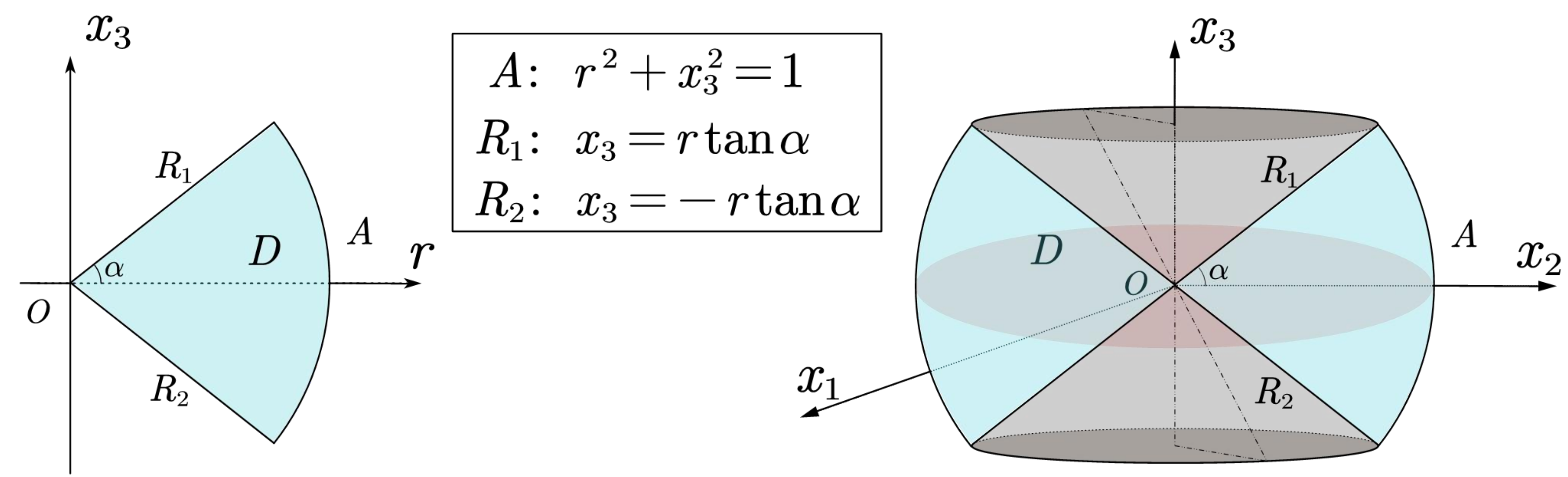}
	\caption{Domain $D$ in cylindrical coordinates}
	\label{Fig,domain-cyl}
\end{figure}

\begin{defn}\label{Def, domain}
	Let $\al\in\big(0,\frac{\pi}{2}\big)$ be any fixed angle.  The domain $D$ with boundary surfaces $R_1$, $R_2$ and $A$ is defined in the cylindrical coordinates as follows  (also see Figure \ref{Fig,domain-cyl}):
	\be\label{domain-cyl}
	D=\big\{(r,\theta,x_3): 0<r^2+x_3^2<1, \, -r\tan\al <x_3< r\tan\al, \, \theta\in [0,2\pi) \big\}.
	\ee
	Moreover, for convenience of notation, we denote
	\[\p^{R}D = R_{1} \cup R_2,\quad \p^{A}D = A,\]
	where the superscripts $R$ and $A$ stand for the radial boundary and the annular boundary respectively.
\end{defn}

The boundary condition adopted in \cite{LPYZZZ24} was the Navier-Hodge-Lions boundary condition:
\be\label{NHL slip bdry}
\bm{v}\cdot \bm{n}=0, \quad \bm{\o}\times \bm{n}=0, \quad \text{on}\quad \p D,\ee
where $\bm{n}$ is the outward unit normal direction on the smooth part of $\p D$, and $\bm{\o}$ is the vorticity defined as
\bn
\bm{\o} =\text{curl $\bm{v}$} =\nabla\times \bm{v}. \en
This boundary condition was abbreviated as the NHL boundary condition in \cite{LPYZZZ24} and it was also named as the vorticity slip boundary condition in some literature. 

The goal of this paper is to study the ASNS on domains $\t{D}$ defined in (\ref{domain_general}), which are more general than the domains $D$ in (\ref{domain-cyl}), with the Navier total-slip boundary condition defined as below:
\be\label{NTS bdry}
\bm{v} \cdot \bm{n}=0, \quad (\mathbb{S}\bm{v}\cdot\bm{n})_{\tan }=0, \quad \text{on}\quad \p D;
\ee
see Theorem \ref{Thm_main} and Remark \ref{Remark, gen_domain} in the sequel.
For convenience of statement, the rest of this paper will still focus on the symmetric domains $D$, but we point out that all the results obtained in this paper are also valid on domains $\t{D}$ (\ref{domain_general}) since the symmetry of the domains $D$ is not used in any proof in this paper.

Condition (\ref{NTS bdry}) is a special case in a family of boundary conditions proposed by Navier \cite{Nav}:
\[
\bm{v} \cdot \bm{n}=0, \quad (\mathbb{S}\bm{v}\cdot\bm{n})_{\tan }+\beta\bm{v}_{\tan}=0, \quad \text{on}\quad \p D\,.
\]
This condition has been studied extensively in the literature and was attributed to different authors, see e.g. \cite{WatanabeJCAM, AACG, MuchaAAM, Kel06, LPYSCM, LLZJDE, CQ10, MR}, where $\beta\geq0$ is the slip ratio. Here the case $\beta>0$ (``non-total slip'') is often technically easier than the total-slip limit case $\beta=0$ because \(\beta>0\) provides boundary dissipation. In the basic energy identity, one gains the coercive term
$\beta\|\bm{v}_\tau\|_{L^2(\p_D)}^2$, which can absorb boundary integrals created via integration by parts. However, this damping disappears when \(\beta=0\). Then boundary terms involving tangential velocity or swirl flux (e.g., Robin-type traces for \(v_\theta\) induced by \((\mathbb{S}\bm{v}\cdot\bm{n})_{\tan }=0\)) generally have bad sign in our case and cannot be controlled solely by the interior dissipation, especially near non-smooth geometries such as cones where trace/curvature effects are delicate. This makes closing energy inequalities harder, and is a primary reason why we initially focused on the NHL boundary condition rather than the Navier total-slip boundary condition in \cite{LPYZZZ24}.

Due to Leray \cite{Le2}, if $D=\bR^3$ and $v_0 \in L^2(\bR^3)$, then the Cauchy problem (\ref{nse}) has a weak solution in the energy space (c.f. \ref{enorm} below). By finite energy, we mean the solutions are in the energy space $\mathbf{E}=L^2_t H^{1}_x \cap L^\infty_t L^2_x$.  Here and throughout, the norm in $\mathbf{E}$ for a function $v$ on $D \times [0, T]$ is taken as
\be
\label{enorm}
\Vert \bm{v} \Vert^2_\mathbf{E} = \int^T_0 \int_D |\nabla \bm{v}|^2 \d x \d t + \sup_{t \in [0, T]} \int_D |\bm{v}(x, t)|^2 \d x \,.
\ee
Here, $ T>0 $ and the function $\bm{v}$ can be  vector-valued or scalar-valued, depending on the context. The solutions with finite energy include the so-called Leray-Hopf solutions which need to satisfy an energy inequality. In general, it is not known if Leray-Hopf solutions stay bounded or regular for all $t>0$. Recently, by allowing a super-critical forcing term in (\ref{nse}), it was shown in \cite{ABC22} that Leray-Hopf solutions may not be unique 
even with zero initial value and identical forcing term.

In this paper, we will focus on a special case of \eqref{nse}, namely when $\bm{v}$ and $P$ are independent of the azimuthal angle $ \th $ in the cylindrical coordinate system $(r,\,\th,\,x_3)$. Although ASNS (\ref{eqasns}) seems more complicated than the full 3D equation (\ref{nse}), a simplification happens in (\ref{eqasns})$_{2}$ where the pressure term disappears. For a succinct derivation of the ASNS (\ref{eqasns}) using the tensor notations, we refer the readers to \cite{Zha22}. If the swirl $v_\theta=0$, then it is well-known that finite energy solutions to the Cauchy problem of (\ref{eqasns}) in $\mathbb{R}^3$ are smooth for all time $t>0$, see e.g. \cite{La, UY, LMNP}. In the presence of swirl, it is still not known in general whether finite energy solutions blow up in finite time.

By the partial regularity result in \cite{CKN}, possible singularity for suitable weak solutions of ASNS can only appear at the $x_3$ axis. See also \cite{Linf} for a simplified proof and \cite{BuZh} for  the same statement but without the "suitable" requirement. Moreover, in \cite{CSTY1, CSTY2, KNSS, SS}, it was shown that if
\be
\label{v<1/r}
|\bm{v}(x, t)| \le
\frac{C}{r},
\ee
where $C$ is any positive constant, then finite energy solutions to the Cauchy problem of ASNS are smooth for all time. Later, there are some logarithmic improvements on the order of the criterion (\ref{v<1/r}), see e.g. \cite{Panx, Ser22a, Ser22b, CTZarxiv}. Also see \cite{Tao21} for a similar improvement in full 3D Navier-Stokes equations. In contrast, the energy bound scales as
$-1/2$. So even with axial symmetry, there is a finite scaling gap which makes the ASNS supercritical, just like the full 3D equations. In \cite{CFZ, LZ17},  the authors revealed that the vortex stretching term of the ASNS is critical after a  suitable change of dependent variables. Thus, the aforementioned scaling gap is zero, which makes the regularity problem of ASNS appear less formidable. Nevertheless, all major open problems are still open.

The main result in \cite{LZ17} includes the following statement.
Let $\delta_0 \in (0, \frac{1}{2})$ and $C_{*} > 1$.
If
\begin{equation}\label{CD}
	\sup_{0 \leq t < T}|r v_\theta (r, x_3, t)| \leq C_{*} |\ln r|^{- 2},\ \ r \leq \delta_0,
\end{equation}
then the velocity $\bm{v}$ is regular globally in time. Noting that a priori we have $|r v_\theta (r, x_3, t)| \leq C$ by the maximal principle applied on equation
\eqref{eqvth} of $ \Gamma $:
\be
\label{eqvth}
\Delta \Gamma - \bm{b} \cdot \nabla \Gamma- \frac{2}{r} \p_r
\Gamma-\p_t \Gamma=0,
\ee
where $\Gamma= r v_\theta$ and $\bm{b}=v_r \bm{e_r} + v_3 \bm{e_3}$, so there is still a gap of logarithmic nature from regularity. Later, the power index $ -2 $ in (\ref{CD}) was improved to $ -\frac32 $ in  \cite{Weid}.

In \cite{LPYZZZ24}, by requiring the angle $\al\in(0,\frac{\pi}{6}]$ in the domain $D$, as defined in  (\ref{domain-cyl}), and restricting 
\[ \sup_{x\in D} |\Gamma(x,0)| \leq \frac{1}{100},\]
the authors demonstrated the solution $\bm{v}$ to the Navier-Stokes equations with the NHL boundary condition and with certain symmetry is bounded and regular globally in time. However, they were not able to treat the more natural Navier total-slip (NTS) condition due to the difficulty induced by the boundary conditions. Moreover, the symmetry assumption of the velocity in \cite{LPYZZZ24} forces the domains to be symmetric, which undermines the applicability of the theory in reality.

Now we specify the meaning of solutions to ASNS (\ref{eqasns}) associated with the Navier total-slip (NTS) boundary condition (\ref{NTS bdry}). In the rest of this paper, functions and vector fields are always assumed to be axially symmetric with respect to the $ x_3 $-axis unless stated otherwise. Fix any $ T>0 $ and any $ \bm{v}_0\in H^2 (D) $ which is divergence free in $ D $ and satisfies the NTS boundary condition (\ref{NTS bdry}). Consider
\be\label{NS1} \left\{\, \begin{aligned}
	&\Delta \bm{v} - (\bm{v}\cdot \nabla) \bm{v} - \nabla P - \p_{t} \bm{v} = 0 \quad\text{in}\quad  D\times (0,T], \\
	&\nabla \cdot \bm{v} = 0  \quad \text{in} \quad  D\times (0,T], \\
	&\bm{v}\cdot \bm{n} = 0,\quad (\mathbb{S}\bm{v}\cdot\bm{n})_{\tan } = 0 \quad\text{on} \quad  \p D\times (0,T],\\
	&\bm{v}(\cdot, 0) = \bm{v}_0(\cdot) \quad\text{in} \quad  D.
\end{aligned} \right.\ee
In this paper, we are looking for strong solutions of (\ref{NS1}) which are defined as below.
\begin{defn}\label{Def, ss}
	\sl{If there exist $ \bm{v}\in L_t^2 H_x^2 \cap H_t^1 L_x^2 \big(D\times[0,T]\big) $ and $ P\in L_t^2 H_x^1  \big(D\times[0,T]\big) $ such that $ (\bm{v},P) $ satisfies (\ref{NS1}) in $ L_{tx}^2 $ sense, then $ \bm{v} $ or $ (\bm{v},P) $ is called a strong solution of (\ref{NS1}) on $ D\times [0,T] $.}
\end{defn}

For the bounded domain $ D $ in (\ref{domain-cyl}) with $ \al\in\big(0, \frac{\pi}{6}\big] $ and under the NTS boundary condition (\ref{NTS bdry}), we manage to obtain a strong solution to ASNS (\ref{NS1}) under the assumptions (i) and (\ref{COND}) in the main result, Theorem \ref{Thm_main}, of this paper. 
By taking advantage of the geometric feature of the targeted domain (\ref{domain-cyl}), we  adopt the spherical coordinates to study it. Under the spherical coordinates, the domain $D$ in (\ref{domain-cyl}) is equivalent to the following (also see Figure \ref{Fig,domain-sph})
\bn
D = \Big\{(\rho,\phi,\th): 0<\rho <1, \, \frac{\pi}{2}-\al < \phi < \frac{\pi}{2}+\al, \, \theta\in [0,2\pi) \Big\}.
\en

\begin{figure}[!ht]
	\centering
	\includegraphics[scale=0.25]{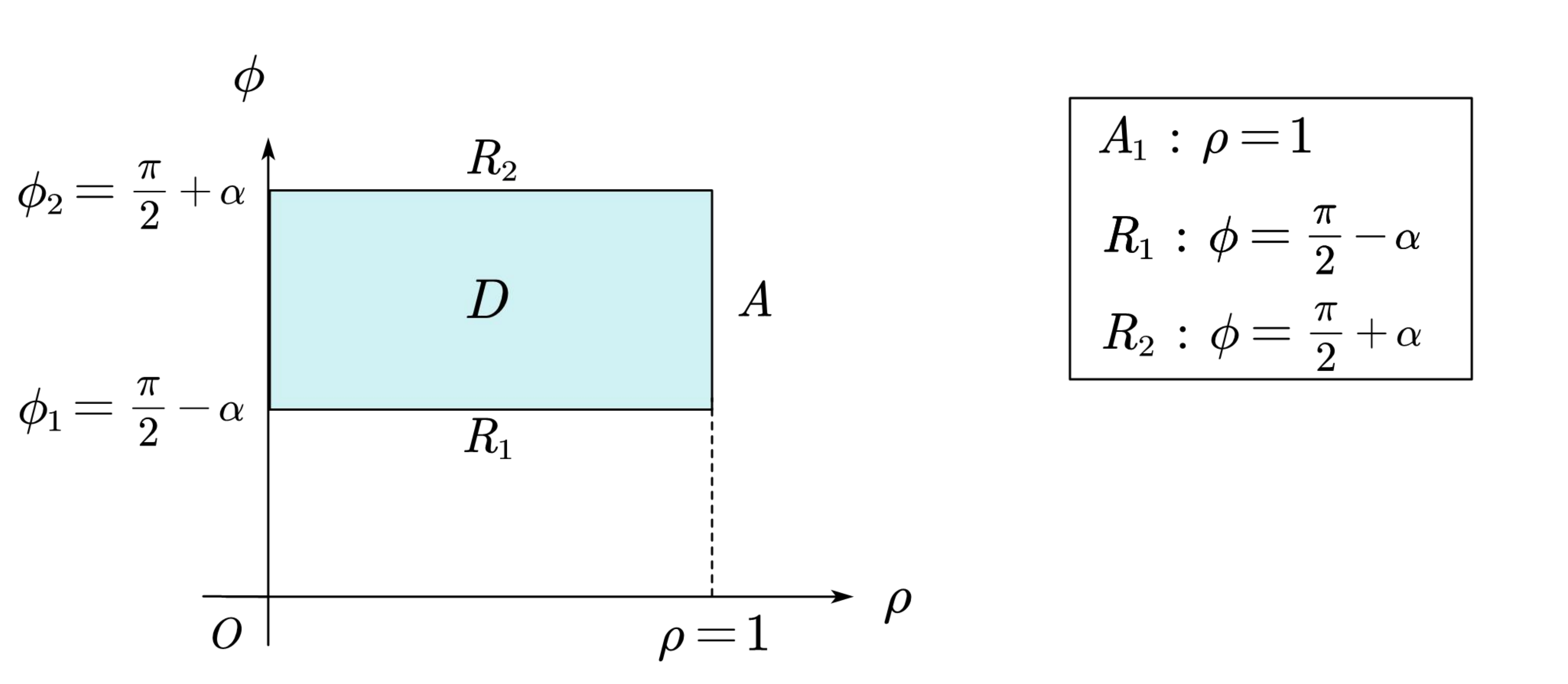}
	\caption{Domain $D$ in spherical coordinates}
	\label{Fig,domain-sph}
\end{figure}

The axially symmetric Navier-Stokes equations in the spherical coordinate system are as below (see \cite[(2.7) and Appendix A.1]{LPYZZZ24} for computational details in spherical coordinates):
\be\label{NS}
\left\{
\begin{aligned}
	&\left(\Delta+\frac{2}{\rho} \partial_\rho+\frac{2}{\rho^2}\right) v_\rho-\bm{b} \cdot \nabla v_\rho+\frac{1}{\rho}\left(v_\phi^2+v_\theta^2\right)-\partial_\rho P-\partial_t v_\rho=0\,, \\
	&\left(\Delta-\frac{1}{\rho^2 \sin ^2 \phi}\right) v_\phi-\bm{b} \cdot \nabla v_\phi+\frac{2}{\rho^2} \partial_\phi v_\rho-\frac{1}{\rho} v_\rho v_\phi+\frac{\cot \phi}{\rho} v_\theta^2-\frac{1}{\rho} \partial_\phi P-\partial_t v_\phi=0\,, \\
	&\left(\Delta-\frac{1}{\rho^2 \sin ^2 \phi}\right) v_\theta-\bm{b} \cdot \nabla v_\theta-\frac{1}{\rho}\left(v_\rho + \cot \phi \, v_\phi\right) v_\theta-\partial_t v_\theta=0\,, \\
	&\frac{1}{\rho^2} \partial_\rho\left(\rho^2 v_\rho\right)+\frac{1}{\rho \sin \phi} \partial_\phi\left(\sin \phi \, v_\phi\right)=0\,.
\end{aligned}\right.
\ee
Here, $\rho$ is the radial distance and $\phi$ is the angle between the radius vector and the positive $x_3$ axis:
\[
\rho\ed\sqrt{r^2 + x_3^2},\q\q\phi\ed\left\{
\begin{aligned}
	&\arctan\f{r}{x_3}, &\text{for}\q x_3 > 0\,,\\
	& \frac{\pi}{2}, &\text{for}\q x_3=0\,,\\
	&\pi+\arctan\f{r}{x_3}, &\text{for}\q x_3<0\,,
\end{aligned}
\right.\q\q\th\ed\arctan\f{x_2}{x_1}\,.
\] 
And $\bm{v}=v_\rho(\rho, \phi, t) \bm{e_\rho}+v_\phi(\rho, \phi, t) \bm{e_\phi}+v_\theta(\rho, \phi, t)\bm{e_\theta}\,,$
where 
\[
\bm{e_\rho}\ed\left(\begin{array}{c}
	\sin \phi \cos \theta \\
	\sin \phi \sin \theta \\
	\cos \phi
\end{array}\right), \quad \bm{e_\phi}\ed\left(\begin{array}{c}
	\cos \phi \cos \theta \\
	\cos \phi \sin \theta \\
	-\sin \phi
\end{array}\right), \quad \bm{e_\theta}\ed\left(\begin{array}{c}
	-\sin \theta \\
	\cos \theta \\
	0
\end{array}\right)\,.
\]
The components of the vorticity $ \bm{\o} := \nabla\times \bm{v} $ is given below: 
\be\label{vor-sph}
\left\{
\begin{aligned}
	&\omega_\rho=\frac{1}{\rho}\left(\partial_\phi+\cot \phi\right) v_\theta=\frac{1}{\rho \sin \phi} \partial_\phi\left(\sin \phi v_\theta\right)\,,\\
	&\omega_\phi=-\left(\partial_\rho+\frac{1}{\rho}\right) v_\theta=-\frac{1}{\rho} \partial_\rho\left(\rho v_\theta\right)\,, \\
	&\omega_\theta=\left(\partial_\rho+\frac{1}{\rho}\right) v_\phi-\frac{1}{\rho} \partial_\phi v_\rho=\frac{1}{\rho} \partial_\rho\left(\rho v_\phi\right)-\frac{1}{\rho} \partial_\phi v_\rho\,.
\end{aligned}
\right.
\ee

Now we state the main result of this paper.
\begin{thm}\label{Thm_main}
	Let the domain $D$ be as defined in \eqref{domain-cyl} with the angle $\alpha \in\left(0, \frac{\pi}{6}\right]$. Suppose the initial velocity $\bm{v}_0$ in the admissible class $\mathscr{A} \subset C^2(\ol{D})$ (see Definition \ref{Def, admissible sets}) satisfies
	\begin{itemize}
		\item[(i)] $\int_{D} r \, v_{0,\th}\d x=0$, where $v_{0,\th}$ represents the swirl component of $\bm{v}_0$;
		\item[(ii)] there exists an absolute positive constant $C_*$, given in (\ref{small_G0}), such that 
		\ba\l{COND}
		\sup_{x\in D} r |v_{0,\th}|\leq C_*\,.
		\ea
	\end{itemize}
	Then for any $T>0$, the problem (\ref{NS1}) for the Navier-Stokes equation with the initial data $\bm{v}_0$ under the Navier total-slip boundary condition has a strong solution $(\bm{v}, P)$ on $D \times[0, T]$ such that $\bm{v}$ is bounded uniformly in time, and satisfies
	\bn
	\|\bm{v}\|_{L_{t x}^{\infty}(D \times[0, T])}+\|\bm{v}\|_{H_t^1 L_x^2(D \times[0, T])}+\|\bm{v}\|_{L_t^2 H_x^2(D \times[0, T])}+\|P\|_{L_t^2 H_x^1(D \times[0, T])} \leq C,
	\en
	and 
	\be\label{gamma_int0}
	\int_{D} r v_{\th} (x,t)\d x = 0, \quad \forall\, t\in [0,T].
	\ee
	Here, $C$ is a constant that only depends on $\left\|\bm{v}_0\right\|_{C^2(\bar{D})}$. In addition, the following energy inequality holds:
	\be\label{energy_decay}
	\int_D|\bm{v}(x, T)|^2 \d x+\f{3}{4}\int_0^T \int_D|\nabla \bm{v}(x, t)|^2 \d x \d t \leq \int_D\left|\bm{v}_0(x)\right|^2 \d x\,.
	\ee
	On the other hand, if $ (\t{v}, \t{P}) $ is another strong solution of (\ref{NS1}) on $ D\times[0,T] $ under the Navier total-slip boundary condition with the same initial data $\bm{v}_0$, then $ \t{v} $ coincides with the above strong solution $ v $.
\end{thm}

\begin{rem}\label{Remark, gen_domain}
	\sl{Compared with \cite{LPYZZZ24}, we do not require the velocity field $\bm{v}$ to satisfy any odd or even symmetry condition in the present paper. As a consequence, our argument is applicable to a more general class of domains $\t{D}$:
		\be\label{domain_general}
		\t{D}= \bigg\{(r,\theta,x_3): 0< r^2+x_3^2 < 1, \, -r\tan\al_1 <x_3< r\tan\al_2, \, \theta\in [0,2\pi) \bigg\}\,,
		\ee
		with $0\leq\al_1\,,\,\al_2\leq\f{\pi}{6}$, $\al_1$ and $\al_2$ are not both zero. 	
		In particular, by taking $\al_1=0$, our results remain valid on the following ``coral-type'' exterior domains $\t{D}_0$ (see Figure \ref{Fig,Coral}).
		\[
		\t{D}_0= \bigg\{(r,\theta,x_3): 0< r^2+x_3^2 < 1, \, 0<x_3< r\tan\al_2, \, \theta\in [0,2\pi) \bigg\}\,.
		\] 
		\begin{figure}[!ht]
			\centering
			\includegraphics[scale=0.22]{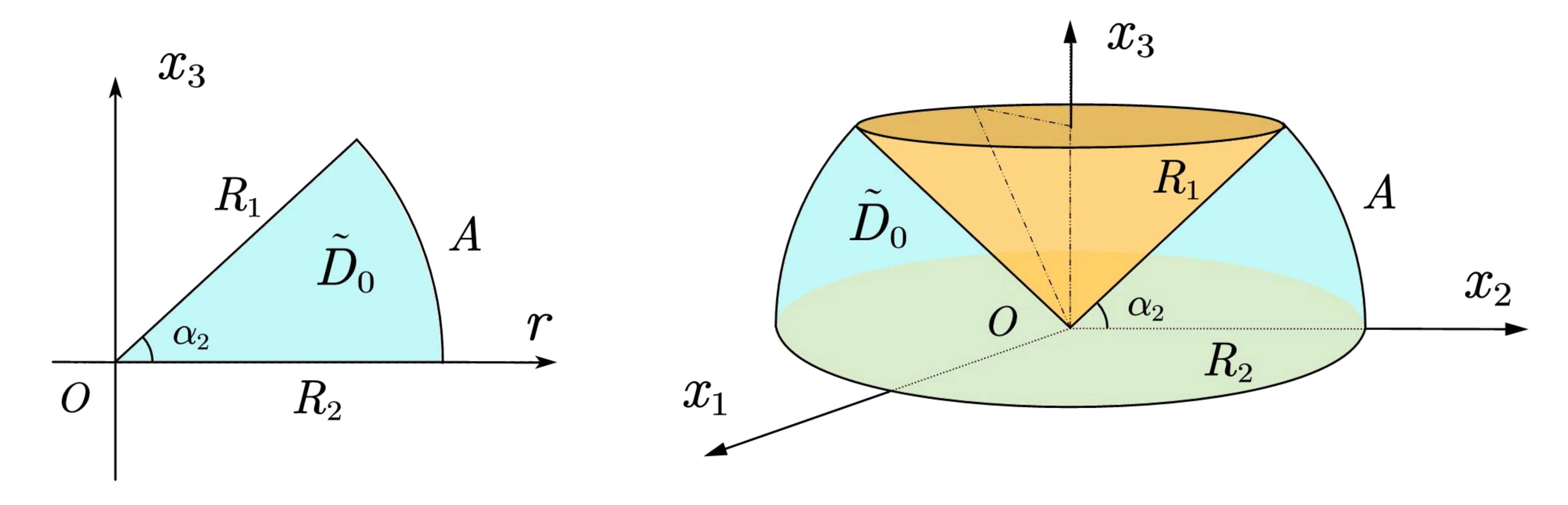}
			\caption{Coral-type exterior domain}
			\label{Fig,Coral}
	\end{figure}}
\end{rem}

\begin{rem}
	\sl{We would like to emphasize that the assumption (i) in Theorem \ref{Thm_main} is necessary for the solution $\bm{v}$ to decay and to satisfy the energy inequality (\ref{energy_decay}). Here is an exact example: for any constant $A>0$, the velocity field $\bm{v} := v_{\th}\bm{e_{\th}}$, where
		\be\l{Eacts}
		v_\th:= Ar
		\ee
		is a solution to equation (\ref{NS}) on the domain $D$ and the Navier total-slip boundary condition (\ref{NTS bdry}).  When $A$ is sufficiently small, the assumption (ii) in Theorem \ref{Thm_main} is satisfied. However, the assumption (i) is violated and (\ref{energy_decay}) does not hold since 
		\[\int_0^T \int_D|\nabla \bm{v}(x, t)|^2 \d x \d t = A^2 T \int_{D} |\nabla (r \bm{e_\th}) |^2 \d x = 2 A^2 T |D| \to \infty \quad \text{as} \quad T\to\infty.\]
		We also point out that this difficulty does not arise for the corresponding boundary value problem with the NHL boundary condition (\ref{NHL slip bdry}) owing to the lack of such special solutions as \eqref{Eacts}. }
\end{rem}

The problem studied in Theorem \ref{Thm_main} is connected to the general open question which asks that if an absolute smallness of one component of the initial velocity implies the global smoothness. We refer readers to \cite[page 873]{CZZ17} and to some of its subsequent follow-up works.  
Similar conic or wedge-type geometries have also appeared in other fluid problems, including studies of Euler singularity formation \cite{EJ19}, and Navier–Stokes flows in curved thin domains \cite{Miura2022} or exterior domains \cite{GigaSohr1991}. 
These works illustrate how geometric constraints and boundary conditions can influence regularity, stability, and long-time behavior of viscous flows.

Next, we will offer applications of Theorem \ref{Thm_main} in two opposite directions: global existence and finite-time blowup respectively.
Firstly, we discuss the direction of the global existence. Although Theorem \ref{Thm_main} is still far away from solving the regularity problem of ASNS due to the smallness assumption of $rv_{0,\th}$, it yields the following consequence which we call the controlled regularity:

\begin{defn}\label{Def, control_reg}
	\sl{We say that the controlled regularity holds for ASNS in a domain $D_*$ with a given boundary condition if for any divergence-free initial velocity $v_0$ in $C^2\cap H^2(\ol{D}_*)$ under the given boundary condition, and for any time $T>0$, there exists a forcing term $\bm{F}\in L^{2}_{tx}(D_*\times[0,T])$}, supported away from the $x_3$ axis, such that the ASNS with the forcing $\bm{F}$ under the prescribed boundary condition has a strong solution in $D_*\times [0,T]$. Noting that if $D_*$ is the whole space $\mathbb{R}^3$, then no boundary condition is needed.
\end{defn}

In the spirit of the mathematical control theory, the significance of the above concept is that the support of the forcing term $\bm{F}$ is away from the symmetric axis, otherwise the control will be much easier. In the current domain $D$ with the Navier total-slip boundary condition as stated in Theorem \ref{Thm_main}, the distance between the support of $\bm{F}$ and the axis of symmetry only depends on a single component of the initial velocity, see the next Theorem \ref{Thm2} for the precise statement.

\begin{thm}\l{Thm2}
	Let the domain $D$ be as defined in \eqref{domain-cyl} with the angle $\alpha \in\left(0, \frac{\pi}{6}\right]$. Suppose the initial velocity $\bm{v}_0$ is in the admissible class $\mathscr{A} \subset C^2(\ol{D})$ (see Definition \ref{Def, admissible sets}) and satisfies $v_{0,\th}$ is odd with respect to $\phi=\f{\pi}{2}$. Then the controlled regularity holds for the problem (\ref{NS1}), namely, for any $T>0$, there exists a force $\bm{F} \in L^{2}_{tx}(D\times [0,T])$ which is supported in 
	$\ol{D}\cap \{(\rho,\phi,\th): \rho \geq \rho_0/4\}$, where $\rho_0$ is a positive constant depending only on $v_{0,\th}$, such that the following problem 
	\be\l{PP1}\left\{\, \begin{aligned}
		&\Delta \t{\bm{v}} - (\t{\bm{v}}\cdot \nabla) \t{\bm{v}} - \nabla \t{P} - \p_{t} \t{\bm{v}} = \bm{F} \quad\text{in}\quad  D\times (0,T], \\
		&\nabla \cdot \t{\bm{v}} = 0  \quad \text{in} \quad  D\times (0,T], \\
		&\t{\bm{v}}\cdot \bm{n} = 0,\quad (\mathbb{S}\t{\bm{v}}\cdot\bm{n})_{\tan } = 0 \quad\text{on} \quad  \p D\times (0,T],\\
		&\t{\bm{v}}(\cdot, 0) = {\bm{v}}_0(\cdot) \quad\text{in} \quad  D,
	\end{aligned} \right.\ee
	has a strong solution in $D\times [0,T]$.
\end{thm}

We remark that the oddness assumption of $v_{0,\th}$ in Theorem \ref{Thm2} is part of the even-odd-odd parity of $(v_{0,\rho}, v_{0,\phi}, v_{0,\th})$ which is well-known to be preserved for the flow,
but the oddness for $v_{0,\th}$ alone may not be preserved for $t>0$. Fortunately, the proof of Theorem \ref{Thm2} only requires the oddness for $v_\th$ at the initial time. On the other hand, the controlled regularity is not limited to the current situation, we can also generalize this phenomenon to the ASNS model in the whole $\mathbb{R}^3$ space. 

\begin{thm}\l{Thm3}
	Let $\bm{v_0}$ be a divergence-free vector field on $\mR^{3}$ such that $\bm{v_0}\in H^2(\mR^3)\cap C^2(\mR^3)$ and $r v_{0,\th}\in L^2(\mR^3)$. Then the controlled regularity holds for the problem (\ref{eqasns}) on the whole space $\mR^3$, namely, for any $T>0$, there exists a force $\bm{F} \in L^{\infty}_{t}L^{2}_{x} \cap L^{\infty}_{tx}\big(\mR^3 \times [0,T]\big) \subset L^2_{tx}(\mR^3 \times[0,T])$ which is supported in $\mR^3\cap \{(r,\th,x_3): r \geq r_0/4\}$, where $r_0>0$ is a constant depending on $\bm{v_0}$, such that the following problem 
	\be\l{PP2}\left\{\, \begin{aligned}
		&\Delta \t{\bm{v}} - (\t{\bm{v}}\cdot \nabla) \t{\bm{v}} - \nabla \t{P} - \p_{t} \t{\bm{v}} = \bm{F} \quad\text{in}\quad  \mathbb{R}^3 \times (0,T], \\
		&\nabla \cdot \t{\bm{v}} = 0  \quad \text{in} \quad  \mR^3\times (0,T], \\
		&\t{\bm{v}}(\cdot, 0) = {\bm{v}}_0(\cdot) \quad\text{in} \quad \mR^3,
	\end{aligned} \right.\ee
	has a strong solution in $\mR^3\times [0,T]$.
\end{thm}

\begin{rem}
	\sl{From the point of view of control theory, the result is better if the forcing term has smaller support and less dependence. In this sense, Theorem \ref{Thm2} is stronger than Theorem \ref{Thm3} since the distance between the support of the forcing term and the axis of symmetry only depends on the azimuthal component of $\bm{v_0}$ instead of the whole $\bm{v_0}$.
	On the other hand, the forcing term $\bm{F}$ can be guaranteed in $L^{\infty}_{tx}$ in Theorem \ref{Thm3} which is stronger than that in Theorem \ref{Thm2}. The reason why $\bm{F}$ is not justified in $L^{\infty}_{tx}$ in Theorem \ref{Thm2} is due to complicated boundary estimates.
	}
\end{rem}

\begin{rem}
	\sl{Theorem \ref{Thm2} and Theorem \ref{Thm3} seem to add a little weight to the regularity scenario for ASNS, since the force is supported away from the axis of symmetry which is the only place where the regularity may break down.}
\end{rem}
Secondly, we discuss the application of Theorem \ref{Thm_main} in the direction of finite-time blowup. A challenging but natural future study is to investigate if the smallness assumption (\ref{COND}) can be removed. At this moment, it is not clear if the solution will blow up in finite time if the initial data $r v_{0,\th}$ is large. But if there exists a solution that blows up in finite time, then we can take advantage of Theorem \ref{Thm_main} to find an unstable blow-up solution. The precise meaning of this statement is given below.

\begin{cor}\label{Cor, unstable-bus}
	Let the domain $D$ be as defined in \eqref{domain-cyl} with the angle $\alpha \in\left(0, \frac{\pi}{6}\right]$. Suppose the problem (\ref{NS1}) possesses a strong solution $\bm{v}$ which blows up in finite time with an initial data $\bm{v}_0$ in the admissible class $\mathscr{A}$ such that $\int_{D} r v_{0,\th} \d x = 0$. Then there exists a strong solution $\bm{v}^{*}$ to the problem (\ref{NS1}) which blows up in finite time and is unstable in the sense that for any $\delta>0$, there exists a global strong solution $\wt{\bm{v}}$ to the problem (\ref{NS1}) such that 
	\[
	\| \wt{\bm{v}}_{\th}(x,0) - \bm{v}^{*}_{\th}(x,0) \|_{C^2(\ol{D})} \leq \delta,
	\]
	where $\wt{\bm{v}}_{\th}(x,0)$ and $\bm{v}^{*}_{\th}(x,0)$ represent the azimuthal components of the initial values of $\wt{\bm{v}}$ and $\bm{v}^{*}$ respectively.
\end{cor}

The idea for Corollary \ref{Cor, unstable-bus} can be generalized to arbitrary setup of Navier-Stokes equations as long as an absolute smallness condition, such as (\ref{COND}), guarantees the global existence of a strong solution. For example in the full 3D Navier-Stokes model (\ref{nse}), if it possesses a strong solution $\bm{v}$ which blows up in finite time with an initial data $\bm{v}_0$ in the Schwartz space, then one can follow the proof of Corollary \ref{Cor, unstable-bus} to find a strong solution $\bm{v}^{*}$ to (\ref{nse}) which blows up in finite time and is unstable in the sense that for any $\delta>0$, there exists a global strong solution $\wt{\bm{v}}$ to (\ref{nse}) such that 
\[
\| \wt{\bm{v}}(x,0) - \bm{v}^{*}(x,0) \|_{L^{\infty}(\mathbb{R}^3)} \leq \delta,
\]
where $\wt{\bm{v}}(x,0)$ and $\bm{v}^{*}(x,0)$ represent the initial values of $\wt{\bm{v}}$ and $\bm{v}^{*}$ respectively.

\subsection{Difficulties and strategies}
We now explain the main analytical difficulties arising from the Navier total-slip boundary condition and outline the strategy used to prove Theorem \ref{Thm_main}.
A central part of proving Theorem \ref{Thm_main} is to construct a closed energy estimate for a collection of higher-order good unknowns. In \cite{LPYZZZ24}, the authors employed the following group of good unknowns
\[
K\ed\f{\o_\rho}{\rho}\,,\q F\ed\f{\o_\phi}{\rho}\,, \q \O\ed\f{\o_\th}{\rho\sin\phi}\,.
\] 
Here the quantity $\O$ could date back to Ladyzhenskaya \cite{La} and Ukhovskii-Yudovich \cite{UY} in 1960s, whilst $K$ and $F$ are new inputs which were inspired by the quantity $J\ed \frac{w_r}{r}$ in \cite{CFZ}. However, this group of good unknowns could not be applied directly in the current paper, owing to the change of boundary conditions. Instead of the traditional good unknown
\[
\O=\f{\o_\th}{\rho\sin\phi}=\f{\p_{x_3}v_r-\p_rv_3}{r}
\]
we observe that
\[
\t{\O}\ed\p_\rho\Big(\f{v_\phi}{\rho\sin\phi}\Big) - \p_\phi\Big(\f{v_\rho}{\rho^2}\Big)\,
\]
vanishes on the arc boundary provided $\bm{v}$ satisfies the Navier total-slip boundary condition \eqref{NS1}$_3$. However, owing to the singularity of the Navier total-slip boundary condition near the origin, we choose to impose a mixed boundary condition of both NTS type and NHL type, see (\ref{NTS-NHL}), on the artificial inner arc $A_{1,m}$ of the approximating domain $D_m$ (see Section \ref{Sec2} below).  Unfortunately, the boundary condition for $\t{\O}$ on $A_{1,m}$ makes it very challenging to control the following type of boundary integration 
\[
\int_{A_{1,m}}\t{\O}\p_\rho\t{\O}\d S.
\]
To overcome this difficulty, we ``connect" $\Omega$ and $\t{\O}$ by $\m{O}$ which is defined as 
\be\l{DO}
\m{O}\ed\eta\t{\O}+(1-\eta){\O}=\Omega-\frac{2 v_\phi \eta}{\rho^{2} \sin \phi}\,,
\ee
where $\eta$ is any smooth cut-off function such that $0\leq\eta(\rho)\leq1$, $0 \leq \eta'(\rho) \leq 6$ and
\ba\l{DCUT}
\eta(\rho)=\left\{
\begin{aligned}
	&0,\q\text{for}\q &0\leq\rho\leq\f{1}{3}\,;\\
	&1,\q\text{for}\q &\rho\geq\f{2}{3}\,.\\
\end{aligned}
\right.
\ea

Compared with $\Omega$, here we notice that $\m{O}$ not only depends on $\omega_\theta$, but also on $v_\phi$. Therefore, before obtaining a closed energy estimate of good unknowns, we must first establish an $H^1$ estimate for the pressure $P$, since the equation of $v_\phi$ contains it. To do this, we formulate a boundary value problem satisfied by the pressure $P$ (Lemma \ref{BVPP}), and use it to derive elliptic estimates for $P$ (Lemma \ref{EPs}). In addition, we develop techniques to estimate the nonlinear terms such as $\| v_\phi \p_{\rho} v_\rho \|_{L^2}$ (Lemma \ref{NonC}), which arise in the elliptic estimates for the pressure $P$. 

To close the estimate of $\m{O}$, we also need to pair it with quantities involving $v_\theta$ of the same order. Unlike the pair $(K,F)$ given in \cite{LPYZZZ24}, or traditional quantity $J=\f{\omega_r}{r}=-\p_z\big(\f{v_\th}{r}\big)$ \cite{CFZ}, we introduce the following good unknowns:
\be\l{DmK}
\m{K}\ed K-\frac{2 v_{\theta}\cot \phi}{\rho^{2}} = \f{\sin\phi}{\rho^2}\p_\phi\Big(\f{v_\th}{\sin\phi}\Big)
\,,\q\q
\mF \ed F+\f{2v_\th}{\rho^2} = -\p_\rho\Big(\f{v_\th}{\rho}\Big)\,.
\ee
The advantage of $\m{K}$ is that it vanishes on both $R_1$ and $R_2$ while $K$ and $J$ do not. Here we mention that $R_1$ and $R_2$ touch the origin, which may generate higher-order singularities near the origin if the related boundary condition is not of Robin type on $R_1\cup R_2$. For example, the related boundary integrals of $K$ or $J$ on $R_1\cup R_2$ have bad signs and are not controllable under the Navier total-slip boundary condition. Moreover, $\m{K}$ also satisfies boundary conditions of Robin-type on the inner arc $A_{1,m}$, see \eqref{BCOKF} below, which makes the estimate for $\m{K}$ manageable on the whole boundary $\p D_m$.  
Meanwhile, we introduce \(\m{F}\), which provides the \(\rho\)-derivative of \(v_\theta\), satisfies a Robin condition on the rays, and vanishes on the arcs. As a result, the boundary integrals of $\m{F}$ will not cause trouble in the energy estimate for $\m{F}$.

Technically, for a 2nd-order good unknown $\m{S}$, the boundary term
\[
I\ed\int_{\p D}\m{S} (\p_n\m{S}) \d S\,
\]
which is generated via integration by parts of the viscous term, can only be controlled by the following two methods: (A) $\m{S}=0$ on $\p D$ (homogeneous Dirichlet boundary condition), then $I$ vanishes directly; (B) $\p_n\m{S}$ equals some lower order terms on $\p D$ (Robin-type boundary condition), which can be used to reduce the order of derivatives appearing in $I$, and then one recovers $I$ to a volume integration by applying the Newton-Leibniz formula. We could not find a $v_\theta$-related 2nd-order good unknown as $\m{K}$ if the spherical boundary $\{\rho=1\}$ is replaced by the cylinder $\{r=1\}$, owing to the item (A) or (B) stated above cannot be simultaneously satisfied on both the ray boundary and the cylindrical boundary. This motivates us to consider the domains $D$ with arcs as boundary instead of vertical boundary. Geometrically, when the boundary normals of two adjacent boundary pieces are not mutually orthogonal, spurious second-order normal derivatives on the boundary may arise during the analysis. However, our NTS boundary condition \eqref{NTS bdry} provides no control or information on the 2nd-order normal derivatives of $\bm{v}$ on the boundary, and the normal derivative of $v_\theta$ cannot be converted into a tangential derivative via the divergence-free constraint.

Whether the initial smallness of $\G$ can be propagated for all times is also crucial to solving our problem. Unlike the NHL boundary case, where this can be obtained directly via the maximum principle, in the Navier total-slip case we must handle a “bad” boundary term (a term whose sign is favorable in the NHL case), see Section \ref{Sec7} below. Here, we use a De Giorgi iteration scheme instead of the Moser iteration scheme in which the boundary term can no longer be controlled by the viscous term because the coefficients grow very fast as the power of the function increases under the Navier total-slip boundary condition. 

Now we briefly outline the main difficulty in removing the assumption that \(v_\theta\) is odd with respect to \(\phi\), which was essentially used in \cite{LPYZZZ24}. In fact, if \(v_\theta\) is odd with respect to \(\phi\), then by applying the Poincar\'e inequality in the \(\phi\)-variable (see Lemma \ref{Poin0} below), we can obtain the following Hardy-type inequality for \(v_\theta\):
\ba\l{0319Ineq}
\int_{D_m}\Big|\f{1}{\rho} \f{v_\th}{\rho} \Big|^2\d x
\leq \f{2}{19}\int_{D_m} \Big| \frac{1}{\rho}\p_\phi \Big(\f{v_\th}{\rho}\Big) \Big|^2\d x
\leq \f{2}{19}\int_{D_m}\Big| \na \Big(\f{v_\th}{\rho}\Big) \Big|^2\d x\,.
\ea
Here, the small constant \(\frac{2}{19}\) in the inequality is crucial for closing the higher-order energy estimates. More precisely,  if the constant $\f{2}{19}$ in \eqref{0319Ineq}
were $4$ as in the standard three-dimensional Hardy's inequality, our method in this paper would no longer work. To obtain an alternative estimate as (\ref{0319Ineq}), we take advantage of the assumption (i) of Theorem \ref{Thm_main} to establish the following anisotropic Hardy's inequality for $v_\th$ (see Corollary \ref{Cor, vth} in Section \ref{Sec, App_Hardy}):
\bn
\int_{D_m} 
\Big| \frac{1}{\rho} \frac{v_\theta}{\rho} \Big|^2 \dx
\leq 
\frac{8}{\sqrt{3}} \int_{D_m}
\Big|\partial_\rho \Big(\frac{v_\theta}{\rho}\Big) \Big|^2 \dx
+ \frac{1}{5\sqrt{3}} 	\int_{D_m} 
\Big| \frac{1}{\rho}\partial_\phi \Big(\frac{v_\theta}{\rho} \Big) \Big|^2
\dx.
\en
The above estimate enables us to control $L^2$ norm of $\frac{v_\th}{\rho^2}$ by a certain combination of $L^2$ norms of first-order derivatives of $\frac{v_\th}{\rho}$, which in turn helps to control $\nabla \big(\frac{v_\th}{\rho}\big)$ through the good unknown pair \((\m{K},\m{F})\); see Lemma \ref{Lem35} which is also a key part in closing the energy estimates for $(\m{K},\m{F},\m{O})$. 
Meanwhile, following the same approach, we derive in Corollary \ref{Cor, Korn_ineq} a Korn's inequality for $v_\th\bm{e_\th}$, with its constant independent of the approximation index $m$, which is crucial for deriving the fundamental energy estimate in Proposition \ref{Funden}.
We emphasize that condition (i) of Theorem \ref{Thm_main} is weaker than the oddness assumption of $v_\th$, which allows more flexibility of the domain $\t{D}$ in Remark \ref{Remark, gen_domain} compared with the symmetric domain $D$ in (\ref{domain-cyl}). 

Finally, we discuss the explicit constructions in the proofs of Theorems \ref{Thm2} and Theorem \ref{Thm3}. We start with Theorem \ref{Thm2}.
The key step is to split the initial velocity $\bm{v}_0$ only in the azimuthal component: $\bm{v}_0 = \bm{v}^{(1)}_0 + \bm{v}^{(2)}_0$,
where 
\be\label{init_split}
\bm{v}^{(1)}_0 = v_{0,\rho}\bm{e_\rho} + v_{0,\phi}\bm{e_\phi} + \eta_1 v_{0,\th}\bm{e_\th}, \qquad 
\bm{v}^{(2)}_0 = (1-\eta_1) v_{0,\th}\bm{e_\th}.
\ee
Here, $\eta_1$ is a smooth cut-off function which equals $1$ near the $x_3$ axis and decays to $0$ quickly. The advantage of splitting the azimuthal component alone is to preserve the divergence-freeness of both $\bm{v}^{(1)}_0$ and $\bm{v}^{(2)}_0$ since the azimuthal component does not contribute to the divergence for the axisymmetric vector fields. Owing to the fact that the support of $\eta_1$ is near the $x_3$ axis, the smallness assumption (\ref{COND}) in Theorem \ref{Thm_main} is guaranteed for the initial velocity $\bm{v}^{(1)}_0$. As a result, the problem with the initial velocity $\bm{v}^{(1)}_0$ admits a global strong solution $\bm{v}^{(1)}$. 

Meanwhile, since the support of $1-\eta_1$ is away from the $x_3$ axis, we may regard $\bm{v}^{(2)}_0$ as the initial data on a truncated domain that is away from the axis of symmetry.  By Proposition \ref{Prop, local soln in ad} below,  $\bm{v}^{(2)}_0$ admits a global strong solution on this truncated domain. So we consider $\bm{v}^{(3)}\ed\bm{v}^{(1)}+\eta_2\bm{v}^{(2)}$, where $\eta_2$ is a delicately chosen cut-off function so that its support is away from the $x_3$ axis and the initial value of $\bm{v}_3$ matches $\bm{v}_0$. With this choice, $\bm{v}^{(3)}$ is regular and the support of the forcing term $\bm{F}$ can also be made away from the $x_3$ axis. 

However, due to the presence of $\eta_2$, the vector $\bm{v}^{(3)}$ is not divergence-free any more. So we perform suitable corrections to take care of both the divergence-freeness and the Navier total-slip boundary condition. At this point, the construction correction terms require particular care, since all the boundary conditions must be met simultaneously, while the initial value of $\bm{v}^{(3)}$ and its values in a neighborhood of the $x_3$ axis need to remain unaffected.

The idea for the construction in Theorem \ref{Thm3} is similar, except that we take advantage of a result in \cite{LZ17} in place of Theorem \ref{Thm_main} to address the global well-posedness in the whole space $\mR^3$. Fortunately, the criterion in \cite{LZ17} can also be met for $\bm{v}^{(1)}_0$ as long as $\eta_1$ is suitably chosen. The remaining steps are analogous to the argument for Theorem \ref{Thm2}. The difference is that the delicate matching of boundary conditions is unnecessary in this case, but one instead needs to construct the correction term to be $L^2$-integrable over the whole space $\mR^3$, which requires particular care. Moreover, without the boundary effect, the forcing term $\bm{F}$ in Theorem \ref{Thm3} can be justified in $L^\infty_{tx}$ which is stronger than that in Theorem \ref{Thm2}.

The organization of this paper is as follows: In Section \ref{Sec2}, preliminaries of setup on approximating domains $D_m$ and the equations for the key triple $(\m{K}, \m{F}, \m{O})$ in the spherical coordinates are presented. Section \ref{Sec, pressure} is devoted to elliptic estimates of the pressure $P$. Various Poincar\'e's inequalities, Hardy's inequalities and their applications could be found in Sections \ref{Sec4} and \ref{Sec, App_Hardy}. Section \ref{Sec, v_energy_est} is concerned with the fundamental energy estimate for the velocity $\bm{v}$ and Section \ref{Sec7} establishes the $L^\infty$ estimate of $\Gamma$ by its initial data. Sections \ref{Sec8} and \ref{Sec9} are devoted to the energy estimates for the triple of good unknowns $(\mK, \mF, \mO)$. In Section \ref{Sec10}, we prove Theorem \ref{Thm_main}. Finally, Section \ref{Sec11}, Section \ref{Sec, reg_control_R3} and Section \ref{Sec12} are devoted to the proof of Theorem \ref{Thm2}, Theorem \ref{Thm3} and Corollary \ref{Cor, unstable-bus}, respectively.


\section{Setup on approximating domains $D_m$}\l{Sec2}

\subsection{Approximating Domains $D_m$}
Firstly, we introduce the approximating domains $D_m$ and the admissible class $ \mathscr{A} $ of the initial vector fields that we consider in this paper. Since the original domain $ D $ touches the $ x_3 $ axis with an angle, the velocity may be more likely to develop singularities. Moreover, the solution may not be expected to have higher regularity than $ L_t^2 H_x^1$. In order to acquire more regularity and to prove the boundedness of the velocity, we first cut the corner of $ D $ and then study the problem in approximating domains $ D_m $ $ (m\geq 10^3) $, which are defined as
\be\label{app domain-cyl}
D_m = \bigg\{(r,\theta,x_3): \frac{1}{m^2} < r^2+x_3^2 < 1, \, -r\tan\al <x_3< r\tan\al, \, \theta\in [0,2\pi) \bigg\}.\ee
Under the spherical coordinates, the domain $ D_m $ in (\ref{app domain-cyl}) is equivalent to the following (also see Figure \ref{Fig,app domain-sph}):
\bn
D_m=\Big\{(\rho,\phi,\th): \frac{1}{m}<\rho <1, \, \frac{\pi}{2}-\al < \phi < \frac{\pi}{2}+\al, \, \theta\in [0,2\pi) \Big\}\,.
\en
In addition, for convenience of notation, we denote the four pieces of the boundary $\p D_m $ to be $ R_{1,m} $, $ R_{2,m} $, $ A_{1,m} $ and $ A_{2,m} $, and write $\p^{R}D_{m} = R_{1,m} \cup R_{2,m}$, $ \p^{A}D_{m} = A_{1,m} \cup A_{2,m}$.
\begin{figure}[!ht]
	\centering
	\includegraphics[scale=0.23]{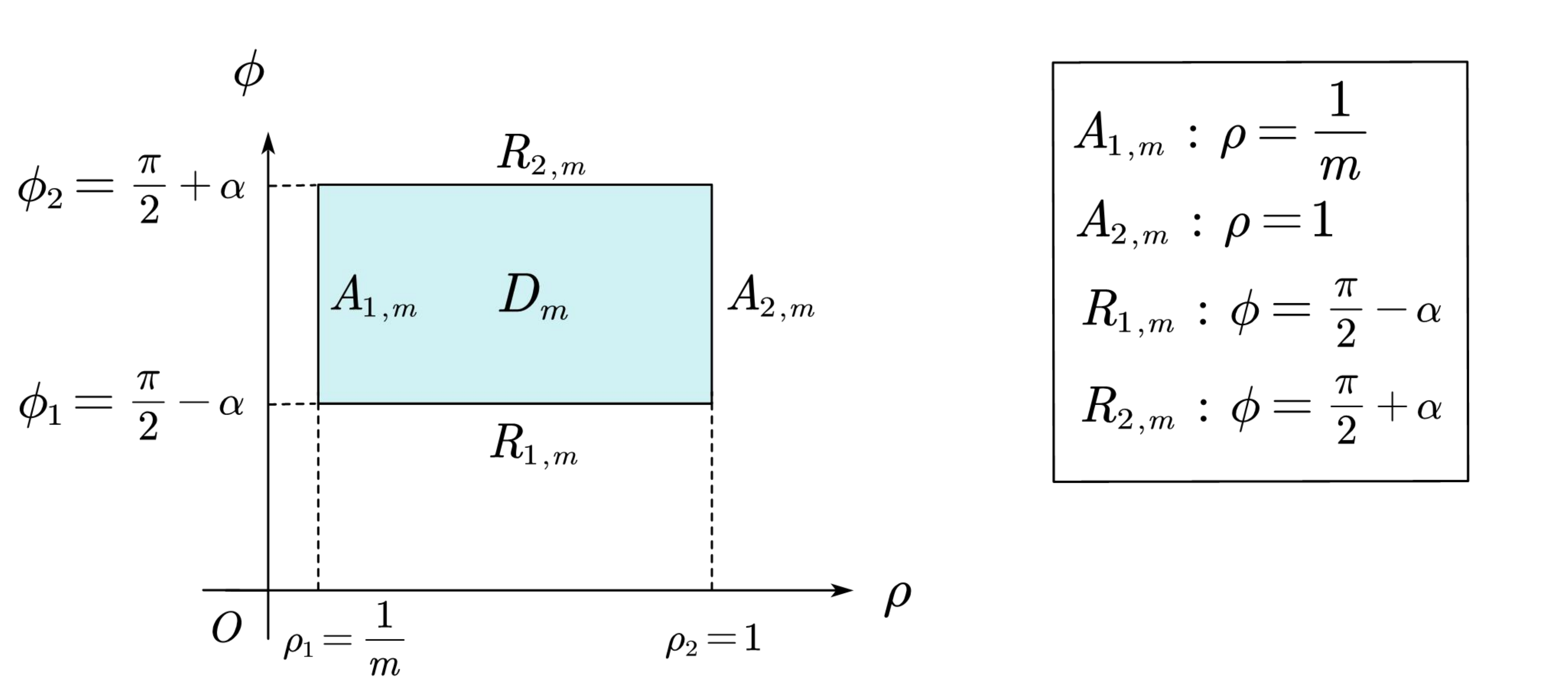}
	\caption{Domain $D_m$ in spherical coordinates}
	\label{Fig,app domain-sph}
\end{figure}

The boundary condition associated with $D_m$ is adopted as a combination of the NTS condition (\ref{NTS bdry}) and the NHL condition (\ref{NHL slip bdry}):
\be\label{bdry for Dm}
\begin{aligned}
	\bm{v}\cdot \bm{n}=0 \text{ on } \p D_m, \quad \text{(\ref{NTSR}) on } R_{1,m}\cup R_{2,m}, \quad \text{(\ref{NTSA}) on } A_{2,m}, \quad \text{(\ref{NTS-NHL}) on } A_{1,m}.
\end{aligned}\ee	

On $R_{1,m} \cup R_{2,m}$, the NTS condition (\ref{NTS bdry}) can be represented by
\be\label{NTSR}
\left\{\begin{aligned}
	&v_{\phi}=\partial_{\rho} v_{\phi}=\partial_{\phi} v_{\rho}=0, \\
	&\partial_{\phi} v_{\theta}=\cot \phi \, v_{\theta}, \\
	&\omega_{\theta}=0 .
\end{aligned}\right.
\ee
On $A_{2,m}$, the NTS condition (\ref{NTS bdry}) is
\be\label{NTSA}
\left\{\begin{aligned}
	&v_{\rho}=\partial_{\phi} v_{\rho}=0, \quad \partial_{\rho} v_{\theta}=\frac{1}{\rho} v_{\theta}, \\
	&\partial_{\rho} v_{\phi}=\frac{1}{\rho} v_{\phi},  \\
	&\omega_{\theta}=\frac{2}{\rho} v_{\phi}.
\end{aligned}\right.
\ee
Meanwhile, the mixed NTS condition (\ref{NTS bdry})-NHL condition (\ref{NHL slip bdry}) on $A_{1,m}$ is defined as 
\be\label{NTS-NHL}
\left\{
\begin{aligned}
	& v_{\rho} = \p_\phi v_\rho = 0, \quad \q\p_{\rho} v_{\th} = \frac{1}{\rho} v_{\theta}, \\ 
	& \partial_{\rho} v_{\phi} = -\frac{1}{\rho} v_{\phi}, \\	
	& \omega_{\theta} = 0.
\end{aligned}
\right.
\ee

\begin{rem}[On the mixed boundary conditions on $A_{1,m}$]
	The mixed boundary condition (\ref{NTS-NHL}), i.e. NTS for $v_\th$ and NHL for $v_\rho$ and $v_\phi$, is introduced on the inner arc $A_{1,m}$ for the approximating problem in $D_m$. As mentioned in the introduction, if one chooses other types of boundary conditions on the inner arc, some challenging boundary terms will emerge. 
	As $m\to\infty$, the inner arc $A_{1,m}$ shrinks to the cone vertex and the problem involves only the Navier total-slip boundary condition in the limit case. 
\end{rem}

Due to the above strategy, it is natural to choose the elements in $ \mathscr{A} $ to be the  limits of vector fields on $ D_m $.
\begin{defn}[Admissible classes $ \mathscr{A}_m $ and $ \mathscr{A} $]
	\label{Def, admissible sets}
	Fix any angle $ \al\in\big(0, \frac{\pi}{6}\big] $.
	\begin{itemize}
		\item[(1)] For any integer $ m\geq 10^3 $, we define the admissible class $ \mathscr{A}_{m} $ on $ D_m $ to be the space of vector fields $ \bm{v}_{0}^{(m)} $ in $ C^2(\overline{D_m}) $ that are axisymmetric and divergence free in $ D_m $ that satisfy the boundary condition (\ref{bdry for Dm}) on $ \p D_m$.
		
		\item[(2)] For the domain $ D $, we define the admissible class $ \mathscr{A} $ on it to be the space of vector fields $ \bm{v}_0 $ in $ C^2(\ol{D}) $ such that there exist vector fields $ \big\{\bm{v}_0^{(m)}\big\}_{m\geq 10^3} $ such that $ \bm{v}_0^{(m)}\in\mathscr{A}_m $ and
		\[ \lim_{m\to\infty} \big\| \bm{v}_0 - \bm{v}_0^{(m)} \big\|_{C^2(\overline{D_m})} = 0. \]
	\end{itemize}
\end{defn}
Based on Definition \ref{Def, admissible sets}, it is not difficult to see that the admissible class $\mathscr{A}$ is not empty, but it is not clear whether every function in $C^2(\ol{D})$ that satisfies the Navier total-slip boundary condition (\ref{NTS bdry}) belongs to $\mathscr{A}$.

Besides the fixed cut-off function $\eta$ defined in (\ref{DCUT}), this paper will introduce another sequence of cut-off functions $\{\eta_m\}$ which are adapted to the approximating domains $\{D_m\}$. These cut-off functions $\{\eta_m\}$ vanish on the inner arc $A_{1,m}$ and equal to 1 on the outer arc $A_{2,m}$. More precisely, we choose these $\{\eta_m\}$ as in (\ref{cut_fn}) which satisfy the crucial estimates in (\ref{est_cut_fn}).

\begin{lemma}\label{Lemcutoff}
	Let $m \geq 10^3$ and define
	\be\label{cut_fn}
	\eta_{m}(\rho)=\sin \left[\frac{\pi}{2}\left(\frac{m}{m-1} \rho-\frac{1}{m-1}\right)\right]\,, \quad \forall\, \frac{1}{m} \leq \rho \leq 1\,.
	\ee
	Then for any $\rho \in [1/m, 1]$, we have 
	\be\label{est_cut_fn}
	0 \leq \eta_{m}(\rho) + \rho \eta_{m}^{\prime}(\rho) \leq \frac{7}{5}\,, \quad \text{and}\quad 
	0 \leq 3\eta_{m}(\rho) + \rho \eta_{m}^{\prime}(\rho) \leq \frac{13}{4}\,.
	\ee
\end{lemma}

\begin{proof}
	Direct calculation shows
	$$
	\eta_{m}(\rho)+\rho \eta_{m}^{\prime}(\rho)
	=\sin \Big[\frac{\pi}{2} \Big(\frac{m}{m-1} \rho-\frac{1}{m-1}\Big)\Big]
	+ \frac{\pi}{2} \cdot \frac{m}{m-1} \rho \cos \Big[\frac{\pi}{2} \Big(\frac{m}{m-1} \rho - \frac{1}{m-1}\Big)\Big]\,.
	$$
	Let $x=\frac{\pi}{2}(\frac{m}{m-1} \rho-\frac{1}{m-1})$. As $\rho \in\left[\frac{1}{m}, 1\right]$, we have $x \in\left[0, \frac{\pi}{2}\right]$ and 
	\[\left\{\begin{aligned}
		\eta_{m}(\rho)+\rho \eta_{m}^{\prime}(\rho) & = \sin (x)+x \cos (x)+\frac{\pi}{2(m-1)} \cos (x) := f_1(x), \\
		3\eta_{m}(\rho)+\rho \eta_{m}^{\prime}(\rho) & = 3\sin (x)+x \cos (x)+\frac{\pi}{2(m-1)} \cos (x) := f_2(x).
	\end{aligned}	\right.\]
	For the function $\sin (x)+x \cos (x)$ on $[0,\pi/2]$, its maximum is attained at $x_1\in [0,\pi/2]$ such that $x_1 \tan(x_1) = 2$, then $x_1$ is between $1.075$ and $1.08$, so 
	\[ f_1(x) \leq \sin (1.08) + 1.08 \cos (1.075) + \frac{\pi}{1998} < \frac75 \,,\q\text{for any}\q x\in [0, \pi/2]\,.\]
	Meanwhile, for the function $3\sin (x)+x \cos (x)$ on $[0,\pi/2]$, its maximum is attained at $x_2\in [0,\pi/2]$ such that $x_2 \tan(x_2) = 4$, then $x_2$ is between $1.26$ and $1.265$, so 
	\[ f_2(x) \leq 3\sin (1.265) + 1.265 \cos (1.26) + \frac{\pi}{1998} < \frac{13}{4} \,,\q\text{for any}\q x\in [0, \pi/2]\,.\]
	This concludes the lemma.
\end{proof}

\subsection{Strong Solutions on $D_m$}

\begin{prop}\label{Prop, local soln in ad}
	Let	$\al\in\big(0,\frac{\pi}{6}\big]$ and $m\geq 10^3$. Assume that the initial velocity $ \bm{v}_0\in H^{2}(D_m) $ is divergence free in $ D_m $ and satisfies the boundary  condition (\ref{bdry for Dm}) on $ \p D_m $. Then for any time $ T>0 $, the problem (\ref{NS}) on $ D_m\times[0,T] $ with the initial data $ \bm{v}_0 $ and the boundary condition (\ref{bdry for Dm}) possesses a strong solution $ (\bm{v},P) $ such that
	\[
	\bm{v}\in H_t^1 L_x^2 \cap L_t^2 H_x^2\cap L_{tx}^{\infty}\big(D_m\times[0,T]\big), \quad P\in L_t^2 H_x^1(D_m\times[0,T]). 
	\]
	In addition, if $ (\hat{\bm{v}}, \hat{P}) $ is another strong solution, then $ \hat{\bm{v}}$ coincides with $ \bm{v} $ on $ D_m\times[0,T] $. 
	Moreover, the azimuthal angular momentum of the velocity $\bm{v}$ is preserved, that is 
	\be\label{aam_con}
	\int_{D_m} \Gamma(x,t) \d x = \int_{D_m} \Gamma_0(x)\d x,\q\text{for any}\q t\in(0,T]\,,
	\ee
	where $\Gamma_0$ is the initial value of $\Gamma  \ed \rho\sin\phi\, v_\th $. In particular, if the initial azimuthal angular momentum $\int_{D_m}\Gamma_0(x) \d x$ vanishes, then $\int_{D_m} \Gamma(x,t) \d x = 0$ for any $t\in[0,T]$.
\end{prop}

\begin{proof}
	Since the domain $D_m$ is away from the symmetric axis and is a polyhedral domain, the existence of a strong solution $(\bm{v},P)$ and the uniqueness of the velocity $\bm{v}$ are standard. See for example the paper \cite{Benes}, where existence of strong solutions with a related boundary value problems on more general polyhedral domains was proven. For another instance, one can also follow the strategy in \cite[Section 3]{LPYZZZ24} to justify these results. We omit the details here.
	
	Next, we will verify the conservation property (\ref{aam_con}) of the azimuthal angular momentum. Recalling \eqref{NS}$_3$, the equation of $v_\th$, one has $\Gamma\ed \rho\sin\phi v_\theta$ satisfies
	\be\l{EGMA}
	\Delta \Gamma-\bm{b} \cdot \nabla \Gamma-\frac{2}{\rho} \partial_{\rho} \Gamma-\frac{2 \cot \phi}{\rho^{2}} \partial_{\phi} \Gamma-\partial_{t} \Gamma=0\,.
	\ee
	Integrating \eqref{EGMA} over $D_m$, one has
	\[
	\f{\d}{\d t}\int_{D_m}\Gamma\d x = 
	-\un{\int_{D_m} \Big(v_{\rho} \partial_{\rho} + \frac{1}{\rho} v_{\phi} \partial_{\phi}\Big)\Gamma\d x}_{I_1}
	+ \un{\int_{D_m}\Big(\Delta-\f{2}{\rho}\p_\rho-\f{2\cot\phi}{\rho^2}\p_\phi\Big)\Gamma\d x}_{I_2}\,.
	\]
	We will show both $I_1$ and $I_2$ vanish in order to justify the lemma.
	Noticing that $v_\rho=0$ on $A_{1,m}\cup A_{2,m}$, $v_\phi=0$ on $R_{1,m}\cup R_{2,m}$, it follows by integration by parts that
	\[
	I_1 = -\int_{D_m}\Gamma\Big(\frac{1}{\rho^2} \partial_\rho\left(\rho^2 v_\rho\right)+\frac{1}{\rho \sin \phi} \partial_\phi\left(\sin \phi v_\phi\right)\Big) \d x = 0\,.
	\]
	Moreover, recalling
	\[
	\Dl=\frac{1}{\rho^2}\p_\rho(\rho^2\p_\rho\,\cdot)+\frac{1}{\rho^2\sin\phi}\p_\phi(\sin\phi \, \p_\phi\,\cdot)
	\]
	in spherical coordinates, we have the following by integrating by parts
	\bn
	I_2 = & 2\pi\int_{{\f{\pi}{2}-\al}}^{{\f{\pi}{2}+\al}}\rho^2 \p_\rho\G \sin\phi\Big|_{\rho=\f{1}{m}}^1\d\phi 
	+ 2\pi\int_{\f{1}{m}}^1\sin\phi \, \p_\phi\G \Big|_{\phi={\f{\pi}{2}-\al}}^{{\f{\pi}{2}+\al}}\d\rho
	+4\pi\int_{\f{1}{m}}^1\int_{{\f{\pi}{2}-\al}}^{{\f{\pi}{2}+\al}}\sin\phi \, \G\d\phi\d\rho\\
	& -4\pi\int_{{\f{\pi}{2}-\al}}^{{\f{\pi}{2}+\al}}\sin\phi \, \G\Big|_{\rho=\f{1}{m}}^1\d\phi
	-4\pi\int_{\f{1}{m}}^1\int_{{\f{\pi}{2}-\al}}^{{\f{\pi}{2}+\al}}\sin\phi \, \G\d\phi\d\rho-4\pi\int_{\f{1}{m}}^1\G\cos\phi\Big|_{\phi={\f{\pi}{2}-\al}}^{{\f{\pi}{2}+\al}}\d\rho\\
	=&2\pi\int_{{\f{\pi}{2}-\al}}^{{\f{\pi}{2}+\al}}\rho^2\p_\rho\G\sin\phi \, \Big|_{\rho=\f{1}{m}}^1\d\phi+2\pi\int_{\f{1}{m}}^1\sin\phi \, \p_\phi\G\Big|_{\phi={\f{\pi}{2}-\al}}^{{\f{\pi}{2}+\al}}\d\rho
	-4\pi\int_{{\f{\pi}{2}-\al}}^{{\f{\pi}{2}+\al}}\sin\phi \, \G\Big|_{\rho=\f{1}{m}}^1\d\phi\\
	&-4\pi\int_\f{1}{m}^1\G\cos\phi\Big|_{\phi={\f{\pi}{2}-\al}}^{{\f{\pi}{2}+\al}}\d\rho\,.\\
	\en
	Finally, owing to $\p_\rho\G=\f{2}{\rho}\G$ on $A_{1,m}\cup A_{2,m}$ and $\p_\phi\G=2\cot\phi \, \G$ on $R_{1,m}\cup R_{2,m}$, we conclude
	\bn
	I_2=0\,.
	\en
	This completes the proof of the proposition.
\end{proof}

In the rest of this paper, all the discussions on the approximating domain $D_m$ assume that the velocity $\bm{v}$ is the strong solution of the problem (\ref{NS}) on $D_m$ under the boundary condition (\ref{bdry for Dm}) as stated in Proposition \ref{Prop, local soln in ad} with zero initial azimuthal angular momentum, that is 
\[
\int_{D_m} \Gamma_0(x) \d x = \int_{D_m} \rho \sin\phi \, v_{\th,0}(x) \d x = 0.
\]
Then thanks to the conservation property (\ref{aam_con}), the following identity holds for any time $t$:
\[
\int_{D_m} \Gamma(x,t) \d x = \int_{D_m} \rho \sin\phi \, v_{\th}(x,t) \d x = 0.
\]
Moreover, the following lemma shows $v_\rho$ has zero mean value for any $\rho$. This property is crucial to apply the Poincar\'e inequality in Lemma \ref{Poin0}.

\begin{lemma}\label{v_rho_mean0}
	Let
	\[
	H(\rho) = \int_{\f{\pi}{2}-\al}^{\f{\pi}{2}+\al} v_\rho(\rho,\phi)\sin\phi\d\phi, \quad \forall \, \rho\in \Big[\frac{1}{m}, 1\Big].
	\]
	Then $H(\rho)=0$ for any $\rho\in[\frac{1}{m}, 1]$.
\end{lemma}

\begin{proof}
	Denote
	\[
	f\ed\rho v_\rho\,.
	\]
	Direct calculation from \eqref{DO} shows
	\[
	\left\{\begin{aligned}
		&\Delta f = -\dfrac{\rho}{\sin\phi}\partial_\phi \Big(\m{O} \sin^2\phi+\f{2v_\phi\sin\phi \,\eta(\rho)}{\rho^2} \Big), \quad \text{ in } D_m,\\
		&\partial_\phi f = 0, \quad \text{ on } R_{1,m} \cup R_{2,m},\\
		&f = 0,  \quad \text{ on } A_{1,m} \cup A_{2,m}.
	\end{aligned}\right.
	\]
	This is equivalent to
	\be\label{f_eq2}
	\left\{\begin{aligned}
		&\left(\partial_\rho^2+\dfrac{2}{\rho}\partial_\rho
		+\dfrac{1}{\rho^2}\partial_\phi^2
		+\dfrac{\cot\phi}{\rho^2}\partial_\phi\right)f
		= -\dfrac{\rho}{\sin\phi}\p_\phi \Big(\m{O} \sin^2\phi+\f{2v_\phi\sin\phi \,\eta(\rho)}{\rho^2} \Big), \quad \text{ in }  D_m,\\
		&\partial_\phi f = 0,  \quad \text{ on }  R_{1,m} \cup R_{2,m},\\
		&f = 0, \quad \text{ on } A_{1,m} \cup A_{2,m}.
	\end{aligned}\right.
	\ee
	Multiplying (\ref{f_eq2}) by $\sin\phi$ and then integrating from $\f{\pi}{2}-\al$ to $\f{\pi}{2}+\al$,
	\begin{align*}
		& \quad H''(\rho)+\frac{2}{\rho}H'(\rho)
		+\frac{1}{\rho^2}
		\int_{\f{\pi}{2}-\al}^{\f{\pi}{2}+\al}
		(\sin\phi\,\partial_\phi^2 f+\cos\phi\,\partial_\phi f)\d\phi \\
		&= -\rho \bigg(\m{O} \sin^2\phi+\f{2v_\phi\sin\phi \,\eta(\rho)}{\rho^2} \bigg) \bigg|_{\phi = \f{\pi}{2}-\al}^{\f{\pi}{2}+\al}\,.
	\end{align*}
	Since
	\[
	\sin\phi\,\partial_\phi^2 f+\cos\phi\,\partial_\phi f = \partial_\phi(\sin\phi\,\partial_\phi f),
	\]
	and $\partial_\phi f=\m{O}=v_\phi=0$ on $\phi = {\f{\pi}{2}\pm\al}$, we have  
	\[
	H''(\rho) + \frac{2}{\rho} H'(\rho) = 0, \quad \forall\, \rho\in \Big( \frac{1}{m}, 1\Big).
	\]
	In addition, $H(\frac{1}{m}) = H(1) = 0$ since $v_\rho=0$ on $A_{1,m} \cup A_{2,m}$. 
	Hence, $H(\rho)=0$ for any $ \rho\in [\frac{1}{m}, 1]$.
\end{proof}

\subsection{A system for the triple of good unknowns $(\m{K}, \m{F}, \m{O})$}
In this subsection, we present the key system for three new functions $\m{K}, \m{F}$ and $\m{O}$ in the approximating domain $D_m$. Firstly, the authors in \cite{LPYZZZ24} considered a system consisting of the triple $(K, F, \O)$ defined as 
\[
K\ed\f{\o_\rho}{\rho}\,,\q F\ed\f{\o_\phi}{\rho}\,, \q \O\ed\f{\o_\th}{\rho\sin\phi}\,.
\] 
This triple satisfies a system as below (see \cite[(2.15)]{LPYZZZ24}):
\be\label{eq of K-F-O}
\left\{\begin{array}{ll}
	\big(\Delta +\frac{4}{\rho}\p_\rho +\frac{6}{\rho^2} \big) K - b\cdot\nabla K + \o\cdot\nabla \big(\frac{v_\rho}{\rho}\big)-\p_{t}K = 0,  \vspace{0.1in}\\
	\big(\Delta + \frac{2}{\rho}\p_\rho + \frac{1-\cot^2 \phi}{\rho^2}\big) F - b\cdot \nabla F +\frac{2}{\rho^2}\p_{\phi}K + \o \cdot \nabla \big(\frac{v_\phi}{\rho}\big)-\p_{t}F = 0, \vspace{0.1in}\\
	\big(\Delta +\frac{2}{\rho}\p_\rho +\frac{2\cot\phi}{\rho^2}\p_\phi \big) \O-b\cdot\nabla\O-\frac{2v_{\th}}{\rho\sin\phi}\,(K+\cot\phi\, F)-\p_{t}\O = 0.
\end{array}\right.\ee
As discussed in the introduction, this triple does not work well in the current paper due to the boundary conditions, so we modify them as in (\ref{DO}) and (\ref{DmK}) to adapt to the current boundary condition (\ref{bdry for Dm}). We collect these definitions as below.
\be\label{key-triple}
\left\{\begin{array}{ll}
	\m{K} = \f{\sin\phi}{\rho^2}\p_\phi\big(\f{v_\th}{\sin\phi}\big)= K - \frac{2 v_\th \cot\phi}{\rho^2}\,, \vspace{0.1in}\\
	\mF = -\p_\rho\big(\f{v_\th}{\rho}\big)= F + \frac{2 v_\th}{\rho^2}\, , \vspace{0.1in} \\ \m{O} = \f{1}{\rho\sin\phi}\big(\o_\th-\f{2v_\phi \eta}{\rho} \big)= \O - \frac{2 v_\phi \eta}{\rho^2 \sin\phi} \,,
\end{array}\right.
\ee
where $\eta$ is the cut-off function defined in (\ref{DCUT}).

Based on (\ref{eq of K-F-O}) and (\ref{key-triple}), we can derive the system (\ref{EOKF}) for the new triple $(\m{K}, \m{F}, \m{O})$. 
\be\l{EOKF}
\left\{
\begin{aligned}
	&\left(\Delta+\frac{4}{\rho} \partial_\rho+\frac{2-4\cot^2\phi}{\rho^2}\right) \m{K}-\bm{b} \cdot \nabla \m{K}+\bm{\omega} \cdot \nabla\Big(\frac{v_\rho}{\rho}\Big)-\partial_t \m{K}=f_{\m{K}}\,,\\[3mm]
	&\left(\Delta+\frac{2}{\rho} \partial_\rho-\frac{3+\cot^2\phi}{\rho^2}\right) \m{F}-\bm{b} \cdot \nabla \m{F}+\frac{2}{\rho^2} \partial_\phi \m{K}+\f{4\cot\phi}{\rho^2}\m{K}+\bm{\omega} \cdot \nabla\Big(\frac{v_\phi}{\rho}\Big)-\p_t\m{F}=f_{\mF}\,,\\[3mm]
	&\left(\Delta+\frac{2}{\rho} \partial_{\rho}+\frac{2 \cot \phi}{\rho^{2}} \partial_{\phi}\right) \m{O}-\bm{b} \cdot \nabla \m{O}-\frac{2 v_{\theta}}{\rho \sin \phi}\m{K}-\f{2v_\th\cos\phi}{\rho\sin^2\phi}\mF-\partial_{t} \m{O}=f_{\m{O}}\,.
\end{aligned}
\right.
\ee
Where
\be\label{rhs_KFO}\left\{
\begin{aligned}
	&f_{\m{K}}=-\frac{6\cot\phi}{\rho^3}v_\rho v_\th-\frac{2+2\cos^2\phi}{\rho^3\sin\phi}v_\phi v_\th\,,\\[3mm]
	&f_{\mF}=\frac{6}{\rho^3}v_\rho v_\th+\frac{2\cot\phi}{\rho^3}v_\phi v_\th\,, \\[3mm]
	&f_{\m{O}}=-\frac{2\eta^{\prime\prime}(\rho)}{\rho^2\sin\phi}v_\phi+\f{4\eta(\rho)}{\rho^3\sin\phi}\f{\p_\phi v_\rho}{\rho}+\frac{4}{\rho^2\sin\phi}\left(\f{\eta(\rho)}{\rho}-\eta^\prime(\rho)\right)\p_\rho v_\phi-\frac{2\eta(\rho)}{\rho^2\sin\phi}\f{\p_\phi P}{\rho}\\
	&\hskip 1cm+\frac{2}{\rho^2\sin\phi}\left(\eta^\prime(\rho)-\f{3\eta(\rho)}{\rho}\right)v_\phi v_\rho+\frac{2\eta(\rho)\cos\phi}{\rho^3\sin^2\phi}\left(v_\th^2-v_\phi^2\right)\,.
\end{aligned}
\right.\ee
We emphasize that the pressure term $P$ is involved in $f_O$ because the correction term (based on $\O$) for $\m{O}$ in (\ref{key-triple}) contains $v_\phi$ whose equation in (\ref{NS})$_{2}$ involves $P$. This forces us to control the pressure term $P$ which will be provided in Section \ref{Sec, pressure}. We also point out that although the expression of $f_{\m{O}}$ looks formidable, its support is away from the origin thanks to the cut-off function $\eta$. 

In addition to the system (\ref{EOKF}), we obtain the boundary conditions (\ref{BCOKF}) for the triple $(\m{K}, \m{F}, \m{O})$ on $D_m$ according to \eqref{NTSR}--\eqref{NTSA}--\eqref{NTS-NHL}--\eqref{vor-sph}.
\be\l{BCOKF}
\left\{
\begin{aligned}
	&\m{K}=0\,,\q\text{on}\q R_{1,m}\cup R_{2,m}\,;\q\p_\rho\m{K}=-\f{1}{\rho}\m{K}\,,\q\text{on}\q A_{1,m}\cup A_{2,m}\,;\\[2mm]
	&\p_\phi\m{F}=\cot\phi \, \m{F}\q\text{on}\q R_{1,m}\cup R_{2,m}\,;\q\m{F}=0\,,\q\text{on}\q A_{1,m}\cup A_{2,m}\,;\\[2mm]
	&\m{O}=0\,,\q\text{on}\q \p D_m\,.
\end{aligned}
\right.
\ee

\section{Estimates of the pressure $P$ in $D_m$}\label{Sec, pressure}

Unlike the case studied in \cite{LPYZZZ24} with the NHL boundary condition, the modified good unknown $\m{O}$, which was introduced in (\ref{key-triple}) to adapt to the NTS boundary condition, contains $v_\phi$. Consequently, the pressure term $P$ appearing in the $v_\phi$ equation must be estimated before a closed energy estimate for $\m{O}$ can be obtained. In this section, we first establish the boundary value problem of $P$ in the following Lemma \ref{Lemma, BVPP}, and then solve this problem by giving a gradient $L^2$ bound of $P$ in Lemma \ref{EPs}.

\subsection{Boundary value problem for the pressure}
\begin{lemma}[BVP of pressure]\l{Lemma, BVPP}
	The pressure stated in the \eqref{NS} satisfies the following Neumann boundary value problem:
	\be\label{BVPP}
	\left\{
	\begin{aligned}
		&-\Dl P=F\,,&\q\text{in}\q D_m\,;\\
		&\p_\rho P=P_A\,,&\q\text{on}\q A_{1,m}\cup A_{2,m}\,;\\
		&\frac{1}{\rho}\p_\phi P=P_R\,,&\q\text{on}\q R_{1,m}\cup R_{2,m}\,.\\
	\end{aligned}
	\right.
	\ee
	Where
	\be\label{F}
	\begin{aligned}
		F=&\frac{1}{\rho^{2}} \partial_{\rho}\left[\rho^{2}\left(v_{\rho} \partial_{\rho}+\frac{1}{\rho} v_{\phi} \partial_{\phi}\right) v_{\rho}\right]+\frac{1}{\rho \sin \phi} \partial_{\phi}\left(\sin \phi\left[v_{\rho}\left(\partial_{\rho}+\frac{1}{\rho}\right)+\frac{1}{\rho} v_{\phi} \partial_{\phi}\right] v_{\phi}\right)\\
		&-\frac{1}{\rho} \partial_{\rho}\left(v_{\phi}^{2}+v_{\theta}^{2}\right)-\frac{\cot \phi}{\rho^{2}} \partial_{\phi}\left(v_{\theta}^{2}\right)-\frac{1}{\rho^{2}} v_{\phi}^{2}\,;\\
	\end{aligned}
	\ee
	and
	\ba\l{0428-3}
	\left\{
	\begin{aligned}
		&P_A = \frac{1}{\rho}\left(v_{\phi}^{2}+v_{\theta}^{2}\right) - \frac{2\eta(\rho)}{\rho^{2}\sin\phi}\p_\phi\left(\sin\phi \, v_{\phi}\right)\,;\\
		&P_R = \frac{\cot\phi}{\rho} v_\th^2\,,\\
	\end{aligned}
	\right.
	\ea
	where $\eta$ is the cut-off function in (\ref{DCUT}).
\end{lemma}

\begin{rem}
	We point out that $P_A$ is only evaluated at $\rho=\f{1}{m}, 1$ in \eqref{BVPP}. As a result of (\ref{DCUT})$_1$, 
	\[
	P_{A} = \frac{1}{\rho}(v_{\phi}^{2}+v_{\theta}^{2}) \,\text{ on }\, A_{1,m}, 	\quad\text{and}\quad 
	P_{A} = \frac{1}{\rho}(v_{\phi}^{2}+v_{\theta}^{2}) - 	\frac{2}{\rho^{2}\sin\phi}\p_\phi\left(\sin\phi \, v_{\phi}\right) 
	\,\text{ on } \, A_{2,m}.
	\]
	In addition, we note that the derivatives $\p_\rho P$ and $\frac{1}{\rho}\p_\phi P$ in (\ref{BVPP}) may differ from the normal derivative $\p_{n} P$ by a sign. More precisely, 
	\[
	\p_\rho P = -\p_{n} P \,\text{ on }\, A_{1,m}, 	\quad\text{and}\quad 
	\frac{1}{\rho}\p_\phi P = -\p_{n} P \,\text{ on } \, R_{1,m}.
	\]
\end{rem}

\begin{proof}
	Performing the divergence operator on the momentum equation, one derives
	\[
	-\Delta P=\operatorname{div}((\bm{v} \cdot \nabla) \bm{v})\,.
	\]
	In view of \eqref{NS}$_{1,2}$, it follows that
	\[
	\begin{aligned}
		\operatorname{div}((\bm{v} \cdot \nabla) \bm{v}) = &\operatorname{div}\bigg\{ \Big[\Big(v_\rho\p_\rho + \f{v_\phi}{\rho}\p_\phi\Big)v_\rho - \f{1}{\rho}(v_\phi^2+v_\th^2)\Big]\bm{e_\rho}\\
		& + \bigg( \Big[ v_\rho \Big(\p_\rho + \f{1}{\rho}\Big) + \f{v_\phi}{\rho}\p_\phi\Big] v_\phi - \f{\cot \phi}{\rho}v_\th^2\bigg) \bm{e_\phi}\bigg\}\\
		= & \frac{1}{\rho^{2}} \partial_{\rho}\left[\rho^{2}\left(v_{\rho} \partial_{\rho}+\frac{1}{\rho} 	v_{\phi} \partial_{\phi}\right) v_{\rho}\right]+\frac{1}{\rho \sin \phi} \partial_{\phi}\left(\sin \phi\left[v_{\rho}\left(\partial_{\rho}+\frac{1}{\rho}\right)+\frac{1}{\rho} v_{\phi} \partial_{\phi}\right] v_{\phi}\right)\\
		&-\frac{1}{\rho} \partial_{\rho}\left(v_{\phi}^{2}+v_{\theta}^{2}\right)-\frac{\cot \phi}{\rho^{2}} \partial_{\phi}\left(v_{\theta}^{2}\right)-\frac{1}{\rho^{2}} v_{\phi}^{2}\,.
	\end{aligned}
	\]
	See also \cite[(A.3) and (A.10)]{LPYZZZ24} for detailed calculations. This infers the equation of pressure in \eqref{BVPP}$_1$. In the following, we will derive the boundary conditions of $P$ from equations \eqref{NS}$_{1,2}$.\\[2mm]
	
	\noindent\textbf{Boundary condition on arcs:}
	First, we recall the Laplacian operator in the spherical coordinates (see \cite[(A.5)]{LPYZZZ24}):
	\[
	\Dl = \frac{1}{\rho^2}\p_\rho(\rho^2\p_\rho\,\cdot) + \frac{1}{\rho^2\sin\phi}\p_\phi(\sin\phi 	\,\p_\phi\,\cdot) + \frac{1}{\rho^2\sin^2\phi} \p_\th^2(\cdot).
	\]
	In particular, for axisymmetric functions which are independent of $\th$, the Laplacian operator is reduced to 
	\be\label{Lap_sph}
	\Dl = \frac{1}{\rho^2}\p_\rho(\rho^2\p_\rho\,\cdot) + \frac{1}{\rho^2\sin\phi}\p_\phi(\sin\phi 	\,\p_\phi\,\cdot).
	\ee
	Then it follows from the equation of $v_\rho$ in \eqref{NS} that
	\[
	\begin{aligned}
		\partial_{\rho} P &= \left(\partial_{\rho}^{2} v_{\rho} + \frac{4}{\rho} \partial_{\rho}+\frac{1}{\rho^{2}} \partial_{\phi}^{2}+\frac{\cot \phi}{\rho^{2}} \partial_{\phi}+\frac{2}{\rho^{2}}\right) v_{\rho}
		- \Big( v_\rho \p_\rho + \frac{v_\phi}{\rho} \p_\phi \Big) v_\rho \\
		& \quad +\frac{1}{\rho}\left(v_{\phi}^{2}+v_{\theta}^{2}\right)-\left(v_{\rho} \partial_{\rho}+\frac{1}{\rho} v_{\phi} \partial_{\phi}\right) v_{\rho} - \partial_{t} v_{\rho}\,.
	\end{aligned}
	\]
	On $A_{1,m}\cup A_{2,m}$, $v_\rho$, its tangent derivative $\p_\phi v_\rho$, and its time derivative $\p_t v_\rho$ vanish owing to boundary conditions \eqref{NTSA} and \eqref{NTS-NHL}. Therefore,
	\be\label{Bpp}
	\partial_{\rho} P= \partial_{\rho}^{2} v_{\rho} + \frac{4}{\rho} \partial_{\rho} v_{\rho} + \frac{1}{\rho}\left(v_{\phi}^{2}+v_{\theta}^{2}\right)\,.
	\ee
	The second derivative term $\p_\rho^2 v_\rho$ on the boundary is difficult to control, we will take advantage of the divergence-freeness of $\bm{v}$ to reduce this term to be lower-order terms (see (\ref{2nd_rho_deri}) in the sequel). Since 
	\[
	\mathrm{div\,}\bm{v} = \frac{1}{\rho^2}\p_\rho(\rho^2 v_\rho) + \frac{1}{\rho \sin\phi} \p_\phi(\sin\phi \, v_\phi),
	\]
	it then follows from the divergence-freeness of $\bm{v}$ that $\mathrm{div\,}\bm{v}=0$ and $ \p_\rho\mathrm{div\,}\bm{v}=0$ in $D_m$, which implies
	\be\label{div-free}
	\frac{1}{\rho^2}\p_\rho(\rho^2 v_\rho) = - \frac{1}{\rho \sin\phi} \p_\phi(\sin\phi \, v_\phi)
	\ee
	and
	\be\label{2nd_rho_deri}
	\partial_{\rho}^{2} v_{\rho}=\frac{2}{\rho^{2}} v_{\rho}-\frac{2}{\rho} \partial_{\rho} 	v_{\rho}+\frac{1}{\rho^{2}\sin\phi}\p_\phi\left(\sin\phi \, v_{\phi}\right) - \frac{1}{\rho\sin\phi} \p_\phi\left(\sin\phi \, \p_{\rho} v_{\phi}\right)\,.
	\ee
	Thus on $A_{1,m}$, the combination of the boundary conditions $v_\rho=0$, $\partial_{\rho} v_{\phi}=-\frac{1}{\rho} v_{\phi}$ and the divergence-freeness of $\bm{v}$ indicates
	\be\label{vpp1}
	\partial_{\rho}^{2} v_{\rho} = -\frac{2}{\rho} \partial_{\rho} v_{\rho}+\frac{2}{\rho^{2}\sin\phi}\p_\phi\left(\sin\phi \, v_{\phi}\right) = -\frac{4}{\rho} \partial_{\rho} v_{\rho}\,.
	\ee
	Meanwhile, on $A_{2,m}$, we apply the boundary conditions $v_\rho = 0$ and $\partial_{\rho} v_{\phi}=\frac{1}{\rho} v_{\phi}$ to get
	\be\label{vpp2}
	\partial_{\rho}^{2} v_{\rho} = -\frac{2}{\rho} \partial_{\rho} v_{\rho}\,.
	\ee
	Substituting \eqref{vpp1} and \eqref{vpp2} in \eqref{Bpp}, one finds
	\[
	\p_\rho P=\left\{
	\begin{aligned}
		&\frac{1}{\rho}\left(v_{\phi}^{2}+v_{\theta}^{2}\right),\q\text{on}\q A_{1,m}\,;\\[2mm]
		&\frac{1}{\rho}\left(v_{\phi}^{2}+v_{\theta}^{2}\right) + \frac{2}{\rho} \partial_{\rho} v_{\rho} \q\text{on}\q A_{2,m} \,.
	\end{aligned}
	\right.
	\]
	Using the divergence-freeness (\ref{div-free}) again, we know 
	$\frac{2}{\rho} \partial_{\rho} v_{\rho} = -\frac{2}{\rho^{2}\sin\phi}\p_\phi\left(\sin\phi \, v_{\phi}\right)$ on $A_{2,m}$. This relation is crucial to apply the integration by parts to the terms $I_{121}$ and $I_{122}$ in (\ref{I12}) in order to estimate the term $I_{12}$ in (\ref{EPP1}). So we adopt the following expression for $\p_\rho P$.
	\[
	\p_\rho P=\left\{
	\begin{aligned}
		&\frac{1}{\rho}\left(v_{\phi}^{2}+v_{\theta}^{2}\right),\q\text{on}\q A_{1,m}\,;\\[2mm]
		&\frac{1}{\rho}\left(v_{\phi}^{2}+v_{\theta}^{2}\right)-\frac{2}{\rho^{2}\sin\phi}\p_\phi\left(\sin\phi \, v_{\phi}\right),\q\text{on}\q A_{2,m}\,.
	\end{aligned}
	\right.
	\]
	This concludes the boundary condition (\ref{BVPP})$_{2}$ of $\p_\rho P$ on arcs $A_{1,m} \cup A_{2,m}$. \\[2mm]
	
	\noindent\textbf{Boundary condition on rays:}
	
	By the equation of $v_\phi$ in \eqref{NS},
	\[
	\begin{aligned}
		\frac{1}{\rho} \partial_{\phi} P &= \left( \p_\rho^2 + \frac{2}{\rho}\p_\rho + \frac{1}{\rho^2}\p_\phi^2 + \frac{\cot\phi}{\rho^2}\p_\phi - \frac{1}{\rho^2 \sin ^2 \phi}\right) v_\phi 
		- \Big( v_\rho \p_\rho + \frac{v_\phi}{\rho} \p_\phi \Big) v_\phi+\frac{2}{\rho^2} \partial_\phi v_\rho \\
		&\quad -\frac{1}{\rho} v_\rho v_\phi + \frac{\cot \phi}{\rho} v_\theta^2 - \p_t v_\phi.
	\end{aligned}
	\]
	On $R_{1,m}\cup R_{2,m}$, $v_\phi$ and its tangent derivative $\p_\rho v_\phi$ and temporal derivative $\p_t v_\phi$ vanish owing to boundary conditions \eqref{NTSR}. Therefore,
	\be\label{Bpp2}
	\frac{1}{\rho} \p_{\phi} P = \frac{1}{\rho^2} \p_{\phi}^{2} v_{\phi} + \frac{\cot\phi}{\rho^2} \p_{\phi} v_{\phi} + \frac{\cot \phi}{\rho} v_\theta^2\,.
	\ee
	Again, we will take advantage of the divergence-freeness of $\bm{v}$ to reduce the second-order derivative $\p_\phi^2 v_\phi$ to be lower-order terms (see (\ref{2nd_phi_deri}) in the sequel). Since 
	\[
	\mathrm{div\,}\bm{v} = \frac{1}{\rho^2}\p_\rho(\rho^2 v_\rho) + \frac{1}{\rho \sin\phi} \p_\phi(\sin\phi \, v_\phi),
	\]
	it then follows from the divergence-freeness of $\bm{v}$ that $\mathrm{div\,}\bm{v}=0$ and $ \p_\phi\mathrm{div\,}\bm{v}=0$ in $D_m$, which implies
	\bn
	\frac{1}{\rho \sin\phi} \p_\phi(\sin\phi \, v_\phi) = -\frac{1}{\rho^2}\p_\rho(\rho^2 v_\rho)
	\en
	and
	\be\label{2nd_phi_deri}
	\frac{1}{\rho} \p_{\phi}^{2} v_{\phi} = -\frac{1}{\rho} \p_\phi(\cot\phi \, v_\phi) - \p_\rho \p_\phi v_\rho - \frac{2}{\rho} \p_\phi v_\rho\,.
	\ee
	Then the combination of (\ref{2nd_phi_deri}) and \eqref{NTSR} yields 
	\be\label{vtt2}
	\frac{1}{\rho}\p_\phi^2v_\phi = -\f{\cot\phi}{\rho}\p_\phi v_\phi,\q\text{on}\q R_{1,m}\cup R_{2,m}\,.
	\ee
	Substituting \eqref{vtt2} in \eqref{Bpp2}, one concludes
	\[
	\frac{1}{\rho}\p_\phi P=\frac{\cot\phi}{\rho} v_\th^2,\q\text{on}\q R_{1,m}\cup R_{2,m}\,.
	\]
	This finishes the proof of the lemma.
\end{proof}

\subsection{Gradient $L^2$ bound of the pressure}
The following lemma is the main ingredient of this section. Note that there are no boundary integrals of the pressure $P$ in the estimate (\ref{EPS}). It is well known that the analysis of the boundary behavior of the pressure $P$ is tricky. Fortunately, we manage to control them by subtracting low-order auxiliary functions given in \eqref{0428-3}, see $I_1$ in (\ref{EPP1}) and $I_2$ in (\ref{EPP2}).

\begin{lemma}[Gradient $L^2$ bound of pressure]\l{EPs}
	The pressure given in \eqref{NS} satisfies
	\be\label{EPS}
	\begin{aligned}
		\|\nabla P\|_{L^{2}} \leq & (1+\cot\al)\Big\|\frac{1}{\rho}\left(v_{\phi}^{2}+v_{\theta}^{2}\right)\Big\|_{L^{2}}+6\|\p_\rho v_\phi\|_{L^2}+6\Big\|\frac{\partial_{\phi}v_{\phi}}{\rho} \Big\|_{L^{2}}+18(2+\cot\al)\|v_\phi\|_{L^2}\\
		& + \Big\| \Big(v_{\rho} \partial_{\rho}+\frac{1}{\rho} v_{\phi} \partial_{\phi}\Big) v_{\rho}\Big\|_{L^{2}}+
		\Big\| \Big[v_{\rho}\Big(\partial_{\rho}+\frac{1}{\rho}\Big)+\frac{1}{\rho} v_{\phi} \partial_{\phi}\Big] v_{\phi}\Big\|_{L^{2}}\,,
	\end{aligned}
	\ee
	where $L^2$ means $L^2(D_m)$.
\end{lemma}

\begin{proof}
	Multiplying $P$ on both sides of \eqref{BVPP}$_1$ and integrating over $D_m$, we derive
	\be\label{EP1}
	-\int_{D_{m}}(\Delta P) P \dx =\int_{D_m}PF\dx\,.
	\ee
	Based on the formula (\ref{Lap_sph}) for the Laplacian in spherical coordinates, the left hand side of \eqref{EP1} reads as
	\be\label{Lap}
	-\int_{D_{m}}(\Delta P) P \dx=\un{-\int_{D_{m}} \frac{1}{\rho^{2}} \partial_{\rho}\left(\rho^{2} \partial_{\rho} P\right) P \dx}_{I_1}\un{-\int_{D}\frac{1}{\rho^{2} \sin \phi} \partial_{\phi}\left(\sin \phi \, \partial_{\phi} P\right) P \dx}_{I_2}\,.
	\ee
	For $I_1$ and $I_2$, the absence of boundary values of the pressure $P$ makes it ineligible to perform the integration by parts directly, so we will subtract $P_{A}$ from $\p_\rho P$ in $I_1$ and subtract $\rho P_{R}$ from $\p_\phi P$ in $I_2$ to justify the integration by parts without boundary integrals, see (\ref{BVPP}) and (\ref{0428-3}) for the definitions of $P_A$ and $P_{R}$. More precisely, 
	\be\label{EPP1}
	\begin{aligned}
		I_{1}= & -2 \pi \int_{\frac{\pi}{2}-\alpha}^{\frac{\pi}{2}+\alpha} \bigg[ \int_{\frac{1}{m}}^{1}\Big( \partial_{\rho}\big[ \rho^{2}\left( \partial_{\rho} P-P_{A}\right)\big] P + \partial_{\rho}\left(\rho^2P_{A}\right) P \Big) \d \rho\bigg] \sin \phi \d\phi \\
		= & 2 \pi \int_{\frac{\pi}{2}-\alpha}^{\frac{\pi}{2}+\alpha} \bigg( \int_{\frac{1}{m}}^{1} \rho^{2} \left(  \partial_{\rho} P-P_{A}\right) (\p_{\rho} P) \d \rho\bigg)  \sin \phi  \d \phi \\
		& - 2 \pi \int_{\frac{\pi}{2}-\alpha}^{\frac{\pi}{2}+\alpha} \bigg(\int_{\frac{1}{m}}^{1}\partial_{\rho}\left( \rho^2 P_{A}\right) P \d \rho\bigg) \sin \phi \d \phi \\
		= & \int_{D_{m}}\left|\partial_{\rho} P\right|^{2} \dx-\un{\int_{D_{m}}P_{A} \partial_{\rho} P \dx}_{I_{11}}-\un{\int_{D_m} \frac{1}{\rho^{2}}\partial_{\rho}\left(\rho^2P_{A}\right) P \dx}_{I_{12}}\,, \\
	\end{aligned}
	\ee
	and
	\be\label{EPP2}
	\begin{aligned}
		I_{2}= & -2 \pi \int_{\frac{1}{m}}^{1} \bigg[ \int_{\frac{\pi}{2}-\alpha}^{\frac{\pi}{2}+\alpha} \Big( \partial_{\phi}\big[ \sin \phi \left(\partial_{\phi} P-\rho P_{R}\right)\big] P + \rho\p_{\phi}\left(\sin\phi P_{R}\right) P \Big)\d\phi \bigg] \d\rho \\
		= & 2 \pi \int_{\frac{1}{m}}^{1} \int_{\frac{\pi}{2}-\alpha}^{\frac{\pi}{2}+\alpha} \sin \phi  \left(\partial_{\phi} P-\rho P_{R}\right) \partial_{\phi} P \d \phi\d\rho
		-2 \pi \int_{\frac{1}{m}}^{1} \int_{\frac{\pi}{2}-\alpha}^{\frac{\pi}{2}+\alpha}\rho\p_{\phi}\left(\sin\phi P_{R}\right) P \d\phi\d\rho \\
		= & \int_{D_{m}} \Big|\frac{1}{\rho} \partial_{\phi} P\Big|^{2} \d x - \un{\int_{D_{m}} \frac{1}{\rho} P_{R} \partial_{\phi} P \d x}_{I_{21}} - \un{\int_{D_{m}} \frac{1}{\rho\sin \phi}\partial_{\phi}\left(\sin\phi P_{R}\right) P \d x}_{I_{22}}\,.
	\end{aligned}
	\ee
	The terms $I_{12}$ and $I_{22}$ involve derivatives of $P_{A}$ and $P_{R}$ respectively, so we discuss them first. Recalling the expressions for $P_{A}$ and $P_{R}$ in (\ref{0428-3}):
	\[	P_A=\frac{1}{\rho}\left(v_{\phi}^{2}+v_{\theta}^{2}\right)-\frac{2\eta(\rho)}{\rho^{2}\sin\phi}\p_\phi\left(\sin\phi v_{\phi}\right)\,,\qq P_R=\frac{\cot\phi}{\rho}v_\th^2\,,
	\]
	one has
	\be\label{I12}
	\begin{aligned}
		I_{12}= & \int_{D_{m}}\left[\frac{1}{\rho^{2}}\left(v_{\phi}^{2}+v_{\theta}^{2}\right)+\frac{1}{\rho} \partial_{\rho}\left(v_{\phi}^{2}+v_{\theta}^{2}\right)\right] P \d x-2 \un{\int_{D_{m}} \frac{\eta(\rho) }{\rho^{2}\sin\phi}\p_\phi\left(\sin\phi \, \partial_{\rho} v_{\phi}\right) P \d x}_{I_{121}} \\
		& -2 \un{\int_{D_{m}} \frac{\eta'(\rho)}{\rho^{2}\sin\phi}\p_\phi\left(\sin\phi \, v_\phi\right)P \d x}_{I_{122}}\,,
	\end{aligned}
	\ee
	and
	\be\label{I22}
	\begin{aligned}
		I_{22} = -\int_{D_{m}} \frac{v_{\theta}^{2}}{\rho^{2}} P \d x+\int_{D_{m}} \frac{\cot\phi}{\rho^{2}} \partial_{\phi}\left(v_{\theta}^{2}\right) P \d x\,.
	\end{aligned}
	\ee
	
	Substituting \eqref{EPP1}--\eqref{I22} into \eqref{Lap} yields
	\[\begin{split}
		-\int_{D_{m}}(\Delta P) P \dx &= \int_{D_m}\Big(|\p_\rho P|^2+\big|\frac{1}{\rho}\p_\phi P\big|^2\Big)\d x - (I_{11} + 2I_{121} + 2I_{122} + I_{21})\\
		& \quad - \int_{D_m} \bigg[ \frac{1}{\rho^{2}} v_{\phi}^{2} + \frac{1}{\rho} \partial_{\rho}\left(v_{\phi}^{2} + v_{\theta}^{2} \right)
		+ \frac{\cot\phi}{\rho^{2}} \partial_{\phi}\left(v_{\theta}^{2}\right) \bigg] \, P \d x \,.
	\end{split}\]	
	Combining the above identity with \eqref{EP1} and \eqref{F}, one finds
	\be\label{ESSS}
	\int_{D_m}\Big(|\p_\rho P|^2+\big|\frac{1}{\rho}\p_\phi P\big|^2\Big)\d x = I_{11}+2I_{121}+2I_{122}+I_{21}+\int_{D_m}F_1P\d x\,,
	\ee
	where 
	\[
	F_1=\frac{1}{\rho^{2}} \partial_{\rho}\left[\rho^{2}\left(v_{\rho} \partial_{\rho}+\frac{1}{\rho} v_{\phi} \partial_{\phi}\right) v_{\rho}\right]+\frac{1}{\rho \sin \phi} \partial_{\phi}\left(\sin \phi\left[v_{\rho}\left(\partial_{\rho}+\frac{1}{\rho}\right)+\frac{1}{\rho} v_{\phi} \partial_{\phi}\right] v_{\phi}\right)\,.
	\]
	
	Now we estimate the right hand side of \eqref{ESSS} term by term. First, we see
	\[
	\left|I_{11}\right| \leq\left|\int_{D_{m}} \frac{1}{\rho}\left(v_{\phi}^{2}+v_{\th}^{2}\right)\partial_{\rho} P \d x\right|+2\left|\int_{D_{m}} \frac{\eta(\rho)}{\rho^2}\left(\p_\phi v_\phi+\cot\phi\, v_\phi\right) \p_\rho P \d x\right|.
	\]
	Since the support of $\eta$ is in $\{\rho: \rho \geq \frac13\}$, then 
	$\frac{\eta(\rho)}{\rho^2} \leq \min\{ \frac{3}{\rho}, 9 \}$.
	Consequently, 
	\be\label{EI11}
	\begin{aligned}
		I_{11} & \leq\left\|\partial_{\rho} P\right\|_{L^{2}}\Big\|\frac{1}{\rho}\left(v_{\phi}^{2}+v_{\theta}^{2}\right)\Big\|_{L^{2}}
		+ 6\left\|\partial_{\rho} P\right\|_{L^{2}}\Big\|\frac{\partial_{\phi} v_\phi}{\rho} \Big\|_{L^{2}}+18\cot\al\|\p_\rho P\|_{L^2}\|v_\phi\|_{L^2}\,.
	\end{aligned}
	\ee
	Next, we write $I_{121}$ in polar coordinates as
	\[
	I_{121} = 2\pi \int_{\frac1m}^{1} \int_{\pi/2-\al}^{\pi/2+\al} \eta(\rho)\, \p_\phi(\sin\phi \, \p_\rho v_\phi) P \d\phi \d\rho. 
	\]
	By noticing $\p_\rho v_\phi=0$ on rays, we apply the integration by parts to obtain
	\[\begin{split}
		I_{121} &= -2\pi \int_{\frac1m}^{1} \int_{\pi/2-\al}^{\pi/2+\al} \eta(\rho) \sin\phi \, (\p_\rho v_\phi) \, (\p_\phi P) \d\phi \d\rho \\
		&= -\int_{D_m}\frac{\eta(\rho)}{\rho} (\p_\rho v_\phi) \Big(\frac{1}{\rho}\p_\phi P\Big) \d x\,.
	\end{split}\]
	This indicates
	\be\label{EI121}
	|I_{121}|\leq 3\|\p_\rho v_\phi\|_{L^2} \Big\|\frac{1}{\rho}\p_\phi P \Big\|_{L^2}\,.
	\ee
	Similarly, 
	\[
	I_{122} = 2\pi \int_{\frac1m}^{1} \int_{\pi/2-\al}^{\pi/2+\al}\eta'(\rho)\p_\phi(\sin\phi \, v_\phi)P \d\phi \d\rho
	= -\int_{D_m} \Big(\frac{\eta'(\rho)}{\rho}\Big) v_\phi \Big(\frac{1}{\rho}\p_\phi P\Big) \d x\,,
	\]
	which follows that
	\be\label{EI122}
	|I_{122}|\leq 18\|v_\phi\|_{L^2}\Big\|\frac{1}{\rho}\p_\phi P\Big\|_{L^2}\,.
	\ee
	
	On the other hand, 
	\be\label{EI21}
	\begin{aligned}
		|I_{21}|\leq& \cot\alpha\left|\int_{D_{m}} \frac{v_{\theta}^{2}}{\rho} \cdot \frac{1}{\rho} \partial_{\phi}P \d x\right|\leq\cot \alpha\cdot \Big\|\frac{1}{\rho}\partial_{\phi}P\Big\|_{L^{2}}\Big\| \frac{v_{\theta}^{2}}{\rho}\Big\|_{L^{2}}\,.
	\end{aligned}
	\ee
	Now it remains to bound the last term in \eqref{ESSS}. Indeed
	\[
	\begin{aligned}
		\int_{D_m}F_1P\d x=&\un{\int_{D_{m}} \frac{1}{\rho^{2}} \partial_{\rho}\left[\rho^{2}\left(v_{\rho} \partial_{\rho}+\frac{1}{\rho} v_{\phi} \partial_{\phi}\right) v_{\rho}\right] P \d x}_{I_3}\\
		&+\un{\int_{D_{m}} \frac{1}{\rho \sin \phi} \partial_{\phi}\left(\sin \phi\left[v_{\rho}\left(\partial_{\rho}+\frac{1}{\rho}\right)+\frac{1}{\rho} v_{\phi} \partial_{\phi}\right] v_{\phi}\right) P \d x}_{I_4}\,.
	\end{aligned}
	\]
	Representing $I_3$ in the spherical coordinates and using integration by parts, noticing that $v_\rho=\p_\phi v_\rho=0$ on arcs, one deduces
	\[
	\begin{aligned}
		I_3&=2 \pi \int_{\frac{\pi}{2}-\alpha}^{\frac{\pi}{2}+\alpha} \int_{\frac{1}{m}}^{1} \sin \phi \, \partial_{\rho}\left[\rho^{2}\left(v_{\rho} \partial_{\rho}+\frac{1}{\rho} v_{\phi} \partial_{\phi}\right) v_{\rho}\right] P \d \rho \d \phi \\
		& =-2 \pi \int_{\frac{\pi}{2}-\alpha}^{\frac{\pi}{2}+\alpha} \int_{\frac{1}{m}}^{1} \sin \phi \, \rho^{2}\left(v_{\rho} \partial_{\rho}+\frac{1}{\rho} v_{\phi} \partial_{\phi}\right) v_{\rho} \partial_{\rho} P \d \rho \d \phi \\
		& =-\int_{D_{m}}\left(v_{\rho} \partial_{\rho}+\frac{1}{\rho} v_{\phi} \partial_{\phi}\right) v_{\rho} \partial_{\rho} P \d x\,.
	\end{aligned}
	\]
	This indicates
	\be\label{EI3}
	\left|I_3\right| \leq\left\|\partial_{\rho} P\right\|_{L^{2}}\left\|\left(v_{\rho} \partial_{\rho}+\frac{1}{\rho} v_{\phi} \partial_{\phi}\right) v_{\rho}\right\|_{L^{2}}\,.
	\ee
	Similarly, it holds for $I_4$ that
	\[
	\begin{aligned}
		I_{4} & =2 \pi \int_{\frac{1}{m}}^{1} \int_{\frac{\pi}{2}-\alpha}^{\frac{\pi}{2}+\alpha} \rho \partial_{\phi}\left(\sin \phi\left[v_{\rho}\left(\partial_{\rho}+\frac{1}{\rho}\right)+\frac{1}{\rho} v_{\phi} \partial_{\phi}\right] v_{\phi}\right) P \d \phi \d \rho \\
		& =-2 \pi \int_{\frac{1}{m}}^{1} \int_{\frac{\pi}{2}-\alpha}^{\frac{\pi}{2}+\alpha} \rho \sin \phi\left[v_{\rho}\left(\partial_{\rho}+\frac{1}{\rho}\right)+\frac{1}{\rho} v_{\phi} \partial_{\phi}\right] v_{\phi} \partial_{\phi} P \d \phi \d \rho \\
		& =-\int_{D_{m}}\left[v_{\rho}\left(\partial_{\rho}+\frac{1}{\rho}\right)+\frac{1}{\rho} v_{\phi} \partial_{\phi}\right] v_{\phi} \Big(\frac{1}{\rho} \partial_{\phi} P\Big) \d x.
	\end{aligned}
	\]
	This follows that
	\be\label{EI4}
	|I_4|\leq\Big\|\frac{1}{\rho} \partial_{\phi} P\Big\|_{L^{2}}\Big\|\Big[v_{\rho}\Big(\partial_{\rho}+\frac{1}{\rho}\Big)+\frac{1}{\rho} v_{\phi} \partial_{\phi}\Big] v_{\phi}\Big\|_{L^{2}}\,.
	\ee
	
	Finally, substituting \eqref{EI11}--\eqref{EI121}--\eqref{EI122}--\eqref{EI21}--\eqref{EI3}--\eqref{EI4} into \eqref{ESSS}, one deduces
	\[
	\begin{aligned}
		&\|\p_\rho P\|_{L^2}^2+\big\|\frac{1}{\rho}\p_\phi P\big\|_{L^2}^2\\
		\leq& \left\|\partial_{\rho} P\right\|_{L^{2}}\left(\Big\|\frac{1}{\rho}\left(v_{\phi}^{2}+v_{\theta}^{2}\right)\Big\|_{L^{2}}+6\Big\|\frac{\partial_{\phi}}{\rho} v_{\phi}\Big\|_{L^{2}}+18\cot\al\|v_\phi\|_{L^2} + \Big\|\Big(v_{\rho} \partial_{\rho}+\frac{1}{\rho} v_{\phi} \partial_{\phi}\Big) v_{\rho}\Big\|_{L^{2}}\right)\\
		& + \Big\|\frac{1}{\rho}\p_\phi P\Big\|_{L^2}\left(6\|\p_\rho v_\phi\|_{L^2}+36\|v_\phi\|_{L^2}+\cot\al\Big\| \frac{v_{\theta}^{2}}{\rho}\Big\|_{L^{2}}+\Big\|\Big[v_{\rho}\Big(\partial_{\rho}+\frac{1}{\rho}\Big)+\frac{1}{\rho} v_{\phi} \partial_{\phi}\Big] v_{\phi}\Big\|_{L^{2}}\right)\,.
	\end{aligned}
	\]
	Thus, we conclude \eqref{EPS} since 
	\[
	\|\nabla P\|_{L^2}^2 = \|\p_\rho P\|_{L^2}^2 + \Big\|\frac{1}{\rho}\p_\phi P \Big\|_{L^2}^2.
	\]
\end{proof}

\section{Poincar\'e and Anisotropic Hardy's inequalities}\l{Sec4}

In this section, we will introduce some important inequalities which will be heavily used in later energy estimates.

\subsection{Poincar\'e inequalities}\label{Subsec, Poin}
We first recall two classical Poincar\'e inequalities. 

\begin{lemma}[Poincar\'e inequality with zero mean value, Corollary 2.4 in \cite{LPYZZZ24}]\l{Poin0}
	Let $0<\alpha<\frac{\pi}{2}$, $a=\frac{\pi}{2}-\alpha$, $b=\frac{\pi}{2}+\alpha$. Then for any $f \in H^1(a, b) \backslash\{0\}$ with $\int_a^b f(\phi)  \sin \phi \d \phi = 0$, we have
	\ba\l{Poin0E}
	\int_a^b f^2(\phi)  \sin \phi \d\phi \leq C_{\alpha, A} \int_a^b [f^{\prime}(\phi) ]^2 \sin \phi \d \phi,
	\ea
	where
	$$
	C_{\alpha, A}=\frac{(b-a)^2}{\pi^2+2 \alpha^2}=\frac{4 \alpha^2}{\pi^2+2 \alpha^2}\,.
	$$
	In particular, $C_{\al,A}$ is an increasing function in $\al$ and $C_{\pi/6, A} = \frac{2}{19}$.
\end{lemma}

\begin{lemma}[Poincar\'e inequality with zero boundary, Corollary 2.6 in \cite{LPYZZZ24}]\l{PoinB}
	Let $0<\alpha \leq \frac{\pi}{4}$, $a=\frac{\pi}{2}-\alpha$, $b=\frac{\pi}{2}+\alpha$. Then for any $f \in H_0^1(a, b) \backslash\{0\}$, we have
	\ba\l{PoinBE}
	\int_a^b f^2(\phi)  \sin \phi \d\phi  \leq C_{\alpha, B} \int_a^b [f^{\prime}(\phi) ]^2 \sin \phi \d \phi,
	\ea
	where
	$$
	C_{\alpha, B}=\frac{(b-a)^2}{\pi^2-\frac{2 \alpha^2}{\cos ^2 \alpha}}=\frac{4 \alpha^2}{\pi^2-\frac{2 \alpha^2}{\cos ^2 \alpha}}\,.
	$$
	In particular, $C_{\al,B}$ is an increasing function in $\al$ and $C_{\pi/6, B} = \frac{3}{25}$.
\end{lemma}

We note that Lemma \ref{Poin0} will be used to control $\|v_\rho/\rho\|_{L^2(D_m)}$ by $\|\nabla v_\rho\|_{L^2(D_m)}$ thanks to the zero-average condition for $v_\rho$ in Lemma \ref{v_rho_mean0}, and Lemma \ref{PoinB} will be applied to control $\|v_\phi/\rho\|_{L^2(D_m)}$ by $\|\nabla v_\phi\|_{L^2(D_m)}$ according to the boundary condition $v_\phi=0$ on rays $R_{1,m}\cup R_{2,m}$, see Corollary \ref{Cor, vSob} for more details. Previously in \cite{LPYZZZ24}, Lemma \ref{Poin0} was also used to control $\|v_\th/\rho\|_{L^2(D_m)}$ by $\|\nabla v_\th\|_{L^2(D_m)}$ due to the odd symmetry assumption of $v_\th$. But this symmetry assumption is no longer available in this paper, therefore we have to develop new tool, the anisotropic Hardy's inequality, to deal with $\|v_\th/\rho\|_{L^2(D_m)}$, see Lemma \ref{Lemma, Hardy_Dm} and Corollary \ref{Cor, vth}. 

Next, based on the conservation property (\ref{aam_con}) of the azimuthal angular momentum, we know $\int_{D_m} \Gamma(x,t) \d x = 0$ for any $t$ as long as the initial data satisfies $\int_{D_m} \Gamma_0(x) \d x = 0$. By taking advantage of this property, we will derive a Poincar\'e inequality for functions whose integral on $D_m$ vanish.

\begin{lemma}\l{Poin3D}
	Let $m \geq 10^3$ and $0<\alpha \leq \frac{\pi}{6}$. Then for any axisymmetric function $f \in H^1(D_m) \backslash\{0\}$ with 
	\[
	\int_{D_m} f(x) \d x = 0\,,
	\]
	we have
	\be\label{Poin_Gamma}
	\int_{D_m} f^2\d x \leq \frac{1}{11} \int_{D_m} |\p_\rho f|^2  \d x + \frac{2}{19} \int_{D_m} |\p_\phi f|^2  \d x \leq \frac{2}{19} \int_{D_m} |\nabla f|^2  \d x.
	\ee
\end{lemma}

\begin{proof}
	For convenience of notations, we denote $\phi_1 = \frac{\pi}{2}-\al$ and $\phi_2 = \frac{\pi}{2}+\al$. Using spherical coordinates, (\ref{Poin_Gamma}) is equivalent to 
	\[\begin{split}
		& \int_{\frac1m}^1 \int_{\phi_1}^{\phi_2} |f|^2 \rho^2\sin\phi \d\phi\d\rho \\
		\leq\,\, & \frac{1}{11	} \int_{\frac1m}^1 \int_{\phi_1}^{\phi_2} |\partial_\rho f|^2 \rho^2\sin\phi \d\phi\d\rho
		+ \frac{2}{19}	\int_{\frac1m}^1 \int_{\phi_1}^{\phi_2}
		|\partial_\phi f|^2 \rho^2 \sin\phi \d\phi\d\rho.
	\end{split}\]
	We introduce the following notations:
	\[
	\d\mu \ed \rho^2\sin\phi\d\phi\d\rho,
	\qquad
	\|f\|_{L^2(\d\mu)}^2 \ed \int_{\frac1m}^{1} \int_{\phi_1}^{\phi_2} |f|^2 \rho^2\sin\phi\d\phi \d\rho\,,
	\]
	\[
	E_{rho}(f)
	\ed
	\int_{\frac1m}^1 \int_{\phi_1}^{\phi_2}
	|\partial_\rho f|^2 \rho^2\sin\phi\d\phi\d\rho,
	\qquad 
	E_{phi}(f)
	\ed
	\int_{\frac1m}^1 \int_{\phi_1}^{\phi_2}
	|\partial_\phi f|^2 \rho^2 \sin\phi\d\phi\d\rho.
	\]
	Then the proof is reduced to justifying 
	\be\label{reduction0}
	\|f\|_{L^2(\d\mu)}^2 \leq \frac{1}{11} E_{rho}(f) + \frac{2}{19} E_{phi}(f), \quad\forall\, f\in\m{A}(\al),
	\ee
	where
	\[
	\m A(\alpha)
	\ed
	\left\{
	f \in H^1(D_m) \setminus\{0\}
	:
	\int_{\frac1m}^{1}\int_{\phi_1}^{\phi_2}
	f \rho^2\sin\phi \d\phi \d\rho = \frac{1}{2\pi} \int_{D_m} f(x) \d x = 0
	\right\}.
	\]
	
	For any $f\in \m{A}(\alpha)$, we define
	\[
	g(\rho)
	\ed
	\frac{1}{2\sin\al}
	\int_{\phi_1}^{\phi_2} f(\rho,\phi)\sin\phi\d\phi,
	\qquad
	h(\rho,\phi) \ed f(\rho,\phi) - g(\rho),
	\]
	where $2\sin\al = \int_{\phi_1}^{\phi_2} \sin\phi \d \phi$. Then $g$ depends only on $\rho$ and $g\in \m{A}_{\rho}$, where
	\[
	\m A_\rho
	=
	\Big\{
	R \in H^1 \big(\frac1m, 1\big) \setminus\{0\}
	\;:\;
	\int_{\frac1m}^1 R(\rho)\,\rho^2\d\rho = 0
	\Big\}.
	\]
	As a result,
	\bn
	E_{rho}(g)&=\int_{\frac1m}^1 \int_{\phi_1}^{\phi_2}
	|g^\prime|^2 \rho^2\sin\phi\d\phi\d\rho\\
	&\geq\int_{\phi_1}^{\phi_2}\left(\la_\rho\int_{\f{1}{m}}^1|g|^2\rho^2\d\rho\right)\sin\phi\d\phi= \lambda_\rho \|g\|_{L^2(\d\mu)}^2,
	\en
	where 
	\be\label{lam_rho_eigen1}
	\lambda_{rho}
	\ed
	\inf_{R \in \m A_\rho}
	\frac{\displaystyle \int_{\frac1m}^1 |R'(\rho)|^2 \rho^2 \d\rho}
	{\displaystyle \int_{\frac1m}^1 |R(\rho)|^2 \rho^2 \d\rho}\,.
	\ee
	According to Lemma 2.3 in \cite{LPYZZZ24}, 
	\[
	\lambda_{rho} \geq \frac{\pi^2}{(1 - \frac1m)^2} - \max_{\rho\in [\f1m, 1]} 	\bigg( \frac{p_1''}{2p_1} - \frac{3(p_1')^2}{4p_1^2}  \bigg),
	\]
	where $p_1(\rho) := \rho^2$. By direct computation, we find 
	\[
	\max_{\rho\in [\f1m, 1]} \bigg( \frac{p_1''}{2p_1} - \frac{3(p_1')^2}{4p_1^2}  \bigg) = \max_{\rho\in [\f1m, 1]} -\frac{2}{\rho^2} = -2,
	\]
	which implies that 
	\[
	\lambda_{rho} \geq \pi^2 + 2 > 11\,. 
	\]
	Hence, 
	\[
	\|g\|_{L^2(\d\mu)}^2 \leq \frac{1}{11} E_{rho}(g)\,.
	\]
	
	For any fixed $\rho$, $h(\rho, \cdot)$, as a function in $\phi$, satisfies
	\[\int_{\phi_1}^{\phi_2} h(\rho, \phi) \sin\phi \d\phi = 0\,.\] 
	So $h(\rho, \cdot)$ belongs to $\m{A}_{\phi}(\alpha)$, where 
	\[
	\m A_\phi(\alpha)
	\ed
	\left\{
	Q \in H^1(\phi_1,\phi_2)\setminus\{0\}
	\;:\;
	\int_{\phi_1}^{\phi_2} Q(\phi)\sin\phi\d\phi = 0
	\right\}.
	\]
	Then
	\bn
	E_{phi}(h)&=\int_{\frac1m}^1 \int_{\phi_1}^{\phi_2}
	|\partial_\phi h|^2 \rho^2 \sin\phi\d\phi\d\rho\\
	&\geq\int_{\frac1m}^1\left(\la_\phi(\al)\int_{\phi_1}^{\phi_2}|h|^2\sin\phi\d\phi\right)\rho^2\d\rho= \lambda_\phi(\alpha) \|h\|_{L^2(\d\mu)}^2\,,
	\en
	where
	\bn
	\lambda_{phi}(\alpha)
	\ed
	\inf_{Q \in \m A_\phi(\alpha)}
	\frac{\displaystyle \int_{\phi_1}^{\phi_2} |Q'(\phi)|^2 \sin\phi \d\phi}
	{\displaystyle \int_{\phi_1}^{\phi_2} |Q(\phi)|^2 \sin\phi \d\phi}.
	\en
	According to Lemma 2.3 in \cite{LPYZZZ24}, 
	\[
	\lambda_{phi}(\alpha) \ge \frac{\pi^2}{4\alpha^2} + \frac12 \geq \frac{19}{2},
	\]
	where the last inequality is due to the assumption that $\al \leq \pi/6$.
	
	One can directly check that $g \perp h$ in $ L^2(\d\mu)$, 
	$\rho\p_\rho g \perp \rho\p_\rho h $ in $L^2(\d\mu)$ and $\partial_\phi g = 0$. 
	Therefore, 
	$\|f\|_{L^2(\d\mu)}^2=\|g\|_{L^2(\d\mu)}^2 + \|h\|_{L^2(\d\mu)}^2$ and
	\[ 
	E_{rho}(f) = E_{rho}(g) + E_{rho}(h) \geq E_{rho}(g), \qquad E_{phi}(f) = E_{phi}(h)\,.
	\]			
	Taking advantage of the above argument leads to
	\begin{align*}
		\|f\|_{L^2(\d\mu)}^2 = \|g\|_{L^2(\d\mu)}^2 + \|h\|_{L^2(\d\mu)}^2 
		& \leq \frac{1}{\lambda_{rho}} E_{rho}(g) + \frac{1}{\lambda_{phi}} E_{phi}(h) \\
		& \leq \frac{1}{11} E_{rho}(f) + \frac{2}{19} E_{phi}(f)\,,
	\end{align*}
	which justifies (\ref{reduction0}).
\end{proof}


\subsection{An anisotropic Hardy's inequality}\l{HardySec}
This paper will need to control norms related to $\frac{v_\th}{\rho}$ in the energy estimate later, however, the previous Poincar\'e inequalities in Section \ref{Subsec, Poin} do not apply to it. Noticing 
\[
\int_{D_m} \Big(\frac{v_\theta}{\rho}\Big) \rho^2\sin\phi\d x = \int_{D_m} \Gamma \,\d x = 0, 
\]
we will take advantage of this constraint to derive a Hardy's inequality in Lemma \ref{Lemma, Hardy_Dm}, the significance of this result is that the constants in (\ref{Hardy_Dm}) are uniform in $m$.
\begin{lemma}\label{Lemma, Hardy_Dm}
	Let $m\geq 10^3$ and $\al\in (0,\pi/6]$. Let $i$, $j$ and $k$ be integers such that $i\geq -2$, $k\geq -1$ and $0\leq j \leq i+2$. Then for any $f\in H^{1}(D_m)$ with 
	\be\label{Hardy_m0}
	\int_{D_m} f(\rho,\phi) \rho^{i} \sin^{k}\phi \d x = 0,
	\ee
	we have 
	\be\label{Hardy_Dm}
	\int_{D_m} \Big(\f{f}{\rho}\Big)^2 \rho^{j} \sin^{k}\phi \d x \leq \frac{4}{(j+1)^2} \int_{D_m} |\p_\rho f|^2 \rho^{j} \sin^{k}\phi \d x + \frac{2}{19+k} \int_{D_m} \Big|\f{ \p_\phi f}{\rho}\Big|^2 \rho^{j} \sin^{k}\phi \d x.
	\ee
\end{lemma}

\begin{rem}
	We point out two comments about the above result.
	\begin{itemize}
		\item The idea of proof is to convert the Hardy inequality to a weighted eigenvalue problem. The general structure of this proof is analogous to that for Lemma \ref{Poin3D}, but the treatment of the $\p_\rho$ part in (\ref{Hardy_Dm}) is different from that since the weight $\rho^{j}$ for $|\p_\rho f|^2$ is smaller than the weight $\rho^{j-2}$ for $f^2$ and $|\p_\phi f|^2$ by the factor $\rho^2$. As a compensate for the smaller weight on $|\p_\rho f|^2$, the size constant $\frac{4}{(j+1)^2}$ is much larger than $\frac{1}{20}$ in (\ref{Poin_Gamma}) when $j$ is small. For the application of Lemma \ref{Lemma, Hardy_Dm} in this paper, $j$ will be taken as $0$ and $2$, see Corollary \ref{Cor, vth} and Corollary \ref{Cor, Korn_ineq}. 
		
		\item The size constant $\frac{2}{19+k}$ in (\ref{Hardy_Dm}) can be replaced by $C_\al$, where $C_\al = \frac{4\al^2}{\pi^2 + 2(k+1)\al^2}$. With this replacement, (\ref{Hardy_Dm}) is valid for any $\al\in(0,\pi/2)$. The constant $C_\al$ decays to $0$ as $\al\to 0$ and it is an increasing function on $(0, \pi/6]$, so its maximum on $(0,\pi/6]$ is 
		\[
		C_{\pi/6} = \frac{4(\pi/6)^2}{\pi^2 + 2(k+1)(\pi/6)^2} = \frac{2}{19+k},
		\]
		which is chosen as the size constant in (\ref{Hardy_Dm}).
	\end{itemize}
\end{rem}

\begin{proof}
	Using spherical coordinates, (\ref{Hardy_Dm}) is equivalent to 
	\bn
	& \int_{\frac1m}^1 \int_{\phi_1}^{\phi_2} f^2 \rho^{j} \sin^{k+1}\phi \, \d\phi\d\rho \\
	\leq\,\, & \frac{4}{(j+1)^2} \int_{\frac1m}^1 \int_{\phi_1}^{\phi_2} |\partial_\rho f|^2 \rho^{j+2}\sin^{k+1}\phi \, \d\phi\d\rho
	+ \frac{2}{19+k} \int_{\frac1m}^1 \int_{\phi_1}^{\phi_2}
	|\partial_\phi f|^2 \rho^{j} \sin^{k+1}\phi \, \d\phi\d\rho,
	\en
	where $\phi_1 := \frac{\pi}{2}-\al$ and $\phi_2 := \frac{\pi}{2}+\al$.
	Meanwhile, the constraint (\ref{Hardy_m0}) becomes 
	\be\label{Hardy_m0_sph}
	\int_{\frac1m}^1 \int_{\phi_1}^{\phi_2} f(\rho,\phi) \rho^{i+2} \sin^{k+1}\phi \, \d\phi\d\rho = 0.
	\ee
	We introduce the following notations:
	\[
	\d\mu \ed \rho^{j} \sin^{k+1} \phi\d\phi\d\rho,
	\qquad
	\|f\|_{L^2(\d\mu)}^2 \ed \int_{\frac1m}^{1} \int_{\phi_1}^{\phi_2} f^2 \rho^{j} \sin^{k+1} \phi\d\phi \d\rho\,,
	\]
	\[
	E_{rho}(f)
	\ed
	\int_{\frac1m}^1 \int_{\phi_1}^{\phi_2}
	|\partial_\rho f|^2 \rho^{j+2}\sin^{k+1}\phi\d\phi\d\rho\,,
	\quad 
	E_{phi}(f)
	\ed
	\int_{\frac1m}^1 \int_{\phi_1}^{\phi_2}
	|\partial_\phi f|^2 \rho^{j} \sin^{k+1} \phi\d\phi\d\rho\,.
	\]
	Then the goal is to justify the following (\ref{reduction1}) for all $f\in H^{1}(D_m)$ that satisfies  (\ref{Hardy_m0_sph}).
	\be\label{reduction1}
	\|f\|_{L^2(\d\mu)}^2 \leq \frac{4}{(j+1)^2} E_{rho}(f) + \frac{2}{19+k} E_{phi}(f)\,.
	\ee

	Based on the structure of (\ref{Hardy_m0_sph}), we split
	\ba\l{042501}
	f(\rho,\phi)\ed g(\rho)+h(\rho,\phi)\,,
	\ea
	where
	\[
	g(\rho)
	\ed
	\frac{1}{L_{k}(\al)}
	\int_{\phi_1}^{\phi_2} f(\rho,\phi)\sin^{k+1}\phi\d\phi\,,\q h\ed f-g\,.
	\]
	Here $L_{k}(\al) \ed \int_{\phi_1}^{\phi_2} \sin^{k+1}\phi \d \phi$. Thanks to (\ref{Hardy_m0_sph}), we have
	\[
	g\in\m A_\rho
	\ed
	\Big\{
	R \in H^1 \big(\frac1m, 1\big) \setminus\{0\}
	\;:\;
	\int_{\frac1m}^1 R(\rho)\,\rho^{i+2}\d\rho = 0
	\Big\}.
	\]
	We define
	\be\label{lam_rho_eigen2}
	\lambda_{rho}
	\ed
	\inf_{R \in \m A_\rho}
	\frac{\displaystyle \int_{\frac1m}^1 |R'(\rho)|^2 \rho^{j+2} \d\rho}
	{\displaystyle \int_{\frac1m}^1 |R(\rho)|^2 \rho^{j} \d\rho}\,.
	\ee
	Clearly, we have
	\bn
	E_{rho}(g)=&\int_{\phi_1}^{\phi_2} \bigg( \int_{\f{1}{m}}^1|g^\prime(\rho)|^2\rho^{j+2}\d\rho \bigg) \sin^{k+1}\phi\d\phi\\
	\geq&\int_{\phi_1}^{\phi_2}\left(\la_{rho}\int_{\f{1}{m}}^1|g(\rho)|^2\rho^{j}\d\rho\right)\sin^{k+1}\phi\d\phi=\lambda_{rho} \|g\|_{L^2(\d\mu)}^2\,.
	\en
	Comparing (\ref{lam_rho_eigen2}) with (\ref{lam_rho_eigen1}), the weights for both the numerator and the denominator in (\ref{lam_rho_eigen1}) are $\rho^2$, while the weight $\rho^{j+2}$ for the numerator is smaller than the weight $\rho^{j}$ for the denominator in (\ref{lam_rho_eigen2}) by the factor $\rho^2$, so it is more challenging to obtain a lower bound for $\la_\rho$ in (\ref{lam_rho_eigen2}) and we need to adopt a different method. 
	
	Applying the change of variable: $u(x) := e^{- \beta x} R(e^{-x})$ for $x = -\ln \rho\in(0, \ln m)$, where $R$ is the function in (\ref{lam_rho_eigen2}) and 
	\[\beta := \frac{j+1}{2},\]
	we find 
	\[
	\int_{\frac1m}^1 |R(\rho)|^2 \rho^{j} \d\rho
	=\int_0^{\ln m} u^2(x) \d x,
	\qquad
	\int_{\frac1m}^1 |R'(\rho)|^2 \rho^{j+2} \d\rho
	=\int_0^{\ln m} \bigl|u'(x)+\beta u(x)\bigr|^2\d x\,,	
	\]
	and the constraint $\int_{\frac1m}^1 R(\rho)\,\rho^{i+2}\d\rho = 0$ becomes 
	\[
	\int_0^{\ln m} e^{-\frac{5+2i-j}{2}x}u(x)\dx=0\,.
	\]
	Hence
	\[
	\lambda_{rho}
	=
	\inf_{u\in\m{B}_\rho}
	\frac{\int_0^{\ln m} |u' + \beta u|^2\d x}
	{\int_0^{\ln m} u^2\d x}\,,
	\]
	where 
	\[
	\m B_\rho
	\ed
	\left\{
	u \in H^1(0, \ln m)\setminus\{0\}
	\;:\;
	\int_0^{\ln m} e^{-\frac{5+2i-j}{2}x} u(x) \d x = 0
	\right\}.
	\]
	Noticing that 
	\[
	\int_0^{\ln m} \bigl| u' + \beta u\bigr|^2\d x
	=
	\int_0^{\ln m} |u'|^2\d x
	+ \beta^2 \int_0^{\ln m} u^2\d x
	+ \beta |u(\ln m)|^2 
	- \beta |u(0)|^2\,,
	\]
	so
	\begin{equation}\label{lam_rho}
		\lambda_{rho} = \beta^2 + \un{\inf_{u \in \m{B}_\rho} \frac{\int_0^{\ln m} |u'|^2 \d x 
				+ \beta |u(\ln m)|^2 - \beta |u(0)|^2}{\int_0^{\ln m} u^2\d x}}_{I}\,.
	\end{equation}
	In the following we will show $I\geq0$. Since \(\int_0^{\ln m} e^{-\frac{5+2i-j}{2}x}u(x)\d x = 0\), then applying the integration by parts yields
	\[
	u(0) = m^{-\frac{5+2i-j}{2}} u(\ln m) - \int_0^{\ln m} e^{-\frac{5+2i-j}{2}x} u'(x) \d x.
	\]
	This implies that
	\[
	|u(0)|^2
	\le
	2 m^{-(5+2i-j)} |u(\ln m)|^2 + 2 \bigg( \int_0^{\ln m} e^{-(5+2i-j)x} \d x \bigg) \bigg( \int_0^{\ln m} |u'(x)|^2 \d x \bigg).
	\]
	By the assumption $m\geq 10^3$, $j\leq i+2$ and $i\geq -2$, we know 
	\[
	m^{-(5+2i-j)} \leq m^{-(3+i)} \leq m^{-1} \leq \frac12,
	\]
	and 
	\[
	\int_0^{\ln m} e^{-(5+2i-j)x} \d x \leq \frac{1}{5+2i-j}.
	\]
	As a result, 
	\[
	|u(0)|^2
	\le
	|u(\ln m)|^2 + \frac{2}{5+2i-j} \int_0^{\ln m} |u'(x)|^2 \d x.
	\]
	On the other hand, since $j\geq 0$, we have $\beta = (j+1)/2 > 0$, so multiplying the above inequality by $\beta$ yields
	\[
	\beta |u(0)|^2
	\le
	\beta |u(\ln m)|^2 + \frac{2\beta}{5+2i-j} \int_0^{\ln m} |u'(x)|^2 \d x.
	\]
	Again, one can check that $2\beta \leq 5+2i-j$ due to the assumption that $j\leq i+2$, so 
	\[
	\beta |u(0)|^2
	\le
	\beta |u(\ln m)|^2 + \int_0^{\ln m} |u'(x)|^2 \d x.
	\]
	Plugging this estimate into (\ref{lam_rho}), we justify $I\geq0$. This leads to 
	\[
	\lambda_{rho} \geq \beta^2 = \frac{(j+1)^2}{4}.
	\]
	which implies that 
	\be\label{energy_rho}
	\|g\|_{L^2(\d\mu)}^2 \leq \frac{4}{(j+1)^2} E_{rho}(g).
	\ee

	Now we estimate term $h$ in \eqref{042501}. For any fixed $\rho$, one can check that
	\[
	\int_{\phi_1}^{\phi_2} h(\rho, \phi) \sin^{k+1} \phi \d\phi = \int_{\phi_1}^{\phi_2} \big[f(\rho, \phi) - g(\rho)\big] \sin^{k+1}\phi \d\phi = 0\,,
	\]
	so $h(\rho, \cdot)$ belongs to $\m{A}_{\phi}(\alpha)$, where
	\[
	\m A_\phi(\alpha)
	\ed
	\left\{
	Q \in H^1(\phi_1,\phi_2)\setminus\{0\}
	\;:\;
	\int_{\phi_1}^{\phi_2} Q(\phi)\sin^{k+1}\phi\d\phi = 0
	\right\}\,.
	\]
	Consequently, 
	\bn
	E_{phi}(h) = & \int_{\frac1m}^1 \bigg( \int_{\phi_1}^{\phi_2}
	|\partial_\phi h|^2 \rho^{j} \sin^{k+1} \phi\d\phi \bigg) \d\rho\\
	\geq&\int_{\frac1m}^1 \left(\la_{phi}(\al)\int_{\phi_1}^{\phi_2}
	| h|^2 \sin^{k+1} \phi\d\phi\right)\rho^{j}\d\rho= \lambda_{phi}(\alpha) \|h\|_{L^2(\d\mu)}^2,
	\en
	where 
	\bn
	\lambda_{phi}(\alpha)
	\ed
	\inf_{Q \in \m A_\phi(\alpha)}
	\frac{\displaystyle \int_{\phi_1}^{\phi_2} |Q'(\phi)|^2 \sin^{k+1}\phi \d \phi}
	{\displaystyle \int_{\phi_1}^{\phi_2} |Q(\phi)|^2 \sin^{k+1} \phi \d \phi}.
	\en
	According to Lemma 2.3 in \cite{LPYZZZ24}, 
	\[
	\lambda_{phi}(\al) \geq \frac{\pi^2}{(\phi_2 - \phi_1)^2} - \max_{\phi\in [\phi_1, \phi_2]} 	\bigg( \frac{p_2''}{2p_2} - \frac{3(p_2')^2}{4p_2^2}  \bigg),
	\]
	where $p_2(\phi) := \sin^{k+1}\phi$. By direct computation and the assumption that $k\geq -1$, we find 
	\[
	\max_{\phi\in [\phi_1, \phi_2]} \bigg( \frac{p_2''}{2p_2} - \frac{3(p_2')^2}{4p_2^2}  \bigg) = -\frac{k+1}{2},
	\]
	which implies that 
	\[
	\lambda_{phi}(\al) \geq \frac{\pi^2}{4\al^2} + \frac{k+1}{2} = \frac{\pi^2 + 	2(k+1)\al^2}{4\al^2}\,. 
	\]
	Hence, 
	\[
	\|h\|_{L^2(\d\mu)}^2 \leq \frac{4\al^2}{\pi^2 + 2(k+1)\al^2} E_{phi}(h)\,.
	\]
	When $\al$ is restricted on $(0,\frac{\pi}{6}]$, we obtain 
	\be\label{energy_phi}
	\|h\|_{L^2(\d\mu)}^2 \leq \frac{2}{19+k} E_{phi}(h)\,.
	\ee
	
	Finally, one can directly check that $g \perp h$ in $L^2(\d\mu)$,
	$\rho\p_\rho g \perp \rho\p_\rho h$ in $L^2(\d\mu)$ and $\partial_\phi g = 0$. 
	Therefore, 
	$\|f\|_{L^2(\d\mu)}^2=\|g\|_{L^2(\d\mu)}^2 + \|h\|_{L^2(\d\mu)}^2$ and 
	\ba\l{042601}
	E_{rho}(f) = E_{rho}(g) + E_{rho}(h) \geq E_{rho}(g), \qquad E_{phi}(f) = E_{phi}(h).
	\ea			
	Taking advantage of \eqref{energy_rho} and \eqref{energy_phi}, and then using \eqref{042601}, we conclude that
	\begin{align*}
		\|f\|_{L^2(\d\mu)}^2 = \|g\|_{L^2(\d\mu)}^2 + \|h\|_{L^2(\d\mu)}^2 
		& \leq \frac{4}{(j+1)^2} E_{rho}(g) + \frac{2}{19+k} E_{phi}(h) \\
		& \leq \frac{4}{(j+1)^2} E_{rho}(f) + \frac{2}{19+k} E_{phi}(f).
	\end{align*}
	This proves \eqref{reduction1}, thereby completing the proof of the lemma.
\end{proof}

The inequalities established in this section will be applied in Section \ref{Sec, App_Hardy} to obtain estimates for $\|v_\th/\rho\|_{L^2}$, $\|v_\th/\rho^2\|_{L^2}$ and a Korn inequality for $v_\th \bm{e_\th}$. These estimates are the key ingredients for the energy estimate in Proposition \ref{Funden} and for controlling the good unknowns $\m{K}$ and $\m{F}$ in Section \ref{Subsec, K and F}.

\section{Applications of the anisotropic Hardy's inequality in $D_m$}
\label{Sec, App_Hardy}

\subsection{Hardy's inequality for $v_\th$ and $v_\th/\rho$}
In this section, we will present several applications of the anisotropic Hardy's inequality introduced earlier in Lemma \ref{Lemma, Hardy_Dm}. The main purpose of our first result is to control the singular quantities $v_\th/\rho$ and $v_\th/\rho^2$ by $\nabla v_\th$, which will later be used in the Biot–Savart estimates and the energy estimates for $(\m{K},\m{F},\m{O})$.

\begin{cor}\label{Cor, vth}
	Let $m\geq 10^3$ and $\alpha \in (0,\pi/6]$. Then for any $v_\th\in H^{1}(D_m)$ such that 
	\[
	\int_{D_m} \rho \sin\phi \, v_\th\d x = 0,
	\]
	we have 
	\be\label{wP2}
	\int_{D_m} 
	\frac{f^2}{\rho^2} \d x
	\leq 
	\frac{8}{\sqrt{3}} \int_{D_m}
	|\partial_\rho f|^2 \dx
	+ \frac{1}{5\sqrt{3}} \int_{D_m} 
	\frac{1}{\rho^2} |\partial_\phi f|^2
	\dx\,,
	\ee
	where $f = v_\th$ or $\frac{v_\th}{\rho}$.
\end{cor}
\begin{proof}
	We first treat the case when $f = \frac{v_\th}{\rho}$. Plugging $f = \frac{v_\theta}{\rho}$, $i=2$, $k=1$ and $j=0$ into Lemma \ref{Lemma, Hardy_Dm} yields
	\[
	\int_{D_m} 
	\Big| \frac{1}{\rho} \frac{v_\theta}{\rho} \Big|^2 \sin\phi \dx
	\leq 
	4 \int_{D_m}
	\Big|\partial_\rho \Big(\frac{v_\theta}{\rho}\Big) \Big|^2 \sin\phi \dx
	+ \frac{1}{10}	\int_{D_m} 
	\Big| \frac{1}{\rho}\partial_\phi \Big(\frac{v_\theta}{\rho} \Big) \Big|^2 \sin\phi
	\dx\,.
	\]
	Since $\alpha\in (0, \pi/6]$, then $\sin\phi \geq \frac{\sqrt{3}}{2}$ for any $\phi\in [\phi_1, \phi_2]$. As a result, 
	\[\begin{split}
		\int_{D_m} \Big| \frac{1}{\rho} \frac{v_\theta}{\rho} \Big|^2 \dx & \leq
		\frac{2}{\sqrt{3}} \int_{D_m} \Big| \frac{1}{\rho} \frac{v_\theta}{\rho} \Big|^2 \sin\phi \dx \\
		&\leq \frac{2}{\sqrt{3}} \bigg(4 \int_{D_m}
		\Big|\partial_\rho \Big(\frac{v_\theta}{\rho}\Big) \Big|^2 \dx
		+ \frac{1}{10} 	\int_{D_m} 
		\Big| \frac{1}{\rho}\partial_\phi \Big(\frac{v_\theta}{\rho} \Big) \Big|^2 \d x \bigg)\,.
	\end{split}\]
	which justifies (\ref{wP2}) for the case $f=v_\th/\rho$.
	
	The proof for the case $f=v_\th$ can be carried out similarly by plugging $f = v_\theta$, $i=1$, $k=1$ and $j=0$ into Lemma \ref{Lemma, Hardy_Dm}, we omit the details.
\end{proof}

\begin{cor}\label{Cor, vSob}
	Let $m\geq 10^3$ and $\alpha \in (0,\pi/6]$. Let $\bm{v}$ be the strong solution of the problem (\ref{NS}) on $D_m$ under the boundary condition (\ref{bdry for Dm}) as stated in Proposition \ref{Prop, local soln in ad} with 
	\[\int_{D_m} \rho\sin\phi\, v_{\th,0}(x) \d x = 0.\]
	Then 
	\be\label{vSob}\left\{\begin{array}{ll}
		\big\| \frac{v_\rho}{\rho} \big\|_{L^2(D_m)} &\leq \sqrt{\frac{2}{19}} \big\| \frac{1}{\rho} \p_\phi v_\rho \big\|_{L^2(D_m)} \leq \sqrt{\frac{2}{19}} \| \nabla v_\rho\|_{L^2(D_m)}, 
		\vspace{0.1in}\\
		\big\| \frac{v_\phi}{\rho} \big\|_{L^2(D_m)} &\leq \frac{\sqrt{3}}{5} \big\| \frac{1}{\rho} \p_\phi v_\phi \big\|_{L^2(D_m)} \leq  \frac{\sqrt{3}}{5} \| \nabla v_\phi\|_{L^2(D_m)}, \vspace{0.1in} \\
		\big\| \frac{v_\th}{\rho} \big\|_{L^2(D_m)} &\leq \sqrt{5} \| \nabla v_\th\|_{L^2(D_m)}. \\
	\end{array}\right.
	\ee
\end{cor}
\begin{proof}
	Firstly, (\ref{vSob})$_{3}$ follows directly from Corollary \ref{Cor, vth} by choosing $f=v_\th$ in (\ref{wP2}) since $\max\Big\{ \frac{8}{\sqrt{3}},\, \frac{1}{5\sqrt{3}} \Big\} < 5$. Secondly, 
	\[
	\Big\| \frac{v_\rho}{\rho} \Big\|_{L^2(D_m)}^2 = 2\pi \int_{\frac1m}^{1}\int_{\frac{\pi}{2}-\al}^{\frac{\pi}{2}+\al} v_\rho^2 \sin\phi\, \d\phi \d\rho.
	\]
	Thanks to Lemma \ref{v_rho_mean0}, we can apply Poincar\'e inequality (\ref{Poin0E}) to obtain 
	\[
	\int_{\frac{\pi}{2}-\al}^{\frac{\pi}{2}+\al} v_\rho^2 \sin\phi\, \d\phi \leq \frac{2}{19} \int_{\frac{\pi}{2}-\al}^{\frac{\pi}{2}+\al} (\p_\phi v_\rho)^2 \sin\phi\, \d\phi.
	\]
	As a result, 
	\[
	\Big\| \frac{v_\rho}{\rho} \Big\|_{L^2(D_m)}^2  \leq \frac{2}{19} \cdot 2\pi \int_{\frac1m}^{1}\int_{\frac{\pi}{2}-\al}^{\frac{\pi}{2}+\al} (\p_\phi v_\rho)^2 \sin\phi\, \d\phi \d\rho = \frac{2}{19} \| \nabla v_\phi\|_{L^2(D_m)}^2,
	\]
	which implies (\ref{vSob})$_{1}$. Finally, similar to the derivation of (\ref{vSob})$_{1}$, we take advantage of the fact that $v_\phi = 0$ on rays $R_{1,m}\cup R_{2,m}$ to apply Poincar\'e inequality (\ref{PoinB}) to justify (\ref{vSob})$_{2}$.
\end{proof}

Corollary \ref{Cor, vSob} shows that all three velocity components enjoy uniform $L^2$
control after division by $\rho$. This will be repeatedly used in Section \ref{Sec, v_energy_est} through Section \ref{Sec9} when estimating singular coefficients arising in spherical coordinates.

The following corollary plays a crucial role in the proof of Lemma \ref{Lem35} (see the estimate after (\ref{0319EE})). We emphasize that it is absolutely necessary for the constant appearing in the inequality to be strictly less than $1$, which will be achieved thanks to Corollary \ref{Cor, vth}.

\begin{cor}\l{COR0416-1}
	Let $g = \frac{v_\th}{\rho}$, where $v_\th$ is the function in Corollary \ref{Cor, vth}. Then for any $m\geq 10^3$ and $\al\in (0, \pi/6]$,
	\[
	\bigg| \int_{D_m} \cot\phi \, \frac{\partial_\phi g}{\rho} \frac{g}{\rho}  \bigg| \,dx \leq 0.7 \|\nabla g\|_{L^2(D_m)}^2\,.
	\]
\end{cor}

\begin{proof}
	Firstly, it follows from Corollary \ref{Cor, vth} that 
	\[
	\Big\| \frac{g}{\rho} \Big\|_{L^2(D_m)}^2 \leq \frac{8}{\sqrt{3}} \|\partial_\rho g\|_{L^2(D_m)}^2 + \frac{1}{5\sqrt{3}} \Big\| \frac{1}{\rho} \partial_\phi g \Big\|_{L^2(D_m)}^2\,,
	\]
	which implies that 
	\[
	\Big\| \frac{g}{\rho} \Big\|_{L^2(D_m)} \leq 2^{\frac32} 3^{-\frac14} \|\partial_\rho g\|_{L^2(D_m)} + 5^{-\frac12} 3^{-\frac14} \Big\| \frac{1}{\rho} \partial_\phi g \Big\|_{L^2(D_m)}\,.
	\]
	Consequently, 
	\[\begin{split}
		\bigg| \int_{D_m} \cot\phi \frac{\partial_\phi g}{\rho} \frac{g}{\rho}\dx  \bigg| 
		&\leq 3^{-\frac12} \Big\|\frac{\partial_\phi g}{\rho} \Big\|_{L^2(D_m)} \Big\| \frac{g}{\rho} \Big\|_{L^2(D_m)} \\
		&\leq 2^{\frac32} 3^{-\frac34} \Big\|\frac{\partial_\phi g}{\rho} \Big\|_{L^2(D_m)} \|\partial_\rho g\|_{L^2(D_m)} + 5^{-\frac12} 3^{-\frac34}  \Big\|\frac{\partial_\phi g}{\rho} \Big\|_{L^2(D_m)}^2\,.
	\end{split}\]
	Applying Cauchy-Schwarz inequality, for any $\lambda > 0$,
	\[\begin{split}
		\bigg| \int_{D_m} \cot\phi \frac{\partial_\phi g}{\rho} \frac{g}{\rho}\dx  \bigg| &\leq 
		\lambda \|\partial_\rho g\|_{L^2(D_m)}^2 + \bigg( \frac{2^3 3^{-\frac32}}{4\lambda} + 5^{-\frac12} 3^{-\frac34}  \bigg) \Big\|\frac{\partial_\phi g}{\rho} \Big\|_{L^2(D_m)}^2 \\
		&\leq \lambda \|\partial_\rho g\|_{L^2(D_m)}^2 + \bigg( \frac{0.4}{\lambda} + \frac19 \bigg) \Big\|\frac{\partial_\phi g}{\rho} \Big\|_{L^2(D_m)}^2\,.
	\end{split}\]
	Taking $\lambda = 0.7$ yields 
	\[
	\bigg| \int_{D_m} \cot\phi \frac{\partial_\phi g}{\rho} \frac{g}{\rho}\dx  \bigg| \leq 0.7 \|\partial_\rho g\|_{L^2(D_m)}^2 + 0.7 \Big\|\frac{\partial_\phi g}{\rho} \Big\|_{L^2(D_m)}^2 = 0.7 \|\nabla g\|_{L^2(D_m)}^2\,.
	\]
\end{proof}

\subsection{Korn's inequality for $v_\th e_\th$}

Let $\mathbb{S}(\bm{v})$ be the strain tensor, defined as $\mathbb{S}(\bm{v}) := \frac12 (\nabla \bm{v} + \nabla \bm{v}^{\top})$. In literature, for domains with smooth boundary, the well-known Korn's inequality would be valid under certain conditions. For example, if we pretend that our approximating domain $D_m$ has a smooth boundary (say $C^2$), then based on \cite[Theorem 1]{Falocchi2022}, the kernel of the strain operator in
\[
\{\bm{h}\in H^1(\Omega)\,:\,\bm{h}\cdot\bm{n}=0\text{ on }\p \Omega\}
\]
is
\[
\m{N}\ed\{C\rho\sin\phi \, \bm{e_\th}\,:\,C\in\mathbb{R}\}\,.
\]
Denote $P_\m{N}:\,H^1(D_m)\to \m{N}$ to be the related projection operator, then it follows from \cite[Lemma 4.1 and Lemma 4.2]{WatanabeJCAM} that 
\[
\|\na\bm{v}\|_{L^2(D_m)}\leq C_m\|\mathbb{S}\bm{v}\|_{L^2(D_m)}\,,
\]
as long as $\bm{v}\in\m{N}^\perp$. However, the domain $D_m$ in our paper is only Lipschitz, so the above inequality does not seem to be directly applicable.
Meanwhile, it is not clear if the constants $\{C_m\}_{m\geq 10^3}$ have an upper bound that is uniform in $m$. Next, by taking advantage of the anisotropic Hardy's inequality in Lemma \ref{Lemma, Hardy_Dm}, we manage to establish a Korn's inequality for the $\th$-component of $\bm{v}$ with an absolute constant $\frac83$, see (\ref{Korn_vth}) in Corollary \ref{Cor, Korn_ineq}.

For future references, we introduce some notations below. 
Recalling the appendix A.1 in \cite{LPYZZZ24}, we know that
\ba\l{gradientmatrix}
\na\bm{v}=\left(\begin{array}{ccc}
	\partial_\rho v_\rho & \frac{1}{\rho}\left(\partial_\phi v_\rho-v_\phi\right) & -\frac{1}{\rho} v_\theta \vspace{0.07in}\\
	\partial_\rho v_\phi & \frac{1}{\rho}\left(\partial_\phi v_\phi+v_\rho\right) & -\frac{\cot \phi}{\rho} v_\theta \vspace{0.07in}\\
	\partial_\rho v_\theta & \frac{1}{\rho} \partial_\phi v_\theta & \frac{1}{\rho}\left(v_\rho+\cot \phi v_\phi\right)
\end{array}\right)
\ea
in the spherical coordinates. 
Let $\bm{b} = v_\rho \bm{e_{\rho}} + v_\phi \bm{e_{\phi}}$ be the vector obtained by subtracting the $\th$-component from $\bm{v}$. Then 
\be\label{grad_b}
\na\bm{b} = \left(\begin{array}{ccc}
	\partial_\rho v_\rho & \frac{1}{\rho}\left(\partial_\phi v_\rho-v_\phi\right) & 0 \vspace{0.07in}\\
	\partial_\rho v_\phi & \frac{1}{\rho}\left(\partial_\phi v_\phi+v_\rho\right) & 0 \vspace{0.07in}\\
	0 & 0  & \frac{1}{\rho}\left(v_\rho+\cot \phi v_\phi\right)
\end{array}\right),
\ee
and 
\be\label{grad_vth}
\na (v_\th \bm{e_\th}) = \left(\begin{array}{ccc}
	0 & 0 & -\frac{1}{\rho} v_\theta \vspace{0.07in}\\
	0 & 0 & -\frac{\cot \phi}{\rho} v_\theta \vspace{0.07in}\\
	\partial_\rho v_\theta & \frac{1}{\rho} \partial_\phi v_\theta & 0
\end{array}\right).
\ee
Combining (\ref{gradientmatrix})--(\ref{grad_vth}) leads to 
\be\label{grad_v_decom}\begin{split}
	|\nabla \bm{v}|^2 &=  |\nabla\bm{b}|^2 + |\nabla(v_\th \bm{e_\th})|^2 \\
	&=  |\nabla\bm{b}|^2 + (\partial_\rho v_\theta)^2
	+\left(\frac1\rho\p_\phi v_\theta\right)^2
	+ \frac{1}{\sin^2\phi}\left(\frac{v_\theta}{\rho}\right)^2.
\end{split}\ee
On the other hand, it follows from (\ref{grad_vth}) and the definition of the strain tensor $\mathbb{S}$ that 
\be\label{grad_S_th}
\mathbb{S}(v_\th \bm{e_\th})
=\left(\begin{array}{ccc}
	0 & 0 & \frac{1}{2} \partial_\rho v_\theta-\frac{1}{2 \rho} v_\theta \vspace{0.07in}\\
	0 & 0 & \frac{1}{2 \rho} \partial_\phi v_\theta-\frac{\cot \phi}{2 \rho} v_\theta \vspace{0.07in}\\
	\frac{1}{2} \partial_\rho v_\theta-\frac{1}{2 \rho} v_\theta & \frac{1}{2 \rho} \partial_\phi v_\theta-\frac{\cot \phi}{2 \rho} v_\theta & 0
\end{array}\right)\,.
\ee

\begin{cor}[Korn's inequality for $v_\th\bm{e_\th}$]\label{Cor, Korn_ineq}
	Let $m\geq 10^3$ and $\al\in (0, \pi/6]$. Let \(v_\theta\in H^1(D_m)\) be an axially symmetric function that satisfies
	\begin{equation}\label{eq:vtheta-mean-zero}
		\int_{D_m}\rho\sin\phi \, v_\th\d x=0\,.
	\end{equation}
	Then 
	\be\label{Korn_vth}
	\int_{D_m} |\nabla (v_\th\bm{e_\th})|^2 \d x \leq \frac83 \int_{D_m} |\mathbb{S}(v_\th \bm{e_\th})|^2 \d x.
	\ee
	
\end{cor}
\begin{proof}
	According to the expressions in (\ref{grad_vth}) and (\ref{grad_S_th}), we know 
	\[
	|\nabla (v_\th\bm{e_\th})|^2 = (\partial_\rho v_\theta)^2
	+\left(\frac1\rho\p_\phi v_\theta\right)^2
	+ (1+\cot^2\phi)\left(\frac{v_\theta}{\rho}\right)^2
	\]
	and 
	\[
	|\mathbb{S}(v_\th \bm{e_\th})|^2 = \frac12 \left(\partial_\rho v_\theta-\frac{v_\theta}{\rho}\right)^2
	+
	\frac12\left(\frac{\partial_\phi v_\theta}{\rho}-\frac{\cot\phi}{\rho}v_\theta\right)^2.
	\]
	Therefore, (\ref{Korn_vth}) boils down to the following estimate.
	\begin{align}
		&\int_{D_m}
		\left[
		(\partial_\rho v_\theta)^2
		+\left(\frac1\rho\p_\phi v_\theta\right)^2
		+ \frac{1}{\sin^2\phi}\left(\frac{v_\theta}{\rho}\right)^2
		\right]\d x \nonumber\\
		&\qquad\leq
		\frac43
		\int_{D_m}
		\left[
		\left(\partial_\rho v_\theta-\frac{v_\theta}{\rho}\right)^2
		+
		\left(\frac{\partial_\phi v_\theta}{\rho}-\frac{\cot\phi}{\rho}v_\theta\right)^2
		\right]\d x\,.
		\label{eq:coercive-vtheta}
	\end{align}
	
	Define
	\begin{equation}\label{eq:def-w-vtheta}
		w(\rho,\phi)\ed\frac{v_\theta(\rho,\phi)}{\rho\sin\phi}.
	\end{equation}
	Since $\al\in (0,\pi/6]$, we know 
	\(\phi\in [\f{\pi}{2}-\al,\f{\pi}{2}+\al] \subseteq [\pi/3, 2\pi/3]\), hence $\sin\phi \geq \frac{\sqrt{3}}{2}$ and \eqref{eq:def-w-vtheta} is well-defined.
	From \(v_\theta=\rho\sin\phi\,w\), a direct computation gives
	\bn
	\partial_\rho v_\theta-\frac{v_\theta}{\rho}
	=\rho\sin\phi\,\partial_\rho w
	\qquad\text{and}\qquad
	\frac1\rho\p_\phi v_\theta-\frac{\cot\phi}{\rho}v_\theta
	=\sin\phi\,\partial_\phi w.
	\en
	Therefore,
	\ba\l{0428-2}
	&\quad \int_{D_m}
	\left[
	\left(\partial_\rho v_\theta-\frac{v_\theta}{\rho}\right)^2
	+
	\left(\frac{\partial_\phi v_\theta}{\rho}-\frac{\cot\phi}{\rho}v_\theta\right)^2
	\right]\d x \\
	&= \int_{D_m}
	\left[
	\rho^2\sin^2\phi\,|\partial_\rho w|^2
	+
	\sin^2\phi\,|\partial_\phi w|^2
	\right]\d x.
	\ea
	Meanwhile, the mean-zero condition \eqref{eq:vtheta-mean-zero} becomes
	\bn
	\int_{D_m}\rho^2\sin^2\phi\,w(\rho,\phi)\d x=0\,.
	\en
	Applying Lemma \ref{Lemma, Hardy_Dm} (with $i=2$, $j=2$ and $k=2$), it follows that
	\begin{equation}\label{eq:weighted-poincare-w}
		\int_{D_m}\sin^2\phi\, w^2\d x
		\leq
		\frac49
		\int_{D_m}
		\rho^2\sin^2\phi\,|\partial_\rho w|^2 \d x
		+
		\frac{2}{21}
		\int_{D_m}
		\sin^2\phi\,|\partial_\phi w|^2
		\d x\,.
	\end{equation}
	
	We now estimate the three terms on the left-hand side of \eqref{eq:coercive-vtheta}. Indeed, since 
	\bn\left\{\begin{array}{crl}
		v_\theta/\rho &= &\sin\phi\,w\,,\\[2mm]
		\partial_\rho v_\theta &= & \rho\sin\phi\,\partial_\rho w + \sin\phi\,w\,,\\[2mm]
		\frac1\rho\p_\phi v_\theta &= & \sin\phi\,\partial_\phi w + \cos\phi\,w\,,
	\end{array}\right.
	\en
	it follows from Cauchy-Schwarz inequality that 
	\bn
	(\p_\rho v_\th)^2 + \Big( \frac{1}{\rho} \p_\phi v_\th \Big)^2 + \frac{1}{\sin^2\phi}\Big( \frac{v_\th}{\rho} \Big)^2 
	&\leq \Big( \la_1 \rho^2 \sin^2\phi\, (\partial_\rho w)^2 + \frac{1}{4\la_1} w^2 \sin^2\phi\Big) \\ 
	& \quad + \Big( \la_2 \sin^2 \phi \,(\p_\phi w)^2 + \frac{1}{4\la_2} w^2 \cos^2\phi \Big) \\
	& \quad + (1+\cot^2\phi) w^2 \sin^2\phi\,,
	\en
	where $\la_1$ and $\la_2$ are arbitrary positive constants. 
	Since $\phi\in [\pi/3, 2\pi/3]$, then $\cot^2\phi \leq 1/3$, so rearranging the above estimate yields 
	\bn
	(\p_\rho v_\th)^2 + \Big( \frac{1}{\rho} \p_\phi v_\th \Big)^2 + \frac{1}{\sin^2\phi}\Big( \frac{v_\th}{\rho} \Big)^2 &\leq \la_1 \rho^2 \sin^2\phi\, (\partial_\rho w)^2 +  \la_2 \sin^2 \phi \,(\p_\phi w)^2 \\
	& \quad + \Big( \frac43 + \frac{1}{4\la_1}  + \frac{1}{12\la_2}  \Big) w^2 \sin^2\phi\,.
	\en
	
	Integrating the above inequality on $D_m$ and applying (\ref{eq:weighted-poincare-w}) leads to 
	\bn
	& \quad \int_{D_m} \Big[(\p_\rho v_\th)^2 + \Big( \frac{1}{\rho} \p_\phi v_\th \Big)^2 + \frac{1}{\sin^2\phi}\Big( \frac{v_\th}{\rho} \Big)^2\Big] \d x \\
	& \leq \Big[ \la_1 + \frac49 \Big( \frac43 + \frac{1}{4\la_1}  + \frac{1}{12\la_2} \Big) \Big] \int_{D_m} \rho^2 \sin^2\phi\, (\partial_\rho w)^2 \d x \\
	& \quad + \Big[ \la_2 + \frac{2}{21} \Big( \frac43 + \frac{1}{4\la_1}  + \frac{1}{12\la_2} \Big) \Big] \int_{D_m} \sin^2 \phi \,(\p_\phi w)^2 \d x.
	\en
	Choosing $\la_1 = \frac13$ and $\la_2 = 1$ in the above estimate yields 
	\bn
	& \quad \int_{D_m} (\p_\rho v_\th)^2 + \Big( \frac{1}{\rho} \p_\phi v_\th \Big)^2 + \frac{1}{\sin^2\phi}\Big( \frac{v_\th}{\rho} \Big)^2 \d x \\
	& \leq \frac43 \int_{D_m} \rho^2 \sin^2\phi\, (\partial_\rho w)^2 \d x + \frac43\int_{D_m} \sin^2 \phi \,(\p_\phi w)^2 \d x \\
	& = \frac43 \int_{D_m} \left(\partial_\rho v_\theta-\frac{v_\theta}{\rho}\right)^2 \d x
	+ \frac43\int_{D_m} \left(\frac{\partial_\phi v_\theta}{\rho}-\frac{\cot\phi}{\rho}v_\theta\right)^2
	\d x\,,
	\en
	which justifies (\ref{eq:coercive-vtheta}). Here we have used \eqref{0428-2}. The proof is finished.
\end{proof}

\section{Fundamental energy estimates for the velocity in $D_m$}\l{Sec, v_energy_est}

In this section, we will establish the fundamental energy estimate for the solution $\bm{v}$ on the approximating domains $D_m$.
\begin{lemma}\label{Lemma, lap_int}
	Let $m\geq 10^3$ and $\al\in (0,\pi/6]$. Let $\bm{v}$ be the strong solution of \eqref{NS} on $D_m \times [0,T]$ under the boundary condition (\ref{bdry for Dm}).
	Then for any fixed $t\in (0,T)$, the following integral identity holds.
	\be\label{lap_int}\begin{split}
		\int_{D_{m}} \bm{v} \cdot \Delta \bm{v} \d x &= - \int_{D_m} |\na\bm{b}|^2\d x 
		- 2 \int_{D_m} |\mathbb{S}(v_\th \bm{e_\th})|^2 \d x \\
		& \quad + \int_{D_m} \frac{1}{\rho^2} v_\phi^2 (2\eta_m + 2\rho\eta_m' - 1) \d x + \int_{D_m} \frac{1}{\rho} (\p_\rho v_\phi^2)(2\eta_m - 1) \d x,
	\end{split}\ee
	where $\eta_m(\rho)\ed\sin \left[\frac{\pi}{2}\left(\frac{m}{m-1} \rho-\frac{1}{m-1}\right)\right]$ is the cut-off function defined in (\ref{cut_fn}) in Lemma \ref{Lemcutoff}.
\end{lemma}

\begin{proof}
	Using integration by parts, we find 
	\be\label{lap_int_ibp}
	\int_{D_{m}} \bm{v} \cdot \Delta \bm{v} \d x = 
	-\int_{D_{m}} |\na\bm{v}|^2 \d x
	+\frac{1}{2}\int_{\p D_{m}}\na (|\bm{v}|^2) \cdot \bm{n}\d S\,.
	\ee
	Based on (\ref{grad_v_decom}), 
	\be\label{grad_int_est1}\begin{split}
		-\int_{D_{m}} |\na\bm{v}|^2 \d x &= -\int_{D_{m}} |\na\bm{b}|^2 \d x 
		- \int_{D_m} \bigg[ (\p_\rho v_\th)^2 + \Big(\frac{1}{\rho} \p_\phi v_\th \Big)^2 + \frac{1}{\sin^2\phi} \Big(\frac{v_\th}{\rho}\Big)^2 \bigg] \d x\,.
	\end{split}\ee
	
	Now we handle the boundary term. 
	\be\label{EB0}
	\begin{aligned}
		\int_{\p D_{m}}\na (|\bm{v}|^2) \cdot\bm{n}\d S = 
		&\un{\int_{R_{1,m}\cup R_{2,m}}\na (|\bm{v}|^2) \cdot \bm{n} \d S}_{B_1}
		+ \un{\int_{A_{1,m}\cup A_{2,m}}\na (|\bm{v}|^2) \cdot \bm{n} \d S}_{B_2}.
	\end{aligned}
	\ee
	Firstly, 
	\[
	\begin{aligned}
		B_{1}= & -\int_{R_{1,m}} \frac{1}{\rho} \partial_{\phi} (|\bm{v}|^{2}) \d S
		+ \int_{R_{2,m}} \frac{1}{\rho} \partial_{\phi} (|\bm{v}|^{2}) \d S \\
		= & -2 \pi \int_{\frac{1}{m}}^{1} \partial_{\phi} (|\bm{v}|^{2})\Big|_{\phi=\frac{\pi}{2}-\alpha} \sin \left(\frac{\pi}{2}-\alpha\right) \d \rho
		+ 2 \pi \int_{\frac{1}{m}}^{1} \partial_{\phi} (|\bm{v}|^{2}) \Big|_{\phi=\frac{\pi}{2}+\alpha} \sin \left(\frac{\pi}{2}+\alpha\right) \d \rho\,. \\
	\end{aligned}
	\]
	Noticing that $v_\phi=\p_\phi v_\rho=0$ and $\partial_{\phi} v_{\theta}=\cot \phi \, v_{\theta}$ on $R_{1,m}\cup R_{2,m}$, one deduces
	\be\label{EB1}
	\begin{aligned}
		B_1= &4 \pi \int_{\frac{1}{m}}^{1} \Big[\cos \left(\frac{\pi}{2}+\alpha\right) v_{\theta}^{2}\left(\rho, \frac{\pi}{2}+\alpha\right)-\cos \left(\frac{\pi}{2}-\alpha\right) v_{\theta}^{2}\left(\rho, \frac{\pi}{2}-\alpha\right)\Big] \d \rho \\
		= & 4 \pi \int_{\frac{1}{m}}^{1} \int_{\frac{\pi}{2}-\alpha}^{\frac{\pi}{2}+\alpha} \partial_{\phi}\left[ \cos \phi \, v_{\theta}^{2}(\rho, \phi)\right] \d \phi \d \rho \\
		= & 4 \int_{D_{m}} \frac{\cot \phi}{\rho^{2}} v_{\theta} \partial_{\phi} v_{\theta} \d x - 2 \int_{D_{m}} \frac{v_{\theta}^{2}}{\rho^{2}} \d x\,.\\
	\end{aligned}
	\ee
	Meanwhile, since $v_\rho=0$, $\partial_{\rho} v_{\phi}=\frac{1}{\rho} v_{\phi,}$ and $\partial_{\rho} v_{\theta}=\frac{1}{\rho} v_{\theta}$ on $A_{2,m}$, while $v_\rho=0$, $\partial_{\rho} v_{\phi}=-\frac{1}{\rho} v_{\phi,}$ and $\partial_{\rho} v_{\theta}=\frac{1}{\rho} v_{\theta}$ on $A_{1,m}$, it follows that 
	\[
	\begin{aligned}
		B_{2}= &2 \pi \int_{\frac{\pi}{2}-\alpha}^{\frac{\pi}{2}+\alpha} \partial_{\rho} (|\bm{v}|^{2}) \rho^{2} \sin \phi \d \phi\Big|_{\rho=1}
		- 2 \pi \int_{\frac{\pi}{2}-\alpha}^{\frac{\pi}{2}+\alpha} \partial_{\rho} (|\bm{v}|^{2}) \rho^{2} \sin \phi \d \phi\Big|_{\rho=\f{1}{m}} \\
		= &2 \pi \int_{\frac{\pi}{2}-\alpha}^{\frac{\pi}{2}+\alpha}\left(2 v_{\phi} \partial_{\rho} v_{\phi}+2 v_{\theta} \partial_{\rho} v_{\theta}\right) \rho^{2} \sin \phi \d \phi\Big|_{\rho=1}\\
		& -2 \pi \int_{\frac{\pi}{2}-\alpha}^{\frac{\pi}{2}+\alpha}\left(2 v_{\phi} \partial_{\rho} v_{\phi}+2 v_{\theta} \partial_{\rho} v_{\theta}\right) \rho^{2} \sin \phi \d \phi\Big|_{\rho=\f{1}{m}} \\
		= & 4 \pi \int_{\frac{\pi}{2}-\alpha}^{\frac{\pi}{2}+\alpha}\left(v_{\phi}^{2}+v_{\theta}^{2}\right)\rho\sin \phi \d \phi\Big|_{\rho=1}+4 \pi \int_{\frac{\pi}{2}-\alpha}^{\frac{\pi}{2}+\alpha}\left(v_{\phi}^{2}-v_{\theta}^{2}\right)\rho\sin \phi \d \phi\Big|_{\rho=\frac1m} \,.
	\end{aligned}
	\]
	Rearranging the above terms leads to 
	\[
	B_2 = \un{4\pi\int_{\f{\pi}{2}-\al}^{\f{\pi}{2}+\al}v_\phi^2 \, \rho\sin\phi\d\phi\Big|_{\rho=1}}_{B_{21}} + \un{4\pi\int_{\f{\pi}{2}-\al}^{\f{\pi}{2}+\al}v_\phi^2 \, \rho\sin\phi\d\phi\Big|_{\rho=\f{1}{m}}}_{B_{22}} + \un{4\pi\int_{\f{\pi}{2}-\al}^{\f{\pi}{2}+\al}v_\th^2 \, \rho\sin\phi\d\phi\Big|_{\rho=\f{1}{m}}^{1}}_{B_{23}}\,.
	\]
	
	To rewrite $B_{21}$ and $B_{22}$ in a form of volume integration, we use the cut-off function $\eta_m(\rho)$ which is defined in Lemma \ref{Lemcutoff}. By the Newton-Leibniz formula, we have, since $\eta_m=1$ at $\rho=1$ while $\eta_m=0$ at $\rho=\f{1}{m}$:
	\be\label{EB2}
	\begin{aligned}
		B_{21}= & 4 \pi \int_{\frac{\pi}{2}-\alpha}^{\frac{\pi}{2}+\alpha} \int_{\frac{1}{m}}^{1} \partial_{\rho}\left[\rho \sin \phi \, v_{\phi}^{2} \,\eta_{m}(\rho)\right] \d \rho \d \phi \\
		= & 2 \int_{D_{m}} \frac{v_{\phi}^{2}}{\rho^{2}}\eta_{m}(\rho) \d x
		+ 2 \int_{D_{m}} \frac{1}{\rho} (\partial_{\rho}v_{\phi}^{2}) \eta_{m}(\rho) \d x
		+ 2 \int_{D_{m}} \frac{v_{\phi}^{2}}{\rho}\eta_{m}^{\prime}(\rho) \d x\,,
	\end{aligned}
	\ee
	and
	\be\label{EB3}
	\begin{aligned}
		B_{22} & = 4 \pi \int_{\frac{\pi}{2}-\alpha}^{\frac{\pi}{2}+\alpha} \rho\sin \phi \, v_{\phi}^{2}\d \phi\Big|_{\rho=\frac1m}-B_{21}+B_{21} \\
		& = -4 \pi \int_{\frac{\pi}{2}-\alpha}^{\frac{\pi}{2}+\alpha} \int_{\frac{1}{m}}^{1} \partial_{\rho}\left(\rho \sin \phi \, v_{\phi}^{2}\right) \d \rho \d \phi + B_{21} \\
		& = -4 \pi \int_{\frac{\pi}{2}-\alpha}^{\frac{\pi}{2}+\alpha} \int_{\frac{1}{m}}^{1}\left[\sin \phi \,v_{\phi}^{2}+\rho \sin \phi \, (\p_{\rho} v_{\phi}^{2})\right] \d \rho \d \phi + B_{21} \\
		& =-2 \int_{D_{m}}\Big[\frac{1}{\rho^{2}}v_{\phi}^{2}+\frac{1}{\rho} (\p_{\rho}v_{\phi}^{2})\Big] \d x+B_{21}\,.
	\end{aligned}
	\ee
	Finally, using Newton-Leibniz formula yields 
	\ba\l{EB23}
	B_{23}=&4\pi\int_{\f{1}{m}}^1\int_{\f{\pi}{2}-\al}^{\f{\pi}{2}+\al}\p_\rho(\rho v_\th^2)\sin\phi\d\phi\d\rho=2\int_{D_m}\f{v_\th^2}{\rho^2}\d x
	+ 4\int_{D_m}\f{1}{\rho}v_\th\p_\rho v_\th\d x\,.
	\ea
	
	Substituting \eqref{EB1}, \eqref{EB2}, \eqref{EB3} and \eqref{EB23} into \eqref{EB0} yields 
	\[\begin{split}
		\int_{\p D_{m}}\na|\bm{v}|^2\cdot\bm{n}\d S &= 4\int_{D_m}\f{1}{\rho}v_\th\p_\rho v_\th\d x
		+  4 \int_{D_{m}} \frac{\cot \phi}{\rho^{2}} v_{\theta} \partial_{\phi} v_{\theta} \d x \\
		& \quad + 2\int_{D_m} \frac{1}{\rho^2} v_\phi^2 (2\eta_m + 2\rho\eta_m' - 1) \d x + 2\int_{D_m} \frac{1}{\rho} (\p_\rho v_\phi^2)(2\eta_m - 1) \d x.
	\end{split}\]
	Combining this identity with (\ref{grad_int_est1}), it then follows from (\ref{lap_int_ibp}) that 
	\be\label{lap_int_est2}\begin{split}
		\int_{D_m} \bm{v} \cdot \Delta \bm{v} \d x &=  -\int_{D_{m}} |\na\bm{b}|^2 \d x 
		- \int_{D_m} \bigg[ (\p_\rho v_\th)^2 + \Big(\frac{1}{\rho} \p_\phi v_\th \Big)^2 + \frac{1}{\sin^2\phi} \Big(\frac{v_\th}{\rho}\Big)^2 \bigg] \d x \\
		&\quad + 2\int_{D_m} \bigg[ \f{1}{\rho}v_\th\p_\rho v_\th +  \frac{\cot \phi}{\rho^{2}} v_{\theta} \partial_{\phi} v_{\theta} \bigg] \d x \\
		& \quad + 2\int_{D_m} \frac{1}{\rho^2} v_\phi^2 (2\eta_m + 2\rho\eta_m' - 1) \d x + 2\int_{D_m} \frac{1}{\rho} (\p_\rho v_\phi^2)(2\eta_m - 1) \d x.
	\end{split}\ee
	Noting that 
	\[\begin{split}
		&\quad - \bigg[ (\p_\rho v_\th)^2 + \Big(\frac{1}{\rho} \p_\phi v_\th \Big)^2 + \frac{1}{\sin^2\phi} \Big(\frac{v_\th}{\rho}\Big)^2 \bigg] + 2 \bigg[ \f{1}{\rho}v_\th\p_\rho v_\th +  \frac{\cot \phi}{\rho^{2}} v_{\theta} \partial_{\phi} v_{\theta} \bigg] \\
		& = - \bigg( \p_\rho v_\th - \frac{v_\th}{\rho} \bigg)^2 - \bigg( \frac{1}{\rho} \p_\phi v_\th - \frac{\cot \phi}{\rho} v_{\theta} \bigg)^2 = -2 |\mathbb{S}(v_\th \bm{e_\th})|^2,
	\end{split}\]
	where the last equality is due to (\ref{grad_S_th}), we have justified (\ref{lap_int}) from (\ref{lap_int_est2}).
	
\end{proof}

Now we are ready for the fundamental energy inequality. The key new ingredient compared with the NHL case is the Korn inequality established in Corollary \ref{Cor, Korn_ineq}, whose constant is independent of $m$. This allows us to absorb the swirl-related terms appearing in the energy identity (\ref{lap_int}).

\begin{prop}\label{Funden}
	Let $m\geq 10^3$ and $\al\in (0,\pi/6]$. Let $\bm{v}$ be the strong solution of \eqref{NS} on $D_m \times [0,T]$ under the boundary condition (\ref{bdry for Dm}) whose initial data satisfies
	\bn
	\int_{D_m}\rho\sin\phi \, v_{\th,0}\d x=0\,.
	\en
	Then the following energy estimate holds:
	\be\label{energy_est_Dm}
	\int_{D_m}|\bm{v}(x, T)|^{2} d x + \frac34\int_{0}^{T} \int_{D_m}|\nabla \bm{v}(x, s)|^{2} \d x \d s \leq \int_{D_m}|\bm{v}_{0}|^2\d x\,.
	\ee
\end{prop}
\begin{proof}
	Multiplying $\bm{v}$ on both sides of the original Navier-Stokes equation (\ref{NS1}) and then integrating over $D_m$, we find
	\[
	\int_{D_m} \bm{v} \cdot 
	\big[\Delta \bm{v}-(\bm{v} \cdot \nabla) \bm{v}-\nabla P-\partial_{t} \bm{v}\big] \d x = 0\,.
	\]
	Since $\bm{v}\cdot\bm{n}=0$ on $\p D_m$, one notices that
	\[
	\int_{D_{m}} \bm{v} \cdot \nabla P \d x=-\int_{D_{m}} P (\na\cdot\bm{v}) \d x=0\,,
	\] 
	and
	\[
	\int_{D_{m}}\bm{v}\cdot [(\bm{v} \cdot \nabla) \bm{v}] \d x
	= \frac{1}{2} \int_{D_{m}}\bm{v} \cdot \nabla (|\bm{v}|^{2}) \d x
	= -\frac{1}{2} \int_{D_{m}}(\na\cdot\bm{v}) \cdot |\bm{v}|^{2} \d x = 0\,.
	\]
	Thus we see
	\[
	\frac{1}{2} \f{\d}{\d t} \int_{D_{m}}|\bm{v}|^{2} \d x=\int_{D_{m}} \bm{v} \cdot \Delta \bm{v} \d x\,.
	\]
	Now we apply Lemma \ref{Lemma, lap_int} to find
	\ba\l{E04141}
	&\frac12 \f{\d}{\d t} \int_{D_{m}}|\bm{v}|^{2} \d x + \int_{D_m}|\na\bm{b}|^2\d x + 2 \int_{D_m} |\mathbb{S}(v_\th \bm{e_\th})|^2 \d x\\
	=& \int_{D_{m}} \frac{1}{\rho^{2}} v_{\phi}^{2}\left(2 \eta_{m}(\rho)
	+ \rho \eta_{m}^{\prime}(\rho)-1\right) \d x+\int_{D_{m}} \frac{1}{\rho} (\p_{\rho}v_\phi^2) \left(2 \eta_{m}(\rho)-1\right) \d x\,.
	\ea
	
	According to Lemma \ref{Lemcutoff}, we know 
	\[
	|2 \eta_{m}(\rho)+2 \rho \eta_{m}^{\prime}(\rho)-1| \leq \frac95 \quad\text{and}\quad 
	|2 \eta_{m}(\rho)-1| \leq 1.
	\]
	Then it follows from (\ref{E04141}) that 
	\bn
	& \f{\d}{\d t} \int_{D_{m}}|\bm{v}|^{2} \d x + 2\int_{D_m}|\na\bm{b}|^2\d x + 4 \int_{D_m} |\mathbb{S}(v_\th \bm{e_\th})|^2 \d x\\
	\leq & \frac{18}{5}\int_{D_{m}} \frac{1}{\rho^{2}} v_{\phi}^{2} \d x + 2\int_{D_{m}} \Big| \frac{1}{\rho} (\p_{\rho}v_\phi^2)\Big| \d x\,.
	\en
	Now we apply Corollary \ref{Cor, Korn_ineq} to obtain 
	\bn
	& \f{\d}{\d t} \int_{D_{m}}|\bm{v}|^{2} \d x + 2\int_{D_m}|\na\bm{b}|^2\d x + \frac32 \int_{D_m} |\nabla (v_\th \bm{e_\th})|^2 \d x\\
	\leq & \frac{18}{5}\int_{D_{m}} \frac{1}{\rho^{2}} v_{\phi}^{2} \d x + 2\int_{D_{m}} \Big| \frac{1}{\rho} (\p_{\rho}v_\phi^2)\Big| \d x\,.
	\en
	
	Using the Young inequality, for any $\la>0$, one has
	\[
	\begin{aligned}
		&\f{\d}{\d t} \int_{D_{m}}|\bm{v}|^{2} \d x+2\int_{D_m}|\na\bm{b}|^2\d x
		+ \frac32 \int_{D_m} |\nabla (v_\th \bm{e_\th})|^2 \d x \\
		\leq& \left(\frac{18}{5} + \frac{4}{\lambda}\right)\big\|\frac{v_\phi}{\rho}\big\|_{L^{2}}^{2}
		+ \lambda\left\|\partial_{\rho} v_{\phi}\right\|_{L^{2}}^{2}\,.
	\end{aligned}
	\]
	Setting $\la=\frac{5}{4}$, one derives
	\ba\l{042401}
	&\f{\d}{\d t} \int_{D_{m}}|\bm{v}|^{2} \d x + 2\int_{D_m}|\na\bm{b}|^2\d x
	+ \frac32 \int_{D_m} |\nabla (v_\th \bm{e_\th})|^2 \d x \\
	\leq\,\, &\frac{5}{4}\big\|\partial_{\rho} v_{\phi}\big\|_{L^{2}}^{2} + 7\Big\|\frac{v_{\phi}}{\rho}\Big\|_{L^{2}}^{2} \\
	\leq\,\, & \f{5}{4}\big\|\partial_{\rho} v_{\phi}\big\|_{L^{2}}^{2}
	+ \frac{21}{25}\Big\|\frac{\p_\phi v_{\phi}}{\rho}\Big\|_{L^{2}}^{2}\,,
	\ea
	where the last inequality took advantage of Lemma \ref{PoinB} with the constant $C_{\pi/6, B} = \frac{3}{25}$ thanks to the boundary condition $v_\phi = 0$ on $R_{1,m}\cup R_{2,m}$.
	
	Next, we will control the right hand side of (\ref{042401}) by $\frac54\int_{D_m} |\nabla \bm{b}|^2 \d x$ which can be absorbed by the left hand side of (\ref{042401}).  Based on (\ref{grad_b}), 
	\[
	\|\nabla \bm{b}\|_{L^2}^2 \geq \big\|\partial_{\rho} v_{\phi}\big\|_{L^{2}}^{2} + \|F_1\|_{L^2}^2 + \|F_2\|_{L^2}^2,
	\]
	where 
	\[
	F_1 \ed \frac{1}{\rho}(\p_\phi v_\phi + v_\rho)\,, \quad\text{and} \quad 
	F_2 \ed \frac{1}{\rho}(\p_\phi v_\rho - v_\phi)\,.
	\]
	So we aim to prove
	\be\label{F1F2ineq}
	\frac{21}{25}\Big\|\frac{\p_\phi v_{\phi}}{\rho}\Big\|_{L^{2}}^{2} 
	\leq \frac54 \big( \|F_1\|_{L^2}^2 + \|F_2\|_{L^2}^2 \big)\,.
	\ee
	Noting that $v_\phi=0$ on rays and the weighted zero mean value of $v_\rho$ on arcs is guaranteed by Lemma \ref{v_rho_mean0}, so we can use the triangle inequality and Poincar\'e inequalities in Lemma \ref{Poin0} and Lemma \ref{PoinB} to obtain
	\be\label{gra_b_1}
	\| F_1 \|_{L^2} \geq \Big\| \frac{1}{\rho} \p_\phi v_\phi \Big\|_{L^2} - \Big\| \frac{v_\rho}{\rho} \Big\|_{L^2}
	\geq \Big\| \frac{1}{\rho} \p_\phi v_\phi \Big\|_{L^2} - \sqrt{\frac{2}{19}} \Big\| \frac{1}{\rho} \p_\phi v_\rho \Big\|_{L^2},
	\ee
	and 
	\be\label{gra_b_2}
	\| F_2 \|_{L^2} \geq \Big\| \frac{1}{\rho} \p_\phi v_\rho \Big\|_{L^2} - \Big\| \frac{v_\phi}{\rho} \Big\|_{L^2}
	\geq \Big\| \frac{1}{\rho} \p_\phi v_\rho \Big\|_{L^2} - \frac{\sqrt{3}}{5} \Big\| \frac{1}{\rho} \p_\phi v_\phi \Big\|_{L^2}.
	\ee
	Multiplying (\ref{gra_b_1}) by $\sqrt{19/2}$ and then adding to (\ref{gra_b_2}) yields
	\[
	\sqrt{\frac{19}{2}} \| F_1 \|_{L^2} + \| F_2 \|_{L^2} \geq \bigg( \sqrt{\frac{19}{2}} - \frac{\sqrt{3}}{5} \bigg) \Big\| \frac{1}{\rho} \p_\phi v_\phi \Big\|_{L^2}.
	\]
	Now we square both sides and taking advantage of the Cauchy-Schwarz inequality to find
	\[
	\bigg( \sqrt{\frac{19}{2}} - \frac{\sqrt{3}}{5} \bigg)^2 \Big\| \frac{1}{\rho} \p_\phi v_\phi \Big\|_{L^2}^2 \leq \Big( \frac{19}{2} + 1\Big)(\|F_1\|_{L^2}^2 + \|F_2\|_{L^2}^2)\,.
	\]
	By direct calculation, we get 
	\be\label{F1F2_est1}
	\Big\| \frac{1}{\rho} \p_\phi v_\phi \Big\|_{L^2}^2  \leq \frac{10}{7} (\|F_1\|_{L^2}^2 + \|F_2\|_{L^2}^2)\,,
	\ee
	which further justifies  (\ref{F1F2ineq}).
	
	Plugging (\ref{F1F2ineq}) into (\ref{042401}) leads to 
	\bn
	&\f{\d}{\d t} \int_{D_{m}}|\bm{v}|^{2} \d x+2\int_{D_m}|\na\bm{b}|^2\d x
	+ \frac32 \int_{D_m} |\nabla (v_\th \bm{e_\th})|^2 \d x \\
	\leq\,\, & \f{5}{4} \Big( \big\|\partial_{\rho} v_{\phi}\big\|_{L^{2}}^{2} + \|F_1\|_{L^2}^2 + \|F_2\|_{L^2}^2\Big) \leq \frac54 \| \nabla \bm{b}\|_{L^2}^2\,,
	\en
	where the last inequality is due to the expression $(\ref{grad_b})$ of $\bm{b}$. Hence, 
	\[
	\f{\d}{\d t} \int_{D_{m}}|\bm{v}|^{2} \d x + \frac{3}{4}\int_{D_m}|\na\bm{b}|^2\d x + \frac32 \int_{D_m} |\nabla (v_\th \bm{e_\th})|^2 \d x \leq 0\,.
	\]
	Since $| \nabla \bm{v}|^{2} = |\nabla \bm{b}|^{2} + |\nabla (v_\th \bm{e_\th})|^{2}$, we conclude that 
	\[
	\f{\d}{\d t} \int_{D_{m}} |\bm{v}|^{2} \d x + \frac{3}{4}\int_{D_m}|\na\bm{v}|^2\d x 
	\leq 0\,.
	\]
	Integrating with time on $[0,T]$, we find
	\[
	\int_{D_{m}}|\bm{v}(x, T)|^{2} d x+\f{3}{4}\int_0^T\int_{D_m}|\na\bm{v}(x,s)|^2\dx \d s 	\leq \int_{D_{m}}|\bm{v}_{0}|^2\d x\,.
	\]
\end{proof}

Proposition \ref{Funden} provides a uniform energy bound independent of $m$ for the velocity $\bm{v}$. This estimate forms the foundation for the higher-order estimates of the good unknowns $(\m{K},\m{F},\m{O})$ in Section \ref{Sec9}.

\section{$L^\infty$ estimate of $\G$ by its initial data in $D_m$}\label{Sec7}

Recall that $\G = r v_\th$. Since the higher-order energy estimates in Section \ref{Sec9} require the smallness of $\G$, the purpose of this section is to show that the initial smallness assumption (\ref{COND}) propagates for all times. More precisely, we will derive an \(L^\infty\) bound for \(\Gamma\) in terms of its initial data: 
\[
\|\G(t,\cd)\|_{L^\i(D_m)}\leq C\|\G_0\|_{L^\i(D_m)}\,,
\]
where $C$ is an absolute constant which is independent of $m$.
Under the NHL boundary condition as that in \cite{LPYZZZ24}, the above estimate follows directly from the maximum principle applied to the equation for $\G$. However, under the Navier total-slip boundary condition, the induced boundary condition for $\G$ generates a boundary contribution with an unfavorable sign, preventing a direct application of the standard maximum principle. Our new strategy consists of two parts. First, by deriving the energy estimate of $\G$, we show the space-time $L^2$ bound of $\G(x,t)$ is uniformly bounded by initial data. Secondly, by virtue of the De Giorgi iteration, we prove that $\|\G(\cd,t)\|_{L^\i(D_m)}$ can also be controlled by its initial data.

\begin{lemma}[Fundamental energy of $\G$]\label{lem3.2}
	Let $m \geqslant 10^3$ and $\alpha \in\left(0, \frac{\pi}{6}\right]$. Then 
	\bn
	\|\Gamma\|_{L_{tx}^{2}\left(D_{m} \times[0, T]\right)} \leq 3\left\|\Gamma_{0}\right\|_{L^{2}\left(D_{m}\right)}\,.
	\en
\end{lemma}
\begin{proof}
	Recalling the equation (\ref{eqvth}) for $\G$ and the boundary condition (\ref{bdry for Dm}) for $v_\th$, we find $\G$ is a solution to the following system:
	\ba\l{EG}
	\left\{\begin{aligned}
		&\Delta \Gamma-\bm{b} \cdot \nabla \Gamma-\frac{2}{\rho} \partial_{\rho} \Gamma-\frac{2 \cot \phi}{\rho^{2}} \partial_{\phi} \Gamma-\partial_{t} \Gamma=0, \quad \text { in } D_{m} . \\
		&\partial_{\phi} \Gamma=2(\cot \phi) \Gamma, \quad \text { on }\left(R_{1, m} \cup R_{2, m}\right) \times(0, T)\,, \\
		&\partial_{\rho} \Gamma=\frac{2}{\rho} \Gamma, \text { on } \left(A_{1 ,m}\cup A_{2 ,m}\right) \times(0, T)\,.
	\end{aligned}\right.
	\ea
	Testing equation \ef{EG}$_1$ by $\Gamma$ yields 
	\be\label{ESTG}
	\begin{aligned}
		& \frac{1}{2} \int_{D_m}\Gamma^2(x,T) \d x+\int_0^T\int_{D_m}|\nabla \Gamma|^2 \d x \d t \\
		= & \frac{1}{2} \int_{D_m} \Gamma_0^2(x) \d x+\un{\int_0^T \int_{\partial D_m} \Gamma \partial_{n} \Gamma \d S \d t-\int_0^T \int_{D_m} \frac{2}{\rho} \Gamma \partial_\rho \Gamma \d x \d t}_{G_1}\\
		&-\int_0^T \int_{D_m} \frac{2 \cot \phi}{\rho^2} \Gamma \partial_\phi \Gamma \d x \d t\,.
	\end{aligned}
	\ee
	Applying integration by parts, we get 
	\ba\l{0410E1}
	-\int_{D_m} \frac{2}{\rho} \Gamma \partial_\rho \Gamma \d x=&-2\pi\int_{\f{\pi}{2}-\al}^{\f{\pi}{2}+\al}\int_{\f{1}{m}}^1\rho\sin\phi\p_\rho\G^2\d\rho\d\phi\\
	=&\int_{D_m}\f{\G^2}{\rho^2}\d x-2\pi\int_{\f{\pi}{2}-\al}^{\f{\pi}{2}+\al}\G^2\rho\sin\phi\Big|_{\rho=\f{1}{m}}^1\d\phi\,.
	\ea
	Using the boundary condition in \eqref{EG}, the boundary integration satisfies
	\be\l{0410E2}
	\begin{aligned}
		\int_{\partial D_m} \Gamma \partial_n \Gamma \d S=&
		4\pi\int_{\pi/2-\al}^{\pi/2+\al}\G^2\rho\sin\phi\Big|_{\rho=\f{1}{m}}^{1}\d\phi + 
		4\pi\int_{\frac1m}^{1}\G^2\cos\phi\Big|_{\phi=\pi/2-\al}^{\pi/2+\al}\d\rho.
	\end{aligned}
	\ee
	Noticing the $\cos\phi$ is negative when $\phi=\pi/2+\al$ and negative when $\phi=\pi/2-\al$, the second term on the right hand side of \eqref{0410E2} is non-positive. Meanwhile, the integral along the inner arc also carries a good sign, so we can omit it. Thus we infer from \eqref{0410E1}--\eqref{0410E2} that
	\[
	\begin{aligned}
		&\int_{\partial D_m} \Gamma \partial_n \Gamma \d x-\int_{D_m} \frac{2}{\rho} \Gamma \partial_\rho \Gamma \d x\\
		\leq& 2\pi\int_{\pi/2-\al}^{\pi/2+\al}\G^2\rho\sin\phi\Big|_{\rho=1}\d\phi + 2\pi\int_{\frac1m}^{1}\G^2\cos\phi\Big|_{\phi=\pi/2-\al}^{\pi/2+\al}\d\rho+\int_{D_m}\f{\G^2}{\rho^2}\d x\\
		=& 2\pi\int_{\frac1m}^{1}\int_{\pi/2-\al}^{\pi/2+\al}\p_\rho(\G^2\rho^3\eta_{m})\sin\phi\d\phi\d\rho + 2\pi\int_{\frac1m}^{1}\int_{\phi=\pi/2-\al}^{\pi/2+\al}\p_\phi(\G^2\cos\phi)\d\phi\d\rho+\int_{D_m}\f{\G^2}{\rho^2}\d x\,,
	\end{aligned}
	\]
	where $\eta_{m} = \eta_{m}(\rho)$ refers to the cut-off function given in Lemma \ref{Lemcutoff}. Direct calculation shows
	\bn
	\int_{\partial D_m} \Gamma \partial_n \Gamma \d x-\int_{D_m} \frac{2}{\rho} \Gamma \partial_\rho \Gamma \d x\leq
	& \bigg(2 \int_{D_m} \rho\Gamma (\p_\rho \Gamma) \eta_{m} \d x+3 \int_{D_m}\Gamma^2\eta_{m}\d x + \int_{D_m}\rho|\Gamma|^2\eta_{m}'\d x\bigg)\\
	& + 2 \int_{D_m} \frac{\Gamma}{\rho} \frac{\partial_\phi \Gamma}{\rho}\cot\phi \d x.
	\en
	Substituting above in \eqref{ESTG}, one deduces
	\bn
	& \frac{1}{2} \int_{D_m}\Gamma^2(x,T) \d x+\int_0^T\int_{D_m}|\nabla \Gamma|^2 \d x \d t \\
	=\,\, & \frac{1}{2} \int_{D_m} \Gamma_0^2(x) \d x+2 \int_0^T\int_{D_m} \rho\Gamma (\p_\rho \Gamma) \eta_{m} \d x\d t \\
	& +3\int_0^T\int_{D_m}\Gamma^2\eta_{m}\d x\d t + \int_0^T\int_{D_m}\rho|\Gamma|^2\eta_{m}'\d x\d t\,.
	\en
	Using Lemma \ref{Lemcutoff} and the assumption that $m\geq 10^3$, we have 
	$|3\eta_m+ \rho\eta_{m}^{\prime}| \leq \f{13}{4}$.
	Therefore,
	$$
	\begin{aligned}
		&\frac{1}{2} \int_{D_{m}} \Gamma^{2}(x, T) \d x+\int_{0}^{T} \int_{D_{m}} |\nabla \Gamma|^{2} \d x \d t \\
		\leq & \frac{1}{2} \int_{D_{m}} \Gamma_{0}^{2}(x) d x+\f{13}{4}\int_0^T\int_{D_m}\G^2\d x\d t+2\int_0^T\int_{D_m}|\G| |\p_\rho\G|\d x\d t \,.
	\end{aligned}
	$$
	Using Cauchy-Schwartz inequality and the Poincar\'e inequality Lemma \ref{Poin3D}, for any positive constant $\lambda_1$, 
	$$
	\begin{aligned}
		& \frac{1}{2} \int_{D_{m}} \Gamma^{2}(x, T) d x+\int_{0}^{T} \int_{D_{m}}|\na\Gamma|^{2} \d x \d t \\
		\leq & \frac{1}{2} \int_{D_{m}} \Gamma_{0}^{2}(x) d x + \bigg(\f{13}{4} + \lambda_1 \bigg)\int_0^T\int_{D_m}\G^2\d x\d t + \lambda_{1}^{-1} \int_{0}^{T}\int_{D_{m}}|\partial_{\rho}\Gamma|^{2} \d x \d t \\
		\leq &\frac{1}{2} \int_{D_{m}} \Gamma_{0}^{2}(x) d x 
		+ \f{2}{19}\left(\f{13}{4} + \la_1\right) \int_{0}^{T} \int_{D_{m}} |\na\Gamma|^2\d x \d t
		+ \lambda_{1}^{-1} \int_{0}^{T}\int_{D_{m}}|\partial_{\rho}\Gamma|^{2} \d x \d t\,.
	\end{aligned}
	$$
	By choosing $\la_1 = \sqrt{19/2}$, we deduce that
	\bn
	\frac{1}{2} \int_{D_{m}} \Gamma^{2}(x, T) \d x+\int_{0}^{T} \int_{D_m} |\nabla \Gamma|^2\d x \d t\leq\frac{1}{2} \int_{D_{m}} \Gamma_{0}^{2}(x) \d x + 0.992\int_{0}^{T} \int_{D_m} |\nabla \Gamma|^2\d x \d t\,.
	\en
	This indicates that
	\bn
	\int_{0}^{T} \int_{D_m} |\nabla \Gamma|^2\d x \d t\leq \frac{125}{2}\int_{D_{m}} \Gamma_{0}^{2}(x) \d x\,.
	\en
	Then by Poincar\'e inequality Lemma \ref{Poin0} again, 
	$$
	\|\G\|_{L^2_{xt}(D_m\times[0,T])}\leq
	\sqrt{\frac{2}{19}}\|\nabla\G\|_{L^2_{xt}(D_m\times[0,T])}\leq 
	3\left\|\Gamma_{0}\right\|_{L^{2}\left(D_{m}\right)}\,.
	$$
	This completes the proof of the lemma.
\end{proof}
Based on Lemma \ref{lem3.2}, the following lemma gives a uniform bound for $\G$ by De Giorgi iteration.
\begin{lemma}\l{LEMGA}
	Let $m \geq 10^3, \alpha \in\left(0, \frac{\pi}{6}\right]$, then for any $T>0$,
	\ba\l{DeGest}
	\|\Gamma\|_{L^{\infty}(D_m \times[0, T])} \leq C_{\G}\left\|\Gamma_{0}\right\|_{L^{\infty}\left(D_{m}\right)}\,,
	\ea
	where $C_\G$ is an absolute constant which is independent of $m$ and $T$, given in Remark \ref{Remark, Cstar}.
\end{lemma}
\begin{proof}
	We only prove the positive part of \eqref{DeGest}, i.e.
	\[
	\|\max\{\Gamma,0\}\|_{L^{\infty}(D_m \times[0, T])} \leq C\left\|\Gamma_{0}\right\|_{L^{\infty}\left(D_{m}\right)}\,,
	\]
	since $-\G$ satisfies the same equation as $\G$. 
	
	\leftline{\textbf{Step 1: Energy estimates for truncated $\Gamma$.}}
	
	Fix $k\geq0$ and set
	\[
	w\ed(\G-k)_+=\max\{\G-k,0\}\,.
	\]
	Test the weak form of \eqref{EG}$_1$ by $w$. 
	\ba\l{EGAMMA}
	& - \int_{D_m} (\partial_t\Gamma)w\d x + \int_{D_m} (\Delta\Gamma)w\d x-\int_{D_m} (\bm{b}\cdot\nabla\Gamma)w\d x
	-\int_{D_m} \frac{2}{\rho} (\p_\rho\Gamma) w\d x\\
	-&\int_{D_m} \frac{2\cot\phi}{\rho^{2}} (\p_\phi\Gamma) w\d x = 0.
	\ea
	Using $w=(\G-k)_+$, the term for time derivative is
	\ba\l{ETem}
	\int_{D_m} (\partial_t\Gamma)w\d x=\frac12\frac{\d}{\d t}\int_{D_m} w^2\d x\,.
	\ea
	Since $\na\G=\na w$ a. e. on $\{\G>k\}$ and $\na w=0$ on $\{\G\leq k\}$, the Laplacian term gives 
	\ba\l{Lap0428}
	\int_{D_m} (\Delta\Gamma)w\d x=& -\int_{D_m} \nabla\Gamma\cdot\nabla w\d x + \int_{\partial D_m} (\partial_n\Gamma)w\d S\\
	=& -\int_{D_m} |\nabla w|^2\d x + \un{2\pi\int_{\f{\pi}{2}-\al}^{\f{\pi}{2}+\al}w (\p_\rho\G) \rho^2\sin\phi\d\phi\bigg|_{\rho=\f{1}{m}}^1}_{B_1}\\
	& + \un{2\pi\int_{\f{1}{m}}^{1} w (\p_\phi \G) \sin\phi\d\rho\bigg|_{\phi=\f{\pi}{2}-\al}^{\f{\pi}{2}+\al}}_{B_2}\,.
	\ea
	In addition, 
	\be\label{drift_zero}
	\int_{D_m} (\bm{b}\cdot\nabla\Gamma)w\d x = \int_{D_m} (\bm{b}\cdot\nabla w)w\d x = 0.
	\ee
	Plugging (\ref{ETem})--(\ref{drift_zero}) into (\ref{EGAMMA}) yields 
	\be\label{Lap1}\begin{split}
		&\frac12\frac{\d}{\d t}\int_{D_m} w^2\d x + \int_{D_m} |\nabla w|^2\d x \\
		=\quad & B_1 + B_2 - \int_{D_m} \frac{2}{\rho} (\partial_\rho\Gamma) w\d x - \int_{D_m} \frac{2\cot\phi}{\rho^{2}} (\partial_\phi\Gamma) w\d x,
	\end{split}\ee
	where $B_1$ and $B_2$ are as defined in (\ref{Lap0428}).
	
	Using \eqref{EG}$_2$, one deduces
	\ba\l{BG2}
	B_2=2\pi\int_{\f{1}{m}}^12 w(w+k)\cos\phi\d\rho\bigg|_{\phi=\f{\pi}{2}-\al}^{\f{\pi}{2}+\al}\,.
	\ea
	Since $\cos\phi$ is negative at $\phi = \frac{\pi}{2} + \al$ and positive at $\phi = \frac{\pi}{2} - \al$, the right hand side of (\ref{BG2}) is negative, so 
	\[
	B_2 \leq 2\pi\int_{\f{1}{m}}^1 w(w+k)\cos\phi\d\rho\bigg|_{\phi=\f{\pi}{2}-\al}^{\f{\pi}{2}+\al} \leq 2\pi\int_{\f{1}{m}}^1 w^2\cos\phi\d\rho\bigg|_{\phi=\f{\pi}{2}-\al}^{\f{\pi}{2}+\al}.
	\]
	To exploit cancellation properties, we use the fundamental theorem of calculus to find
	\begin{align*}
		B_2 &\leq 2\pi\int_{\f{1}{m}}^1 \int_{\f{\pi}{2}-\al}^{\f{\pi}{2}+\al} \p_{\phi} [w^2\cos\phi] \d \phi \d \rho \\
		&= 2\pi\int_{\f{1}{m}}^1 \int_{\f{\pi}{2}-\al}^{\f{\pi}{2}+\al} [ 2w (\p_\phi w) \cos\phi - w^2 \sin\phi] \d \phi \d \rho \\
		&= \int_{D_m} \frac{2 \cot\phi}{\rho^2} w (\p_\phi w) \d x - \int_{D_m} \frac{w^2}{\rho^2} \d x.
	\end{align*}
	Noticing $w\p_\phi w = w \p_\phi \Gamma$, so 
	\be\label{B2}
	B_2 \leq \int_{D_m} \frac{2 \cot\phi}{\rho^2} w (\p_\phi \Gamma) \d x - \int_{D_m} \frac{w^2}{\rho^2} \d x.
	\ee
	Putting (\ref{B2}) into (\ref{Lap1}) yields 
	\be\label{Lap2}
	\frac12\frac{\d}{\d t}\int_{D_m} w^2\d x + \int_{D_m} |\nabla w|^2\d x 
	\leq B_1 - \int_{D_m} \frac{2}{\rho} (\partial_\rho\Gamma) w\d x 
	- \int_{D_m} \frac{w^2}{\rho^2} \d x.
	\ee
	
	Next, we compute $B_1$. Recalling the boundary condition \eqref{EG}$_3$, we have
	\[
	B_1 = 4\pi \int_{\f{\pi}{2}-\al}^{\f{\pi}{2}+\al} \rho (w^2+kw)(\rho,\phi,t)\Big|_{\rho=\f{1}{m}}^1 \sin\phi\d\phi.
	\]
	Meanwhile,
	\bn
	-\int_{D_m} \f{2}{\rho} (\p_\rho\G) w\d x =& -2\pi\int_{\frac1m}^1\int_{\f{\pi}{2}-\al}^{\f{\pi}{2}+\al} 2\rho w (\p_\rho w) \sin\phi\d\rho\d\phi\\
	= & 2\pi\int_0^1\int_{\f{\pi}{2}-\al}^{\f{\pi}{2}+\al}w^2\sin\phi\d\rho\d\phi - 2\pi\int_{\f{\pi}{2}-\al}^{\f{\pi}{2}+\al}w^2\rho\sin\phi\Big|_{\rho=\frac1m}^{1} \d\phi\\
	=& \int_{D_m} \frac{w^2}{\rho^2} \d x - 2\pi\int_{\f{\pi}{2}-\al}^{\f{\pi}{2}+\al}w^2\rho\sin\phi\Big|_{\rho=\frac1m}^{1}\d\phi\,.
	\en
	Plugging the above two estimates into (\ref{Lap2}) leads to 
	\bn
	& \frac12\frac{\d}{\d t}\int_{D_m} w^2\d x + \int_{D_m} |\nabla w|^2\d x \\
	\leq\,\, & 2\pi\int_{\f{\pi}{2}-\al}^{\f{\pi}{2}+\al}\rho w^2(\rho,\phi,t)\Big|_{\rho=\f{1}{m}}^1\sin\phi\d\phi
	+ 4\pi\int_{\f{\pi}{2}-\al}^{\f{\pi}{2}+\al} k\rho w(\rho,\phi,t)\Big|_{\rho=\f{1}{m}}^1\sin\phi\d\phi\,.
	\en
	Next, we drop all the integral terms on the inner arc, which have the good sign, to obtain
	\be\label{Lap3}\begin{split}
		& \frac12\frac{\d}{\d t}\int_{D_m} w^2\d x + \int_{D_m} |\nabla w|^2\d x \\
		\leq\,\, & \un{2\pi\int_{\f{\pi}{2}-\al}^{\f{\pi}{2}+\al}  w^2(1,\phi,t) \sin\phi\d\phi}_{B_3}
		+ \un{4\pi k\int_{\f{\pi}{2}-\al}^{\f{\pi}{2}+\al} w(1,\phi,t) \sin\phi\d\phi}_{B_4}.
	\end{split}\ee
	For the integral terms on the outer arc, we recall $\eta_{m} = \eta_{m}(\rho)$, the cut-off function given in Lemma \ref{Lemcutoff}. Using the Newton-Leibniz formula:
	\bn
	B_3 &= 2\pi\int_{\f{\pi}{2}-\al}^{\f{\pi}{2}+\al} \int_{\frac1m}^{1} \p_\rho (\rho^3 w^2 \eta_m ) \sin\phi \d \rho \d\phi \\
	&= \int_{D_m} \big( 3 w^2 \eta_m + 2\rho w (\p_\rho w) \eta_m + \rho w^2 \eta_m'\big) \d x \\
	&= \int_{D_m} (3\eta_m + \rho \eta_m') w^2 \d x + 2 \int_{D_m} \rho w (\p_\rho w) \eta_m \d x.
	\en	
	Taking advantage of Lemma \ref{Lemcutoff} and applying the Cauchy-Schwarz inequality, we find 
	\be\label{B3}\begin{split}
		|B_3| &\leq \frac{13}{4} \int_{D_m} w^2 \d x + \Big(4 \int_{D_m} w^2 \d x  + \frac14 \int_{D_m} |\p_\rho w|^2 \d x\Big) \\
		&=  \frac{29}{4} \int_{D_m} w^2 \d x + \frac14 \int_{D_m} |\p_\rho w|^2 \d x.
	\end{split}\ee
	Meanwhile, 
	\bn
	B_4 &= 4\pi k \int_{\f{\pi}{2}-\al}^{\f{\pi}{2}+\al} \int_{\frac1m}^{1} \p_\rho (\rho^3 w \eta_m ) \sin\phi \d \rho \d\phi \\
	&= 2k \int_{D_m} (3\eta_m + \rho \eta_m') w \d x + 2k \int_{D_m} \rho (\p_\rho w) \eta_m \d x\,.
	\en
	Let $S$ denote the support of the function $w$ within the space-time domain $D_m\times[0,T]$:
	\be\label{S_supp}
	S \ed \{(x,t)\in D_{m}\times [0,T]: \Gamma(x,t) > k\}\,.
	\ee
	Accordingly, we define $\chi_{S}$ to be the characteristic function of $S$. Then it follows from Lemma \ref{Lemcutoff} that 
	\[
	|B_4| \leq \frac{13}{2} \int_{D_m} k w \chi_{S} \d x + 2 \int_{D_m} k |\p_\rho w| \chi_{S}\d x\,.
	\]
	Now we apply Cauchy-Schwarz inequality to obtain
	\be\label{B4}\begin{split}
		|B_4| &\leq \frac{13}{2} \Big( \frac12\int_{D_m} k^2 \chi_S \d x + \frac12 \int_{D_m} w^2 \d x\Big) + \Big( 4\int_{D_m} k^2 \chi_S \d x + \frac14 \int_{D_m} |\p_\rho w|^2 \d x\Big) \\
		&\leq  8\int_{D_m} k^2 \chi_S \d x + \frac{13}{4} \int_{D_m} w^2 \d x + \frac14 \int_{D_m} |\p_\rho w|^2 \d x\,.
	\end{split}\ee
	Plugging (\ref{B3}) and (\ref{B4}) into (\ref{Lap3}) yields 
	\[\begin{split}
		& \frac12\frac{\d}{\d t}\int_{D_m} w^2\d x + \int_{D_m} |\nabla w|^2\d x\leq8\int_{D_m} k^2 \chi_S \d x + \frac{21}{2} \int_{D_m} w^2 \d x + \frac12 \int_{D_m} |\p_\rho w|^2 \d x\,,
	\end{split}\]
	which implies that 
	\[
	\frac{\d}{\d t}\int_{D_m} w^2\d x + \int_{D_m} |\nabla w|^2\d x
	\leq 16\int_{D_m} k^2 \chi_S \d x + 21 \int_{D_m} w^2 \d x\,.
	\]
	Integrating in time $t$ over $[0,T]$ leads to
	\[\begin{split}
		& \int_{D_m} w^2(x,T) \d x + \int_{0}^{T}\int_{D_m} |\nabla w|^2\d x \d t \\
		\leq\,\, & \int_{D_m} w^2(x,0) \d x + 16 k^2 |S| + 21\int_{0}^{T}\int_{D_m} w^2 \d x \d t\,,
	\end{split}\]
	where $|S|$ represents the Lebesgue measure of the set $S$ in (\ref{S_supp}).
	Since $T$ is arbitrary, the above result implies that 
	\be\label{Lap4}\begin{split}
		&\sup_{t\in[0,T]}\int_{D_m}|w(x,t)|^2\d x + \int_{0}^{T}\int_{D_m} |\nabla w|^2\d x \d t\\
		\leq\,\,& 2\int_{D_m} w^2(x,0) \d x + 42\int_{0}^{T}\int_{D_m} w^2 \d x \d t + 32 k^2 |S|\,.
	\end{split}\ee
	
	\leftline{\textbf{Step 2: Iteration with $k$}.}
	We set
	\ba\l{DefB}
	B \ed C_{\Gamma} \|\G_0\|_{L^\i}\,,
	\ea
	where $C_{\Gamma}$ is a positive constant to be determined later. Then for any integer $n\geq 0$, we denote
	\[
	k_n = (1-2^{-n})B\,, \quad w_n\ed(\Gamma-k_n)_+\,, \quad S_{n} = \{(x,t) \in D_m\times [0,T]: \Gamma > k_n\}\,.
	\]
	We first require $C_{\G} \geq 2$, then for any $n\geq 1$, $k_n \geq B/2 \geq \|\Gamma_0\|_{L^\infty}$, so 
	\[
	w_n(x,0) = (\Gamma_0(x) - k_n)_{+} = 0, \quad \forall\, x\in D_m\,.
	\]
	
	For $n\geq 1$, replacing $w$ with $w_n$ in (\ref{Lap4}) yields
	\be\label{Lap5}\begin{split}
		&\sup_{t\in[0,T]}\int_{D_m}|w_n(x,t)|^2\d x + \int_{0}^{T}\int_{D_m} |\nabla w_n|^2\d x \d t + \int_{0}^{T}\int_{D_m} w_n^2\d x \d t\\
		\leq\,\, & 43 \int_{0}^{T}\int_{D_m} w_n^2 \d x \d t + 32 k^2 |S_n|\,.
	\end{split}\ee	
	For $n\geq 0$, define
	\[
	Y_n\ed \int_0^T\int_{D_m}|w_n|^2\d x\d t\,.
	\]
	Using H\"older's inequality and the Sobolev embedding theorem, one deduces
	\be\label{Yn_est1}
	Y_{n} = \int_0^T\int_{D_m} |w_{n}|^2\d x\d t
	\leq \| w_{n} \|_{L^{10/3}_{tx}(D_m\times [0,T])}^{2} \big| S_{n} \big|^{\f{2}{5}} .
	\ee
	Applying the H\"older's inequality, we have
	\[
	\| w_{n} \|_{L^{10/3}_{tx}(D_m\times [0,T])} \leq\|w_{n}\|_{L^{\infty}_{t} L^{2}_{x}(D_m\times [0,T])}^{2/5} 
	\|w_{n}\|_{L^{2}_{t} L^{6}_{x}(D_m\times [0,T])}^{3/5}\,,
	\]
	and this further implies by Young's inequality that
	\[
	\| w_{n} \|_{L^{10/3}_{tx}(D_m\times [0,T])} \leq \frac25 \|w_{n}\|_{L^{\infty}_{t} L^{2}_{x}(D_m\times [0,T])} + \frac{3}{5} \|w_{n}\|_{L^{2}_{t} L^{6}_{x}(D_m\times [0,T])}\,.
	\]
	So 
	\[
	\| w_{n} \|_{L^{10/3}_{tx}(D_m\times [0,T])}^{2} \leq\|w_{n}\|_{L^{\infty}_{t} L^{2}_{x}(D_m\times [0,T])}^{2} + \|w_{n}\|_{L^{2}_{t} L^{6}_{x}(D_m\times [0,T])}^{2}\,.
	\]
	Plugging this estimate into (\ref{Yn_est1}) leads to 
	\bn
	Y_{n} &\leq \Big( \|w_{n}\|_{L^{\infty}_{t} L^{2}_{x}(D_m\times [0,T])}^{2} +  \|w_{n}\|_{L^{2}_{t} L^{6}_{x}(D_m\times [0,T])}^{2} \Big) \big| S_{n} \big|^{\f{2}{5}}.
	\en
	From the Sobolev embedding theorem, there exists a constant $C_{sob} > 0$ such that 
	\be\label{Sob_emb}
	\|w_{n}\|^2_{L^{2}_{t} L^{6}_{x}(D_m\times [0,T])} \leq C_{sob} \|w_{n}\|^2_{L^{2}_{t} H^{1}_{x}(D_m\times [0,T])}\,.
	\ee
	Combining the above three inequalities together leads to 
	\[
	Y_{n} \leq\Big( \|w_{n}\|_{L^{\infty}_{t} L^{2}_{x}(D_m\times [0,T])}^{2} + C_{sob} \|w_{n}\|_{L^{2}_{t} H^{1}_{x}(D_m\times [0,T])}^{2} \Big) \big| S_{n} \big|^{\f{2}{5}}, \quad \forall\, n\geq 0.
	\]
	Plugging this into (\ref{Lap5}) with $n\geq 1$ yields 
	\[\begin{split}
		Y_{n} &\leq C_1 \Big( \|w_{n}\|_{L^{\infty}_{t} L^{2}_{x}(D_m\times [0,T])}^{2} + \| w_{n}\|_{L^{2}_{t} H^{1}_{x}(D_m\times [0,T])}^{2} \Big) \big| S_{n} \big|^{\f{2}{5}} \\
		& \leq 43 C_1 \bigg( \int_{0}^{T}\int_{D_m} w_n^2 \d x \d t + k_n^2 |S_n| \bigg) |S_n|^{\frac25},
	\end{split}\]
	where 
	\bn
	C_1 := \max\{ 1, C_{sob} \}.
	\en
	Since $k_n^2 \leq B^2$, then 
	\be\label{Yn_est3}
	Y_{n} \leq C_2 ( Y_n + B^2 |S_n|) |S_n|^{\frac25}\,, \qquad\forall\, n\geq 1\,,
	\ee
	where $C_2 := 43 C_1$. For $n\geq 0$, noticing that 
	\[
	w_{n}(x,t) = (\Gamma(x,t) - k_n)_{+} > (k_{n+1} - k_n) = 2^{-(n+1)} B \quad\text{on}\quad S_{n+1}\,,
	\] 
	so
	\be\label{Sn_est1}
	|S_{n+1}| \le {2^{2n+2}}{B^{-2}}\int_0^T\int_{D_m} |w_n|^2\d x\d t = 2^{2n+2} B^{-2} Y_n\,.
	\ee
	As a result, for any $n\geq 0$, by applying (\ref{Yn_est3}) with $n+1$ and using \eqref{Sn_est1} yields
	\be\label{Yn_est4}
	Y_{n+1} \leq C_2  (Y_{n+1} + 2^{2n+2}Y_{n}) (2^{2n+2} B^{-2} Y_n)^{\frac25}.
	\ee
	Noticing $Y_{n+1}\leq Y_{n}$, it then follows from (\ref{Yn_est4}) that 
	\[
	Y_{n+1} \leq C_2 (1 + 2^{2n+2}) Y_n(2^{2n+2} B^{-2} Y_n)^{\frac25}.
	\]
	Define $\widetilde{Y}_{n} = B^{-2} Y_n$. Then the above estimate indicates that 
	\ba\l{Degiorgi}
	\widetilde{Y}_{n+1}\leq C_3 8^n \widetilde{Y}_n^{\f{7}{5}}, \qquad \forall\, n\geq 0,
	\ea
	where $C_3 := 10C_2$.
	
	Taking logarithm of both sides of (\ref{Degiorgi}) yields 
	\[
	\ln \widetilde{Y}_{n+1} \leq \ln C_3 + \la n + \frac75 \ln \widetilde{Y}_n,
	\]
	where $\la := \ln 8$. The above inequality can be converted into a recurrence relation: 
	\[g_{n+1} \leq \frac75 g_{n}, \quad\forall\, n\geq 0,\]
	where 
	\[
	g_{n}:= \ln \widetilde{Y}_n + \frac52 \la n + \frac{25}{4}\la + \frac52 \ln C_3.
	\]
	As a result, 
	\[
	g_{n} \leq \Big(\frac75\Big)^{n} g_0, \quad\forall\, n\geq 0.
	\]
	This implies that 
	\[
	\ln \widetilde{Y}_n \leq \Big(\frac75\Big)^{n} \Big(\ln \widetilde{Y}_0 + \frac{25}{4}\la + \frac52 \ln C_3 \Big) - \frac52 \la n - \frac{25}{4}\la - \frac52 \ln C_3, \quad\forall\, n\geq 0.
	\]
	Noticing that if 
	\be\label{coef_neg}
	\ln \widetilde{Y}_0 + \frac{25}{4}\la + \frac52 \ln C_3  \leq 0,
	\ee
	then one concludes that $\widetilde{Y}_n\to0$ as $n\to\infty$, and hence,
	\[
	\|\Gamma_+\|_{L^\infty(D_m\times[0,T])} \leq B = C_{\G} \|\Gamma_0\|_{L^\infty}.
	\]
	Since $\widetilde{Y}_0 = B^{-2} \|\Gamma_+\|_{L^2(D_m\times [0,T])}$ and $\la = \ln 8$, then (\ref{coef_neg}) reduces to 
	\be\label{B_large1}
	B^2 \geq 8^{25/4} C_3^{5/2} \|\Gamma_+\|^2_{L^2(D_m\times[0,T])}.
	\ee
	Thanks to Lemma \ref{lem3.2}, 
	\[
	\|\Gamma_+\|^2_{L^2(D_m\times[0,T])} \leq 9 \|\Gamma_0\|^2_{L^2(D_m)}.
	\]
	By H\"older's inequality, one finds that 	
	\[\|\Gamma_0\|_{L^2(D_m)} \leq \|\Gamma_0\|_{L^{\infty}(D_m)} |D_m|^{1/2} \leq \frac{2\pi}{3}  \|\Gamma_0\|_{L^{\infty}(D_m)}, \]
	so 
	$\|\Gamma_+\|^2_{L^2(D_m\times[0,T])} \leq 4\pi^2 \|\Gamma_0\|^2_{L^\infty(D_m)}$. Then based on (\ref{B_large1}) and the definition of $B$ in (\ref{DefB}), it suffices to guarantee
	\[
	C_{\G}^{2} \|\Gamma_0\|^{2}_{L^{\infty}(D_m)} \geq  8^{25/4} C_3^{5/2} 4\pi^2 \|\Gamma_0\|^2_{L^\infty(D_m)}.
	\]
	This goal can be achieved by choosing 
	\be\label{C_star}
	C_{\G} = 2^{11} C_3^{5/4} \pi.
	\ee
	With this choice of $C_\G$, 
	\[
	\|\Gamma_+\|_{L^\infty(D_m\times[0,T])} \leq C_{\G} \|\Gamma_0\|_{L^\infty(D_m)}.
	\]
	This estimate also holds for $\Gamma_{-}$ for the similar reason. Since $C_{sob}$ is an absolute constant which is  independent of $m$ and $T$, the lemma is established.
\end{proof}

\begin{rem}\label{Remark, Cstar}
	We remark that the constant $C_\G$ defined in (\ref{C_star}) has the following expression:
	\[
	C_{\G} = 2^{11} \big( 430\max \{ 1, C_{sob}\} \big)^{\frac54} \pi,
	\]
	where $C_{sob}$ is the Sobolev embedding constant in (\ref{Sob_emb}). Noting that the domains $D_m$ eventually converge to $D$ which is a bounded Lipschitz domain, so the Sobolev embedding constant $C_{sob}$ in (\ref{Sob_emb}) can be chosen uniformly in $m$.
\end{rem}

As a consequence of Lemma \ref{LEMGA}, if the initial data $\Gamma_0$ satisfies (\ref{COND}), then $\Gamma$ remains uniformly small for all time. This smallness will be used repeatedly in Section \ref{Sec9} to absorb nonlinear terms in the energy estimates for $(\m{K},\m{F},\m{O})$.

\section{Estimates based on the Biot-Savart law}\l{Sec8}

In this section, we will take advantage of the Biot-Savart law to control $\| \nabla (\frac{\bm{v}}{\rho})\|_{L^2(D_m)}$ and $\|\frac{1}{\rho} \nabla (\frac{\bm{v}}{\rho})\|_{L^2(D_m)}$ in terms of the triple of good unknowns $(\m{K}, \m{F}, \m{O})$. 
The goal is to control singular velocity derivatives in terms of $\m{K}$, $\m{F}$ and $\m{O}$, so that Section \ref{Sec9} can close the energy estimate.

\subsection{Biot-Savart law involving $(\m{K}, \m{F}, \m{O})$}
We first recall the classical Biot-Savart law $\Delta v = -\nabla \times \o$ in spherical coordinates. According to \cite[(4.12) and the discussion before (4.13)]{LPYZZZ24}, the Biot-Savart law can be written as 
\be\l{Biot-Savart}
\left\{
\begin{aligned}
	&\left(\Dl+\f{2}{\rho}\p_\rho+\f{2}{\rho^2}\right)v_\rho=-\f{1}{\rho\sin\phi}\p_\phi(\sin\phi\o_\th)\,,\\
	&\left(\Dl+\f{2}{\rho}\p_\rho+\f{1-\cot^2\phi}{\rho^2}\right)v_\phi=\f{1}{\rho^3}\p_\rho(\rho^3\o_\th)\,,\\
	&\left(\Dl-\f{1}{\rho^2\sin^2\phi}\right)v_\th=-\f{1}{\rho}\p_\rho(\rho\o_\phi)+\f{1}{\rho}\p_\phi\o_\rho\,.
\end{aligned}
\right.
\ee
Based on the definition (\ref{key-triple}), we know 
\be\label{KFO_vor}
\left\{
\begin{aligned}
	&\m{K} = \frac{\o_\rho}{\rho} - \frac{2 v_\th \cot\phi}{\rho^2}, \vspace{0.1in}\\
	&\mF = \frac{\o_\phi}{\rho} + \frac{2 v_\th}{\rho^2}, \vspace{0.1in} \\ 
	&\m{O} = \f{\o_\th}{\rho\sin\phi} - \f{2v_\phi \eta}{\rho^2\sin\phi}\,.
\end{aligned}
\right.
\ee
Plugging these relations into the right hand side of (\ref{Biot-Savart}) yields 
\be\l{Biot-Savart2}
\left\{
\begin{aligned}
	&\left(\Dl+\f{2}{\rho}\p_\rho+\f{2}{\rho^2}\right)v_\rho = - \f{1}{\rho\sin\phi} \p_\phi \Big( \rho \sin^2\phi \, \m{O} + \frac{2\sin\phi}{\rho} v_\phi \eta \Big)\,,\\
	&\left(\Dl+\f{2}{\rho}\p_\rho+\f{1-\cot^2\phi}{\rho^2}\right)v_\phi = \f{1}{\rho^3} \p_\rho \big( \rho^4 \sin\phi \,\m{O} + 2\rho^2 v_\phi \eta \big)\,,\\
	&\left(\Dl - \frac{2}{\rho}\p_\rho - \frac{2\cot\phi}{\rho^2}\p_\phi + \f{1}{\rho^2\sin^2\phi}\right)v_\th = -\f{1}{\rho}\p_\rho(\rho^2 \m{F}) + \p_\phi \m{K}\,.
\end{aligned}
\right.
\ee
Noticing that the right hand side of the equation for $v_\rho$ in (\ref{Biot-Savart2})$_{1}$ contains the function $v_\phi$, so our strategy is to estimate $v_\phi$ first based on (\ref{Biot-Savart2})$_{2}$ and then estimate $v_\rho$ by taking advantage of the estimate for $v_\phi$.

\subsection{Control of $v_\phi/\rho$ via $\m{O}$}

\begin{lemma}\label{Lemma, f_2_est}
	Let the region $D_m$ be as defined in \eqref{app domain-cyl} with $m\geq 10^3$ and the angle $\alpha\in (0,\frac{\pi}{6}]$. Then for any $T>0$ and a.e. $t\in[0,T]$, the following hold:
	\begin{align}
		\|\nabla (v_\phi/\rho)(\cdot,t)\|_{L^2(D_m)}
		&\le \sqrt{3}\,\|\m{O}(\cdot,t)\|_{L^2(D_m)},  \label{v_phi_est1}\\
		\left\|\frac{1}{\rho}\nabla (v_\phi/\rho)(\cdot,t)\right\|_{L^2(D_m)}
		&\le 10\sqrt{3}\,\|\nabla\m{O}(\cdot,t)\|_{L^2(D_m)}.  \label{v_phi_est2}
	\end{align}
\end{lemma}

\begin{proof}
	Let $f_2 \ed \frac{v_\phi}{\rho}$. Recalling the equation (\ref{Biot-Savart2})$_{2}$ for $v_\phi$ and the boundary conditions \eqref{NTSR}--\eqref{NTS-NHL}, we replace $v_\phi$ by $\rho f_2$ to obtain
	\be\label{f2_eq2}
	\left\{\begin{aligned}
		&\left(\Delta+\dfrac{4}{\rho}\partial_\rho
		+\dfrac{5-\cot^2\phi}{\rho^2}\right)f_2
		= \dfrac{1}{\rho^4}\partial_\rho
		(\rho^4 \m{O} \sin\phi\, + 2\rho^3 f_2 \eta), \quad \text{ in } D_m,\\
		&f_2=0, \quad \text{ on } R_{1,m} \cup R_{2,m},\\
		&\partial_\rho f_2=-\dfrac{2}{\rho}f_2, \quad \text{ on } A_{1,m},\\
		&\partial_\rho f_2=0, \quad \text{ on } A_{2,m}.
	\end{aligned}\right.
	\ee
	
	Testing (\ref{f2_eq2}) by $f_2$ yields
	\be\label{f2_energy}
	\begin{aligned}
		&\int_{D_m} f_2 \Delta f_2 \d x
		+ 4\int_{D_m}\frac{f_2}{\rho}\partial_\rho f_2 \d x
		+\int_{D_m}\frac{5-\cot^2\phi}{\rho^2}f_2^2 \d x \\
		= \quad & \int_{D_m}\frac{1}{\rho^4} f_2
		\partial_\rho(\rho^4 \m{O} \sin\phi) \d x +  \int_{D_m}\frac{1}{\rho^4} f_2
		\partial_\rho(2\rho^3 f_2\eta) \d x.
	\end{aligned}
	\ee
	Using integration by parts and the boundary conditions for $f_2$ in (\ref{f2_energy}), we obtain
	\[
	\int_{D_m} f_2 \Delta f_2 \d x
	= -\int_{D_m}|\nabla f_2|^2 \d x + \int_{A_{1,m}}\frac{2}{\rho}f_2^2 \d S.
	\]
	For any point $x$ on $A_{1,m}$, its corresponding $\rho$ is $\frac1m$. For convenience of notation, we denote \[\rho_1 \ed \frac1m.\] 
	Meanwhile, we recall the cut-off function $\eta$ in (\ref{DCUT}). Then the above surface integral can be rewritten as 
	\begin{align*}
		\int_{A_{1,m}}\frac{2}{\rho}f_2^2 \d S &= 4\pi \int_{{\f{\pi}{2}-\al}}^{{\f{\pi}{2}+\al}} f_2^2(\rho_1, \phi) \rho_1 \sin\phi \d\phi \\
		&= -4\pi \int_{{\f{\pi}{2}-\al}}^{{\f{\pi}{2}+\al}} \bigg(\int_{\frac1m}^{1} \p_{\rho} 
		\big[\rho f_2^2(1-\eta)\big] \d\rho \bigg) \sin\phi\d\phi \\
		&= -2 \int_{D_m} \bigg[ \frac{f_2^2}{\rho^2}(1-\eta) + \frac{2}{\rho} f_2 (\p_{\rho} f_2) (1-\eta) - \frac{f_2^2}{\rho}\eta' \bigg]\d x. 
	\end{align*}
	So
	\be\label{vphiest1}
	\begin{split}
		\text{LHS of (\ref{f2_energy})} &= -\int_{D_m}|\nabla f_2|^2 \d x + 4\int_{D_m}\frac{f_2}{\rho} (\partial_\rho f_2)\eta \d x + \int_{D_m} (3 + 2\eta - \cot^2\phi) \frac{f_2^2}{\rho^2} \d x \\
		&\qquad + \int_{D_m} \frac{2}{\rho} f_2^2\eta' \d x.
	\end{split}\ee
	We turn to the computation of the right hand side of (\ref{f2_energy}). Firstly, it follows from the integration by parts that 
	\bn
	\int_{D_m}\frac{1}{\rho^4} f_2 \partial_\rho(\rho^4 \m{O} \sin\phi) \d x &= 
	2\pi \int_{\frac{\pi}{2}-\al}^{\frac{\pi}{2}+\al} \bigg( \int_{\frac1m}^{1} \frac{\sin\phi}{\rho^2} f_2 \p_\rho(\rho^4 \m{O} \sin\phi) \d\rho \bigg) \d \phi \\
	&= -2\pi \int_{\frac{\pi}{2}-\al}^{\frac{\pi}{2}+\al} \bigg( \int_{\frac1m}^{1} \p_\rho\Big(\frac{\sin\phi}{\rho^2} f_2\Big) \rho^4 \m{O} \sin\phi \d\rho \bigg) \d \phi \\
	&= 2\int_{D_m} \frac{f_2}{\rho} \m{O}\sin\phi \d x - \int_{D_m}(\p_{\rho} f_2) \m{O} \sin\phi \d x.
	\en
	As a result, 
	\be\label{vphiest2}
	\begin{split}
		\text{RHS of (\ref{f2_energy})} &= \bigg(2\int_{D_m} \frac{f_2}{\rho} \m{O}\sin\phi \d x - \int_{D_m}(\p_{\rho} f_2) \m{O} \sin\phi \d x\bigg)\\
		&\quad + \bigg(\int_{D_m} \frac{6 f_2^2}{\rho^2} \eta \d x + \int_{D_m} \frac{2f_2}{\rho} (\p_{\rho} f_2) \eta \d x + \int_{D_m} \frac{2}{\rho} f_2^2 \eta' \d x\bigg)\,.
	\end{split}\ee
	Combining (\ref{vphiest1}) and (\ref{vphiest2}) together yields 
	\[\begin{split}
		\int_{D_m} |\nabla f_2|^2 \d x &=  2\int_{D_m}\frac{f_2}{\rho} (\partial_\rho f_2)\eta \d x + \int_{D_m} (3 - 4\eta - \cot^2\phi) \frac{f_2^2}{\rho^2} \d x \\
		& \quad - 2\int_{D_m} \frac{f_2}{\rho} \m{O}\sin\phi \d x + \int_{D_m}(\p_{\rho} f_2) \m{O} \sin\phi \d x\,.
	\end{split}\]
	By Cauchy-Schwarz inequality, for any positive $\la_1$, $\la_2$ and $\la_3$, 
	\[\begin{split}
		\int_{D_m} |\nabla f_2|^2 \d x & \leq \bigg( \la_1 \int_{D_m} (\p_{\rho} f_2)^2 \d x + \frac{1}{\la_1} \int_{D_m} \frac{f_2^2}{\rho^2} \eta^2\d x \bigg) + \bigg(3 \int_{D_m} \frac{f_2^2}{\rho^2}\d x - 4 \int_{D_m} \frac{f_2^2}{\rho^2} \eta \d x \bigg)\\
		& \quad + \bigg( \la_2 \int_{D_m} \frac{f_2^2}{\rho^2}\d x + \frac{1}{\la_2} \int_{D_m} |\m{O}|^2 \d x  \bigg) \\
		&\quad + \bigg( \la_3 \int_{D_m} (\p_{\rho} f_2)^2 \,dx + \frac{1}{4\la_3} \int_{D_m} |\m{O}|^2 \d x \bigg).
	\end{split}\]
	Rearranging terms yields 
	\[\begin{split}
		\int_{D_m} |\nabla f_2|^2 \d x &\leq (\la_1 + \la_3) \int_{D_m} (\p_{\rho} f_2)^2 \d x + (3+\la_2) \int_{D_m} \frac{f_2^2}{\rho^2} \d x \\
		& \quad + \Big( \frac{1}{\la_1} - 4 \Big)\int_{D_m} \frac{f_2^2}{\rho^2} \eta \d x + \Big( \frac{1}{\la_2} + \frac{1}{4\la_3} \Big) \int_{D_m} |\m{O}|^2 \d x.
	\end{split}\]
	Noticing that $f_2=0$ on rays $R_{1,m} \cup R_{2,m}$, it then follows from Lemma \ref{PoinB} that 
	\[ 
	\int_{D_m} \frac{f_2^2}{\rho^2} \d x \leq \frac{3}{25} \int_{D_m} \frac{1}{\rho^2}(\p_{\phi} f_2)^2 \d x\,. 
	\]
	Thus, by taking $\la_1 = \frac14$, $\la_2 = \frac{12}{5}$, $\la_3 = \frac25$, we get 
	\bn
	\int_{D_m} |\nabla f_2|^2 \d x &\leq \frac{13}{20} \int_{D_m} (\p_{\rho} f_2)^2 \d x 
	+ \frac{81}{125} \int_{D_m} \frac{1}{\rho^2}(\p_{\phi} f_2)^2 \d x
	+ \frac{25}{24} \int_{D_m} |\m{O}|^2 \d x \\
	&\leq \frac{13}{20} \int_{D_m} |\nabla f_2|^2 \d x  + \frac{25}{24} \int_{D_m} |\m{O}|^2 \d x,
	\en
	which implies (\ref{v_phi_est1}).
	
	Next, we turn to prove (\ref{v_phi_est2}). Testing (\ref{f2_eq2}) by $\rho^{-2}f_2$ yields
	\be\label{f2_energy2}
	\begin{aligned}
		&\int_{D_m} \frac{1}{\rho^2} f_2 \Delta f_2 \d x
		+ 4\int_{D_m}\frac{f_2}{\rho^3}\partial_\rho f_2 \d x
		+\int_{D_m}\frac{5-\cot^2\phi}{\rho^4}f_2^2 \d x \\
		= \quad & \int_{D_m}\frac{1}{\rho^6} f_2
		\partial_\rho(\rho^4 \m{O} \sin\phi) \d x +  \int_{D_m}\frac{1}{\rho^6} f_2
		\partial_\rho(2\rho^3 f_2\eta) \d x.
	\end{aligned}
	\ee
	Via integration by parts, we deduce 
	\ba\l{IBP0318}
	\int_{D_m} \frac{1}{\rho^2} f_2 \Delta f_2 \d x &= \int_{\p D_m} \frac{1}{\rho^2} f_2 \p_{n} f_2 \d S - \int_{D_m} \nabla\Big( \frac{1}{\rho^2} f_2 \Big) \nabla f_2 \d x\,.
	\ea
	According to the boundary conditions for $f_2$ in (\ref{f2_energy}), the surface integral in (\ref{IBP0318}) can be computed as 
	\[\begin{split}
		\int_{\p D_m} \frac{1}{\rho^2} f_2 \p_{n} f_2 \d S &= \int_{A_{1,m}} \frac{1}{\rho_1^2} f_2 \p_{n} f_2 \d S \\
		=\,\, & 4\pi \int_{{\f{\pi}{2}-\al}}^{{\f{\pi}{2}+\al}} \frac{1}{\rho_1} f_2^2(\rho_1, \phi) \sin\phi \d\phi \\
		=\,\,&  -4\pi \int_{{\f{\pi}{2}-\al}}^{{\f{\pi}{2}+\al}} \int_{\rho_1}^{\rho_2} \p_{\rho}\Big( \frac{f_2^2}{\rho}(1-\eta)\Big) \sin\phi \d\phi \\
		=\,\, & -2 \int_{D_m} \bigg[ -\frac{f_2^2}{\rho^4}(1-\eta) + \frac{2}{\rho^3} f_2 (\p_{\rho} f_2) (1-\eta) - \frac{f_2^2}{\rho^3}\eta' \bigg]\d x. 
	\end{split}\]
	Plugging this equality into (\ref{IBP0318}) yields 
	\bn
	\int_{D_m} \frac{1}{\rho^2} f_2 \Delta f_2 \d x & = -\int_{D_m} \frac{1}{\rho^2} |\nabla f_2|^2 \d x + \int_{D_m}\frac{f_2}{\rho^3} (\partial_\rho f_2) (4\eta-2) \d x \\
	&\quad + \int_{D_m} \frac{f_2^2}{\rho^4} (2 + 2\rho\eta' - 2\eta) \d x.
	\en
	Hence, 
	\be\label{vphi_rho_est1}\begin{split}
		\text{LHS of (\ref{f2_energy2})} &= -\int_{D_m} \frac{1}{\rho^2} |\nabla f_2|^2 \d x + \int_{D_m}\frac{f_2}{\rho^3} (\partial_\rho f_2) (2+4\eta) \d x \\
		&\quad  + \int_{D_m} \frac{f_2^2}{\rho^4} (7 + 2\rho\eta' - 2\eta - \cot^2\phi) \d x.
	\end{split}\ee
	
	On the other hand, analogous to the derivation of (\ref{vphiest2}), we use integration by parts to find
	\be\label{vphi_rho_est2}\begin{split}
		\text{RHS of (\ref{f2_energy2})} &=  \bigg( 4\int_{D_m} \frac{f_2}{\rho^3} \m{O}\sin\phi \d x - \int_{D_m} \frac{1}{\rho^2} (\p_{\rho} f_2) \m{O} \sin\phi \,dx \bigg) \\
		& \quad + \bigg( 2\int_{D_m} \frac{f_2}{\rho^3} (\p_{\rho} f_2) \eta \d x + \int_{D_m} \frac{ f_2^2}{\rho^4} (6\eta + 2\rho \eta') \d x \bigg)\,.
	\end{split}\ee
	Combining (\ref{vphi_rho_est1}) and (\ref{vphi_rho_est2}) together gives
	\be\label{vphi_rho_est3}
	\begin{split}
		\int_{D_m} \frac{1}{\rho^2} |\nabla f_2|^2 \d x &= \int_{D_m} \frac{1}{\rho^2} (\p_{\rho} f_2) \m{O} \sin\phi \d x - 4\int_{D_m} \frac{f_2}{\rho^3} \m{O}\sin\phi \d x \\
		&\quad + \int_{D_m}\frac{f_2}{\rho^3} (\partial_\rho f_2) (2 + 2\eta) \d x + \int_{D_m} \frac{f_2^2}{\rho^4} (7 - 8\eta - \cot^2\phi) \d x,
	\end{split}
	\ee
	Since $f_2$ vanishes on rays, then it follows from the Poincar\'e inequality (\ref{PoinBE}) that 
	\bn
	\int_{D_m} \frac{f_2^2}{\rho^4} \d x \leq \frac{3}{25} \int_{D_m} \frac{1}{\rho^4}(\p_{\phi} f_2)^2 \d x.
	\en
	Plugging this inequality into (\ref{vphi_rho_est3}) leads to 
	\be\l{EE0416}
	\begin{split}
		\int_{D_m} \frac{1}{\rho^2} |\nabla f_2|^2 \d x &\leq \int_{D_m} \frac{1}{\rho^2} (\p_{\rho} f_2) \m{O} \sin\phi \d x - 4\int_{D_m} \frac{f_2}{\rho^3} \m{O}\sin\phi \d x \\
		&\quad + \un{ 2\int_{D_m}\frac{f_2}{\rho^3} (\p_\rho f_2) \d x}_{I_1} 
		+ \un{2\int_{D_m}\frac{f_2}{\rho^3} (\partial_\rho f_2) \eta \d x}_{I_2} 
		+ \frac{21}{25} \int_{D_m} \frac{1}{\rho^4}(\p_{\phi} f_2)^2 \d x.
	\end{split}
	\ee
	Since the coefficient $\frac{21}{25}$ on the right hand side of (\ref{EE0416}) is very close to $1$, then we have to treat other terms with particular care. The most challenging term is $I_1$ since its integrand is quadratic in $f_2$ and does not involve the cut-off function $\eta$. Next, we will first discuss how to estimate $I_1$ effectively. 
	
	Using integration by parts and the boundary conditions for $f_2$ in (\ref{f2_eq2}), we have
	\bn
	I_1
	=& 2\pi\int_\f{1}{m}^1\int_{\f{\pi}{2}-\al}^{\f{\pi}{2}+\al}\f{1}{\rho} 
	(\p_\rho f_2^2) \sin\phi\d\phi\d\rho\\
	=&2\pi\int_\f{1}{m}^1\int_{\f{\pi}{2}-\al}^{\f{\pi}{2}+\al}\f{f_2^2}{\rho^2}\sin\phi\d\phi\d\rho+2\pi\int_{\f{\pi}{2}-\al}^{\f{\pi}{2}+\al}\f{f_2^2}{\rho}\sin\phi\d\phi\Big|_{\rho=\f{1}{m}}^1 \\
	=&  \int_{D_m}\f{f_2^2}{\rho^4}\d x + 2\pi\int_{\f{\pi}{2}-\al}^{\f{\pi}{2}+\al}\f{f_2^2}{\rho}\sin\phi\d\phi\Big|_{\rho=1}
	- 2\pi\int_{\f{\pi}{2}-\al}^{\f{\pi}{2}+\al}\f{f_2^2}{\rho}\sin\phi\d\phi\Big|_{\rho=\frac1m}\,,
	\en
	which implies that 
	\ba\l{E20416}
	I_1 \leq \int_{D_m}\f{f_2^2}{\rho^4}\d x
	+\un{2\pi\int_{\f{\pi}{2}-\al}^{\f{\pi}{2}+\al}\f{f_2^2}{\rho}\sin\phi\d\phi\Big|_{\rho=1}}_{I_{11}}\,.
	\ea
	Moreover, using the cut-off function $\eta$ and the fact that $\rho=1$ in $I_{11}$, we obtain
	\bn
	I_{11}=&2\pi\int_{\f{\pi}{2}-\al}^{\f{\pi}{2}+\al}\rho\eta f_2^2\sin\phi\d\phi\Big|_{\rho=\f{1}{m}}^1\\
	=&2\pi\int_{\f{1}{m}}^1\int_{\f{\pi}{2}-\al}^{\f{\pi}{2}+\al}\p_\rho(\rho\eta f_2^2)\sin\phi\d\phi\d\rho\\
	=&\int_{D_m}\f{2 f_2}{\rho} (\p_\rho f_2)\eta + \f{f_2^2}{\rho^2}(\eta+\rho\eta^\prime)\d x\,.
	\en
	Since $|\eta^\prime|\leq 6$ and $0\leq \eta\leq 1$, then
	\be\label{E0523}\begin{split}
		|I_{11}|\leq& 2\Big\|\f{f_2}{\rho}\Big\|_{L^2(D_m)}
		\|\p_\rho f_2\|_{L^2(D_m)}+7\Big\|\f{f_2}{\rho}\Big\|_{L^2(D_m)}^2\\
		\leq & \f{2\sqrt{3}}{5}\Big\|\f{\p_\phi f_2}{\rho}\Big\|_{L^2(D_m)}\|\p_\rho f_2\|_{L^2(D_m)}
		+\f{21}{25}\Big\|\f{\p_\phi f_2}{\rho}\Big\|_{L^2(D_m)}^2 
		\leq \|\nabla f_2\|_{L^2(D_m)}^2\,,
	\end{split}\ee
	Now we take advantage of estimate (\ref{v_phi_est1}) to infer from (\ref{E0523}) that 
	\[
	|I_{11}| \leq 3 \|\m{O}\|_{L^2(D_m)}^2\,.
	\]
	Since $\m{O}$ vanishes on the rays according to the boundary condition (\ref{BCOKF}), 	
	then we apply the Poincar\'e inequality in Lemma \ref{PoinB} to deduce 
	\be\label{O_Poincare}
	\|\m{O}\|_{L^2(D_m)}^2 \leq \frac{3}{25} \|\nabla \m{O}\|_{L^2(D_m)}^2. 
	\ee
	Thus, 	
	\ba\l{E30416}
	|I_{11}|\leq\f{9}{25}\|\na\m{O}\|_{L^2(D_m)}^2\,.
	\ea
	Substituting (\ref{E30416}) into (\ref{E20416}) yields 
	\ba\l{E0523-2}
	I_1 &\leq \int_{D_m}\f{f_2^2}{\rho^4}\d x + \f{9}{25}\|\na\m{O}\|_{L^2(D_m)}^2 \\
	&\leq \frac{3}{25} \int_{D_m} \frac{1}{\rho^4}(\p_{\phi} f_2)^2 \d x + \f{9}{25}\|\na\m{O}\|_{L^2(D_m)}^2\,.
	\ea
	
	On the other hand, since $\eta\equiv 0$ when $\rho\leq\f{1}{3}$, direct calculation shows
	\ba\l{E40416}
	I_2 = 2\int_{D_m}\f{f_2}{\rho^3} (\p_\rho f_2) \eta\d x
	\leq 18 \|\p_\rho f_2\|_{L^2(D_m)}\Big\|\f{f_2}{\rho}\Big\|_{L^2(D_m)}.
	\ea
	Next, we apply Poincar\'e inequality (\ref{PoinBE}) and estimate (\ref{v_phi_est1}) to (\ref{E40416}) to find
	\[
	I_2 \leq 18 \frac{\sqrt{3}}{5} \|\nabla f_2\|_{L^2(D_m)}^2 \leq \frac{54\sqrt{3}}{5} \|\m{O}\|_{L^2(D_m)}^2\,.
	\]
	Combining this estimate with (\ref{O_Poincare}) gives 
	\be\label{E0523-3}
	I_2 \leq \frac{162\sqrt{3}}{125} \|\na\m{O}\|_{L^2(D_m)}^2\,.
	\ee

	Putting (\ref{E0523-2}) and (\ref{E0523-3}) into (\ref{EE0416}) leads to 
	\[
	\begin{split}
		\int_{D_m} \frac{1}{\rho^2} |\nabla f_2|^2 \d x &\leq \int_{D_m} \frac{1}{\rho^2} (\p_{\rho} f_2) \m{O} \sin\phi \d x - 4\int_{D_m} \frac{f_2}{\rho^3} \m{O}\sin\phi \d x \\
		&\quad + \frac{24}{25} \int_{D_m} \frac{1}{\rho^4}(\p_{\phi} f_2)^2 \d x + 3\|\nabla\m{O}\|_{L^2(D_m)}^2\,.
	\end{split}
	\]
	
	Then by applying the Cauchy-Schwarz inequality, for any positive $\la_1$, $\la_2$ and $\la_3$,
	\[\begin{split}
		\int_{D_m} \frac{1}{\rho^2} |\nabla f_2|^2 \,dx &\leq \bigg( \la_1 \int_{D_m} \Big(\frac{1}{\rho} \p_{\rho} f_2 \Big)^2 \,dx + \frac{1}{4\la_1} \int_{D_m} \Big( \frac{1}{\rho} \m{O}\Big)^2 \d x \bigg) \\
		&\quad + \bigg( \la_2 \int_{D_m} \Big(\frac{f_2}{\rho^2}\Big)^2 \,dx + \frac{4}{\la_2} \int_{D_m} \Big(\frac{1}{\rho} \m{O}\Big)^2 \d x \bigg) \\
		&\quad + \frac{24}{25} \int_{D_m} \frac{1}{\rho^4}(\p_{\phi} f_2)^2 \d x + 3\|\nabla\m{O}\|_{L^2(D_m)}^2\,.
	\end{split}\]
	Rearranging the above terms, and noticing (\ref{O_Poincare}) and
	\[ 
	\int_{D_m} \frac{f_2^2}{\rho^4} \d x \leq \frac{3}{25} \int_{D_m} \frac{1}{\rho^4}(\p_{\phi} f_2)^2 \d x, 
	\]
	we obtain 
	\[\begin{split}
		\int_{D_m} \frac{1}{\rho^2} |\nabla f_2|^2 \d x &\leq \la_1\int_{D_m} \Big(\frac{1}{\rho} \p_{\rho} f_2 \Big)^2 \d x + \Big(\frac{3\la_2}{25} + \frac{24}{25}\Big) \int_{D_m} \Big(\frac{1}{\rho^2} \p_\phi f_2\Big)^2 \d x \\
		&\quad + \left(\frac{3}{25}\Big( \frac{1}{4\la_1} + \frac{4}{\la_2} \Big)+3\right) \int_{D_m} | \nabla \m{O} |^2 \d x.
	\end{split}\]
	By choosing $\la_1 = \frac{49}{50}$ and $\la_2 = \frac16$, we find 
	\[
	\int_{D_m} \frac{1}{\rho^2} |\nabla f_2|^2 \d x \leq 300 \int_{D_m} | \nabla \m{O} |^2 \d x,
	\]
	which justifies \eqref{v_phi_est2}.
\end{proof}


\subsection{Control of $v_\rho/\rho$ via $\m{O}$}
\begin{lemma}\label{Lemma, f_1_est}
	Let the region $D_m$ be as defined in \eqref{app domain-cyl} with $m\geq 10^3$ and the angle $\alpha\in (0,\frac{\pi}{6}]$. Then for any $T>0$ and a.e. $t\in[0,T]$, the following hold
	\begin{align}
		\|\nabla (v_\rho/\rho)(\cdot,t)\|_{L^2(D_m)}
		&\le 4\,\|\m{O}(\cdot,t)\|_{L^2(D_m)},  \label{v_rho_est1}\\
		\left\|\frac{1}{\rho}\nabla (v_\rho/\rho)(\cdot,t)\right\|_{L^2(D_m)}
		&\le 240\,\|\nabla\m{O}(\cdot,t)\|_{L^2(D_m)}.   \label{v_rho_est2}
	\end{align}
\end{lemma}

\begin{proof}
	Let $f_1 \ed \frac{v_\rho}{\rho}$. Recalling the equation (\ref{Biot-Savart2})$_{1}$ for $v_\rho$ and the boundary conditions \eqref{NTSR}--\eqref{NTS-NHL}, we replace $v_\rho$ by $\rho f_1$ to obtain
	\be\label{f1_eq2}
	\left\{\begin{aligned}
		&\left(\Delta + \dfrac{4}{\rho}\partial_\rho
		+ \dfrac{6}{\rho^2}\right) f_1
		= -\frac{1}{\rho \sin\phi} \partial_\phi
		\Big(\m{O} \sin^2\phi\, + \frac{2 f_2 \eta \sin\phi}{\rho}\Big), \quad \text{ in } D_m,\\
		& \p_\phi f_1 = 0, \quad \text{ on } R_{1,m} \cup R_{2,m},\\
		& f_1 = 0, \quad \text{ on } A_{1,m} \cup A_{2,m}.
	\end{aligned}\right.
	\ee
	
	Testing (\ref{f1_eq2}) by $f_1$ yields
	\be\label{f1_energy}
	\begin{aligned}
		&\int_{D_m} f_1 \Delta f_1 \d x
		+ 4\int_{D_m}\frac{f_1}{\rho}\partial_\rho f_1 \d x
		+ 6 \int_{D_m}\frac{f_1^2}{\rho^2} \d x \\
		= \quad & -\int_{D_m} \frac{1}{\rho \sin\phi} f_1
		\partial_\phi (\m{O} \sin^2\phi) \d x - 2 \int_{D_m}\frac{\eta f_1}{\rho^2 \sin\phi}
		\partial_\phi( f_2\sin\phi) \d x.
	\end{aligned}
	\ee
	Noting
	\[
	\int_{D_m} f_1 \Delta f_1 \d x = -\int_{D_m}|\nabla f_1|^2 \d x \qquad \text{and} \qquad 4\int_{D_m}\frac{f_1}{\rho}\partial_\rho f_1 \d x = -2 \int_{D_m}\frac{f_1^2}{\rho^2} \d x\,, \\
	\]
	so 
	\[
	\text{LHS of (\ref{f1_energy})} = -\int_{D_m}|\nabla f_1|^2 \d x + 4\int_{D_m}  \frac{f_1^2}{\rho^2} \d x\,.
	\]
	On the other hand, we apply integration by parts to find
	\[
	\text{RHS of (\ref{f1_energy})} = \int_{D_m} \frac{1}{\rho} (\p_\phi f_1)
	\m{O} \sin\phi \d x + 2 \int_{D_m} \frac{1}{\rho^2} (\p_\phi f_1) f_2 \eta \d x\,.
	\]
	Combining the above two equations yields 
	\[
	\int_{D_m} |\nabla f_1|^2 \d x = 4\int_{D_m}  \frac{f_1^2}{\rho^2} \d x -  \int_{D_m} \frac{1}{\rho} (\p_\phi f_1) \m{O} \sin\phi \d x - 2 \int_{D_m}\frac{(\p_\phi f_1) f_2 \eta}{\rho^2} \d x\,.
	\]
	By Cauchy-Schwarz inequality, for any positive $\la_1$ and $\la_2$, 
	\[\begin{split}
		\int_{D_m} |\nabla f_1|^2 \d x & \leq 4\int_{D_m}  \frac{f_1^2}{\rho^2} \d x +
		\bigg( \la_1 \int_{D_m} \Big(\frac{1}{\rho} \p_{\phi} f_1 \Big)^2 \d x + \frac{1}{4\la_1} \int_{D_m} |\m{O}|^2 \d x \bigg) \\
		& \quad + \bigg( \la_2 \int_{D_m} \Big(\frac{1}{\rho} \p_{\phi} f_1 \Big)^2 \d x + \frac{1}{\la_2} \int_{D_m} \Big(\frac{1}{\rho} f_2\Big)^2 \d x \bigg)\,.
	\end{split}\]
	Thanks to Lemma \ref{v_rho_mean0} and the Poincar\'e inequality in Lemma \ref{Poin0},
	\[
	\int_{D_m} \Big(\frac{1}{\rho} f_1\Big)^2 \d x \leq \frac{2}{19} \int_{D_m} \Big(\frac{1}{\rho} \p_\phi f_1\Big)^2 \d x.
	\]
	In addition, due to Lemma \ref{PoinBE} and Lemma \ref{Lemma, f_2_est}, 
	\[ \int_{D_m} \Big(\frac{1}{\rho} f_2\Big)^2 \d x \leq \frac{3}{25} \int_{D_m} \Big(\frac{1}{\rho} \p_\phi f_2\Big)^2 \d x \leq \frac{9}{25} \int_{D_m} |\m{O}|^2 \d x.\]
	Combining the above three inequalities leads to 
	\[ 
	\int_{D_m} |\nabla f_1|^2 \d x \leq \Big(\frac{8}{19} + \la_1 + \la_2 \Big) \int_{D_m} \Big(\frac{1}{\rho} \p_\phi f_1\Big)^2 \d x + \Big(\frac{1}{4\la_1} + \frac{9}{25 \la_2}\Big) \int_{D_m} |\m{O}|^2 \d x.
	\]
	By choosing $\la_1 = \frac{1}{10}$ and $\la_2 = \frac15$, we verified (\ref{v_rho_est1}).

	Next, we turn to prove (\ref{v_rho_est2}). Testing (\ref{f1_eq2}) by $\rho^{-2} f_1$ yields
	\be\label{f1_energy2}
	\begin{aligned}
		&\int_{D_m} \frac{1}{\rho^2} f_1 \Delta f_1 \d x
		+ 4\int_{D_m}\frac{f_1}{\rho^3}\partial_\rho f_1 \d x
		+ 6 \int_{D_m}\frac{f_1^2}{\rho^4} \d x \\
		= \quad & -\int_{D_m} \frac{f_1}{\rho^3 \sin\phi} 
		\partial_\phi (\m{O} \sin^2\phi) \d x - 2 \int_{D_m}\frac{\eta f_1}{\rho^4 \sin\phi}
		\partial_\phi( f_2\sin\phi) \d x.
	\end{aligned}
	\ee
	Based on the boundary conditions for $f_1$ in (\ref{f1_eq2}), we apply integration by parts to obtain 
	\[
	\int_{D_m}  \frac{1}{\rho^2} f_1 \Delta f_1 \d x = -\int_{D_m}  \frac{1}{\rho^2} |\nabla f_1|^2 \d x + 2 \int_{D_m} \frac{f_1}{\rho^3} \p_\rho f_1 \d x 
	\]
	and 
	\[
	2\int_{D_m}\frac{f_1}{\rho^3}\partial_\rho f_1 \d x = \int_{D_m}\frac{f_1^2}{\rho^4} \d x\,. \\
	\]
	Thus, 
	\[
	\text{LHS of (\ref{f1_energy2})} = -\int_{D_m}  \frac{1}{\rho^2} |\nabla f_1|^2 \d x + 9\int_{D_m}  \frac{f_1^2}{\rho^4} \d x\,.
	\]
	On the other hand, we apply integration by parts again to deduce
	\[
	\text{RHS of (\ref{f1_energy2})} = \int_{D_m} \frac{1}{\rho^3} (\p_\phi f_1) \m{O} \sin\phi \d x + 2 \int_{D_m} \frac{1}{\rho^4} (\p_\phi f_1) f_2 \eta \d x\,.
	\]
	Combining the above two equations yields 
	\[
	\int_{D_m}  \frac{1}{\rho^2} |\nabla f_1|^2 \d x = 9\int_{D_m}  \frac{f_1^2}{\rho^4} \d x - \int_{D_m} \frac{1}{\rho^3} (\p_\phi f_1) \m{O} \sin\phi \d x - 2 \int_{D_m} \frac{1}{\rho^4} (\p_\phi f_1) f_2 \eta \d x.
	\]
	By Cauchy-Schwarz inequality, for any positive $\la_1$ and $\la_2$, 
	\be\label{vrho_rho1}\begin{split}
		\int_{D_m} \frac{1}{\rho^2} |\nabla f_1|^2 \d x & \leq 9\int_{D_m}  \frac{f_1^2}{\rho^4} \d x +
		\bigg( \la_1 \int_{D_m} \Big(\frac{1}{\rho^2} \p_{\phi} f_1 \Big)^2 \d x + \frac{1}{4\la_1} \int_{D_m} \Big( \frac{1}{\rho}\m{O} \Big)^2 \d x \bigg) \\
		& \quad + \bigg( \la_2 \int_{D_m} \Big(\frac{1}{\rho^2} \p_{\phi} f_1 \Big)^2 \d x + \frac{1}{\la_2} \int_{D_m} \Big(\frac{1}{\rho^2} f_2\Big)^2 \d x \bigg).
	\end{split}\ee
	Thanks to Lemma \ref{v_rho_mean0} and Poincar\'e inequality \eqref{Poin0E},
	\[
	\int_{D_m} \frac{f_1^2}{\rho^4} \d x \leq \frac{2}{19} \int_{D_m} \Big(\frac{1}{\rho^2} \p_\phi f_1\Big)^2 \d x.
	\]
	In addition, due to Poincar\'e inequality (\ref{PoinBE}) and Lemma \ref{Lemma, f_2_est}, 
	\[ 
	\int_{D_m} \Big(\frac{1}{\rho^2} f_2\Big)^2 \d x \leq \frac{3}{25} \int_{D_m} \Big(\frac{1}{\rho^2} \p_\phi f_2\Big)^2 \d x \leq 36\int_{D_m} \big| \nabla \m{O} \big|^2 \d x.
	\]
	Plugging the above two inequalities and estimate (\ref{O_Poincare}) into (\ref{vrho_rho1}) leads to
	\[ 
	\int_{D_m} \frac{1}{\rho^2} |\nabla f_1|^2 \d x \leq \Big(\frac{18}{19} + \la_1 + \la_2 \Big) \int_{D_m} \Big(\frac{1}{\rho} \p_\phi f_1\Big)^2 \d x + \Big(\frac{3}{100\la_1} + \frac{36}{\la_2}\Big) \int_{D_m} | \nabla \m{O}|^2 \d x.
	\]
	By choosing $\la_1 = \frac{1}{38} \cdot \frac{1}{37}$ and $\la_2 = \frac{1}{38} \cdot \frac{36}{37}$, we find 
	\[
	\int_{D_m} \frac{1}{\rho^2} |\nabla f_1|^2 \d x \leq \frac{37}{38}\int_{D_m} \frac{1}{\rho^2} |\nabla f_1|^2 \d x + 1450 \int_{D_m} | \nabla \m{O}|^2 \d x,
	\]
	which implies (\ref{v_rho_est2}).
\end{proof}

The estimates in Lemma \ref{Lemma, f_2_est} and Lemma \ref{Lemma, f_1_est} provide $L^2$ control of the first-order derivative of singular quantities $v_\phi/\rho$ and $v_\rho/\rho$ through the good unknown $\m{O}$. However, nonlinear terms such as $v_\phi \p_\rho v_\rho$ and $v_\rho\p_\rho v_\rho$ appear in the higher-order energy estimate for $\m{O}$ carried out in Section \ref{Subsec, Est_O}. To control these terms, one also needs bounds for $L^2$ norm of second-order derivatives of $v_\rho$. The next lemma provides such estimates by recovering second-order information of $v_\rho$ from $\m{O}$ through the Biot–Savart law.
We also mention that the left hand side of (\ref{E2nd}) is the second-order while the right hand side is the third-order derivatives of the velocity in the sense of scaling. 

\begin{lemma}\l{Lem2nd}
	Let the region $D_m$ be as defined in \eqref{app domain-cyl} with $m \geq 10^3$ and the angle $\alpha \in \left(0, \frac{\pi}{6}\right]$. Then for any $T>0$ and for a.e. $t \in[0, T]$, there exists an absolute constant $C$ such that
	\ba\l{E2nd}
	\|\p_\rho^2 v_\rho (\cdot, t)\|_{L^2(D_m)} + \Big\|\f{\p_\phi\p_\rho v_\rho}{\rho} (\cdot, t)\Big\|_{L^2(D_m)}
	\leq C \|\nabla\m{O}(\cdot,t)\|_{L^2(D_m)}\,.
	\ea
\end{lemma}
\begin{proof}
	Recall the Biot-Savart law of $v_\rho$ in \eqref{Biot-Savart}$_1$
	\ba\l{bsrho}
	\left(\Dl+\f{2}{\rho}\p_\rho+\f{2}{\rho^2}\right)v_\rho=-\f{1}{\rho\sin\phi}\p_\phi(\sin\phi \, \o_\th)\,.
	\ea
	Multiplying \eqref{bsrho} by $\p_\rho^2v_\rho$ and integrating over $D_m$, we have
	\ba\l{bsrM}
	&\int_{D_m}(\partial^2_\rho v_\rho)^2\d x
	+ \int_{D_m}\frac{4}{\rho} \partial_\rho v_\rho (\partial_\rho^2 v_\rho)\d x
	+ \int_{D_m}\frac{2}{\rho^2} v_\rho (\partial_\rho^2 v_\rho) \d x \\
	& + \un{\int_{D_m}\frac{\p_\rho^2v_\rho}{\rho^2 \sin\phi} \partial_\phi\left(\sin\phi \, \partial_\phi v_\rho\right)\d x}_{M}
	= -\int_{D_m}\f{1}{\rho\sin\phi}\p_\phi(\sin\phi \,\o_\th)\p_\rho^2v_\rho\d x\,.
	\ea
	
	We first compute $M$ to obtain 
	\[
	M = 2 \pi \int_{\pi / 2-\alpha}^{\pi / 2+\alpha} \int_{1 / m}^1 \partial_\phi\left(\sin \phi\, \partial_\phi v_\rho\right) \partial_\rho^2 v_\rho \d \rho \d \phi.
	\]
	Noticing that $v_\rho$, together with its derivatives in $\phi$, vanish on arcs $A_{1,m}\cup A_{2,m}$, so it follows from integration by parts in $\rho$ that 
	\bn
	M &= -2 \pi \int_{\pi / 2-\alpha}^{\pi / 2+\alpha} \int_{1 / m}^1 \p_\rho\big[ \partial_\phi\left(\sin \phi \,\partial_\phi v_\rho\right)\big] \partial_\rho v_\rho \d \rho \d \phi \\
	&= -2 \pi \int_{\pi / 2-\alpha}^{\pi / 2+\alpha} \int_{1 / m}^1 \partial_\phi\left(\sin \phi\, \partial_\phi \partial_\rho v_\rho\right) \partial_\rho v_\rho \d \rho \d \phi\,.
	\en
	Now we perform the integration by parts in $\phi$, and using the fact that $\p_\phi v_\rho$ and $\p_\rho \p_\phi v_\rho$ vanish on rays $R_{1,m} \cup R_{2,m}$, to get 	
	\ba\l{Er1}
	M =  2 \pi \int_{\pi / 2-\alpha}^{\pi / 2+\alpha} \int_{1 / m}^1\left(\partial_\phi \partial_\rho v_\rho\right)^2 \sin \phi \d \rho \d \phi 	
	= \int_{D_m} \Big(\f{\p_\phi\p_\rho v_\rho}{\rho}\Big)^2\d x\,.
	\ea
	
	Next, we will treat other terms in (\ref{bsrM}). Using integration by parts and the Poincar\'e inequality \eqref{Poin0E}, we deduce
	\ba\l{Er2}
	\bigg |\int_{D_m}\frac{2}{\rho^2} v_\rho (\partial_\rho^2 v_\rho) \d x\bigg|
	= 2 \int_{D_m}\Big(\frac{\p_\rho v_\rho}{\rho}\Big)^2 \d x
	\leq\f{4}{19}\int_{D_m}\Big(\f{\p_\phi\p_\rho v_\rho}{\rho}\Big)^2\d x\,.
	\ea
	By the Young inequality and Poincar\'e inequality \eqref{Poin0E} again, one has
	\ba\l{Er3}
	\bigg|\int_{D_m}\frac{4}{\rho} 
	(\p_\rho v_\rho) (\p_\rho^2 v_\rho) \d x\bigg|
	\leq &\f{3}{4}\int_{D_m}(\p_\rho^2 v_\rho)^2\d x+\f{16}{3}\int_{D_m} \Big(\frac{\p_\rho v_\rho}{\rho}\Big)^2\d x\\
	\leq &\f{3}{4}\int_{D_m}(\p_\rho^2 v_\rho)^2\d x+\f{32}{57}\int_{D_m}\Big(\frac{\p_\phi\p_\rho v_\rho}{\rho}\Big)^2\d x\,.
	\ea
	Combining (\ref{bsrM})--(\ref{Er1})--(\ref{Er2})--(\ref{Er3}) yields 
	\bn
	& \quad \int_{D_m}(\partial^2_\rho v_\rho)^2\d x + \int_{D_m} \Big(\frac{\p_\phi\p_\rho v_\rho}{\rho}\Big)^2\d x \\
	& \leq \frac34 \int_{D_m}(\partial^2_\rho v_\rho)^2\d x + \bigg(\frac{4}{19} + \frac{32}{57}\bigg) \int_{D_m} \Big(\frac{\p_\phi\p_\rho v_\rho}{\rho}\Big)^2\d x + \bigg| \int_{D_m}\f{1}{\rho\sin\phi}\p_\phi(\sin\phi \,\o_\th)\p_\rho^2v_\rho\d x \bigg| \,,
	\en
	which further implies that 
	\be\label{260514}
	\int_{D_m}(\partial^2_\rho v_\rho)^2\d x + \int_{D_m} \Big(\frac{\p_\phi\p_\rho v_\rho}{\rho}\Big)^2\d x \leq 5 \bigg| \int_{D_m}\f{1}{\rho\sin\phi}\p_\phi(\sin\phi \,\o_\th)\p_\rho^2v_\rho\d x\bigg|\,.
	\ee
	
	Recalling the relation (\ref{KFO_vor}), we know 
	$\o_\th = \rho\sin\phi\, \m{O} + \frac{2 v_\phi \eta}{\rho}$. 
	Consequently, 
	\bn
	\f{1}{\rho\sin\phi}\p_\phi(\sin\phi \, \o_\th) &= \f{\o_\th\cot\phi}{\rho} + \f{\p_\phi\o_\th}{\rho}\\
	&= 2\cos\phi \, \m{O} + \sin\phi \, \p_\phi\m{O} + \f{2\eta \cot\phi}{\rho^2}v_\phi + \f{2\eta}{\rho^2}\p_\phi v_\phi\,.
	\en
	Thus, it follows from Poincar\'e inequality \eqref{Poin0E} and Lemma \ref{Lemma, f_2_est} that 
	\bn
	\Big\|\f{1}{\rho\sin\phi}\p_\phi(\sin\phi \,\o_\th)\Big\|_{L^2(D_m)}
	\leq C\bigg( \Big\|\f{\p_\phi \m{O}}{\rho} \Big\|_{L^2(D_m)} + \Big\|\f{\p_\phi v_\phi}{\rho} \Big\|_{L^2(D_m)}\bigg) 
	\leq C \| \nabla \m{O} \|_{L^2(D_m)}\,.
	\en
	This implies that
	\ba\l{Er4}
	\bigg |\int_{D_m}\f{1}{\rho\sin\phi}\p_\phi(\sin\phi\, \o_\th)\p_\rho^2v_\rho\d x\bigg|
	\leq \f{1}{10}\|\p_\rho^2 v_\rho\|_{L^2(D_m)}^2+C\|{\na}\m{O}\|_{L^2(D_m)}^2\,.
	\ea
	Substituting \eqref{Er4} in \eqref{260514}, one concludes \eqref{E2nd}. This finishes the proof of Lemma \ref{Lem2nd}.
\end{proof}
\subsection{Control of ${v_\th}/{\rho}$ via $\m{K}$ and $\mF$}
\begin{lemma}\l{Lem35}
	Let the region $D_m$ be as defined in \eqref{app domain-cyl} with $m\geq 10^3$ and the angle $\alpha\in (0,\frac{\pi}{6}]$. Then for any $T>0$ and for a.e. $t \in[0, T]$, the next inequality holds
	\be\l{estg1}
	\Big\| \na\Big(\frac{v_\th}{\rho}\Big)(\cd,t)\Big\|_{L^2(D_m)}
	\leq \f{10}{3}\left( \|\m{K}(\cd,t)\|_{L^2(D_m)}^2 + \|\m{F}(\cd,t)\|_{L^2(D_m)}^2\right)^{1/2}\,.
	\ee
\end{lemma}

\begin{proof}
	Denote $g\ed\f{v_\th}{\rho}$. Recalling the equation (\ref{Biot-Savart2})$_{3}$ for $v_\th$ and the boundary conditions \eqref{NTSR}--\eqref{NTS-NHL}, we replace $v_\th$ by $\rho g$ to obtain
	\be\l{Eg}
	\left(\Dl-\frac{2\cot\phi}{\rho}\frac{\p_\phi}{\rho}+\frac{\csc^2\phi}{\rho^2}\right)g=-\frac{1}{\rho^2}\p_\rho(\rho^2\m{F})+\frac{\p_\phi\m{K}}{\rho}\,.
	\ee
	and
	\be\l{Bdrg}
	\left\{
	\begin{aligned}
		&\p_\rho g=0\,,\q\text{on}\q A_{1,m}\cup A_{2,m}\,,\\
		&\p_\phi g=\cot\phi \, g\,,\q\text{on}\q R_{1,m}\cup R_{2,m}\,.
	\end{aligned}
	\right.
	\ee
	Testing \eqref{Eg} by $g$ and integrating on $D_m$, one deduces
	\be\l{Egmain}
	\begin{aligned}
		\int_{D_m}|\nabla g|^2 \d x-\int_{D_m} \frac{\csc^2\phi}{\rho^2} g^2 \d x=&-2 \int_{D_m} \frac{\cot \phi}{\rho^2} g \partial_\phi g \d x+\int_{D_m} \frac{1}{\rho^2} \partial_\rho\left(\rho^2 \m{F}\right) g \d x\\
		&-\int_{D_m} \frac{\partial_\phi \m{K}}{\rho}g \d x+\int_{\partial D_m} g \partial_n g \d S\,.
	\end{aligned}
	\ee
	Using \eqref{Bdrg}, we notice that
	\ba\l{Ebdg}
	\int_{\partial D_m} g \partial_n g \d S&=2\pi\int_{\f{1}{m}}^{1} g^{2} \cot \phi \sin \phi \d \rho\Big|_{\phi=\f{\pi}{2}-\al} ^{\f{\pi}{2}+\al}=\un{2\pi\int_{\frac1m}^{1} \int_{{\f{\pi}{2}-\al}}^{{\f{\pi}{2}+\al}}\partial_\phi\left(g^2 \cos \phi\right) \d \phi\d \rho}_{B_1}\, .
	\ea
	First we see
	\be\l{Ebdg1}
	\begin{aligned}
		B_1= & 4\pi \int_{\f{1}{m}}^{1} \int_{{\f{\pi}{2}-\al}}^{{\f{\pi}{2}+\al}} \frac{1}{\rho^2}g (\partial_\phi g) \cot \phi (\rho^2 \sin \phi) \d \phi \d \rho-2\pi\int_{\f{1}{m}}^{1} \int_{{\f{\pi}{2}-\al}}^{{\f{\pi}{2}+\al}} \frac{g^2}{\rho^2} (\rho^2\sin\phi) \d \phi \d \rho \\
		= & 2 \int_{D_m} \frac{\cot \phi}{\rho^2} g \partial_\phi g \d x-\int_{D_m} \frac{g^2}{\rho^2} \d x\, .
	\end{aligned}
	\ee
	Moreover, noticing that $\m{F}=0$ {on} $A_{1,m}\cup A_{2,m}$, one has
	\be\l{EFg}
	\int_{D_m} \frac{1}{\rho^2} \partial_\rho\left(\rho^2 \m{F}\right) g \d x=\int_{{\f{\pi}{2}-\al}}^{{\f{\pi}{2}+\al}} \int_{\f{1}{m}}^{1}g\p_\rho(\rho^2\m{F})\sin\phi\d \rho\d\phi=-\int_{D_m}\m{F}\p_\rho g\d x\,.
	\ee
	Owing to  $\m{K}=0$ {on} $R_{1,m}\cup R_{2,m}$, one deduces
	\be\label{EKg}
	-\int_{D_m} \frac{\partial_\phi \m{K}}{\rho}g \d x 
	= -\int_{{\f{\pi}{2}-\al}}^{{\f{\pi}{2}+\al}} \int_{\f{1}{m}}^{1}\rho g (\p_\phi\m{K}) \sin\phi\d \rho\d\phi
	=\int_{D_m} \frac{\partial_\phi g}{\rho}\m{K} \d x+\int_{D_m} \frac{\cot \phi}{\rho} \m{K} g \d x\,.
	\ee
	Substituting \eqref{Ebdg}, \eqref{Ebdg1}, \eqref{EFg} and \eqref{EKg} in \eqref{Egmain}, one derives
	\be\l{Estg}
	\begin{aligned}
		&\hskip .5cm \int_{D_m}|\nabla g|^2 \d x+\int_{D_m} \frac{1-\csc^2\phi}{\rho^2} g^2 \d x \\
		&=-\int_{D_m}\m{F}\p_\rho g\d x+\int_{D_m} \frac{\partial_\phi g}{\rho}\m{K} \d x+\int_{D_m} \frac{\cot \phi}{\rho} \m{K} g \d x\,.
	\end{aligned}
	\ee
	Here the last term follows that
	\bn
	\int_{D_m} \frac{\cot \phi}{\rho} \m{K} g \d x=&2\pi\int_\f{1}{m}^1\int_{\f{\pi}{2}-\al}^{\f{\pi}{2}+\al}\f{\sin\phi\cos\phi}{\rho} g \p_\phi \Big(\f{v_\th}{\sin\phi}\Big) \d\phi\d\rho\\
	=& 2\pi\int_\f{1}{m}^1\int_{\f{\pi}{2}-\al}^{\f{\pi}{2}+\al} \sin\phi\cos\phi \, g \, \p_\phi\Big(\f{g}{\sin\phi}\Big)\d\phi\d\rho\\
	=&2\pi\int_\f{1}{m}^1\int_{\f{\pi}{2}-\al}^{\f{\pi}{2}+\al}\cos\phi g\p_\phi g\d\phi\d\rho-2\pi\int_\f{1}{m}^1\int_{\f{\pi}{2}-\al}^{\f{\pi}{2}+\al}g^2\f{\cos^2\phi}{\sin\phi}\d\phi\d\rho\\
	=&\int_{D_m}\cot\phi\f{\p_\phi g}{\rho} \f{g}{\rho}\d x-\int_{D_m}\f{g^2}{\rho^2}\cot^2\phi\d x\,.
	\en
	Thus one infers from \eqref{Estg} that
	\ba\l{0319EE}
	\int_{D_m}|\nabla g|^2 \d x=-\int_{D_m}\m{F}\p_\rho g\d x+\int_{D_m} \frac{\partial_\phi g}{\rho}\m{K} \d x+\int_{D_m}\cot\phi\f{\p_\phi g}{\rho} \f{g}{\rho}\d x\,.
	\ea
	With the help of Corollary \ref{COR0416-1}, we infer from \eqref{0319EE} that
	\[
	\f{3}{10}\int_{D_m}|\nabla g|^2 \d x 
	\leq \|\m{F}\|_{L^2(D_m)}\|\p_\rho g\|_{L^2(D_m)}
	+ \|\m{K}\|_{L^2(D_m)} \Big\|\f{\p_\phi g}{\rho}\Big\|_{L^2(D_m)}.
	\]
	Then it follows from Cauchy inequality that 
	\[
	\f{3}{10}\int_{D_m}|\nabla g|^2 \d x \leq \|\nabla g\|_{L^2(D_m)}
	\left( \|\m{F}\|_{L^2(D_m)}^2 + \|\m{K}\|_{L^2(D_m)}^2\right)^{1/2}.
	\]
	which concludes \eqref{estg1}. 
\end{proof}

\section{Energy estimate for the triple $(\m{K}, \m{F}, \m{O})$ in $D_m$}\l{Sec9}
In this section, we will obtain a uniform energy bound for \((\m K, \m F, \m O)\).

\subsection{Statement of the energy bound for $(\m{K},\m{F},\m{O})$}
Recall the system \eqref{EOKF}, \eqref{rhs_KFO} and their boundary conditions \eqref{BCOKF} given in Section \ref{Sec2}. We will establish the following uniform energy bound for $(\m{K}, \m{F}, \m{O})$:

\begin{prop}\label{Prop_OKFest}
	Let the region $D_m$ be as defined in \eqref{app domain-cyl} with $m \geq 10^3$ and the angle $\alpha \in \left(0, \frac{\pi}{6}\right]$. Let $\mK$, $\m{F}$ and $\m{O}$ be defined as in \eqref{DO}--\eqref{DmK}. Then there exists an absolute constant $C_*$ such that if
	\be\label{small_G0}
	\|\G_0\|_{L^\i(D_m)}\leq C_*\,,
	\ee
	then the following estimate holds for any $T>0$:
	\be\label{KFOee0}
	\begin{aligned}
		& \int_{D_m}\left(\m{K}^2+\m{F}^2 + \m{O}^2\right)(x, T) \d x + \frac16\int_0^T \int_{D_m}\left(|{\nabla} \m{K}|^2+|{\nabla} \m{F}|^2 + |{\nabla} \m{O}|^2\right)(x,t) \d x \d t \\
		\leq & \exp\Big[C  \big( 1+\|\bm{v}_{0}\|_{L^2(D_m)}^4\big)\Big]\int_{D_m}\left(\m{K}_0^2+\m{F}_0^2 + \m{O}_0^2 + 1\right)(x) \d x\,,
	\end{aligned}
	\ee
	where $C$ is a positive constant independent of $T$ and $m$.
\end{prop}
To prove this proposition, we split the argument into three lemmas, corresponding to the estimates of $\mK$, $\mF$, and $\mO$ respectively in Section \ref{Subsec, K and F} and Section \ref{Subsec, Est_O}. For convenience of readers, we recall the system for $(\m{K}, \m{F}, \m{O})$ from (\ref{EOKF})--(\ref{BCOKF}):
\be
\left\{
\begin{aligned}
	&\left(\Delta+\frac{4}{\rho} \partial_\rho+\frac{2-4\cot^2\phi}{\rho^2}\right) \m{K}-\bm{b} \cdot \nabla \m{K}+\bm{\omega} \cdot \nabla\Big(\frac{v_\rho}{\rho}\Big)-\partial_t \m{K}=f_{\m{K}}\,,\\[3mm]
	&\left(\Delta+\frac{2}{\rho} \partial_\rho-\frac{3+\cot^2\phi}{\rho^2}\right) \m{F}-\bm{b} \cdot \nabla \m{F}+\frac{2}{\rho^2} \partial_\phi \m{K}+\f{4\cot\phi}{\rho^2}\m{K}+\bm{\omega} \cdot \nabla\Big(\frac{v_\phi}{\rho}\Big)-\p_t\m{F}=f_{\mF}\,,\\[3mm]
	&\left(\Delta+\frac{2}{\rho} \partial_{\rho}+\frac{2 \cot \phi}{\rho^{2}} \partial_{\phi}\right) \m{O}-\bm{b} \cdot \nabla \m{O}-\frac{2 v_{\theta}}{\rho \sin \phi}\m{K}-\f{2v_\th\cos\phi}{\rho\sin^2\phi}\mF-\partial_{t} \m{O}=f_{\m{O}}\,.
\end{aligned}
\right. \tag{c.f. (2.13)}
\ee
Where
\be\left\{
\begin{aligned}
	&f_{\m{K}}=-\frac{6\cot\phi}{\rho^3}v_\rho v_\th-\frac{2+2\cos^2\phi}{\rho^3\sin\phi}v_\phi v_\th\,,\\[3mm]
	&f_{\mF}=\frac{6}{\rho^3}v_\rho v_\th+\frac{2\cot\phi}{\rho^3}v_\phi v_\th\,, \\[3mm]
	&f_{\m{O}}=-\frac{2\eta^{\prime\prime}(\rho)}{\rho^2\sin\phi}v_\phi+\f{4\eta(\rho)}{\rho^3\sin\phi}\f{\p_\phi v_\rho}{\rho}+\frac{4}{\rho^2\sin\phi}\left(\f{\eta(\rho)}{\rho}-\eta^\prime(\rho)\right)\p_\rho v_\phi-\frac{2\eta(\rho)}{\rho^2\sin\phi}\f{\p_\phi P}{\rho}\\
	&\hskip 1cm+\frac{2}{\rho^2\sin\phi}\left(\eta^\prime(\rho)-\f{3\eta(\rho)}{\rho}\right)v_\phi v_\rho+\frac{2\eta(\rho)\cos\phi}{\rho^3\sin^2\phi}\left(v_\th^2-v_\phi^2\right)\,.
\end{aligned}
\right. \tag{c.f. (2.14)}
\ee
The boundary conditions for $(\m{K}, \m{F}, \m{O})$ are as follows:
\be
\left\{
\begin{aligned}
	&\m{K}=0\,,\q\text{on}\q R_{1,m}\cup R_{2,m}\,;\q\p_\rho\m{K}=-\f{1}{\rho}\m{K}\,,\q\text{on}\q A_{1,m}\cup A_{2,m}\,;\\[2mm]
	&\p_\phi\m{F}=\cot\phi \, \m{F}\q\text{on}\q R_{1,m}\cup R_{2,m}\,;\q\m{F}=0\,,\q\text{on}\q A_{1,m}\cup A_{2,m}\,;\\[2mm]
	&\m{O}=0\,,\q\text{on}\q \p D_m\,.
\end{aligned}
\right. \tag{c.f. (2.15)}
\ee

The main difficulty of Proposition \ref{Prop_OKFest} is that the above system \eqref{EOKF}, \eqref{rhs_KFO} contains nonlinear terms and the pressure term $P$. Fortunately, the $L^2$ estimate of $\nabla P$ in Section \ref{Sec, pressure}, the $L^{\infty}$ estimate of $\Gamma$ obtained in Section \ref{Sec7} and the Biot–Savart estimates established in Section \ref{Sec8} allow us to control these terms and close the energy estimate. 

For convenience of notations in the rest of the paper, we write $f \ls g$ if $|f| \leq C |g|$, where $C$ is some absolute positive constant which is not essential. The notation $f \gs g$ can be interpreted in a similar way, and $f \simeq g$ means that both $f\ls g$ and $f\gs g$ hold. 

\subsection{Energy estimates for $\m{K}$ and $\m{F}$}
\label{Subsec, K and F}

\begin{lemma}[estimate of $\m{K}$]\label{Lemma, K_energy}
	Under the same condition as Proposition \ref{Prop_OKFest}, we have 
	\ba\l{EK}
	&\f{\d}{\d t}\|\m{K}\|_{L^2(D_m)}^2 + \frac15 \|{\na}\m{K}\|_{L^2(D_m)}^2
	\leq C_{\m{K}} \|\G\|_{L^\i(D_m)}^2\|{\na}\m{O}\|_{L^2(D_m)}^2\,,
	\ea
	where $C_{\m{K}}>0$ is an absolute constant.
\end{lemma}
\begin{proof}
	Multipling \eqref{EOKF}$_{1}$ by $\m{K}$, then integrating the results over $D_m$, we deduce
	\ba\l{EKMM}
	\frac{1}{2}\frac{\d}{\d t}\|\m{K}\|_{L^2(D_m)}^2 &= 
	\int_{D_m}\m{K}\left(\Delta+\frac{4}{\rho} \partial_\rho+\frac{2-4\cot^2\phi}{\rho^2}\right) \m{K}\d x \\
	& \quad + \un{\int_{D_m}\m{K}\bm{\omega} \cdot \nabla\left(\frac{v_\rho}{\rho}\right)\d x}_{I_1} -\int_{D_m}\m{K}f_{\m{K}}\d x\,.
	\ea
	
	We first estimate $I_1$ which is the most challenging part in (\ref{EKMM}). According to the vorticity formula in (\ref{vor-sph}), we deduce
	\bn
	I_1 = & \int_{D_m} \m{K} \Big( \omega_\rho \partial_\rho + \omega_\phi \frac{1}{\rho} \partial_\phi \Big)\Big(\frac{v_\rho}{\rho}\Big) \d x\\
	= & \un{2\pi\int_{\f{\pi}{2}-\al}^{\f{\pi}{2}+\al}\int_{\f{1}{m}}^1\m{K} \rho \partial_\phi(\sin\phi \, v_\theta) \partial_\rho\Big(\frac{v_\rho}{\rho}\Big) \d \rho \d \phi}_{I_{11}}
	- \un{2\pi\int_{\f{\pi}{2}-\al}^{\f{\pi}{2}+\al}\int_{\f{1}{m}}^1\m{K} \partial_\rho(\rho v_\theta) \partial_\phi \Big(\frac{v_\rho}{\rho}\Big) \sin \phi \d \rho \d \phi}_{I_{12}}.
	\en
	Based on the boundary condition (\ref{BCOKF}) for $\m{K}$ which vanishes on rays $R_{1,m}\cup R_{2,m}$, so 
	\bn
	I_{11} &= -2\pi\int_{\f{\pi}{2}-\al}^{\f{\pi}{2}+\al}\int_{\f{1}{m}}^1 (\p_\phi\m{K}) \rho(\sin \phi\, v_\theta) \partial_\rho\Big(\frac{v_\rho}{\rho}\Big) \d \rho \d \phi \\
	&\quad - 2\pi\int_{\f{\pi}{2}-\al}^{\f{\pi}{2}+\al}\int_{\f{1}{m}}^1 \m{K} \rho(\sin \phi\, v_\theta) \p_\phi\p_\rho\Big(\frac{v_\rho}{\rho}\Big) \d \rho \d \phi\,.
	\en
	Meanwhile, since $v_\rho$ and $\p_\phi \rho$ vanish on arcs $A_{1,m}\cup A_{2,m}$, then
	\bn
	I_{12} &= -2\pi\int_{\f{\pi}{2}-\al}^{\f{\pi}{2}+\al}\int_{\f{1}{m}}^1 (\p_\rho\m{K})(\rho v_\theta) \partial_\phi\Big(\frac{v_\rho}{\rho}\Big) \sin \phi \d \rho \d \phi \\
	&\quad -2\pi\int_{\f{\pi}{2}-\al}^{\f{\pi}{2}+\al}\int_{\f{1}{m}}^1 \m{K}(\rho v_\theta) \p_\rho \p_\phi\Big(\frac{v_\rho}{\rho}\Big) \sin \phi \d \rho \d \phi \,.
	\en
	Noting that the second term in $I_{11}$ matches that in $I_{12}$, so we obtain 
	\bn
	I_{1} =& -2\pi\int_{\f{\pi}{2}-\al}^{\f{\pi}{2}+\al}\int_{\f{1}{m}}^1 (\p_\phi\m{K}) \rho(\sin \phi\, v_\theta) \partial_\rho\Big(\frac{v_\rho}{\rho}\Big) \d \rho \d \phi \\
	& + 2\pi\int_{\f{\pi}{2}-\al}^{\f{\pi}{2}+\al}\int_{\f{1}{m}}^1 (\p_\rho\m{K})(\rho v_\theta) \p_\phi \Big(\frac{v_\rho}{\rho}\Big) \sin \phi \d \rho \d \phi\\
	=& -\int_{D_m}\f{\partial_\phi\m{K}}{\rho} v_\theta \p_\rho \Big(\frac{v_\rho}{\rho}\Big)\d x
	+\int_{D_m} (\p_\rho\m{K}) v_\theta\f{\partial_\phi}{\rho}\Big(\frac{v_\rho}{\rho}\Big)\d x\,.
	\en
	Since $\Gamma = \rho \sin\phi \, v_\th$ and $\phi \in [\pi/2-\al, \pi/2+\al]\subset [\pi/3, 2\pi/3]$, then 
	\be\label{Gamma_vth}
	|v_\th| = \frac{|\Gamma|}{\rho\sin\phi} \leq \frac{2}{\sqrt{3} \rho} \|\G\|_{L^\i(D_m)}\,.
	\ee
	As a consequence, 
	\[
	|I_1| \ls \|\G\|_{L^\i(D_m)} \bigg( 
	\Big\| \f{\p_\phi\m{K}}{\rho} \Big\|_{L^2(D_m)} 
	\Big\| \f{1}{\rho}\p_\rho\Big(\frac{v_\rho}{\rho}\Big) \Big\|_{L^2(D_m)} + \|\p_\rho\m{K}\|_{L^2(D_m)} \Big\| \f{\p_\phi}{\rho^2} \Big(\frac{v_\rho}{\rho}\Big) \Big\|_{L^2(D_m)}
	\bigg)\,.
	\]
	Combining with the estimate (\ref{v_rho_est2}) in Lemma \ref{Lemma, f_1_est} yields 
	\ba\l{EOO1}
	|I_1| &\les \|\G\|_{L^\i(D_m)}\|{\na}\m{O}\|_{L^2(D_m)}\|{\na}\m{K}\|_{L^2(D_m)}\\
	&\leq \f{1}{5}\|\na\m{K}\|_{L^2(D_m)}^2 + C\|\G\|_{L^\i(D_m)}^2\|{\na}\m{O}\|_{L^2(D_m)}^2\,.
	\ea
	
	We next deal with the term that involves the Laplacian in (\ref{EKMM}). Recalling the boundary condition \eqref{BCOKF} for $\m{K}$ again, we use integration by parts to obtain
	\ba\l{KKK}
	\int_{D_m}\m{K}\Dl\m{K}\d x =&-\int_{D_m}|{\na}\m{K}|^2\d x+2\pi\int_{\f{\pi}{2}-\al}^{\f{\pi}{2}+\al}\m{K}\p_\rho\m{K}\rho^2\sin\phi\Big|_{\rho=\f{1}{m}}^1\d\phi\\
	=& -\int_{D_m}|{\na}\m{K}|^2\d x
	- 2\pi\int_{\f{\pi}{2}-\al}^{\f{\pi}{2}+\al}\m{K}^2 \rho\sin\phi \Big|_{\rho=\f{1}{m}}^1\d\phi\,.
	\ea
	Using the Newton-Leibniz formula, it follows that
	\bn
	\int_{D_m}\m{K}\Dl\m{K}\d x &= -\|{\na}\m{K}(\cd,t)\|_{L^2(D_m)}^2 - 2\pi\int_{\f{1}{m}}^1\int_{\f{\pi}{2}-\al}^{\f{\pi}{2}+\al}\p_\rho(\rho\m{K}^2)\sin\phi\d\phi\d\rho \\
	&= -\|{\na}\m{K}(\cd,t)\|_{L^2(D_m)}^2 - \int_{D_m} \Big( \frac{\m{K}^2}{\rho^2} + \frac{2\m{K}}{\rho}  \p_\rho\m{K}\Big) \d x\,.
	\en
	Hence,
	\ba\l{KKKK}
	&\int_{D_m}\m{K}\left(\Dl+\frac{4}{\rho} \partial_\rho+\frac{2-4\cot^2\phi}{\rho^2}\right) \m{K}\d x\\
	=& -\|{\na}\m{K}\|_{L^2(D_m)}^2 - \int_{D_m} \Big( \frac{\m{K}^2}{\rho^2} + \frac{2\m{K}}{\rho}  \p_\rho\m{K}\Big) \d x + \int_{D_m}\m{K}\left(\frac{4}{\rho} \partial_\rho+\frac{2-4\cot^2\phi}{\rho^2}\right) \m{K}\d x\\
	=& -\|{\na}\m{K}\|_{L^2(D_m)}^2
	+\un{2\int_{D_m}\f{\m{K}}{\rho}\p_\rho\m{K}\d x}_{K_1}
	+\un{\int_{D_m}\f{1-4\cot^2\phi}{\rho^2} \m{K}^2\d x}_{K_2}\,.
	\ea
	Using the Cauchy-Schwarz inequality and the Poincar\'e inequality \eqref{PoinBE},
	\bn
	|K_1| \leq 2\Big\|\f{\m{K}}{\rho}\Big\|_{L^2(D_m)} \|\p_\rho\m{K}\|_{L^2(D_m)} 
	\leq &\f{2\sqrt{3}}{5}\Big\|\f{\p_\phi\m{K}}{\rho} \Big\|_{L^2(D_m)}\|\p_\rho\m{K}\|_{L^2(D_m)}\\
	\leq& \f{1}{2}\|\p_\rho\m{K}\|_{L^2(D_m)}^2 + \f{6}{25}\Big\|\f{\p_\phi\m{K}}{\rho}\Big\|_{L^2(D_m)}^2\,.
	\en
	and
	\bn
	|K_2|\leq\f{3}{25}\Big\|\f{\p_\phi\m{K}}{\rho}\Big\|_{L^2(D_m)}^2\,.
	\en
	The above two estimates indicate that
	\[
	|K_1|+|K_2|\leq\f{1}{2}\|\na\m{K}\|_{L^2(D_m)}^2\,.
	\]
	Plugging this estimate into (\ref{KKKK}) yields
	\be\label{EK1}
	\int_{D_m}\m{K}\left(\Dl+\frac{4}{\rho} \partial_\rho+\frac{2-4\cot^2\phi}{\rho^2}\right) \m{K}\d x \leq - \frac12 \|{\na}\m{K}\|_{L^2(D_m)}^2.
	\ee
	
	Combining (\ref{EOO1}), (\ref{EK1}) with (\ref{EKMM}) leads to 
	\be\label{EKMM1}
	\frac{1}{2}\frac{\d}{\d t}\|\m{K}\|_{L^2(D_m)}^2 
	\leq -\frac{3}{10} \|\na\m{K}\|_{L^2(D_m)}^2 + C\|\G\|_{L^\i(D_m)}^2\|{\na}\m{O}\|_{L^2(D_m)}^2 + \bigg|  \int_{D_m}\m{K}f_{\m{K}}\d x \bigg| \,.
	\ee
	Recalling the expression for $f_{\m{K}}$ in (\ref{rhs_KFO}):
	\[
	f_{\m{K}} = -\frac{6\cot\phi}{\rho^3}v_\rho v_\th - \frac{2+2\cos^2\phi}{\rho^3\sin\phi}v_\phi v_\th,
	\]
	then it follows from (\ref{Gamma_vth}) and the Cauchy-Schwarz inequality that 
	\bn
	\bigg| \int_{D_m}\m{K}f_{\m{K}}\d x \bigg| & \ls 
	\|\G\|_{L^\i(D_m)} \int_{D_m} \frac{|\m{K}|}{\rho} \, \frac{|v_\rho| + |v_\phi|}{\rho^3} \d x \\
	&\ls \|\G\|_{L^\i(D_m)} \bigg(\Big\| \f{v_\rho}{\rho^3} \Big\|_{L^2(D_m)} + \Big\|\f{v_\phi}{\rho^3} \Big\|_{L^2(D_m)}\bigg) \Big\|\f{\m{K}}{\rho}\Big\|_{L^2(D_m)} \,.
	\en
	Applying the Poincar\'e inequality \eqref{Poin0E} for $v_\rho$ and \eqref{PoinBE} for $v_\phi$ and $\m{K}$ yields 
	\[
	\bigg| \int_{D_m}\m{K}f_{\m{K}}\d x \bigg| \ls
	\|\G\|_{L^\i(D_m)}
	\left( \Big\| \f{1}{\rho}\f{\p_\phi}{\rho}\Big(\f{v_\rho}{\rho}\Big) \Big\|_{L^2(D_m)} + \Big\|\f{1}{\rho}\f{\p_\phi}{\rho}\Big(\f{v_\phi}{\rho}\Big)\Big\|_{L^2(D_m)}\right) \Big\|\f{\p_\phi\m{K}}{\rho}\Big\|_{L^2(D_m)} \,.
	\]
	Moreover, by taking advantage of Lemma \ref{Lemma, f_2_est} and Lemma \ref{Lemma, f_1_est}, we find
	\ba\l{ERHSK}
	\bigg| \int_{D_m}\m{K}f_{\m{K}}\d x\bigg|
	\les & \|\G\|_{L^\i(D_m)} \| \nabla \m{O}\|_{L^2(D_m)}
	\Big\| \f{\p_\phi\m{K}}{\rho} \Big\|_{L^2(D_m)}\\
	\leq&\f{1}{5}\|\na\m{K}\|_{L^2(D_m)}^2+C\|\G\|_{L^\i(D_m)}^2\|\na\m{O}\|_{L^2(D_m)}^2\,.
	\ea
	
	Finally, plugging \eqref{ERHSK} into \eqref{EKMM1} leads to 
	\[
	\frac{1}{2}\frac{\d}{\d t}\|\m{K}\|_{L^2(D_m)}^2 
	\leq -\frac{1}{10} \|\na\m{K}\|_{L^2(D_m)}^2 + C\|\G\|_{L^\i(D_m)}^2\|{\na}\m{O}\|_{L^2(D_m)}^2 \,.
	\]
\end{proof}

\begin{lemma}[estimate of $\m{F}$]\l{LemmaF}
	Under the same condition as Proposition \ref{Prop_OKFest}, we have
	\ba\l{EF}
	&\f{\d}{\d t}\|\m{F}\|_{L^2(D_m)}^2 + \frac45\|{\nabla}\m{F}\|_{L^2(D_m)}^2 
	\leq & \frac43 \|{\nabla}\m{K}\|_{L^2(D_m)}^2 + 
	C_\mF \|\G\|_{L^\i(D_m)}^2\|{\nabla}\m{O}\|_{L^2(D_m)}^2\,.
	\ea
	where $C_\mF>0$ is an absolute constant.
\end{lemma}
\begin{proof}
	Multiply \eqref{EOKF}$_2$ by $\m{F}$ and integrating over $D_m$, we get
	\ba\l{LF6}
	\frac{1}{2}\frac{\d}{\d t}\|\m{F}\|_{L^2(D_m)}^2 &= 
	\int_{D_m}\m{F}\left(\Delta+\frac{2}{\rho} \partial_\rho-\frac{3+\cot^2\phi}{\rho^2}\right) \m{F}\d x
	- \int_{D_m}\m{F}\bm{\omega} \cdot \nabla\Big(\frac{v_\phi}{\rho}\Big)\d x
	\\
	&\quad -\int_{D_m}\m{F}\left(\frac{2}{\rho^2} \partial_\phi \m{K}+\f{4\cot\phi}{\rho^2}\m{K}\right)\d x
	- \int_{D_m}\m{F}f_{\m{F}}\d x\,.
	\ea
	Here $f_{\m{F}}$ is the right hand side of \eqref{EOKF}$_2$. 
	
	Similarly as the derivation for \eqref{EOO1}, we obtain
	\ba\l{LF9}
	\bigg|\int_{D_m}\m{F}\bm{\omega} \cdot \nabla\Big(\frac{v_\phi}{\rho}\Big)\d x\bigg|
	\leq\f{1}{5}\|\na\m{F}\|_{L^2(D_m)}^2
	+ C\|\G\|_{L^\i(D_m)}^2\|{\nabla}\m{O}\|_{L^2(D_m)}^2\,.
	\ea
	Then analogous to the argument between \eqref{KKK}--\eqref{KKKK}, we deduce
	\[
	\int_{D_m}\m{F} \Delta \m{F} \d x = -\|{\nabla}\m{F}\|_{L^2(D_m)}^2 + \int_{D_m}\f{2\m{F} (\p_\phi \m{F})\cot\phi}{\rho^2}\dx - \int_{D_m}\f{1}{\rho^2}\mF^2\d x\,,
	\]
	and 	
	\ba\l{LF7}
	&\int_{D_m}\m{F}\left(\Delta+\frac{2}{\rho} \partial_\rho-\frac{3+\cot^2\phi}{\rho^2}\right) \m{F}\d x\\
	= & -\|{\nabla}\m{F}\|_{L^2(D_m)}^2
	+ \un{\int_{D_m}\f{2\m{F} (\p_\phi \m{F})\cot\phi}{\rho^2}\dx}_{F_1}
	+ \un{\int_{D_m}\f{2\m{F}\p_\rho\m{F}}{\rho}\d x}_{F_2}
	- \int_{D_m}\f{4+\cot^2\phi}{\rho^2}\mF^2\d x \,.\\
	\ea
	
	Using the Young inequality and noticing that $|\cot\phi| \leq\f{2}{\sqrt{3}}$,
	\bn
	|F_1| \leq & \f{2}{\sqrt{3}}\Big\|\f{\m{F}}{\rho}\Big\|_{L^2(D_m)} 	\Big\|\f{\p_\phi\m{F}}{\rho}\Big\|_{L^2(D_m)}
	\leq \f{1}{5}\Big\|\f{\p_\phi\m{F}}{\rho}\Big\|_{L^2(D_m)}^2 + 	\f{5}{3}\Big\|\f{\m{F}}{\rho}\Big\|_{L^2(D_m)}^2\,.
	\en
	Meanwhile, since $\m{F}$ vanishes on arcs $A_{1,m}\cup A_{2,m}$ due to the boundary condition (\ref{BCOKF}), we apply integration by parts to get
	\bn
	F_2=2\pi\int_\f{1}{m}^1\int_{\f{\pi}{2}-\al}^{\f{\pi}{2}+\al}\rho \p_\rho(\mF^2) \sin\phi\d\phi\d\rho
	= - 2\pi\int_\f{1}{m}^1\int_{\f{\pi}{2}-\al}^{\f{\pi}{2}+\al}\mF^2\sin\phi\d\phi\d\rho
	= - \Big\|\f{\mF}{\rho}\Big\|_{L^2(D_m)}^2\,.
	\en
	Combining the above estimates of $F_1$ and $F_2$ with (\ref{LF7}) yields 
	\be\label{F_lap_est}
	\int_{D_m}\m{F}\left(\Delta+\frac{2}{\rho} \partial_\rho-\frac{3+\cot^2\phi}{\rho^2}\right) \m{F}\d x \leq -\frac45 \|{\nabla}\m{F}\|_{L^2(D_m)}^2 - \frac{10}{3} \Big\|\f{\mF}{\rho}\Big\|_{L^2(D_m)}^2\,.
	\ee
	Plugging (\ref{LF9}) and (\ref{F_lap_est}) into (\ref{LF6}) yields 
	\ba\l{F_eest1}
	\frac{1}{2}\frac{\d}{\d t}\|\m{F}\|_{L^2(D_m)}^2 
	& \leq -\frac35 \|{\nabla}\m{F}\|_{L^2(D_m)}^2 
	- \frac{10}{3} \Big\|\f{\mF}{\rho}\Big\|_{L^2(D_m)}^2
	+ C\|\G\|_{L^\i(D_m)}^2\|{\nabla}\m{O}\|_{L^2(D_m)}^2 \\
	& \quad + \int_{D_m}\m{F}\left(\frac{2}{\rho^2} \partial_\phi \m{K}+\f{4\cot\phi}{\rho^2}\m{K}\right)\d x
	+ \int_{D_m}\m{F}f_{\m{F}}\d x\,.
	\ea
	
	Moreover, by Young's inequality and the Poincar\'e inequality \eqref{PoinBE}, we get
	\ba\l{LF10}
	&\quad \bigg|\int_{D_m}\m{F}\left(\frac{2}{\rho^2} \partial_\phi \m{K}+\f{4\cot\phi}{\rho^2}\m{K}\right)\d x\bigg| \\
	&\leq \bigg(2\Big\|\f{\m{F}}{\rho}\Big\|_{L^2(D_m)}^2 
	+ \f{1}{2} \Big\|\f{\p_\phi\m{K}}{\rho}\Big\|_{L^2(D_m)}^2\bigg)
	+ \bigg(\Big\|\f{\m{F}}{\rho}\Big\|_{L^2(D_m)}^2
	+ \f{4}{3}\Big\|\f{\m{K}}{\rho}\Big\|_{L^2(D_m)}^2 \bigg)\\
	&\leq 3\Big\|\f{\m{F}}{\rho}\Big\|_{L^2(D_m)}^2 + \Big(\f{1}{2} + \frac43 \cdot \frac{3}{25}\Big)\Big\|\f{\p_\phi\m{K}}{\rho}\Big\|_{L^2(D_m)}^2\,.
	\ea
	Next, using the same method in the derivation for (\ref{ERHSK}), we deduce that
	\ba\l{LF11}
	\bigg|\int_{D_m}\m{F}f_{\m{F}}\d x\bigg|
	\leq \frac{1}{5}\|\na\m{F}\|_{L^2(D_m)}^2 + C\|\G\|_{L^\i(D_m)}^2\|{\nabla}\m{O}\|_{L^2(D_m)}^2\,.
	\ea
	
	Substituting (\ref{LF10}) and (\ref{LF11}) into (\ref{F_eest1}) gives 
	\bn
	\frac{1}{2}\frac{\d}{\d t}\|\m{F}\|_{L^2(D_m)}^2 
	& \leq -\frac25 \|{\nabla}\m{F}\|_{L^2(D_m)}^2 
	+ \f{33}{50} \| \nabla \mK \|_{L^2(D_m)}^2
	+ C\|\G\|_{L^\i(D_m)}^2\|{\nabla}\m{O}\|_{L^2(D_m)}^2 \,,
	\en
	which implies \eqref{EF}. This finishes the proof of Lemma \ref{LemmaF}. 
\end{proof}

\subsection{Energy estimate for $\m{O}$}
\label{Subsec, Est_O}
The equation for $\m{O}$ is substantially more difficult than those for $\m{K}$ and $\m{F}$ because it contains the pressure term $P$. Before carrying out the energy estimate for $\m{O}$, we first introduce some basic inequalities which will be applied to control nonlinear terms in Lemma \ref{LemEO}. These nonlinear terms arise due to the $L^2$ estimate of $\nabla P$ in Lemma \ref{EPs}. 
\begin{lemma}[Nonlinear control]\l{NonC}
	Let the region $D_m$ be as defined in \eqref{app domain-cyl} with $m\geq 10^3$ and the angle $\alpha\in (0,\frac{\pi}{6}]$. Let $f,g\in H^1(D_m)$ be axially symmetric functions. Suppose for any $\rho\in[\f{1}{m},1]$, there exists $\phi_\rho\in[\f{\pi}{2}-\al,\f{\pi}{2}+\al]$ such that $g(\rho,\phi_\rho)=0$. Then the following estimates hold
	\ba\l{NlinE1}
	\|fg\|_{L^2(D_m)}^2&\leq C \Big\|\f{f}{\rho}\Big\|_{L^2(D_m)} \Big( \Big\|\f{\p_\rho f}{\rho}\Big\|_{L^2(D_m)} + \Big\|\f{f}{\rho^2}\Big\|_{L^2(D_m)}\Big) \|g\|_{L^2(D_m)} \|\p_\phi g\|_{L^2(D_m)}\,,
	\ea
	and
	\ba\l{NlinE2}
	\|fg\|_{L^2(D_m)}^2 &\leq C \Big\|\f{f}{\rho}\Big\|_{L^2(D_m)}
	\Big(\|{\p_\rho f}\|_{L^2(D_m)}
	+ \Big\|\f{f}{\rho}\Big\|_{L^2(D_m)}\Big) \|g\|_{L^2(D_m)} \Big\|\f{\p_\phi g}{\rho} \Big\|_{L^2(D_m)}\,,
	\ea
	where $C>0$ denotes an absolute constant.
\end{lemma}
\begin{proof}
	Let $\eta$ be the cut-off function of variable $\rho$ defined in \eqref{DCUT}. For any $\phi\in[\f{\pi}{2}-\al,\f{\pi}{2}+\al]$, one has
	\ba\l{Sup0}
	\int_{\f{1}{m}}^1f^2g^2\rho^2\d\rho 
	\leq & 2\int_{\f{1}{m}}^1(f\eta)^2g^2\rho^2\d\rho + 2\int_{\f{1}{m}}^1[f(1-\eta)]^2 g^2 \rho^2 \d\rho\\
	\leq & 2\Big(\sup_{\rho\in[\f{1}{m},1]}(f\eta)^2+\sup_{\rho\in[\f{1}{m},1]}[f(1-\eta)]^2 \Big) \int_{\f{1}{m}}^1 g^2 \rho^2 \d\rho\,.
	\ea
	Since $\eta=0$ for $\rho\leq\frac{1}{3}$, using the Newton-Leibniz formula, we infer that
	\ba\label{sup1}
	\sup_{\rho \in [\f{1}{m},1]}(f\eta)^2 \leq \int_{\f{1}{m}}^1 \big| \p_\rho \big[(f\eta)^2 \big] \big| \d\rho \les \int_{\f{1}{m}}^1 \bigg( \f{f^2}{\rho^3} + \f{|f \p_\rho f|}{\rho^2} \bigg) \rho^2 \d\rho\,.
	\ea
	Similarly, since $1-\eta=0$ for $\rho\geq\frac{2}{3}$, we also have
	\ba\l{sup2}
	\sup_{\rho\in[\f{1}{m},1]} [f(1-\eta)]^2 \les \int_{\f{1}{m}}^1\bigg(\f{f^2}{\rho^3} + \f{|f \p_\rho f|}{\rho^2}\bigg)\rho^2\d\rho\,.
	\ea
	Moreover, since there exists $\phi_\rho\in[\f{\pi}{2}-\al,\f{\pi}{2}+\al]$ such that $g(\rho,\phi_\rho)=0$, one has
	\be\l{Sup3} 								\sup_{\phi\in[\f{\pi}{2}-\al,\f{\pi}{2}+\al]}\int_{\f{1}{m}}^1g^2\rho^2\d\rho
	\leq\int_{\f{1}{m}}^1\sup_{\phi\in[\f{\pi}{2}-\al,\f{\pi}{2}+\al]} \big(g^2 \rho^2\big) \d\rho
	\leq 2\int_{\f{1}{m}}^1 \bigg(\int_{\f{\pi}{2}-\al}^{\f{\pi}{2}+\al} |g\p_\phi g|\d\phi \bigg) \rho^2\d\rho\,.
	\ee
	Combining \eqref{Sup0}--\eqref{sup1}--\eqref{sup2}--\eqref{Sup3} yields
	\bn
	\|fg\|_{L^2(D_m)}^2 = & 2\pi \int_{\f{\pi}{2}-\al}^{\f{\pi}{2}+\al} \bigg(\int_{\f{1}{m}}^1 f^2 g^2 \rho^2 \d\rho\bigg) \sin\phi \d\phi\\
	\les & \bigg[\int_{\f{\pi}{2}-\al}^{\f{\pi}{2}+\al} \bigg(\int_{\f{1}{m}}^1\bigg(\f{f^2}{\rho^3} + \f{|f\p_\rho f|}{\rho^2}\bigg)\rho^2\d\rho\bigg) \sin\phi \d\phi \bigg]
	\bigg( \int_{\f{1}{m}}^1\int_{\f{\pi}{2}-\al}^{\f{\pi}{2}+\al}|g\p_\phi g|\rho^2\d\phi\d\rho\bigg)\\
	= & \int_{D_m}\bigg( \f{f^2}{\rho^3} + \f{|f\p_\rho f|}{\rho^2}\bigg)\d x  \int_{\f{1}{m}}^1\int_{\f{\pi}{2}-\al}^{\f{\pi}{2}+\al}|g\p_\phi g|\rho^2\d\phi\d\rho\,.
	\en
	Noticing that $\sin\phi$ is away from zero since $0 < \al \leq \f{\pi}{6}$, so
	\[
	\|fg\|_{L^2(D_m)}^2  \ls \int_{D_m}\bigg( \f{f^2}{\rho^3} + \f{|f\p_\rho f|}{\rho^2}\bigg)\d x \int_{D_m}|g\p_\phi g|\d x\,.
	\]
	Then one concludes \eqref{NlinE1} by applying the Cauchy-Schwarz inequality. 
	
	Estimates in \eqref{NlinE2} can be achieved via a similar approach, by distributing the powers of $\rho$ differently. Instead of \eqref{Sup0}, we can derive
	\ba\l{sup01}
	\int_{\f{1}{m}}^1f^2g^2\rho^2\d\rho \leq & 2\Big(\sup_{\rho\in[\f{1}{m},1]}\rho(f\eta)^2+\sup_{\rho\in[\f{1}{m},1]}\rho [f(1-\eta)]^2 \Big) \int_{\f{1}{m}}^1g^2\rho\d\rho\,.
	\ea
	Akin to \eqref{sup1} and \eqref{sup2}, we also have
	\ba\label{sup11}
	\sup_{\rho\in[\f{1}{m},1]}\rho(f\eta)^2 \leq \int_{\f{1}{m}}^1 |\p_\rho(\rho f^2\eta^2)|\d\rho\les\int_{\f{1}{m}}^1\bigg(\f{f^2}{\rho^2} + \f{|f\p_\rho f|}{\rho} \bigg)\rho^2\d\rho\,
	\ea
	and
	\ba\l{sup21}
	\sup_{\rho\in[\f{1}{m},1]}\rho [f(1-\eta)]^2 \les \int_{\f{1}{m}}^1\bigg(\f{f^2}{\rho^2} + \f{|f\p_\rho f|}{\rho} \bigg)\rho^2\d\rho\,.
	\ea
	Similar to \eqref{Sup3}, we also have
	\be\l{Sup31} \sup_{\phi\in[\f{\pi}{2}-\al,\f{\pi}{2}+\al]}\int_{\f{1}{m}}^1g^2\rho\d\rho 
	\leq \int_{\f{1}{m}}^1 \sup_{\phi\in[\f{\pi}{2}-\al,\f{\pi}{2}+\al]} (g^2\rho) \d\rho
	\les \int_{\f{1}{m}}^1 \bigg(\int_{\f{\pi}{2}-\al}^{\f{\pi}{2}+\al} \bigg| g\f{\p_\phi g}{\rho} \bigg| \d\phi \bigg) \rho^2\d\rho\,.
	\ee
	Combining \eqref{sup01}--\eqref{sup11}--\eqref{sup21}--\eqref{Sup31}, one can derive \eqref{NlinE2} by using the same method as we deduce \eqref{NlinE1}. 
\end{proof}

\begin{lemma}[estimate of $\m{O}$]\l{LemEO}
	Under the same condition as Proposition \ref{Prop_OKFest}, we have
	\ba\l{EOmega}
	&\f{\d}{\d t}\|\m{O}\|_{L^2(D_m)}^2+\|{\na}\m{O}\|_{L^2(D_m)}^2\\
	\leq\,\,& 
	C_{1}\left(\|\m{K}\|_{L^2(D_m)}^2 + \|\m{F}\|_{L^2(D_m)}^2 + \|\m{O}\|_{L^2(D_m)}^2+1\right)\|\na\bm{v}\|_{L^2(D_m)}^2\left(1+\|\bm{v}\|_{L^2(D_m)}^2\right) \\
	& + \|\G\|_{L^\i(D_m)}^2\left(\|{\na}\m{K}\|_{L^2(D_m)}^2 + 40\|{\na}\m{F}\|_{L^2(D_m)}^2\right) \,.
	\ea
	where $C_{1}>0$ is an absolute constant. 
\end{lemma}

\begin{proof}
	Multiplying \eqref{EOKF}$_3$ by $\m{O}$ and integrating over $D_m$ yields
	\ba\l{LemEOM}
	\frac{1}{2}\frac{\d}{\d t}\|\m{O}\|_{L^2(D_m)}^2 
	&= \int_{D_m}\m{O}\left(\Delta+\frac{2}{\rho} \partial_{\rho}+\frac{2 \cot \phi}{\rho^{2}} \partial_{\phi}\right) \m{O}\d x\\
	& \quad -\int_{D_m}\m{O}\frac{2 v_{\theta}}{\rho \sin \phi}\m{K}\d x
	-\int_{D_m}\m{O}\f{2v_\th\cos\phi}{\rho\sin^2\phi}\mF\d x
	- \int_{D_m}\m{O}f_{\m{O}}\d x\,.
	\ea  
	Since $\m{O}$ vanishes on $\p D_m$ due to the boundary condition (\ref{BCOKF}), then 
	we apply integration by parts to obtain
	\ba\l{EL1}
	\int_{D_m}\m{O}\Dl\m{O}\d x=-\int_{D_m}|\na\m{O}|^2\d x\,,
	\ea
	and
	\ba\l{EL2}
	&\int_{D_m}\m{O}\left(\frac{2}{\rho} \partial_{\rho}+\frac{2 \cot \phi}{\rho^{2}} \partial_{\phi}\right) \m{O}\d x\\
	=& 2\pi\int_{{\f{\pi}{2}-\al}}^{{\f{\pi}{2}+\al}}\int_{\frac1m}^{1}\p_\rho(\m{O}^2) \rho\sin\phi\d \rho \d\phi
	+ 2\pi\int_{{\f{\pi}{2}-\al}}^{{\f{\pi}{2}+\al}}\int_{\frac1m}^{1}\p_\phi (\m{O}^2) \cos\phi\d\rho\d\phi\\
	=&-2\pi\int_{{\f{\pi}{2}-\al}}^{{\f{\pi}{2}+\al}}\int_{\frac1m}^{1}\m{O}^2\sin\phi\d \rho \d\phi+2\pi\int_{{\f{\pi}{2}-\al}}^{{\f{\pi}{2}+\al}}\int_{\frac1m}^{1}\m{O}^2\sin\phi\d \rho \d\phi = 0\,.
	\ea
	Plugging (\ref{EL1}) and (\ref{EL2}) into (\ref{LemEOM}) yields 
	\ba\l{LemEOM2}
	& \quad \frac{1}{2}\frac{\d}{\d t}\|\m{O}\|_{L^2(D_m)}^2 + \|\nabla \m{O}\|_{L^2(D_m)}^2 
	\\
	& = -\int_{D_m}\m{O}\frac{2 v_{\theta}}{\rho \sin \phi}\m{K}\d x
	-\int_{D_m}\m{O}\f{2v_\th\cos\phi}{\rho\sin^2\phi}\mF\d x
	- \int_{D_m}\m{O}f_{\m{O}}\d x\,.
	\ea  
	
	Next, we will estimate the terms on the right hand side of (\ref{LemEOM2}) one by one. Firstly, since $\sin\phi \geq \frac{\sqrt{3}}{2}$ on $D_m$, then 
	\[
	\bigg|\int_{D_m}\m{O}\frac{2 v_{\theta}}{\rho \sin \phi}\m{K}\d x\bigg| 
	\leq \int_{D_m}\frac{2}{\rho^2 \sin^2 \phi}\big|\G\m{O}\m{K}\big|\d x \leq \frac{8}{3} \|\G\|_{L^\i(D_m)} \int_{D_m}\frac{|\m{O}|}{\rho} \frac{|\m{K}|}{\rho}
	\d x\,.
	\]
	Since $\m{O} = \m{K}=0$ on rays $R_{1,m}\cup R_{2,m}$, then it follows from the Cauchy-Schwarz inequality and the Poincar\'e inequality \eqref{PoinBE} that
	\ba\l{EL3}
	\bigg|\int_{D_m}\m{O}\frac{2 v_{\theta}}{\rho \sin \phi}\m{K}\d x\bigg|
	&\leq \frac{8}{3} \|\G\|_{L^\i(D_m)} 
	\frac{\sqrt{3}}{5} \Big\|\frac{\p_\phi\m{O}}{\rho}\Big\|_{L^2(D_m)} 
	\frac{\sqrt{3}}{5} \Big\|\frac{\p_\phi\m{K}}{\rho}\Big\|_{L^2(D_m)} \\
	&\leq \f{1}{16}\|\na\m{O}\|_{L^2(D_m)}^2 + \frac12\|\G\|_{L^\infty(D_m)}^2\|\na\m{K}\|_{L^2(D_m)}^2\,.
	\ea
	On the other hand, since $\sin\phi \geq \frac{\sqrt{3}}{2}$ and $\cos\phi \leq \frac12$, then 
	\[
	\bigg|\int_{D_m}\m{O}\f{2v_\th\cos\phi}{\rho\sin^2\phi}\mF\d x\bigg| 
	\leq \int_{D_m}\frac{2\cos\phi}{\rho^2 \sin^3 \phi}\big|\G\m{O}\m{F}\big|\d x \leq \frac{8}{3\sqrt{3}} \|\G\|_{L^\i(D_m)} \int_{D_m}\frac{|\m{O}|}{\rho} \frac{|\m{F}|}{\rho}
	\d x\,.
	\]
	Since 
	\[
	\int_{D_m} \Big( \frac{\m{F}}{\rho} \Big)^2 \d x = 2\pi \int_{\frac{\pi}{2} - \al}^{\frac{\pi}{2} + \al} \bigg( \int_{\frac1m}^{1} \m{F}^2 \d\rho \bigg) \sin\phi \d\phi,
	\]
	and $\m{F}=0$ on $A_{1,m}\cup A_{2,m}$, then we apply the classical Hardy's inequality to get 
	\[
	\int_{\frac1m}^{1} \m{F}^2 \d\rho \leq 4 \int_{\frac1m}^{1} \rho^2 |\p_\rho \m{F}|^2 \d\rho,
	\]
	which implies that $\| \frac{\m{F}}{\rho}\|_{L^2(D_m)}^2 \leq 4 \|{\nabla}\m{F}\|_{L^2(D_m)}^2$. As a result, 
	\ba\l{EL4}
	\bigg|\int_{D_m}\m{O}\f{2v_\th\cos\phi}{\rho\sin^2\phi}\mF\d x\bigg|
	&\leq \frac{8}{3\sqrt{3}} \|\G\|_{L^\i(D_m)} 
	\frac{\sqrt{3}}{5} \Big\|\frac{\p_\phi\m{O}}{\rho}\Big\|_{L^2(D_m)} 
	4\| \nabla \m{F}\|_{L^2(D_m)} \\
	&\leq \f{1}{16}\|\na\m{O}\|_{L^2(D_m)}^2 + 20\|\G\|_{L^\infty(D_m)}^2\|\na\m{F}\|_{L^2(D_m)}^2\,.
	\ea
	Combining (\ref{EL3}) and (\ref{EL4}) with (\ref{LemEOM2}) gives 
	\ba\label{LemEOM3}
	& \quad \frac{1}{2}\frac{\d}{\d t}\|\m{O}\|_{L^2(D_m)}^2 + \frac78 \|\nabla \m{O}\|_{L^2(D_m)}^2 
	\\
	& \leq \|\G\|_{L^\infty(D_m)}^2 \Big( \frac12 \|\nabla \m{K}\|_{L^2(D_m)}^2 
	+ 20 \|\nabla \m{F}\|_{L^2(D_m)}^2 \Big)
	+ \bigg|\int_{D_m}\m{O}f_{\m{O}}\d x \bigg|\,.
	\ea  
	
	Now we estimate the term with $f_{\m{O}}$ whose formula was given in (\ref{rhs_KFO}):	
	\[
	\begin{aligned}
		&f_{\m{O}}=-\frac{2\eta^{\prime\prime}(\rho)}{\rho^2\sin\phi}v_\phi+\f{4\eta(\rho)}{\rho^3\sin\phi}\f{\p_\phi v_\rho}{\rho}+\frac{4}{\rho^2\sin\phi}\left(\f{\eta(\rho)}{\rho}-\eta^\prime(\rho)\right)\p_\rho v_\phi-\frac{2\eta(\rho)}{\rho^2\sin\phi}\f{\p_\phi P}{\rho}\\
		&\hskip 1cm+\frac{2}{\rho^2\sin\phi}\left(\eta^\prime(\rho)-\f{3\eta(\rho)}{\rho}\right)v_\phi v_\rho+\frac{2\eta(\rho)\cos\phi}{\rho^3\sin^2\phi}\left(v_\th^2-v_\phi^2\right)\,.
	\end{aligned}
	\]
	Although $f_{\m{O}}$ contains many terms, all of them are supported away from the origin due to the cut-off function $\eta$, so 
	\ba\l{EEFF0}
	\bigg|\int_{D_m}\m{O}f_{\m{O}}\d x\bigg| \les 
	& \int_{D_m}|v_\phi\m{O}|\d x
	+\int_{D_m}\Big|\frac{\p_\phi v_\rho}{\rho}\m{O}\Big|\d x
	+\int_{D_m} | (\p_\rho v_\phi) \m{O} | \d x
	+\int_{D_m}|v_\phi v_\rho\m{O}|\d x\\
	& +\int_{D_m}|v_\th^2\m{O}|\d x+\int_{D_m}|v_\phi^2\m{O}| \d x
	+\int_{D_m}\Big|\frac{\p_\phi P}{\rho}\m{O}\Big|\d x\,.
	\ea
	Since $\rho\leq 1$ in $D_m$, we use Cauchy-Schwarz inequality, Corollary \ref{Cor, vSob} for $v_\phi$ and the Poincar\'e inequality \eqref{Poin0E} for $\m{O}$ to get
	\ba\l{EEFF1}
	&\int_{D_m}|v_\phi\m{O}|\d x
	+ \int_{D_m} | (\p_\rho v_\phi) \m{O} | \d x
	+\int_{D_m} \Big| \frac{\p_\phi v_\rho}{\rho} \m{O} \Big|\d x\\
	\les\,\, &  \big( \| \nabla v_\phi \|_{L^2(D_m)} + \| \nabla v_\rho \|_{L^2(D_m)}  \big) \Big\|\f{\p_\phi \mO}{\rho}\Big\|_{L^2(D_m)} \\
	\les \,\, & \| \nabla \bm{v} \|_{L^2(D_m)} \| \nabla \m{O}\|_{L^2(D_m)}  \,.
	\ea
	Meanwhile, we apply H\"older's inequality and Corollary \ref{Cor, vSob} for $v_\phi$ to deduce
	\ba\l{EEFF2}
	\int_{D_m}|v_\phi v_\rho\m{O}|\d x
	\leq& \|\m{O}\|_{L^6(D_m)}\|v_\phi\|_{L^6(D_m)} \|v_\rho\|_{L^2(D_m)} \| 1 \|_{L^6(D_m)}\\
	\les & \|{\nabla} \m{O}\|_{L^2(D_m)}\|{\na}v_\phi\|_{L^2(D_m)} \|\bm{v}\|_{L^2(D_m)}\,.
	\ea
	Akin to this, we also have
	\ba\l{EEFF3}
	\int_{D_m}|v_\th^2\m{O}|\d x + \int_{D_m}|v_\phi^2\m{O}|\d x
	\les \|{\nabla} \m{O}\|_{L^2(D_m)}\|{\na} \bm{v}\|_{L^2(D_m)} \|\bm{v}\|_{L^2(D_m)}\,.
	\ea
	Moreover, we use Cauchy-Schwarz inequality and Poincar\'e inequality (\ref{PoinB}) for $\m{O}$ to find 
	\be\l{EEFF4}
	\int_{D_m} \Big| \frac{\p_\phi P}{\rho} \m{O} \Big|\d x \les 
	\| \nabla \m{O}\|_{L^2(D_m)} \| \nabla P\|_{L^2(D_m)}.
	\ee
	Combining \eqref{EEFF0}--\eqref{EEFF1}--\eqref{EEFF2}--\eqref{EEFF3}--\eqref{EEFF4} together yields 
	\bn
	\bigg|\int_{D_m}\m{O}f_{\m{O}}\d x\bigg| &\les \bigg( 
	\| \nabla \m{O}\|_{L^2(D_m)} \| \nabla \bm{v} \|_{L^2(D_m)} \big( 1+ \|\bm{v}\|_{L^2(D_m)} \big) 
	+ \| \nabla \m{O}\|_{L^2(D_m)} \| \nabla P\|_{L^2(D_m)} \bigg) \\
	& \leq \frac{1}{8} \| \nabla \m{O}\|_{L^2(D_m)}^2 
	+ C \| \nabla \bm{v} \|_{L^2(D_m)}^2 \Big( 1+ \|\bm{v}\|_{L^2(D_m)}^2 \Big) 
	+ C \| \nabla P\|_{L^2(D_m)}^2 \,.
	\en
	Plugging this estimate into (\ref{LemEOM3}) leads to 
	\ba\label{LemEOM4}
	& \frac{1}{2}\frac{\d}{\d t}\|\m{O}\|_{L^2(D_m)}^2 + \frac34 \|\nabla \m{O}\|_{L^2(D_m)}^2 \\
	\leq\,\, & \|\G\|_{L^\infty(D_m)}^2 \Big( \frac12 \|\nabla \m{K}\|_{L^2(D_m)}^2 
	+ 20 \|\nabla \m{F}\|_{L^2(D_m)}^2 \Big) \\
	&+ C \| \nabla \bm{v} \|_{L^2(D_m)}^2 \Big( 1+ \|\bm{v}\|_{L^2(D_m)}^2 \Big) 
	+ C \| \nabla P\|_{L^2(D_m)}^2 \,,
	\ea  
	
	It remains to handle the term that involves the pressure $P$ in (\ref{LemEOM4}). We apply Lemma \ref{EPs} to get
	\ba\l{EPNT00}
	\|\na P\|_{L^2(D_m)}^2
	\leq\,\, & C\|\na\bm{v}\|_{L^2(D_m)}^2
	+\un{2\Big\|\frac{1}{\rho}\left(v_{\phi}^{2}+v_{\theta}^{2}\right)\Big\|_{L^{2}(D_m)}}_{PN_1}\\
	& + \Big\|\Big(v_{\rho} \partial_{\rho}+\frac{1}{\rho} v_{\phi} \partial_{\phi}\Big) v_{\rho}\Big\|_{L^{2}(D_m)}^2
	+\Big\|\Big[v_{\rho}\Big(\partial_{\rho}+\frac{1}{\rho}\Big) + \frac{1}{\rho} v_{\phi} \partial_{\phi}\Big] v_{\phi}\Big\|_{L^{2}(D_m)}^2\,.
	\ea
	We will treat $PN_1$ via the following claim:
	\begin{claim}\label{Claim, nlv}
		\be\label{nlv}\begin{split}
			\Big\| \frac{1}{\rho} (v_\rho^2 + v_\phi^2) \Big\|_{L^2(D_m)}^2 &\ls \|\m{O}\|_{L^2(D_m)}^2\|\na\bm{v}\|_{L^2(D_m)}^2, \\
			\Big\| \frac{1}{\rho} v_\th^2 \Big\|_{L^2(D_m)}^2 
			&\ls \big( \|\m{K}\|_{L^2(D_m)}^2 + \|\m{F}\|_{L^2(D_m)}^2 \big)\|\na\bm{v}\|_{L^2(D_m)}^2.
		\end{split}\ee
	\end{claim}
	\begin{proof}[Proof of Claim \ref{Claim, nlv}]
		We first estimate $\| v_\rho^2 / \rho\|_{L^2(D_m)}$. By H\"older's inequality,
		\[
		\Big\|\f{v_\rho^2}{\rho}\Big\|_{L^2(D_m)}\leq \Big\|\f{v_\rho}{\rho}\Big\|_{L^6(D_m)} \|v_\rho\|_{L^3(D_m)}.
		\]
		Then it follows from the Sobolev embedding theorem and Corollary \ref{Cor, vSob} that 
		\[
		\|v_\rho\|_{L^3(D_m)} \ls \|v_\rho\|_{H^1(D_m)} \ls \| \nabla v_\rho\|_{L^2(D_m)}.
		\]
		Meanwhile, we apply the Sobolev embedding theorem and Lemma \ref{Lemma, f_1_est} to obtain 
		\[
		\Big\|\f{v_\rho}{\rho}\Big\|_{L^6(D_m)} \ls \Big\| \nabla \Big(\f{v_\rho}{\rho}\Big)\Big\|_{H^1(D_m)} \ls \|\m{O}\|_{L^2(D_m)}.
		\]
		Combining the above estimates together gives 
		$\| v_\rho^2 / \rho\|_{L^2(D_m)} \ls \|\m{O}\|_{L^2(D_m)} \|\na\bm{v}\|_{L^2(D_m)}$.
		
		By a similar argument but replacing Lemma \ref{Lemma, f_1_est} with Lemma \ref{Lemma, f_2_est}, we can also justify 
		$\| v_\phi^2 / \rho\|_{L^2(D_m)} \ls \|\m{O}\|_{L^2(D_m)} \|\na\bm{v}\|_{L^2(D_m)}$.
		Meanwhile, by applying Lemma \ref{Lem35} instead of Lemma \ref{Lemma, f_1_est}, we can verify (\ref{nlv})$_2$. Hence, Claim \ref{Claim, nlv} is justified.
	\end{proof}
	
	Thanks to Claim \ref{Claim, nlv}, it follows from (\ref{EPNT00}) that 
	\ba\l{gradP1}
	\|\na P\|_{L^2(D_m)}^2
	\leq\,\, & C\big(\|\m{K}\|_{L^2(D_m)}^2 + \|\m{F}\|_{L^2(D_m)}^2 + \|\m{O}\|_{L^2(D_m)}^2 + 1\big)\|\na\bm{v}\|_{L^2(D_m)}^2 \\
	& +\un{\Big\| v_{\rho} \partial_{\rho} v_\rho + \frac{1}{\rho} v_{\phi} \partial_{\phi} v_{\rho} \Big\|_{L^{2}(D_m)}^2}_{PN_2}
	+ \un{\Big\| v_{\rho}\partial_{\rho} v_\phi + \frac{1}{\rho} v_{\phi} \partial_{\phi} v_\phi \Big\|_{L^{2}(D_m)}^2}_{PN_3}\,.
	\ea
	Now we consider $PN_2$ and rewrite 
	\[
	v_{\rho} \partial_{\rho} v_\rho = \rho v_\rho \p_\rho \Big( \frac{v_\rho}{\rho}\Big) + \frac{1}{\rho} v_\rho^2\,.
	\]
	The term $\frac{1}{\rho} v_\rho^2$ can be treated using Claim \ref{Claim, nlv}, so we will estimate $\big\| \rho v_\rho \p_\rho \big( \frac{v_\rho}{\rho}\big) \big\|_{L^2}$. Recalling the following identity in Lemma \ref{v_rho_mean0}:
	\[
	\int_{\f{\pi}{2}-\al}^{\f{\pi}{2}+\al}v_\rho(\rho,\phi)\sin\phi\d\phi = 0, \q\forall\rho\in \big[\f{1}{m}, 1\big]\,,
	\]
	which implies that 
	\[
	\int_{\f{\pi}{2}-\al}^{\f{\pi}{2}+\al} \p_\rho\Big( \frac{v_\rho(\rho,\phi)}{\rho} \Big) \sin\phi\d\phi = 0, \q\forall\rho\in \big[\f{1}{m}, 1\big]\,,
	\]
	so for any $\rho\in[\f{1}{m},1]$, there exists $\phi_\rho\in[\f{\pi}{2}-\al,\f{\pi}{2}+\al]$ such that $\p_\rho \big(\f{v_\rho}{\rho}\big)(\rho,\phi_\rho)=0$. 
	Thus, one can apply \eqref{NlinE1} in Lemma \ref{NonC} (with $f=\rho v_\phi$, $g=\p_\rho\f{v_\rho}{\rho}$) to obtain 
	\[
	\Big\| \rho v_\phi \p_\rho \Big( \frac{v_\rho}{\rho}\Big) \Big\|_{L^2}^2 \ls
	\|v_\phi\|_{L^2} \Big( \Big\| \p_\rho v_\phi + \frac{v_\phi}{\rho} \Big\|_{L^2} + \Big\| \frac{v_\phi}{\rho}\Big\|_{L^2} \Big) 
	\Big\| \p_\rho \Big( \frac{v_\rho}{\rho} \Big) \Big\|_{L^2} 
	\Big\| \p_\phi \p_\rho \Big( \frac{v_\rho}{\rho} \Big) \Big\|_{L^2} 
	\]
	Since $\|v_\phi/\rho\|_{L^2} \ls \|\nabla v_\phi\|_{L^2}$ and $\|\p_\rho(v_\rho/\rho)\|_{L^2} \ls \|\m{O}\|_{L^2}$, we have 
	\[
	\Big\| \rho v_\phi \p_\rho \Big( \frac{v_\rho}{\rho}\Big) \Big\|_{L^2}^2 \ls
	\|v_\phi\|_{L^2} \|\nabla v_\phi\|_{L^2} \|\m{O}\|_{L^2}
	\Big\| \p_\phi \p_\rho \Big( \frac{v_\rho}{\rho} \Big) \Big\|_{L^2}.
	\]
	Noticing that 
	\[
	\p_\phi \p_\rho \Big( \frac{v_\rho}{\rho} \Big) = \frac{1}{\rho} \p_\phi \p_\rho v_\rho - \frac{1}{\rho^2} \p_\phi v_\rho,
	\]
	we apply Lemma \ref{Lem2nd} and Lemma \ref{Lemma, f_1_est} to deduce
	\[
	\Big\| \p_\phi \p_\rho \Big( \frac{v_\rho}{\rho} \Big) \Big\|_{L^2} \ls \|\nabla \m{O}\|_{L^2} + \|\m{O}\|_{L^2} \ls \|\nabla \m{O}\|_{L^2} \,.
	\]	
	Therefore, 
	\be\label{PN2-1}
	\Big\| \rho v_\phi \p_\rho \Big( \frac{v_\rho}{\rho}\Big) \Big\|_{L^2}^2 \ls
	\|v_\phi\|_{L^2} \|\nabla v_\phi\|_{L^2} \|\m{O}\|_{L^2} \|\nabla \m{O}\|_{L^2}\,.
	\ee
	
	Next, we estimate the term $\frac{1}{\rho} v_\phi \p_\phi v_\rho$ in $PN_2$. Thanks to the fact that $v_\phi$ vanishes on rays $R_{1,m}\cup R_{2,m}$, one can apply (\ref{NlinE2}) in Lemma \ref{NonC} (with $f = \frac{1}{\rho} \p_\phi v_\rho$ and $g=v_\phi$) to get 
	\[
	\Big\| \frac{1}{\rho} v_\phi \p_\phi v_\rho \Big\|_{L^2}^2 \ls
	\Big\| \frac{1}{\rho^2} \p_\phi v_\rho\Big\|_{L^2} \Big( \Big\| \frac{1}{\rho}\p_\rho \p_\phi v_\rho \Big\| + 2\Big\| \frac{1}{\rho^2}\p_\phi v_\rho \Big\| \Big)
	\|v_\phi\|_{L^2} \Big\| \frac{1}{\rho} \p_\phi v_\phi \Big\|_{L^2}.
	\]
	By taking advantage of Lemma \ref{Lemma, f_1_est} and Lemma \ref{Lem2nd}, 
	\be\label{PN2-2}\begin{split}
		\Big\| \frac{1}{\rho} v_\phi \p_\phi v_\rho \Big\|_{L^2}^2 &\ls \|\m{O}\|_{L^2} \big( \| \nabla \m{O}\|_{L^2} + \|\m{O}\|_{L^2} \big) \|v_\phi\|_{L^2} \| \nabla v_\phi \|_{L^2} \\
		& \ls \|\m{O}\|_{L^2} \|\nabla \m{O}\|_{L^2} \|v_\phi\|_{L^2} \|\nabla v_\phi\|_{L^2} \,.
	\end{split}\ee
	Combining (\ref{PN2-1}) and (\ref{PN2-2}) with (\ref{gradP1}) yields 
	\ba\l{gradP2}
	\|\na P\|_{L^2(D_m)}^2
	\leq\,\, & C\big(\|\m{K}\|_{L^2(D_m)}^2 + \|\m{F}\|_{L^2(D_m)}^2 + \|\m{O}\|_{L^2(D_m)}^2 + 1\big)\|\na\bm{v}\|_{L^2(D_m)}^2 \\
	& + C \|\bm{v}\|_{L^2} \|\nabla \bm{v}\|_{L^2} \|\m{O}\|_{L^2} \|\nabla \m{O}\|_{L^2}
	+ \un{\Big\| v_{\rho}\partial_{\rho} v_\phi + \frac{1}{\rho} v_{\phi} \partial_{\phi} v_\phi \Big\|_{L^{2}(D_m)}^2}_{PN_3}\,.
	\ea
	
	We continue to treat $PN_3$ by reducing to the case for $PN_2$. Firstly, it follows from the divergence-freeness of $\bm{v}$ that 
	\be\label{convert_div}
	\frac{1}{\rho} \p_\phi v_\phi = -\p_\rho v_\rho - \frac{2}{\rho} v_\rho - \frac{\cot\phi}{\rho} v_\phi,
	\ee
	so we deduce the estimate for the second term in $PN_3$:
	\[
	\Big\| \frac{1}{\rho} v_{\phi} \partial_{\phi} v_\phi \Big\|_{L^2} \leq \|v_\phi \p_\rho v_\rho\|_{L^2} + 2 \Big\| \frac{v_\phi v_\rho}{\rho} \Big\|_{L^2} + \Big\| \frac{v_\phi^2}{\rho} \Big\|_{L^2}.
	\]
	Thanks to Claim \ref{Claim, nlv} and similar to the estimate for $\| v_\rho \p_\rho v_\rho \|_{L^2}$ in $PN_2$, we have
	\be\label{PN3-1}
	\Big\| \frac{1}{\rho} v_{\phi} \partial_{\phi} v_\phi \Big\|_{L^2}^2 
	\leq C \|\m{O}\|_{L^2(D_m)}^2\|\na\bm{v}\|_{L^2(D_m)}^2
	+ C \|\bm{v}\|_{L^2} \|\nabla \bm{v}\|_{L^2} \|\m{O}\|_{L^2} \|\nabla \m{O}\|_{L^2}.
	\ee
	Then we rewrite the first term in $PN_3$ as:
	\be\label{PN3-2}
	v_\rho \p_\rho v_\phi = \rho v_\rho \p_\rho\Big( \frac{v_\phi}{\rho} \Big) + \frac{v_\rho v_\phi}{\rho}.
	\ee
	The term $\frac{v_\rho v_\phi}{\rho}$ can be treated using Claim \ref{Claim, nlv}. Meanwhile, we apply \eqref{NlinE1} in Lemma \ref{NonC} (with $f=\rho v_\rho$, $g=\p_\rho\f{v_\phi}{\rho}$) to obtain 
	\[
	\Big\| \rho v_\rho \p_\rho \Big( \frac{v_\phi}{\rho}\Big) \Big\|_{L^2}^2 \ls
	\|v_\rho\|_{L^2} \Big( \Big\| \p_\rho v_\rho + \frac{v_\rho}{\rho} \Big\|_{L^2} + \Big\| \frac{v_\rho}{\rho}\Big\|_{L^2} \Big) 
	\Big\| \p_\rho \Big( \frac{v_\phi}{\rho} \Big) \Big\|_{L^2} 
	\Big\| \p_\phi \p_\rho \Big( \frac{v_\phi}{\rho} \Big) \Big\|_{L^2}.
	\]
	Since $\|v_\rho/\rho\|_{L^2} \ls \|\nabla v_\rho\|_{L^2}$ and $\|\p_\rho(v_\phi/\rho)\|_{L^2} \ls \|\m{O}\|_{L^2}$, we have 
	\be\label{PN3-3}
	\Big\| \rho v_\rho \p_\rho \Big( \frac{v_\phi}{\rho}\Big) \Big\|_{L^2}^2 \ls
	\|v_\rho\|_{L^2} \|\nabla v_\rho\|_{L^2} \|\m{O}\|_{L^2}
	\Big\| \p_\phi \p_\rho \Big( \frac{v_\phi}{\rho} \Big) \Big\|_{L^2}.
	\ee
	Thanks to the conversion (\ref{convert_div}), we know 
	\bn
	\p_\phi \p_\rho \Big( \frac{v_\phi}{\rho} \Big) = \p_\rho \Big( \frac{1}{\rho}\p_\phi v_\phi  \Big) &= \p_\rho \Big( -\p_\rho v_\rho - \frac{2}{\rho} v_\rho - \frac{\cot\phi}{\rho} v_\phi \Big) \\
	&= -\p_\rho^2 v_\rho - 2 \p_\rho \Big(\frac{v_\rho}{\rho}\Big) - \cot\phi\, \p_\rho\Big(\frac{v_\phi}{\rho}\Big).
	\en
	Hence, it follows from Lemma \ref{Lemma, f_2_est}, Lemma \ref{Lemma, f_1_est} and Lemma \ref{Lem2nd} that 
	\[
	\Big\| \p_\phi \p_\rho \Big( \frac{v_\phi}{\rho} \Big) \Big\|_{L^2} \ls \|\nabla \m{O}\|_{L^2} + \|\m{O}\|_{L^2} \ls  \|\nabla \m{O}\|_{L^2}.
	\]
	Plugging this inequality into (\ref{PN3-3}) yields 
	\[
	\Big\| \rho v_\rho \p_\rho \Big( \frac{v_\phi}{\rho}\Big) \Big\|_{L^2}^2 \ls \|\bm{v}\|_{L^2} \|\nabla \bm{v}\|_{L^2} \|\m{O}\|_{L^2} \|\nabla \m{O}\|_{L^2}.
	\]
	Combining this result with (\ref{PN3-2}) and Claim \ref{Claim, nlv} gives 
	\be\label{PN3-4}
	\| v_\rho \p_\rho v_\phi\|_{L^2}^2 \ls \|\bm{v}\|_{L^2} \|\nabla \bm{v}\|_{L^2} \|\m{O}\|_{L^2} \|\nabla \m{O}\|_{L^2} + \|\m{O}\|_{L^2(D_m)}^2\|\na\bm{v}\|_{L^2(D_m)}^2.
	\ee
	Based on (\ref{PN3-1}) and (\ref{PN3-4}), we obtain 
	\be\label{PN3-5}
	|PN_3| \ls \|\bm{v}\|_{L^2} \|\nabla \bm{v}\|_{L^2} \|\m{O}\|_{L^2} \|\nabla \m{O}\|_{L^2} + \|\m{O}\|_{L^2(D_m)}^2\|\na\bm{v}\|_{L^2(D_m)}^2.
	\ee
	
	Substituting (\ref{PN3-5}) into (\ref{gradP2}) leads to 
	\bn
	\|\na P\|_{L^2(D_m)}^2
	\leq\,\, & C\big(\|\m{K}\|_{L^2(D_m)}^2 + \|\m{F}\|_{L^2(D_m)}^2 + \|\m{O}\|_{L^2(D_m)}^2 + 1\big)\|\na\bm{v}\|_{L^2(D_m)}^2 \\
	& + C \|\bm{v}\|_{L^2} \|\nabla \bm{v}\|_{L^2} \|\m{O}\|_{L^2} \|\nabla \m{O}\|_{L^2}\,.
	\en
	Then we plugging this estimate into (\ref{LemEOM4}) to find
	\ba\label{LemEOM5}
	& \frac{1}{2}\frac{\d}{\d t}\|\m{O}\|_{L^2(D_m)}^2 + \frac34 \|\nabla \m{O}\|_{L^2(D_m)}^2 \\
	\leq\,\, & \|\G\|_{L^\infty(D_m)}^2 \Big( \frac12 \|\nabla \m{K}\|_{L^2(D_m)}^2 
	+ 20 \|\nabla \m{F}\|_{L^2(D_m)}^2 \Big) 
	+ C \| \nabla \bm{v} \|_{L^2(D_m)}^2 \Big( 1+ \|\bm{v}\|_{L^2(D_m)}^2 \Big) \\
	& + C\big(\|\m{K}\|_{L^2(D_m)}^2 + \|\m{F}\|_{L^2(D_m)}^2 + \|\m{O}\|_{L^2(D_m)}^2 + 1\big)\|\na\bm{v}\|_{L^2(D_m)}^2 \\
	&+ C\|\bm{v}\|_{L^2} \|\nabla \bm{v}\|_{L^2} \|\m{O}\|_{L^2} \|\nabla \m{O}\|_{L^2}\,.
	\ea  
	Since 
	\[
	C\|\bm{v}\|_{L^2} \|\nabla \bm{v}\|_{L^2} \|\m{O}\|_{L^2} \|\nabla \m{O}\|_{L^2} \leq \frac14 \|\nabla \m{O}\|_{L^2(D_m)}^2 + C^2 \|\bm{v}\|_{L^2}^2 \|\nabla \bm{v}\|_{L^2}^2 \|\m{O}\|_{L^2}^2,
	\]
	it then follows from (\ref{LemEOM5}) that 
	\bn
	& \frac{1}{2}\frac{\d}{\d t}\|\m{O}\|_{L^2(D_m)}^2 + \frac12 \|\nabla \m{O}\|_{L^2(D_m)}^2 \\
	\leq\,\, & \|\G\|_{L^\infty(D_m)}^2 \Big( \frac12 \|\nabla \m{K}\|_{L^2(D_m)}^2 
	+ 20 \|\nabla \m{F}\|_{L^2(D_m)}^2 \Big) \\
	& + C_1\big(\|\m{K}\|_{L^2(D_m)}^2 + \|\m{F}\|_{L^2(D_m)}^2 + \|\m{O}\|_{L^2(D_m)}^2 + 1\big)\|\na\bm{v}\|_{L^2(D_m)}^2 \big( 1+ \|\bm{v}\|_{L^2(D_m)}^2\big) \,,
	\en
	where $C_1$ is some large absolute constant, which justifies (\ref{EOmega}).
	
\end{proof}

\subsection{Proof of the energy bound for $(\m{K}, \m{F}, \m{O})$}
After the preparation of the energy estimates for $\m{K}$, $\m{F}$ and $\m{O}$ in Lemma \ref{Lemma, K_energy}, Lemma \ref{LemmaF} and Lemma \ref{LemEO} respectively, we are ready to justify the main result, Proposition \ref{Prop_OKFest}, in this section.

\begin{proof}[Proof of Proposition \ref{Prop_OKFest}] 
	By calculating 
	\[
	10\times \eqref{EK} + \eqref{EF} + \eqref{EOmega}\,,
	\]
	one derives
	\be\label{KFOee1}\begin{split}
		&\quad\,\, \frac{d}{dt}\Big( 10 \|\m{K}\|_{L^2}^2 + \|\m{F}\|_{L^2}^2 + \|\m{O}\|_{L^2}^2 \Big) 
		+ \frac23 \| \nabla \m{K}\|_{L^2}^2 + \frac45 \|\nabla \m{F}\|_{L^2}^2 + \|\nabla \m{O}\|_{L^2}^2 \\
		& \leq \|\Gamma\|_{L^{\infty}}^2 \Big( \|\nabla \m{K}\|_{L^2}^2 + 40\|\nabla \m{F}\|_{L^2}^2 
		+ (10 C_{\m{K}} + C_{\m{F}}) \|\nabla \m{O}\|_{L^2}^2 \Big) \\
		& \quad + C_{1} \Big( \|\m{K}\|_{L^2}^2 + \|\m{F}\|_{L^2}^2 + \|\m{O}\|_{L^2}^2 + 1\Big) \|\nabla \bm{v}\|_{L^2}^2 \big(1+\|\bm{v}\|_{L^2}^2\big),
	\end{split}\ee
	where $L^2$ denotes $L^2(D_m)$, and $C_\m{K}$, $C_\m{F}$ and $C_1$ are the constants in Lemma \ref{Lemma, K_energy}, Lemma \ref{LemmaF} and Lemma \ref{LemEO} respectively. Denote 
	\[\left\{\begin{array}{ll}
		Y(t) &= 10 \|\m{K}(\cdot,t)\|_{L^2}^2 + \|\m{F}(\cdot,t)\|_{L^2}^2 + \|\m{O}(\cdot, t)\|_{L^2}^2, \vspace{0.1in}\\
		G(t) &= \| \nabla \m{K}(\cdot,t)\|_{L^2}^2 + \|\nabla \m{F}(\cdot,t)\|_{L^2}^2 + \|\nabla \m{O}(\cdot,t)\|_{L^2}^2, \vspace{0.1in}\\
		H(t) &= \|\nabla \bm{v}(\cdot, t)\|_{L^2}^2 \big(1+\|\bm{v}(\cdot, t)\|_{L^2}^2\big).
	\end{array}\right.\]
	Then it follows from (\ref{KFOee1}) that for any $t\in (0,T]$,
	\be\label{KFOee2}
	Y'(t) + \frac23 G(t) \leq \|\Gamma\|_{L^\infty}^2 \max\{40, 10 C_{\m{K}} + C_{\m{F}} \}  G(t) + C_{1} (Y+1) H(t)\,.
	\ee
	According to the constraint (\ref{small_G0}) and Lemma \ref{LEMGA},
	\[
	\|\G\|_{L^\infty} \leq C_\G \|\G_0\|_{L^\infty} \leq C_\G C_*,
	\]
	where $C_\G$ is the positive constant in (\ref{DeGest}) and $C_*$ is a positive constant to be determined. Based on (\ref{KFOee2}), we choose $C_{*}$ as
	\bn
	C_{*}^2 = \frac{1}{2 C_\G^2 \max\{40, 10C_\m{K}+C_\m{F}\}}, 
	\en
	then it follows from (\ref{KFOee2}) that 
	\be\label{KFOee3}
	Y'(t) + \frac16 G(t) \leq C_{1} \big( Y(t)+1 \big) H(t).
	\ee
	
	Applying Gr\"onwall's inequality to (\ref{KFOee3}) and noticing $G\geq 0$, we obtain 
	\[
	Y(t) + 1 \leq \big(Y(0)+1\big) \exp\Big( C_1 \int_{0}^{t} H(\tau) \d \tau \Big).
	\]
	Plugging the above inequality to the right hand side of (\ref{KFOee3}) and then integrating $t$ from $0$ to $T$ yields 
	\be\label{KFOee4}
	Y(T) + \frac16 \int_{0}^{T} G(t) \d t \leq Y(0) + C_1\big(Y(0) + 1\big) \bigg(\int_{0}^{T} H(t) \d t\bigg) \exp\Big(C_1 \int_{0}^{T} H(t) \d t\Big).
	\ee
	Based on Proposition \ref{Funden}, 
	\[
	\int_{0}^{T} H(t) \d t \leq \big(1+\|\bm{v}_0\|_{L^2}^2\big) \int_{0}^{T} \|\nabla \bm{v}(\cdot,t)\|_{L^2}^2 \d t \leq \frac43 \big(1+\|\bm{v}_0\|_{L^2}^2\big) \|\bm{v}_0\|_{L^2}^2.
	\]
	Combining this estimate with (\ref{KFOee4}), we conclude there exists a positive constant $C_2$ such that 
	\[
	Y(T) + \frac16 \int_{0}^{T} G(t) \d t \leq \big( Y(0) + 1 \big)
	\exp\Big[C_2 \big( 1 + \|\bm{v}_0\|_{L^2}^4 \big)\Big],
	\]
	which implies (\ref{KFOee0}).
\end{proof}

\section{Existence and uniqueness of strong solutions in $D$}\l{Sec10}

\label{Sec, exist-uniq}

In this section, we will establish the main results, Theorem \ref{Thm_main} and Corollary \ref{Cor, unstable-bus}, of this paper. 
We will first derive a uniform \(L^\infty_{x,t}\) bound, together with uniform higher-order estimates, for the velocity field $\bm{v}$ on \(D_m\times[0,T]\). These bounds will be used to pass to the limit as $m\to\infty$ and establish the existence theorem of (\ref{NS}) on the original domain $D\times [0,T]$.

\subsection{Uniform boundedness and higher-order regularity estimates of the velocity}

\begin{prop}\label{Prop, inf_vomega}
	Let the region $D_m$ be as defined in \eqref{app domain-cyl} with $m\geq 10^3$ and the angle $\al\in(0,\f{\pi}{6}]$. Let $C_*$ be the positive constant given in Proposition \ref{Prop_OKFest}
	\[
	\|\G_0\|_{L^\i(D_m)}\leq C_*\,,
	\]
	then there exists a constant $C$, which only depends on $\alpha$ and $\|\bm{v}_0 \|_{C^2(\overline{D_m})}$, such that for any $T>0$,
	\bn
	\|\bm{v}\|_{L_{t x}^{\infty}\left(D_m \times[0, T]\right)}+\left\|\omega_\theta\right\|_{L_{t x}^{\infty}\left(D_m \times[0, T]\right)} \leq C\,.
	\en
\end{prop}
\begin{proof}
	The proof can be derived by adapting contents in Section 4.7 of \cite{LPYZZZ24}. The main difference comes from the boundary term when carrying out integration by parts for the major term. However, these can be corrected by adding lower-order terms with cut-off functions of $\rho-$variable. We omit the details here.
\end{proof}

Next, we will derive uniform estimates for $\|\blll{v}\|_{L^2_tH^2_x}$ and $\|\blll{v}\|_{H^1_tL^2_x}$ on $D_m\times[0,T]$.
\begin{prop}\label{Prop, high_deri_est}
	Let the region $D_m$ be defined as in \eqref{app domain-cyl} with $m \geq 10^3$ and the angle $\alpha \in\left(0, \frac{\pi}{6}\right]$. Let $C_*$ be the positive constant given in Proposition \ref{Prop_OKFest}
	\[
	\|\G_0\|_{L^\i(D_m)}\leq C_*\,.
	\]
	Then there exists a constant $C$, which only depends on $\alpha$ and $\left\|\bm{v}_0\right\|_{C^2\left(\overline{D_m}\right)}$, such that for any $T>0$,
	$$
	\|\nabla^2 \bm{v}\|_{L_{t x}^2\left(D_m \times[0, T]\right)}+\|\partial_t \bm{v}\|_{L_{t x}^2\left(D_m \times[0, T]\right)} \leq C\, .
	$$
\end{prop}
\begin{proof}
	The full $H^2_x$ estimate could be derived by adapting the proof of Lemma \ref{Lem2nd}. See also \cite[Section 5]{LPYZZZ24} for a proof for $\bm{v}$ with NHL boundary. The temporal estimate could be deduced by applying the equation \eqref{NS}, together with the pressure estimate in Section \ref{Sec, pressure}. We omit the details here.
\end{proof}

\subsection{Proof of Theorem \ref{Thm_main}}

\begin{proof}[Proof of Theorem \ref{Thm_main}]
	
	\textbf{Step 1: Existence.}
	
	We first show the existence of a strong solution $ (\bm{v},P) $. Pick any $ \bm{v}_0 $ in the admissible class $ \mathscr{A} $ that satisfies the properties (i) and (ii) in Theorem \ref{Thm_main}. By Definition \ref{Def, admissible sets}, there exists a sequence $ \{\bm{v}^{(m)}_0\}_{m\geq 10^3} $ such that $\bm{v}^{(m)}_0\in\mathscr{A}_m$ and
	\be\label{conv-init}
	\lim_{m\to\infty} \|\bm{v}_0 - \bm{v}^{(m)}_0 \|_{C^2(\overline{D_m})} = 0.
	\ee
	
	Although the initial data $\bm{v}_0$ satisfies $\int_{D} r v_{0,\th}(x) \d x = 0$ due to property (i) in Theorem \ref{Thm_main}, $\bm{v}^{(m)}_0$ may not have this property on $D_m$, 
	so we adjust $\bm{v}^{(m)}_0$ to be $\tilde{\bm{v}}^{(m)}_0$ which is defined as 
	\bn
	\tilde{\bm{v}}^{(m)}_{0}(x) = \bm{v}^{(m)}_0(x) - \epsilon_m r \bm{e_{\th}},
	\en
	where $\ep_m$ is chosen so that
	\be\label{init_ker_ad}
	\int_{D_m} r \tilde{v}^{(m)}_{0,\th}(x) \d x = 0. 
	\ee
	In fact, 
	\[
	\int_{D_m} r \tilde{v}^{(m)}_{0,\th}(x) \d x  = \int_{D_m} r v^{(m)}_{0,\th}(x) \d x - \ep_{m} \int_{D_m} r^2 \d x,
	\]
	so choosing 
	\bn
	\ep_m = \frac{1}{ \int_{D_m} r^2 \d x} \int_{D_m} r v^{(m)}_{0,\th}(x) \d x 
	\en
	validates (\ref{init_ker_ad}). With this choice of $\ep_m$, we have
	$\tilde{\bm{v}}^{(m)}_0 \in \mathscr{A}_m$ and 
	\[
	\lim_{m\to\infty} \ep_m = \frac{1}{ \int_{D} r^2 \d x} \int_{D} r v_{0,\th}(x) \d x = 0.
	\]
	Consequently, by noting that $\|r \bm{e_{\th}}\|_{C^2(D)} = \| (-x_2, x_1, 0) \|_{C^2(D)}$ is finite, (\ref{conv-init}) also holds if $\bm{v}^{(m)}_0$ is replaced by $\tilde{\bm{v}}^{(m)}_0$. For ease of notation, we still denote $\t{\bm{v}}^{(m)}_0$ to be $ \bm{v}^{(m)}_0 $.
	
	Let $C_{*}$ be the constant in Proposition \ref{Prop_OKFest}. Then by requiring 
	\[
	\sup_{x\in D} r |v_{0,\th}| \leq \frac12 C_{*},
	\]
	there exists some integer $m_0\geq 10^3$ such that for any $m\geq m_0$,
	\bn
	\|\bm{v}^{(m)}_0\|_{C^2(\ol{D_m})} &\leq \|\bm{v}_0\|_{C^2(\ol{D})} + 1\,,\\
	\| r v^{(m)}_{0,\th} \|_{L^\infty(D_m)} &\leq C_{*}\,.
	\en
	In the following, we will only consider those $\bm{v}^{(m)}_0$ for $m\geq m_0$. Now we fix any time $T>0$. For each $m$, there exists a strong solution $(\bm{v}^{(m)}, P^{(m)})$ of (\ref{NS}) on $D_m\times [0,T]$ with the initial data $\bm{v}^{(m)}_0$ and the boundary condition (\ref{bdry for Dm}). In addition, $\bm{v}^{(m)}$ is bounded and satisfies 
	\be\label{gamma_m_ave0}
	\int_{D_m} r v^{(m)}_\th(x,t) \d x = 0, \quad \forall\, t\in[0,T].
	\ee
	On the other hand, we can assume 
	\be\label{ave_P_0}
	\int_{D_m} P^{(m)}(x,t) \d x = 0, \quad\forall\, t\in[0,T].
	\ee
	Actually, if we define 
	$ \widetilde{P}^{(m)}(x,t) = P^{(m)}(x,t) - \frac{1}{|D_m|}\int_{D_m}P^{(m)}(x,t)\,dx $, 
	then $ \widetilde{P}^{(m)} $ satisfies (\ref{ave_P_0}) and $ (\bm{v}^{m}, \widetilde{P}^{(m)}) $ is also a strong solution.
	
	Based on the above setup and thanks to Proposition \ref{Prop, inf_vomega} and Proposition \ref{Prop, high_deri_est}, we can follow the idea in (Section 6, \cite{LPYZZZ24}) to extract a subsequence of $(\bm{v}^{(m)}, P^{(m)})$ which converge to a solution $(\bm{v}, P)$ of (\ref{NS1}) on $D\times [0,T]$ such that  $ \bm{v}\in L_{tx}^\infty\cap H_t^1 L_x^2\cap L_t^2 H_x^2(D\times[0,T])$, $ P\in L_t^2 H_x^1(D\times[0,T])$ and 
	\bn
	\|\bm{v}\|_{L_{tx}^\infty(D\times[0,T])} + \|\bm{v}\|_{H_t^1 L_x^2(D\times[0,T])} + \|\bm{v}\|_{L_t^2 H_x^2(D\times[0,T])} + \|P\|_{L_t^2 H_x^1(D\times[0,T])} \leq C\,.
	\en
	Meanwhile, since $\bm{v}^{(m)}$ possesses the the identity (\ref{gamma_m_ave0}) and the energy estimate (\ref{energy_est_Dm}) on $D_m$, we can send $m\to\infty$ to justify (\ref{gamma_int0}) and (\ref{energy_decay}). Hence, the existence part is finished. 
	
	\textbf{Step 2: Uniqueness.}
	
	It remains to verify the uniqueness of the strong solution $ \bm{v} $. Suppose  $ (\t{\bm{v}}, \t{P}) $ is another strong solution of (\ref{NS1}) on $ D\times[0,T] $ under the Navier total-slip boundary condition with the same initial data $ \bm{v}_0 $. We will prove that $ \t{\bm{v}}$ coincides with $ \bm{v} $. 
	Let $ \bm{f} = \bm{v}-\t{\bm{v}}\ed f_\rho\bm{e_\rho}+f_\phi\bm{e_\phi}+f_\th\bm{e_\th}$ and $ g= P-\t{P} $. Then $ \bm{f} $ satisfies
	\be\label{diff-ss} \left\{\, \begin{aligned}
		\Delta \bm{f} - (\bm{f}\cdot \nabla) \bm{v} - (\t{\bm{v}} \cdot\nabla)\bm{f} -  \nabla g - \p_{t} \bm{f} = 0 \quad\text{in}\quad & D\times (0,T], \\
		\nabla \cdot \bm{f} = 0  \quad \text{in} \quad & D\times (0,T], \\
		\bm{f}\cdot \bm{n} = 0,\quad  (\mathbb{S} \bm{f}\cdot \bm{n})_{\tan } = 0 \quad\text{on} \quad & \p D\times (0,T],\\
		\bm{f}(\cdot, 0) = 0 \quad\text{in} \quad & D.
	\end{aligned} \right.\ee
	Since both $\bm{v}$ and $ \t{\bm{v}} $ are strong solutions, we test the first equation in (\ref{diff-ss}) by $\bm{f}$ on $ D\times[0,T_1] $, for any $ 0 < T_1 <T $, to find
	\[\begin{split}
		&\quad \int_0^{T_1}\int_{D} (\Delta\bm{f})\cdot \bm{f} \d x\d t - \int_0^{T_1}\int_{D} [(\bm{f}\cdot \nabla)\bm{v}]\cdot \bm{f} \d x\d t \\
		&= \int_0^{T_1}\int_{D} [(\t{\bm{v}} \cdot\nabla)\bm{f}]\cdot \bm{f} \d x\d t +  \int_0^{T_1}\int_{D} (\nabla g)\cdot \bm{f}\d x\d t +  \int_0^{T_1}\int_{D} (\p_t \bm{f})\cdot \bm{f}\d x\d t\,.
	\end{split} \]
	
	Thanks to the boundary condition and the incompressibility condition of $ \t{\bm{v}} $ and $ \bm{f} $, both $ \int_0^{T_1}\int_{D} [(\t{\bm{v}} \cdot\nabla)\bm{f}]\cdot \bm{f} \d x\d t$ and $\int_0^{T_1}\int_{D} (\nabla g)\cdot\bm{f}\d x\d t$ vanish, so
	\be\label{eq1-diff-ss} 
	\int_0^{T_1}\int_{D} (\Delta \bm{f})\cdot \bm{f} \d x\d t - \int_0^{T_1}\int_{D} [(\bm{f}\cdot \nabla) \bm{v}]\cdot \bm{f} \d x\d t = \frac12 \int_{D} |\bm{f}|^2(x,T_1)\d x\,.  
	\ee
	Since $\bm{f}\cdot \bm{n} = 0$ on $\p D$ and 
	\[
	\int_{D} r f_\th(x,t) \d x = \int_{D} r \big[v_\th(x,t) - \tilde{v}_\th(x,t) \big] \d x = 0, \quad \forall\, t\in[0,T],
	\]
	we can apply Lemma \ref{Lemma, lap_int} (replacing $D_m$ with $D$, and replacing $\eta_m$ with the constant function $1$) to obtain 
	\[
	\int_{D} (\Delta \bm{f}) \cdot \bm{f} \d x = - \int_{D} |\na\bm{b_f}|^2\d x 
	- 2 \int_{D} |\mathbb{S}(f_\th \bm{e_\th})|^2 \d x + \int_{D} \frac{1}{\rho^2} f_\phi^2 \d x + \int_{D} \frac{1}{\rho} (\p_\rho f_\phi^2) \d x,
	\]
	where $\bm{b_f} := f_\rho \bm{e_{\rho}} + f_\phi \bm{e_{\phi}}$.
	Plugging this identity into (\ref{eq1-diff-ss}) leads to 
	\be\label{unique3}\begin{split}
		&\quad \frac12 \int_{D} |\bm{f}|^2(x,T_1)\d x + \int_0^{T_1}\int_D |\na\bm{b_f}|^2\d x \d t + 2\int_0^{T_1}\int_D |\mathbb{S}(f_\th \bm{e_\th})|^2 \d x \d t\\
		&= \int_0^{T_1}\int_{D} \frac{1}{\rho^2} f_\phi^2 \d x \d t + \int_0^{T_1}\int_{D} \frac{1}{\rho} (\p_\rho f_\phi^2) \d x \d t
		- \int_0^{T_1}\int_{D} [(\bm{f}\cdot \nabla)\bm{v}]\cdot \bm{f} \d x\d t\,.
	\end{split}\ee
	
	Similar to Corollary \ref{Cor, Korn_ineq}, we have 
	\be\label{ST_est2}
	\int_0^{T_1}\int_D |\mathbb{S}(f_\th \bm{e_\th})|^2 \d x \d t \geq \frac38 \int_0^{T_1}\int_D |\nabla (f_\th \bm{e_\th})|^2 \d x \d t.
	\ee
	Next, we will control the first two terms on the right hand side of (\ref{unique3}) by
	a small multiple of $\int_0^{T_1}\int_D |\na\bm{b_f}|^2\d x \d t$ which can be absorbed by the left hand side of (\ref{unique3}). Based on formula (\ref{grad_b}), for any fixed $t\in (0, T_1)$,
	\be\label{bf_ldd}
	\int_{D} | \nabla \bm{b_f} |^2 \d x \geq \int_{D} (\p_\rho f_\phi)^2 \d x + \int_{D} (F_1^2 + F_2^2) \d x, 
	\ee
	where 
	\[
	F_1 \ed \frac{1}{\rho}(\p_\phi f_\phi + f_\rho)\,, \quad\text{and} \quad 
	F_2 \ed \frac{1}{\rho}(\p_\phi f_\rho - f_\phi)\,.
	\]
	Meanwhile, for any fixed $t\in(0,T_1)$, it follows from Cauchy-Schwarz inequality that 
	\[\begin{split}
		\int_{D} \frac{1}{\rho^2} f_\phi^2 \d x \d t + \int_{D} \Big|\frac{1}{\rho} (\p_\rho f_\phi^2)\Big| \d x \d t 
		&\leq \la\int_{D} (\p_\rho f_\phi)^2 \d x + \Big(1+\frac{1}{\la}\Big) \int_{D} \frac{1}{\rho^2} f_\phi^2 \d x \\
		& \leq \la\int_{D} (\p_\rho f_\phi)^2 \d x + \frac{3}{25}\Big(1+\frac{1}{\la}\Big) \int_{D} \frac{1}{\rho^2} |\p_\phi f_\phi|^2 \d x,
	\end{split}\]
	where the last inequality is due to $f_\phi=0$ on rays and the Poincar\'e inequality in Lemma \ref{PoinB}. Choosing $\la=\frac12$ yields 
	\[
	\int_{D} \frac{1}{\rho^2} f_\phi^2 \d x \d t + \int_{D} \Big|\frac{1}{\rho} (\p_\rho f_\phi^2)\Big| \d x \d t 
	\leq \f12 \int_{D} (\p_\rho f_\phi)^2 \d x + \frac{9}{25} \int_{D} \frac{1}{\rho^2} |\p_\phi f_\phi|^2 \d x.
	\]
	Then by analogous argument as (\ref{F1F2_est1}), we deduce 
	\[
	\Big\| \frac{1}{\rho} \p_\phi f_\phi \Big\|_{L^2(D)}^2  \leq \frac{10}{7} (\|F_1\|_{L^2(D)}^2 + \|F_2\|_{L^2(D)}^2)\,,
	\]
	Thus, 
	\be\label{bf_est1}\begin{split}
		\int_{D} \frac{1}{\rho^2} f_\phi^2 \d x \d t + \int_{D} \Big|\frac{1}{\rho} (\p_\rho f_\phi^2)\Big| \d x \d t 
		&\leq \f12 \int_{D} (\p_\rho f_\phi)^2 \d x + \frac{18}{35} (\|F_1\|_{L^2(D)}^2 + \|F_2\|_{L^2(D)}^2) \,.
	\end{split}\ee
	Combining (\ref{bf_ldd}) with (\ref{bf_est1}) together gives 
	\be\label{bf_est}
	\int_{D} \frac{1}{\rho^2} f_\phi^2 \d x \d t + \int_{D} \Big|\frac{1}{\rho} (\p_\rho f_\phi^2)\Big| \d x \d t \leq \frac{18}{35} \int_{D} |\nabla \bm{b_f}|^2 \d x.
	\ee
	
	Plugging (\ref{ST_est2}) and (\ref{bf_est}) into (\ref{unique3}) leads to 
	\be\label{unique4}\begin{split}
		&\quad \frac12 \int_{D} |\bm{f}|^2(x,T_1)\d x + \int_0^{T_1}\int_D |\na\bm{b_f}|^2\d x \d t+ \frac34\int_0^{T_1}\int_D |\nabla(f_\th \bm{e_\th})|^2 \d x \d t\\
		&\leq \frac{18}{35} \int_0^{T_1}\int_D |\na\bm{b_f}|^2\d x \d t
		+ \bigg| \int_0^{T_1}\int_{D} [(\bm{f}\cdot \nabla)\bm{v}]\cdot \bm{f} \d x\d t \bigg|\,.
	\end{split}\ee
	Since $|\nabla \bm{f}|^2 = |\na\bm{b_f}|^2 + |\nabla(f_\th \bm{e_\th})|^2$, it follows from (\ref{unique4}) that 
	\be\label{unique1}\begin{split}
		\frac12 \int_{D} |\bm{f}|^2(x,T_1)\d x + \frac{17}{35}\int_0^{T_1}\int_D |\na\bm{f}|^2\d x \d t
		\leq \bigg| \int_0^{T_1}\int_{D} [(\bm{f}\cdot \nabla)\bm{v}]\cdot \bm{f} \d x\d t \bigg|\,.
	\end{split}\ee
	
	By definition, 
	\[
	\int_0^{T_1}\int_{D} [(\bm{f}\cdot \nabla)\bm{v}]\cdot \bm{f} \d x\d t = \sum_{j,k = 1}^3 \int_0^{T_1}\int_{D} [f_j (\p_{x_j} v_k)] f_k \dx\d t\,.
	\]
	Then using integration by parts and taking advantage of the boundary condition and the incompressibility condition of $\bm{f} $, we infer that
	\[ 
	\int_0^{T_1}\int_{D} [(\bm{f}\cdot \nabla)\bm{v}]\cdot\bm{f} \d x\d t = -\sum_{j,k = 1}^3 \int_0^{T_1}\int_{D} f_j v_k (\p_{x_j} f_k ) \d x\d t\,. 
	\]
	Since $\|\bm{v}\|_{L^{\infty}_{tx}(D\times[0,T])}$ is finite, there exists some large constant $C_1$ such that 
	\[
	\bigg|\int_0^{T_1}\int_{D} [(\bm{f}\cdot \nabla)\bm{v}]\cdot \bm{f} \d x\d t \bigg| \leq 
	C_1 \int_0^{T_1}\int_{D} |\bm{f}(x,t)|^2\d x\d t + \frac15\int_0^{T_1}\int_{D} |\na\bm{f}|^2 \d x\d t\,.
	\]
	Putting this estimate into (\ref{unique1}) yields 
	\be\label{unique2}
	\int_{D} |\bm{f}(x,T_1)|^2\d x \leq 2 C_1 \int_0^{T_1}\int_{D} |\bm{f}(x,t)|^2\d x\d t\,, \quad\forall\, 0<T_1\leq T.
	\ee
	
	Finally, since both $\bm{v}$ and $ \t{\bm{v}} $ has the same initial data, $ \bm{f}(x,0)=0 $ on $ D $. As a result, it follows from (\ref{unique2}) and Gr\"onwall's inequality that $\bm{f}=0 $ on $ D\times[0,T] $. So $ \t{\bm{v}} = \bm{v} $ on $ D\times[0,T] $, which justifies the uniqueness of the strong solution $ \bm{v} $. This completes the proof of Theorem \ref{Thm_main}.
\end{proof}

\section{Proof of Theorem \ref{Thm2}}\l{Sec11}

\begin{proof}[Proof of Theorem \ref{Thm2}]
Our construction consists of the following four steps.

\textbf{Step 1. Construction of a solution $\bm{v}^{(1)}$ whose initial azimuthal component $v_{0,\th}$ is supported near the axis of symmetry.}

Since $rv_{0,\th}\to 0$ as $\rho\to 0$, there exists $\rho_0\in(0,125^{-1})$ such that
\[
\sup_{x\in D\cap\{\rho\,:\,\rho\leq\rho_0\}}r|v_{0,\th}|\leq C_*\,,
\]
where $C_*$ is the constant in (\ref{COND}) in Theorem \ref{Thm_main}.
Define a new initial velocity $\bm{v}^{(1)}_0$ as 
\bn
\bm{v}^{(1)}_0 = v_{0,\rho}\bm{e_\rho} + v_{0,\phi}\bm{e_\phi} + \eta_1(\rho) v_{0,\th}\bm{e_\th},
\en
where $\eta_1$ is a smooth cut-off function such that $0\leq\eta_1(\rho)\leq1$, $|\eta_1'(\rho)| \leq \f{4}{\rho_0}$, and
\[
\eta_1(\rho)=\left\{
\begin{aligned}
	&1,\q\text{for}\q &0\leq\rho\leq\frac{\rho_0}{2}\,;\\
	&0,\q\text{for}\q &\rho\geq\rho_0\,.\\
\end{aligned}
\right.
\]
Then both the conditions (i) and (ii) in Theorem \ref{Thm_main} are satisfied for the initial value $\bm{v}^{(1)}_0$, and hence the problem \eqref{NS1} with the initial value $\bm{v}^{(1)}_0$ has a bounded strong solution $(\bm{v}^{(1)}, P^{(1)})$ in $D\times[0,T]$. 

\textbf{Step 2. Construction of a solution $\bm{v}^{(2)}$ whose initial value is supported away from the axis of symmetry.}

Fix 
\be\label{m_choice} 
	m = \frac{8}{\rho_0}.
\ee
Then we consider the problem (\ref{NS}) on the domain $D\times [0,T]$ with the mixed boundary condition (\ref{bdry for Dm}) and the initial data $(1-\eta_1)v_{0,\th}\bm{e_\th}$. Noticing that this initial data is divergence free and satisfies the boundary condition (\ref{bdry for Dm}) since the support of $1-\eta_1$ is away from the inner arc of $D_m$, see Figure \ref{Fig, v2}. 
\begin{figure}[!ht]
	\centering
	\includegraphics[scale=0.22]{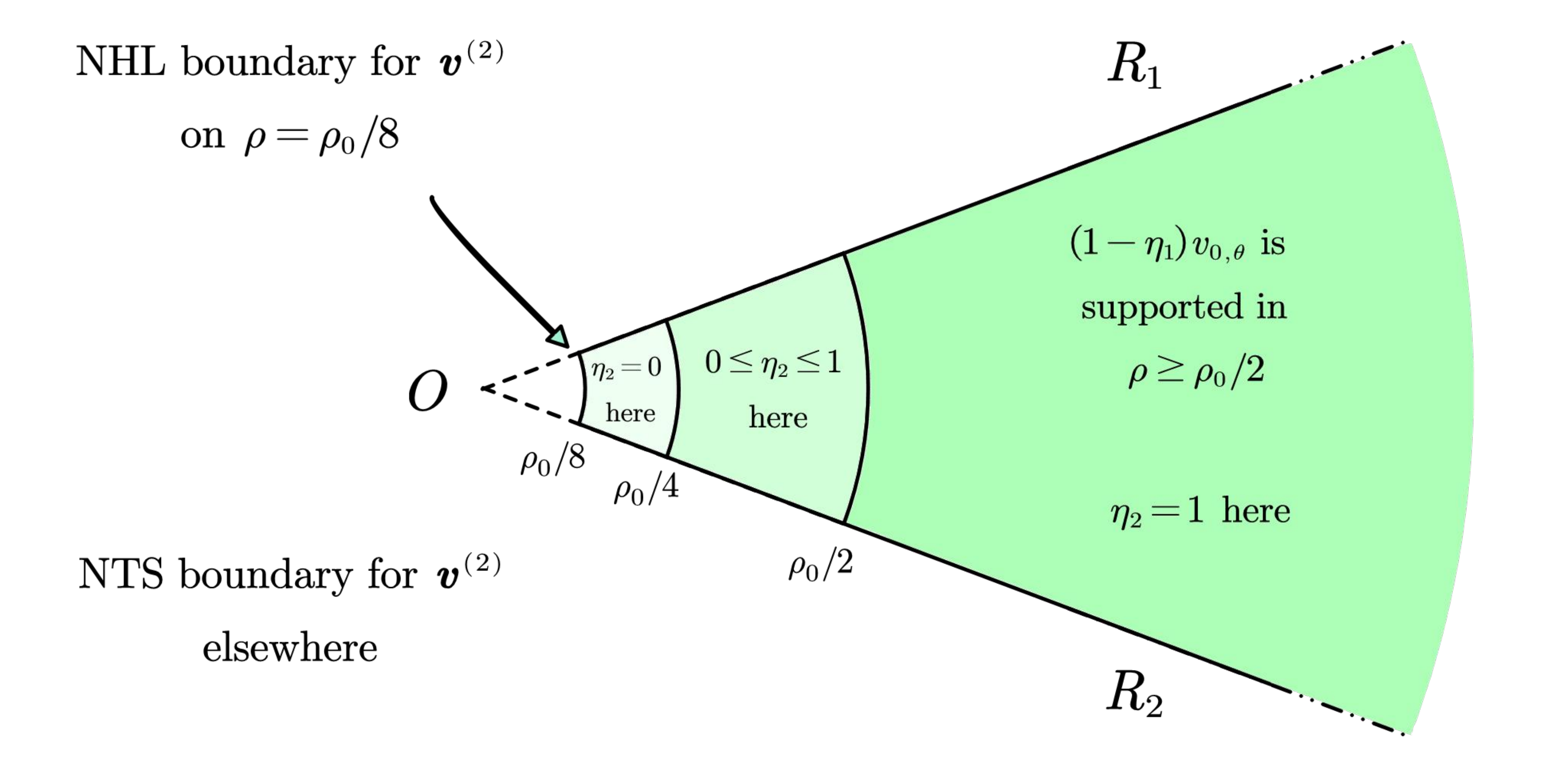}
	\caption{Construction of $\bm{v}^{(2)}$}
	\label{Fig, v2}
\end{figure}
Thanks to Proposition \ref{Prop, local soln in ad}, there exists a unique bounded strong solution $(\bm{v}^{(2)}, P^{(2)})$ of (\ref{NS}) on $D_m\times [0,T]$ with the boundary condition (\ref{bdry for Dm}) such that 
\[
	\bm{v}^{(2)} \in H_t^1 L_x^2 \cap L_t^2 H_x^2\cap L_{tx}^{\infty}\big(D_m\times[0,T]\big), 
	\quad 
	P^{(2)}\in L_t^2 H_x^1(D_m\times[0,T]). 
\]

Then we choose another smooth cut-off function $\eta_2$ such that $0\leq\eta_2(\rho)\leq1$, $|\eta_2'(\rho)| \leq \f{16}{\rho_0}$, and
\[
\eta_2(\rho)=\left\{
\begin{aligned}
	&0,\q\text{for}\q &0\leq\rho\leq\frac{\rho_0}{4}\,;\\
	&1,\q\text{for}\q &\rho\geq \frac{\rho_0}{2}\,.\\
\end{aligned}
\right.
\]
The support of $\eta_2$ is contained in $S_2$:
\[
S_2 \ed \{x\in \ol{D}: \rho\geq \rho_0/4\},
\]
which is away from the $x_3$ axis, and enables one to introduce $\eta_2 \bm{v}^{(2)}$ to extend the domain $D_m$ of $\bm{v}^{(2)}$ to the whole region $D$. In addition, we have
\[
	\eta_2(1-\eta_1)=(1-\eta_1), \quad\forall\, \rho\in[0,1]\,,
\]
which guarantees that the initial values of $\eta_2 \bm{v}^{(2)}$ and $\bm{v}^{(2)}$ match.

Now we combine $\bm{v}^{(1)}$ and $\bm{v}^{(2)}$ together to define $\bm{v}^{(3)}$ as
\[
\bm{v}^{(3)}\ed\bm{v}^{(1)}+\eta_2\bm{v}^{(2)}\,.
\]
Clearly, 
\[
\bm{v}^{(3)}(x,0)=v_{0,\rho}\bm{e_\rho} + v_{0,\phi}\bm{e_\phi} + \eta_1 v_{0,\th}\bm{e_\th} + (1-\eta_1)v_{0,\th}\bm{e_\th} = \bm{v}_0\,.
\]
However, $\bm{v}^{(3)}$ is not divergence-free since
\ba\l{DIV3}
\na\cdot\bm{v}^{(3)}=\na\eta_2\cdot\bm{v}^{(2)}=\eta_2^\pr(\rho)v_\rho^{(2)}(\rho,\phi,t)\,.
\ea

\textbf{Step 3. Divergence correction: the construction of an auxiliary vector field $\bm{X}$.}

To eliminate the divergence part of $\bm{v}^{(3)}$ in \eqref{DIV3}, we introduce
\[
\bm{X}=\left(\f{1}{\rho^2}\int_0^\rho s^2\eta_2^\pr(s)v_\rho^{(2)}(s,\phi,t)\d s\right)\bm{e_\rho}\ed \zeta(\rho,\phi,t)\bm{e_\rho}\,.
\]
Recalling the divergence formula in the spherical coordinates, we find
\bn
\na\cdot\bm{X}=\f{1}{\rho^2}\p_\rho\left(\rho^2\cdot\f{1}{\rho^2}\int_0^\rho s^2\eta_2^\pr(s)v_\rho^{(2)}(s,\phi,t)\d s\right)=\eta_2^\pr(\rho)v_\rho^{(2)}(\rho,\phi,t)\,,
\en
which guarantees the divergence-freeness of $\bm{v}^{(3)}-\bm{X}$ due to (\ref{DIV3}). Noticing that $\zeta(\rho,\phi,0) = 0$ for any $\rho$ and $\phi$ since $v_\rho^{(2)}(s,\phi,0) = 0$ for any $s$ and $\phi$. In addition, $\zeta(\rho,\phi,t)=0$ for $\rho\in(0,\f{\rho_0}{4})$ due to the support of the cut-off function $\eta_2$. Therefore, subtracting $\bm{X}$ does not affect the initial value of $\bm{v}^{(3)}$ or its values in a neighborhood of the axis of symmetry.

Now we consider the boundary condition of $\bm{X}$. Since $\bm{v}^{(2)}$ satisfies the Navier total-slip condition (\ref{NTSR}) on the rays $R_1 \cup R_2$, we have
\[
\p_\phi v_\rho^{(2)}(\rho,\phi,t) = 0 \q \text{on} \q \phi = \frac{\pi}{2}\pm\al\,.
\] 
By interchanging the order of integration, we conclude that $\zeta$ satisfies the same Neumann boundary condition: 	
\be\label{p_phi_m}
\p_\phi \zeta(\rho,\phi,t) = 0 \quad\text{on}\quad \phi = \frac{\pi}{2} \pm \al.
\ee
However, on $\rho=1$,
\[
\zeta(1,\phi,t)=\int_0^1 s^2\eta_2^\pr(s)v_\rho^{(2)}(s,\phi,t)\d s\,,
\]
which is not zero. Therefore, we need to adjust the value of $\bm{v}^{(3)}-\bm{X}$ on the arc boundary $\rho = 1$ in order to satisfy the Navier total-slip boundary condition (\ref{NTSA}) on the arc $\rho = 1$. 

\textbf{Step 4. Boundary correction: the construction of an auxiliary vector $\bm{Y}$.}

We first choose a smooth function $a(\rho)$ which satisfies 
\ba\l{CONDa}
a(\rho) = \left\{
\begin{aligned}
	& 0,  &\q\text{for}\q &0\leq\rho\leq\frac{\rho_0}{2}\,;\\
	& -1, &\q\text{for}\q & \frac12 \leq \rho \leq 1\,.
\end{aligned}
\right.
\ea
Then we define 
\[
\psi(\rho,\phi,t)\ed a(\rho)\int_{\f{\pi}{2}-\al}^\phi \zeta(1, \varphi,t)\sin\varphi \d \varphi
\]
to be the stream function, so that it generates the divergence-free vector field
\[
\bm{Y}=Y_\rho\bm{e_\rho}+Y_\phi\bm{e_\phi}\,,
\]
with
\[
\left\{
\begin{aligned}
	&Y_\rho \ed \f{1}{\rho^2\sin\phi}\p_\phi\psi=\f{a(\rho)}{\rho^2} \zeta(1,\phi,t)\,,\\
	&Y_\phi \ed -\f{1}{\rho\sin\phi}\p_\rho\psi = -\f{a^\pr(\rho)}{\rho\sin\phi}\int_{\f{\pi}{2}-\al}^\phi \zeta(1, \varphi,t)\sin\varphi \d \varphi\,.
\end{aligned}
\right.
\]
It is readily seen that
\[
Y_\rho(1,\phi,t) = -\zeta(1,\phi,t) = -X_\rho(1,\phi,t)\,.
\]
Moreover, direct calculation shows
\begin{itemize}
	\item $\bm{Y}\Big|_{t=0}=0$, since $\zeta(1,\phi,0)=0$ for any $\phi$;\\[2mm]
	
	\item $Y_\phi\Big|_{\phi=\f{\pi}{2}\pm\al}=0$, since $\int_{\f{\pi}{2}-\al}^{\f{\pi}{2}+\al} v_\rho^{(2)}(\rho,\varphi,t)\sin\varphi \d \varphi=0$ thanks to Lemma \ref{v_rho_mean0};\\[2mm]
	
	\item $\p_\phi Y_\rho\Big|_{\phi=\f{\pi}{2}\pm\al}=0$ due to (\ref{p_phi_m});\\[2mm]
	
	\item $\bm{Y}=0$ for $\rho\in[0,\f{\rho_0}{2}]$, and $\p_\rho Y_\phi=\f{1}{\rho}Y_\phi=0$ on $\rho=1$, which are direct conclusions from \eqref{CONDa}.
\end{itemize}

Finally, by denoting 
\[
\t{\bm{v}} \ed \bm{v}^{(3)}-\bm{X}-\bm{Y} = \bm{v}^{(1)}+\eta_2\bm{v}^{(2)}-\bm{X}-\bm{Y}\,,
\]
then 
\be\label{vtilde_eq1}\begin{split}
	\Delta \t{\bm{v}} - (\t{\bm{v}} \cdot \nabla) \t{\bm{v}}  - \p_t \t{\bm{v}}  =& \, \Delta (\bm{v}^{(1)}+\eta_2\bm{v}^{(2)}-\bm{X}-\bm{Y}) \\
	& - \big[ (\bm{v}^{(1)}+\eta_2\bm{v}^{(2)}-\bm{X}-\bm{Y})\cdot \nabla \big] (\bm{v}^{(1)}+\eta_2\bm{v}^{(2)}-\bm{X}-\bm{Y}) \\
	& - \p_t(\bm{v}^{(1)}+\eta_2\bm{v}^{(2)}-\bm{X}-\bm{Y}) = \bm{F}_1 + \bm{F}_2 + \bm{F}_3,
\end{split}\ee
where 
\[\begin{split}
	\bm{F}_1 &= \Delta(\bm{v}^{(1)} + \eta_2\bm{v}^{(2)}) - [(\bm{v}^{(1)} + \eta_2\bm{v}^{(2)})\cdot \nabla ](\bm{v}^{(1)} + \eta_2\bm{v}^{(2)}) - \p_t(\bm{v}^{(1)} + \eta_2\bm{v}^{(2)}), \\
	\bm{F}_2 &= -\Delta(\bm{X}+\bm{Y}) - [(\bm{X} + \bm{Y})\cdot \nabla ](\bm{X} + \bm{Y}) + \p_t(\bm{X} + \bm{Y}), \\
	\bm{F}_3 &= [(\bm{X} + \bm{Y})\cdot \nabla](\bm{v}^{(1)} + \eta_2 \bm{v}^{(2)}) + [( \bm{v}^{(1)} + \eta_2 \bm{v}^{(2)}) \cdot\nabla](\bm{X}+\bm{Y}).
\end{split}\]
Since both $\bm{X}$ and $\bm{Y}$ vanish when $\rho\leq \frac{\rho_0}{4}$, we know $\bm{F}_2 = \bm{F}_3 = 0$ when $\rho\in [ 0,\f{\rho_0}{4}]$. Meanwhile, we can decompose $\bm{F}_1$ as 
$\bm{F}_1 = \bm{F}_{11} + \bm{F}_{12} + \bm{F}_{13}$, where 
\[\begin{split}
	\bm{F}_{11} &= \Delta \bm{v}^{(1)} - \bm{v}^{(1)} \cdot \nabla \bm{v}^{(1)} - \p_t \bm{v}^{(1)}, \\
	\bm{F}_{12} &= \Delta (\eta_2 \bm{v}^{(2)}) - (\eta_2\bm{v}^{(2)} \cdot \nabla)(\eta_2 \bm{v}^{(2)}) - \p_t (\eta_2 \bm{v}^{(2)}),\\
	\bm{F}_{13} &= - (\bm{v}^{(1)}\cdot \nabla)(\eta_2 \bm{v}^{(2)}) - (\eta_2 \bm{v}^{(2)} \cdot \nabla)\bm{v}^{(1)}.
\end{split}\]
Since $(\bm{v}^{(1)}, P^{(1)})$ is a solution to the problem (\ref{NS1}), then $\bm{F}_{11} = \nabla P^{(1)}$. In addition, using spherical coordinates, we find that  
\[\begin{split}
	\Delta (\eta_2 \bm{v}^{(2)}) &= \eta_2 \Delta \bm{v}^{(2)} + 2\eta_2' \p_\rho \bm{v}^{(2)} + \Big( \eta_2'' + \frac{2}{\rho}\eta_2' \Big) \bm{v}^{(2)}, \\
	(\eta_2\bm{v}^{(2)} \cdot \nabla)(\eta_2 \bm{v}^{(2)}) &= \eta_2^2 (\bm{v}^{(2)} \cdot \nabla)\bm{v}^{(2)} + \eta_2 \eta_2' v^{(2)}_\rho \bm{v}^{(2)}.
\end{split}\]
As a result, 
\[\begin{split}
	\bm{F}_{12} &= \eta_2 (\Delta \bm{v}^{(2)} - \p_t \bm{v}^{(2)}) + 2\eta_2' \p_\rho \bm{v}^{(2)} + \Big( \eta_2'' + \frac{2}{\rho}\eta_2' + \eta_2 \eta_2' v^{(2)}_\rho\Big) \bm{v}^{(2)} - \eta_2^2 (\bm{v}^{(2)} \cdot \nabla)\bm{v}^{(2)}.
\end{split}\]
Since $\Delta \bm{v}^{(2)} - \p_t \bm{v}^{(2)} = (\bm{v}^{(2)} \cdot \nabla)\bm{v}^{(2)} + \nabla P^{(2)}$, we obtain 
$\bm{F}_{12} = \nabla (\eta_2 P^{(2)}) + \wt{\bm{F}}_{12}$, 
where 
\[
	\wt{\bm{F}}_{12} = 2\eta_2' \p_\rho \bm{v}^{(2)} + \Big( \eta_2'' + \frac{2}{\rho}\eta_2' + \eta_2 \eta_2' v^{(2)}_\rho\Big) \bm{v}^{(2)} + (\eta_2 - \eta_2^2) (\bm{v}^{(2)} \cdot \nabla)\bm{v}^{(2)} - \eta_2' P^{(2)} \bm{e_\rho}.
\]
Combining all the above computations, we deduce 
\[
	\bm{F}_1 = \nabla \big( P^{(1)} + \eta_2 P^{(2)} \big) + \wt{\bm{F}}_{12} + \bm{F}_{13},
\]
where both $\wt{\bm{F}}_{12}$ and $\bm{F}_{13}$ vanish when $\rho\leq \frac{\rho_0}{4}$ due to the support of $\eta_2$. 
Define 
\be\label{F_decomp}
	\bm{F} = \bm{F}_2 + \bm{F}_{3} + \wt{\bm{F}}_{12} + \bm{F}_{13}, \q \text{and} \q \t{P} = P^{(1)} + \eta_2 P^{(2)}.
\ee
Then 
\be\label{vtilde_eq2}
\Delta \t{\bm{v}} - (\t{\bm{v}} \cdot \nabla) \t{\bm{v}}  - \p_t \t{\bm{v}} = \nabla \t{P} + \bm{F},
\ee
which justifies \eqref{PP1} with 
\[
	\bm{F}=0\q\text{when}\q \rho \in \big[ 0,\f{\rho_0}{4} \big)\,.
\]

Since $\bm{v}^{(1)}$ is a bounded strong solution on the whole region $D$ and $\bm{v}^{(2)}$ is a bounded strong solution on $D_m$ with the fixed $m$ given in (\ref{m_choice}) such that the support of $\bm{F}$ is contained in $D_m$, one can easily check that all the terms in $\bm{F}$ in (\ref{F_decomp}) lie in $L^2_{tx}(D\times [0,T])$.
This completes the proof of Theorem \ref{Thm2}. 
\end{proof}

In the end, we remark that it is not clear whether the forcing term $\bm{F}$ constructed above is bounded in $D \times [0,T]$ due to complicated boundary conditions. In contrast, later in the proof for Theorem \ref{Thm3}, the forcing term $\bm{F}$ by a similar construction process can be proven to live in $L^{\infty}_{tx}(\mR^3\times [0,T])$.

\section{Proof of Theorem \ref{Thm3}}\l{Sec, reg_control_R3}

\begin{proof}
	The proof of Theorem \ref{Thm3} follows the similar strategy as that for Theorem \ref{Thm2}.
	
	\textbf{Step 1. Construction of a solution $\bm{v}^{(1)}$ whose initial azimuthal component $v_{0,\th}$ is supported near the axis of symmetry.}
	
	Firstly, we recall a result in \cite[Theorem 1.4]{LZ17}: there exists an absolute positive constant $\delta_* > 0$ such that if 
	\be\label{LZ_small_id}
	\|\Gamma_0\|_{L^\infty} \|\Gamma_0\|_{L^2} \Big( \|\Omega_0\|_{L^2} + \|V_0^2\|_{L^2} \Big) \leq \delta_*,
	\ee
	then the ASNS (\ref{eqasns}) on $\mR^3$ is globally well-posed, where 
	\[
	\Gamma_0 = r v_{0,\th}, \qquad \Omega_0 = \frac{\o_{0,\th}}{r}, \qquad V_0 = \frac{v_{0,\th}}{r}.
	\]
	Since $\bm{v_0}\in H^2(\mR^3)\cap C^2(\mR^3)$, we know $v_{0,\th} \in L^\infty(\mR^3)$ and all of  $\Gamma_0, \Omega_0, V_0^2$ belong to $L^2(\mR^3)$. On the other hand, since $v_{0,\th} \in L^\infty(\mR^3)$, there exists $r_0 > 0$ such that
	\be\label{small_G_cr}
	\sup_{ r \leq r_0 } r|v_{0,\th}|\leq \frac{\delta_*}{B_0} \,,
	\ee
	where 
	\[
	B_0 \ed \|\Gamma_0\|_{L^2} \big( \|\Omega_0\|_{L^2} + \|V_0^2\|_{L^2} \big).
	\]
	Based on the given initial velocity $\bm{v_0} \ed v_{0,r}\bm{e_r} + v_{0,x_3} \bm{e_{3}} + v_{0,\th}\bm{e_\th}$, we consider the following truncated initial condition $\bm{v}^{(1)}_{0}$:
	\ba\l{id1_cr}
	\bm{v}^{(1)}_0 = v_{0,r}\bm{e_r} + v_{0,x_3} \bm{e_{3}} + \eta_1(r) v_{0,\th}\bm{e_\th},
	\ea
	where $\eta_1$ is a smooth cut-off function such that $0\leq\eta_1(r)\leq 1$, $|\eta_1'(r)| \leq \f{4}{r_0}$, and
	\bn
	\eta_1(r) = \left\{
	\begin{aligned}
		&1,\q\text{for}\q & 0 \leq r \leq \frac{r_0}{2}\,;\\
		&0,\q\text{for}\q & r \geq r_0\,.\\
	\end{aligned}
	\right.
	\en
	Noticing 
	\[
	B_0^{(1)} \ed  \|\Gamma^{(1)}_0\|_{L^2} \Big( \|\Omega^{(1)}_0\|_{L^2} + \|(V^{(1)}_0)^2\|_{L^2} \Big) \leq B_0,
	\]
	so (\ref{LZ_small_id}) holds for the initial velocity $\bm{v}^{(1)}_{0}$ due to (\ref{small_G_cr}). Hence, the ASNS (\ref{eqasns}) with the initial value $\bm{v}_{0}^{(1)}$ in (\ref{id1_cr}) has a global strong solution $(\bm{v}^{(1)}, P^{(1)})$ in $\mR^3\times(0,\infty)$. 
	
	\textbf{Step 2. Construction of a solution $\bm{v}^{(2)}$ whose initial value is supported away from the axis of symmetry.}
	
	Let $\bm{v}^{(2)}$ be the suitable Leray-Hopf solution (i.e. the pressure term belongs to $L^{5/3}_{tx}(\mR^3\times[0,T])$) to the following problem
	\be\l{EV2-3D}\left\{\, \begin{aligned}
		&\Delta {\bm{v}}^{(2)} - ({\bm{v}}^{(2)}\cdot \nabla) {\bm{v}}^{(2)} - \nabla {P}^{(2)} - \p_{t} {\bm{v}}^{(2)} = 0 \quad\text{in}\quad  \mR^3\times (0,T], \\
		&\nabla \cdot {\bm{v}}^{(2)} = 0  \quad \text{in} \quad  \mR^3\times (0,T], \\
		&{\bm{v}}^{(2)}(\cdot, 0) = (1-\eta_1)v_{0,\th}\bm{e_\th} \quad\text{in} \quad  \mR^3.
	\end{aligned} \right.\ee
	Now we choose another smooth cut-off function $\eta_2$ such that $0\leq\eta_2(r)\leq1$, $|\eta_2'(r)| \leq \f{8}{r_0}$, and
	\bn
	\eta_2(r)=\left\{
	\begin{aligned}
		&0,\q\text{for}\q &0\leq r \leq\frac{r_0}{4}\,;\\
		&1,\q\text{for}\q & r \geq \frac{r_0}{2}\,.\\
	\end{aligned}
	\right.
	\en
	As a result, 
	\[
	\eta_2(1-\eta_1)=(1-\eta_1), \quad\forall\, r \in[0,1]\,.
	\]
	Now we introduce a new vector field $\bm{v}^{(3)}$ as
	\[
	\bm{v}^{(3)}\ed\bm{v}^{(1)}+\eta_2\bm{v}^{(2)}\,.
	\]
	Clearly, from \eqref{id1_cr} and \eqref{EV2-3D}$_3$, we have
	\[
	\bm{v}^{(3)}(x,0) = v_{0,r}\bm{e_r} + v_{0,x_3}\bm{e_{3}} + \eta_1 v_{0,\th} \bm{e_\th} + (1-\eta_1) v_{0,\th} \bm{e_\th} = \bm{v}_0\,.
	\]
	However, $\bm{v}^{(3)}$ is not divergence-free since
	\ba\l{DIV4}
	\na\cdot\bm{v}^{(3)}=\na\eta_2\cdot\bm{v}^{(2)} = \eta_2^\pr(r) v_r^{(2)}(r,x_3,t)\,.
	\ea
	
	\textbf{Step 3. Divergence correction: the construction of an auxiliary vector $\bm{X}$.}
	
	To eliminate the divergence part of $\bm{v}^{(3)}$ in \eqref{DIV4}, we introduce
	\[
	\bm{X} \ed -\left(\f{\eta_2^\pr(r)}{r}\int_0^rsv_{x_3}^{(2)}(s,z,t)\d s\right) \bm{e_{3}}\,.
	\]
	Noticing that $\bm{v}^{(2)}$ is divergence-free, so 
	\[
	\frac{1}{s} \p_{s}\big(sv_r^{(2)}(s,x_3,t)\big) + \p_{x_3}v_{x_3}^{(2)}(s,x_3,t) = 0,
	\]
	which implies that 
	\bn
	\na\cdot\bm{X} = -\f{\eta_2^\pr(r)}{r}\int_0^rs\p_{x_3}v_{x_3}^{(2)}(s,x_3,t)\d s =\f{\eta_2^\pr(r)}{r}\int_0^r\p_s\big(sv_r^{(2)}(s,x_3,t)\big)\d s =\eta_2^\pr(r) v_r^{(2)}(r,x_3,t)\,.
	\en
	Thus, $\bm{v}^{(3)}-\bm{X}$ is divergence-free. Noticing that
	\begin{itemize}
		\item $\bm{X}(r,x_3,0)=0$ since $v_{x_3}^{(2)}(r,x_3,0) \equiv 0$; 
		
		\item $\bm{X}(r,x_3,t) = 0$ for any $r \in (0, \f{r_0}{4})$, due to the support of the cut-off function $\eta_2$.
	\end{itemize}
	Therefore, subtracting $\bm{X}$ does not affect the initial value of $\bm{v}^{(3)}$ or its values in a neighborhood of the axis of symmetry.
	
	Finally, by denoting 
	\[
	\t{\bm{v}}\ed\bm{v}^{(3)} - \bm{X} = \bm{v}^{(1)} + \eta_2\bm{v}^{(2)} - \bm{X}\,,
	\]
	and performing similar computations as that in (\ref{vtilde_eq1})--(\ref{vtilde_eq2}), we find that $\t{\bm{v}}$ satisfies \eqref{PP2}, where 
	\[
	\bm{F}=0\q\text{when}\q r \in \big[ 0,\f{r_0}{4} \big)\,.
	\]
	Since the support of $\bm{F}$ is away from the $x_3$ axis, the integrability $L^{\infty}_{t}L^{2}_{x} \cap L^{\infty}_{tx}\big(\mR^3 \times [0,T]\big)$ of $\bm{F}$ follows from standard theory. This completes the proof of Theorem \ref{Thm3}.
	
\end{proof}

\section{Proof of Corollary \ref{Cor, unstable-bus}}\l{Sec12}

According to the setup in Corollary \ref{Cor, unstable-bus}, there exists a strong solution $\bm{v}$ to the problem (\ref{NS1}) which blows up in finite time with an initial value $\bm{v}_0 \in \mathscr{A}$ such that $\int_{D} r v_{0,\th} \d x = 0$. Writing 
\[\bm{v}_0 := v_{0,\th}\bm{e_\th} + v_{0,r}\bm{e_r} + v_{0,x_3}\bm{e_{3}},\]
then we define a class of initial velocities as below:
\[
\bm{v}^{\la}_{0} = \lambda v_{0,\th}\bm{e_\th} + v_{0,r}\bm{e_r} + v_{0,x_3}\bm{e_{3}}, \quad \forall\, \la \geq 0.
\]
Now we denote $S$ to be the set of $\la$ for which the problem (\ref{NS1}) with the initial data $\bm{v}_{0}^{\la}$ has a global strong solution. 
Thanks to Theorem \ref{Thm_main}, the problem (\ref{NS1}) with the initial data $\bm{v}_{0}^{\la}$ possesses a global strong solution for any sufficiently small $\la$, so $S$ contains an interval $[0, A_*]$ for some $A_*>0$. Define 
\[
\la_{*} = \sup \{B \geq 0: [0,B] \subseteq S\}.
\]
Since the solution with the initial data $\bm{v}_0$ blows up in finite time, then we know 
\[0 < A_{*} \leq \la_{*} \leq 1.\]
Now we consider the initial velocity $\bm{v}^{*}_0$ defined as
\[
\bm{v}^{*}_{0} = \la_{*} v_{0,\th}\bm{e_\th} + v_{0,r}\bm{e_r} + v_{0,x_3}\bm{e_{3}}.
\]

If $\la_{*}\in S$, then the strong solution $v^{*}$ to (\ref{NS1}) with the initial data $v^{*}_{0}$ exists globally. As a result, the energy inequality (\ref{energy_decay}) holds for any $T>0$. This energy decay guarantees that $v^{*}$ is globally bounded, which allows a small perturbation of the initial value $v^{*}_{0}$ so that the corresponding strong solution also exists globally. This implies that $\la_{*}+\ep \in S$ for any sufficiently small $\ep>0$, which contradicts to the definition of $\la_*$.
Consequently, we conclude that $\la_{*} \notin S$ and $v^{*}$ blows up in finite time.
Meanwhile, $\la_{*} - \ep \in S$ for any $\ep \in (0, \la_{*})$, and
\[
\| \bm{v^*}_{0} - \bm{v}_{0}^{\la_*-\ep} \|_{C^2(\ol{D})} = \ep \|v_{0,\th}\|_{C^2(\ol{D})},
\]
so by defining 
\[\bm{\wt{v}}_0 = \bm{v}_{0}^{\la_*-\ep},\] 
where $\ep \leq \delta/\|v_{0,\th}\|_{C^2(\ol{D})}$, 
then the strong solution $\bm{\wt{v}}$ to the problem (\ref{NS1}) with the initial data $\bm{\wt{v}}_0$ exists globally and 
\[ \|\bm{\wt{v}}_0 - \bm{v}^{*}_{0}\|_{C^2(\ol{D})} \leq \delta.\] 

As a summary, we obtain a strong solution $\bm{v}^{*}$ to the problem (\ref{NS1}) which blows up in finite time. Furthermore, $\bm{v}^{*}$ is unstable in the sense that for any $\delta>0$, we can find a strong solution $\bm{\wt{v}}$ to the problem (\ref{NS1}) which exists globally, and the $C^{2}(\ol{D})$ norm of the difference between the initial values of $\bm{v}^{*}$ and $\bm{\wt{v}}$ is bounded by $\delta$.

\qed


\section*{Acknowledgments}
We wish to thank Professors Hui Chen, Zhen Lei, Yanlin Liu, Xinghong Pan, Lu Yang, Ping Zhang, Na Zhao and Daoguo Zhou, Dr. Chulan Zeng, Mr. Zili Chen and Mr. Zhipeng Wu for helpful discussions.
X. Yang is supported by National Natural Science Foundation of China (No. 12401299), Natural Science Foundation of Jiangsu Province (No. BK20241260), Scientific Research Center of Applied Mathematics of Jiangsu Province (No. BK20233002). Q. S. Zhang is grateful to the support of the Simons Foundation through grant No. 710364.


\bigskip


\end{document}